\documentclass[a4paper,11pt,twoside]{article} 

\usepackage[utf8]{inputenc}
\usepackage[T1]{fontenc}
\usepackage{indentfirst}
\usepackage{fancyhdr}
\usepackage{graphicx}
\usepackage{physics}
\usepackage{color}
\usepackage[scaled=0.875]{helvet}

\usepackage[english]{babel}
\usepackage{emptypage}
\usepackage{amsmath}
\usepackage{amssymb}
\usepackage{amsfonts}
\usepackage{amsthm, amsbsy}
\usepackage{latexsym}
\usepackage{verbatim}
\usepackage{bm}
\usepackage{tikz}
\usetikzlibrary{shapes.geometric}
\usetikzlibrary{calc,matrix,arrows,}
\usepackage{amsmath}
\usepackage{tkz-euclide}
\usepackage{tikz}
\usepackage{verbatim}
\usepackage{ifthen}
\usepackage{empheq}
\usepackage[italian]{varioref}
\usepackage{algorithm}
\usepackage{algorithmic}
\usepackage{subfig}
\usepackage{tkz-euclide}
\usepackage{adjustbox}

\usetikzlibrary{math}
\usetikzlibrary{calc,matrix,arrows,}
\usetikzlibrary{petri}
\usetikzlibrary{snakes}
\usetikzlibrary{backgrounds}
\usetikzlibrary{positioning}
\usetikzlibrary{calc,matrix}
\usetikzlibrary{shapes}
\usetikzlibrary{arrows,decorations.pathmorphing,fit}

\usepackage{empheq}
\usepackage{import}
\usepackage{rotating}
\usepackage{lscape}
\usepackage[toc,page]{appendix}
\usepackage{caption}

\newtheorem{theorem}{Theorem}
\newtheorem{remark}{Remark}
\newtheorem{proposition}[theorem]{Proposition}

\usepackage{caption}
\usepackage{rotating}
\usepackage{lscape}
\usepackage{geometry}
\allowdisplaybreaks

\newcommand{\Vj}{{\mathbf{j}}}
\newcommand{\Vk}{{\mathbf{k}}}

\newcommand{\VF}{{\mathbf{F}}}

\newcommand{\VU}{{\mathbf{U}}}
\newcommand{\VV}{{\mathbf{V}}}
\newcommand{\VW}{{\mathbf{W}}}
\newcommand{\VX}{{\mathbf{X}}}

\providecommand{\CC}{{\cal C}}
\providecommand{\CD}{{\cal D}}

\providecommand{\CF}{{\cal F}}

\providecommand{\CK}{{\cal K}}

\providecommand{\CN}{{\cal N}}
\providecommand{\CO}{{\cal O}}

\providecommand{\bbE}{\mathbb{E}}

\providecommand{\bbN}{\mathbb{N}}

\providecommand{\bbP}{\mathbb{P}}

\providecommand{\bbR}{\mathbb{R}}

\providecommand{\bbV}{\mathbb{V}}

\newcommand{\Dx}{\ensuremath{\Delta x}}
\newcommand{\Dy}{\ensuremath{\Delta y}}

\newcommand{\dx}{\mathrm{d}x}

\newcommand{\dt}{\mathrm{d}t}

\usepackage{xfrac}

\newcommand{\figsize}{0.4}
\newcommand{\main}{./}

\allowdisplaybreaks

\title{A Monte-Carlo ab-initio algorithm for the multiscale simulation of compressible multiphase flows}
\author{M. Petrella$^1$ \and R. Abgrall$^2$ \and S. Mishra$^1$}
\date{
    $^1$Seminar of Applied Mathematics, ETH Zurich, Switzerland.\\%
    $^2$Department of Mathematics, University of Zurich, Switzerland.\\ %
\today
}
\date{\footnotesize{\today}}

\begin{document}
\maketitle
\begin{abstract}
We propose a novel Monte-Carlo based \emph{ab-initio} algorithm for directly computing the statistics for quantities of interest in an immiscible two-phase compressible flow. Our algorithm samples the underlying probability space and evolves these samples with a sharp interface front-tracking scheme. Consequently, statistical information is generated without resorting to any closure assumptions and information about the underlying microstructure is implicitly included. The proposed algorithm is tested on a suite of numerical experiments and we observe that the ab-initio procedure can simulate a variety of flow regimes robustly and converges with respect of refinement of number of samples as well as number of bubbles per volume. The results are also compared with a state-of-the-art discrete equation method to reveal the inherent limitations of existing macroscopic models.   
\end{abstract}

\section{Introduction}
\emph{Multiphase flows} arise in a wide variety of physical phenomena ranging from bubble dynamics and shock wave interactions with material discontinuities to  detonation of high energetic materials, hypervelocity impacts, cavitating flows and combustion systems \cite{Ishii, Drew&Passman}. Given their importance in applications, the design of a suitable mathematical framework to describe multiphase flows and efficient numerical methods to simulate them is imperative. 

A key observation regarding multiphase flows is the significant amount of uncertainty in the exact locations of particular constituents at any given time. 
Hence, the description of multiphase flow phenomena need to be provided in terms of \emph{statistical} quantities of interest for the flow.  However, derivation of such a suitable statistical description is highly non-trivial. To contextualize the problems, we provide a summary of the underlying mathematical framework, for instance from \cite{Drew&Passman}, here.  

Non-mixing two-phase flow in one space dimension consists of two phases, each of which is assumed to occupy a time-dependent, phase-wise disjoint domain. To be more precise, denote $D\subset\mathbb{R}$ a domain and $T>0$ a time horizon. For any time $t \in [0,T]$ and \emph{realization} $\omega \in \Omega$ (a probability space),  we define the time-dependent (random) domain occupied by phase $k \in \lbrace 1,2\rbrace$ as $D_k(t,\omega) = \lbrace x\in D\, \vert\, \textit{phase } k \textit{ is present at } (x,t) \textit{ in realization }\omega\rbrace$. The phase domains need to verify,
\begin{subequations}
\label{ansatz}
\begin{align}
\label{eq:SaturationCondition}
\textit{(Saturation Condition)}
\qquad
\overline{
D^{(1)}(t,\omega)
\cup
D^{(2)}(t,\omega)
} \setminus \partial D
&= D,\qquad t\in [0,T]\\
\label{eq:NonMixingCondition}
\textit{(Non-mixing Condition)}
\qquad
D^{(1)}(t,\omega)\cap D^{(2)}(t,\omega) &= \emptyset,\qquad t\in [0,T]
\end{align}
\end{subequations}
where $\partial D$ denotes the boundary of $D$, and the interface between two phases is defined as $I(t,\omega) := \overline{
D^{(1)}(t,\omega)}
\cap 
\overline{
D^{(2)}(t,\omega)
} \setminus \partial D $. Within its underlying domain, each phase is assumed to be governed by the (single phase) Euler equations of gas dynamics,
\begin{equation}
\label{eq:Euler_k}
\partial_t \VU^{(k)} + \partial_x \VF^{(k)} = 0 \qquad \textit{ in } D_k(t,\omega)\qquad \forall\, k\in \lbrace 1, 2 \rbrace.
\end{equation}
where 
\begin{equation}
\VU^{(k)} = 
\begin{bmatrix}
\rho^{(k)}\\
\rho^{(k)}u^{(k)}\\
\rho^{(k)}E^{(k)}
\end{bmatrix},
\qquad\quad
\VF^{(k)} = 
\begin{bmatrix}
\rho^{(k)}u^{(k)}\\
\rho^{(k)}{u^{(k)}}^2 + p^{(k)}\\
u^{(k)}\left(\rho^{(k)}E^{(k)} + p^{(k)}\right)
\end{bmatrix}.
\end{equation}
Here $\rho^{(k)}, u^{(k)}, p^{(k)}, E^{(k)}$ denote the density, velocity, pressure and total energy of the phase $\Sigma_k$. The total energy is defined in terms of the velocity field and of the (very generic form of ) internal energy $e^{(k)}$ via
\[
E^{(k)} = \frac{1}{2}{u^{(k)}}^2 + e^{(k)}
\]

Each phase can be identified by its \emph{characteristic function}, a random field given by,
\begin{equation}
\label{eq:CF}
X^{(k)}(x,t;\omega) = \begin{cases}
1 & \textit{if}\quad x\in D_k(t,\omega)\\
0 & \textit{otherwise}.
\end{cases}
\end{equation}
By \eqref{ansatz}, one then deduces that
\begin{equation}
\label{eq:SaturationCondition}
X^{(1)}(x,t;\omega)+X^{(2)}(x,t;\omega) = 1, 
\qquad
\, for\, a.e.\, (x,t)\in D\times (0,T],\,\,
\forall\,\omega\in\Omega.
\end{equation}
Therefore, the statistical description of two-phase flow reduces to the following PDEs,
\begin{equation}
\label{eq:mean-Euler-k}
\mathbb{E}\Bigg[ X^{(k)}\left( \partial_t \VU^{(k)} + \partial_x \VF^{(k)}\right) \Bigg ] = 0, \qquad 
\forall\, (x,t)\in D\times(0,T],\,
\forall\, k \in \lbrace1,2\rbrace
\end{equation}
under the consistency condition (\ref{eq:SaturationCondition}). Here, $\mathbb{E}$ denotes the \emph{expectation} (statistical mean) with respect to the underlying probability space $(\Omega,\mathcal{F},\mathbb{P})$.

As the Equation \eqref{eq:mean-Euler-k} is not tractable in general, further assumptions need to be made. Roughly speaking, one follows \cite{Drew&Passman} by assuming that the expectation and differential operators \emph{commute} and uses Reynolds' transport theorem to obtain, 
\begin{equation}
\label{eq:Etwo-phase}
\begin{cases}
\partial_t \mathbb{E}\left[X^{(k)}\tilde{\VU}^{(k)}\right] + \partial_x\mathbb{E}\left[X^{(k)} \tilde{\VF}^{(k)}\right] = \mathbb{E}\left[\left(\tilde{\VF}_I^{(k)} - \sigma \tilde{\VU}_I^{(k)}\right)\partial_x X^{(k)}\right]\\
\mathbb{E}\left[X^{(1)}\right] + \mathbb{E}\left[X^{(2)}\right] = 1
\end{cases}
\end{equation}
where $\sigma$ denotes the speed of the interface, $\tilde{\VU}^{(k)} := [1,\VU^{(k)}]$, $\tilde{\VF}^{(k)} = [0,\VF^{(k)}]$ and the subindex $I$ stands for evaluation of corresponding quantities at the interface.

In analogy with turbulence modelling, one then tries to rewrite both sides of \eqref{eq:Etwo-phase} in terms of the so-called \emph{Favre-averaged variables}
\begin{equation}
\label{eq:Favre-variables}
\alpha^{(k)} := \mathbb{E}\left[X^{(k)}\right],\quad
\overline{\rho}^{(k)} := \frac{\mathbb{E}\left[X^{(k)}\rho^{(k)}\right]}{\alpha^{(k)}},\quad
\overline{u}^{(k)} := \frac{\mathbb{E}\left[X^{(k)}\rho^{(k)}u^{(k)}\right]}{\alpha^{(k)}\overline{\rho}^{(k)}}, \quad
\textit{etc},
\end{equation} 
hoping to derive a self-consistent set of PDEs. 

This derivation clearly highlights the problems emanating from this procedure as one cannot \emph{close} the equations, purely in terms of the Favre-averaged variables. Rather, \emph{closure} relations have to be specified in order to make the resulting PDEs self-consistent. 

The derivation of such closure models has been the topic of intensive research in the last years with \cite{BaerNunziato,Bdzil,Bo,Gallouet,Kapila,Murrone,Saurel&Abgrall,Stewart} providing a very selected list of references. However, all these models have intrinsic drawbacks such as the appearance of non-conservative products, possibly negative volume fractions and lack of monotonicity with respect to mixture sound speed, among others. See \cite{Saurel2009} and reference therein for an extensive review of the criticisms related to these models. 
In general, it seems unlikely that a universal closure model will be found as information about phenomena at finer scales is missing in such models and will always lead to models that might fit observed data in one flow regime while significantly deviating in others. 

A different, yet related, approach was proposed in \cite{Abgrall&Saurel}. 
With the aim of recovering information lost in averaging over fine scales, the authors proposed a \emph{discrete equation method} (DEM), where underlying numerical schemes were used to simulate the flow and the resulting flows were averaged to recover statistical information. 
This approach belongs to the schemes of the so-called diffuse interface type \cite{Saurel2009} (i.e. one allows each interface to be smeared over several computational cells, as a result of numerical viscosity) and consists of seven equations in one-space dimension. 
Such a methodology has been shown to be able to deal with complex flow regimes, dynamical creation of interfaces as well as strong pressure differences \cite{Saurel2018}. 

Given this context, it is conceivable that taking the limit (with respect to mesh resolution) of the discrete equation method will yield PDEs that provide a suitable mathematical description of multiphase flow. 
The closure relations will implicitly arise from the underlying microstructure in this approach. 
In fact, this procedure was shown to converge to some well-known reduced models for multiphase \cite{Saurel2003, Petrella2022} in the limit of stiff mechanical relaxation, see also  \cite{Saurel2018}.

Unfortunately, such discrete equation models also have considerable drawbacks. In particular and as described in a recent paper \cite{Petrella2022} that provided a comprehensive analysis of the discrete equation method, this approach is \emph{under-determined} and gives rise to an infinite family of possible solutions. These solutions are further characterized in terms of parameters, one modeling the underlying probability coefficients and another the effects of stiff mechanical relaxation. Consequently, \cite{Petrella2022} demonstrated that the discrete equation method is incomplete without further information about the underlying micro-structure.  

It is clear from the above discussion that current modeling strategies for multiphase flow have reached a possible impasse and alternative approaches are needed to supply the missing information about microstructure. Presenting such an approach is the key goal of the current paper. 

Our starting points are the equations \eqref{eq:Etwo-phase} that describe two-phase flow at a \emph{microscopic level}. Instead of commuting expectation and differential operators to derive a macroscropic model in terms of the Favre-averaged variables, we will directly simulate the solutions of \eqref{eq:Etwo-phase}. In order of perform such \emph{ab-initio} simulations, we will require the following key ingredients, 
\begin{itemize}
    \item A Monte-Carlo type sampling and ensembling averaging procedure is used to approximate the expectation operator in \eqref{eq:Etwo-phase}. In particular, an ensemble of flow realizations are generated from the specified initial and boundary conditions, propagated in time with a suitable numerical method and averaged to extract relevant statistical quantities of interest. 
    \item A \emph{front-tracking} algorithm is employed as the numerical solution operator for time-propagation of the ensemble. Front tracking provides a \emph{sharp interface} method that circumvents the issues arising from numerical viscosity in smearing interfaces. 
\end{itemize}
Combining these ingredients into a \emph{novel} ab-initio algorithm for simulating two-phase flows, we will explore various flow configurations to verify the robustness of our procedure and to discover interesting facets of multiphase phenomena. In particular, these ab-initio simulations at the \emph{microscopic scale} will be compared vis a vis macroscopic simulations such as with the discrete equation method in order to glean out the limitations of macroscopic modeling in this context. 

This paper is organized as follows: we first detail the idea of the ab-initio method in Section \ref{sec:alg}, followed by a discussion of each building block. Sections \ref{sec:MSR} introduces a regime-generating strategy, taken as an initial condition for the application of the numerical evolution operator discussed in Section \ref{sec:FT_res}.
The Monte-Carlo ensembling is introduced in Section \ref{sec:MC}. The algorithm is then exemplified on a suite of numerical tests presented in Section \ref{sec:AI:MC:NE}, followed by a discussion of what was observed in Section \ref{sec:disc}. Conclusions are drawn in Section \ref{sec:concl}.

\section{An algorithm to compute two-phase flow solutions}
\label{sec:alg}

Let $(\Omega,\CF, \bbP)$ be a probability space.
Given the above discussion, it results clear that there is a one-to-one correspondence between the sets $D_k(t;\omega)$ and their characteristic functions $X^{(k)}(x,t;\omega)$.
Therefore, we will term the pair $\lbrace X^{(k)}(x,t;\omega) \rbrace_{k=1,2}$ for $\omega\in\Omega$ a \emph{random two-phase distribution}, and will make no distinction between $X^{(k)}(x,t;\omega)$ and $D_k(t;\omega)$.\\

Given a random two-phase distribution $X^{(k)}_0(\cdot;\omega)$, we shall show that
for sufficiently small initial data $\VU^{(k)}_0\in\CD^{(k)}$, there exist a time-parametrized two-phase distribution $D_k(t;\omega)$ such that the problems (\ref{eq:Euler_k}) for any $k=1,2$, admit a random solution in the weak sense.
The notion of random weak entropy solution for systems of hyperbolic conservation laws was firstly defined in \cite{Abgrall2017}.
Essentially such definition requires the underlying deterministic systems to be well-posed, so that random weak entropy solutions are defined by the path-wise equivalent of the deterministic ones.
Such a notion has as been studied for scalar conservation laws \cite{Mishra10} and for linear systems of hyperbolic conservation laws \cite{Sukys2014}, whereas for systems of non-linear hyperbolic systems the situation is more delicate.
Indeed, due to the limited well-posedness in the deterministic case of such systems \cite{Baiti01}, defining a notion for weak entropy solutions is limited to the small-BV bound \cite{Dafermos}.\\

Moreover, unfortunately, the notion of a random weak solution does not apply to systems (\ref{eq:Euler_k}), since the set of equations does not extend to the phasic boundary (i.e. the interface). 
Indeed, across it, solutions are not defined, and supplementary relations need to be added.
To this extent, we assume that the EOS associated to different phases admit a unique parametrization, in the sense that both phasic EOS can be written in a unique form, which defines a thermodynamically consistent EOS \cite{Menikoff}.\\
An example of joint parametrization for the usual EOS of common use is provided by the (shifted) Noble-Abel Stiffened Gas (NASG) EOS \cite{SaurelNASG}
\begin{equation}
\label{eq:NASG}
e^{(k)} = \frac{p^{(k)} + \gamma^{(k)} \pi^{(k)}}{(\gamma^{(k)} - 1)} \left(\frac{1}{\rho^{(k)}}-b^{(k)}\right)
\end{equation}
The sound speed associated to (\ref{eq:NASG}) reads \cite{SaurelNASG}
\begin{equation}
\label{eq:NASG_SoundSpeed}
{a^{(k)}}^2 = \gamma^{(k)}\frac{p^{(k)} + \pi^{(k)}}{(1-b^{(k)}\rho^{(k)})\rho^{(k)}}.
\end{equation}
To show that (\ref{eq:NASG}) can be considered as a parametrization, one first need to acknowledge that is provides a thermodynamical consistent EOS \cite{SaurelNASG}. 
Furthermore, let us denote by $\Vk^{(k)}:=[\gamma^{(k)}, \pi^{(k)}, b^{(k)}]$ the set of parameters involved in (\ref{eq:NASG}).
Then (\ref{eq:NASG}) can be reduced to model the ideal-gas EOS (IG-EOS) \cite[Chapter $1.2.4$]{ToroRS} by setting $\pi=b=0$, the co-volume EOS (CV-EOS) \cite[Chapter $1.2.5$]{ToroRS} by setting $\pi=0$ and the Stiffened Gas Equation of State (SG-EOS) \cite{Chen71}  upon setting $b=0$.\\

By exploiting conservation of mass, one can recast the time-constant behavior of the parameters $\Vk$ of the joint-parametrization via the (trivial) conservation laws
\[
\partial_t \left(\rho \Vk\right) + \partial_x \left( \rho u \Vk \right) = 0.
\]
In turn, for given initial conditions $\VU^{(k)}_0(x)$ and initial random two-phase distributions $X^{(k)}(x,t;\omega)$, one can argue for $\bbP$-a.e. $\omega\in\Omega$ and identify a solution of (\ref{eq:Euler_k}) as a solution of the discontinuous-flux system
\begin{equation}
\label{eq:Euler_dis_path}
\begin{cases}
\partial_t \VU(x,t;\omega) + \partial_x \VF\left(\VU(x,t;\omega)\right) = 0
& (x,t)\in D \times (0,T)\\
\VU(x,0;\omega) = 
\begin{cases}
\begin{bmatrix}
\VU^{(1)}_0(x)\\
\rho^{(1)}_0(x)\Vk^{(1)}
\end{bmatrix}
& x\in D_1(0;\omega)\\
\begin{bmatrix}
\VU^{(2)}_0(x)\\
\rho^{(2)}_0(x)\Vk^{(2)}
\end{bmatrix}
& x\in D_2(0;\omega)
\end{cases}
\end{cases}
\qquad
\forall
\omega\in\Omega
\end{equation}
where $\VU = [\rho,\rho u, \rho E,\rho\Vk]$ and $\VF = [\rho u, \rho u^2 + p, u(\rho E + p),\rho u\Vk]$ and the EOS in the total energy $E$ is given by the joint-parametrization.\\
Notice that, due to the conservation law form in (\ref{eq:Euler_dis_path}), Rankine-Hugoniot relations imply
\[
\VF^{(1)}(\VU^{(1)}(x(t),t;\omega)) - s\VU^{(1)}(x(t),t;\omega)
=
\VF^{(2)}(\VU^{(2)}(x(t),t;\omega)) - s\VU^{(2)}(x(t),t;\omega)
\]
where $s=s(t;\omega)$ denotes the (Lax-admissible) interface speed. 

If a solution to (\ref{eq:Euler_dis_path}) in the weak sense exists, then the random two-phase distribution $X^{(k)}(x,t;\omega)$ can also be equivalently written
\[
X^{(k)}(x,t;\omega)
= 
\begin{cases}
1 & \Vk(x,t;\omega) = \Vk^{(k)}\\
0 & \textit{otherwise}
\end{cases}
\]
and the functions $\VU^{(k)} = X^{(k)}\VU$ verify (\ref{eq:Euler_k}) in the weak sense.

In complete analogy to the standard practice for defining random weak entropy solutions, one would like to establish well-posedness for the deterministic version of (\ref{eq:Euler_dis_path}). Unfortunately, such system is not strictly hyperbolic, in general, and classical results granting well-posedness do not apply \cite{Dafermos,Risebro}.\\
Recently, by exploiting the Front-Tracking (FT) approach, we were able \cite{Petrella22FT} to show that a unique, Lax-admissible weak solution for the deterministic equivalent of (\ref{eq:Euler_dis_path}) exists and that it is unique, under the small-BV assumption.\\

Such a procedure, construct for $\bbP$-a.e. $\omega\in\Omega$ a solution via the FT-algorithm and all the stability properties are inherited immediately $\bbP$-a.s. \cite{Mishra10}.
Hence, the aforementioned description provide a road map for approximating two-phase flow solutions: under the assumption of knowing a random two-phase distribution, one aims at computing statistics of the random field $\VU(x,t;\omega)$.
Specifically, the ab-initio method performs MC-based approximations on the stochastic dimension of the corresponding random field. 
Loosely speaking, the strategy is to adopt a regime-generating procedure to produce the initial two-phase distributions, evolve them via the (numerical) Front-Tracking (FT) method discussed in \cite{Petrella22FT} and make ensemble averaging of such approximate solutions.
For the sake of clarity we will discuss each of such step, starting with the generation of a random two-phase distribution $D_k(0;\omega)$ $k=1,2$.

Since the strategy applies at the numerical level, we consider the computational domain $D = [a,b]$ with $a,b\in\bbR$ and discretize it into $M\in\bbN$ cells (control-volumes), namely
\begin{equation}
\label{eq:AI:FV_discretization}
D = \bigcup_{i=1}^{M} \CC_i, \qquad \CC_i :=[x_{i-\frac{1}{2}},x_{i+\frac{1}{2}}], \qquad x_{i+\frac{1}{2}} = x_i + \frac{\Delta x}{2}, \qquad \forall i=1,\ldots, M.
\end{equation}
where $x_i = a + i\Delta x$ and the mesh width $\Delta x = \frac{\vert b-a\vert }{M}$.

In a similar fashion, we introduce a time-mesh by letting $t^n = n\Delta t$, where $\Delta t>0$ is the time step, whose precise meaning will be specified later.\\
For any realization $\omega_l\in\Omega$, let $q^{(k)}(x_i,t^n;\omega_l) \in (L^\infty\cap BV)\left(D_k(t^n)\right)$ denote the value of the quantity $q$ associated to phase $k$ of the two-phase flow at the spacial location $x =x_i$ and time $t=t^n$.

The (piece-wise constant) Finite Volume (FV) approximation of $q^{(k)}$ over $\CC_i$ then reads
\begin{equation}
\label{eq:AI:FV-approximation}
I_i^{(k)} [q^{(k)}](t^n;\omega_l) := \frac{1}{\Delta x} \int_{x_{i-\frac{1}{2}}}^{x_{i+\frac{1}{2}}} X^{(k)}(x,t^n;\omega_l) q^{(k)}(x,t^n;\omega_l)\,\mathrm{d}x \approx q^{(k)}(x_i,t^n;\omega_l)
\end{equation} 
where the integration is performed via a suitable quadrature rules. 

Notice that if $q^{(k)}(x,t^n;\omega_l) = q$ for any $x\in \mathcal{C}_i$ (i.e. $q^{(k)}$ is constant), then $I^{(k)}_i$ agrees with the space-integral average of $q^{(k)}$ if and only if the characteristic function is of one value throughout $\mathcal{C}_i$.
Such a consistency requirement is mainly the reason why Favre-like averages become very useful in this context.

Given an initial condition $\VU_0^{(k)}$ associated to phase $k$, one computes the corresponding FT approximation over the mesh (\ref{eq:AI:FV_discretization}) by (an approximation of)
\[
\alpha_i^{(k),0}\VU^{(k),0}_i = \frac{1}{\Delta x} \int_{x_{i-\frac{1}{2}}}^{x_{i+\frac{1}{2}}} \bbE\left[X^{(k)}(x,0)\VU^{(k)}_0(x)\right]\,\dx
\]
thus resulting in a pair of data inside each volume $\CC_i$.
Notice that, as it is usual, the initial information does not make any approximation at the stochastic level.

\section{Micro-scale Generation}
\label{sec:MSR}
\noindent
In order to compute Monte-Carlo statistics, we need to design a strategy to generate independent and identically distributed realizations starting from the given initial condition. 
This latter is typically provided in average, so that the aim of this section is to design a volume-wise sub-discretization generating a distribution of phases whose mean coincides with the given initial datum. 

For a given volume $C_i$ in the computational domain with corresponding
volume fractions $\alpha_i^{(k),0}$ with $k\in\lbrace 1,2\rbrace$, the flow realization indexed by $\omega\in\Omega$ is generated by performing the Algorithm \ref{al:AI:MGP}.

\begin{algorithm}
\caption{(Volume-based) Micro-scale Generation algorithm}\label{al:AI:MGP}
\begin{algorithmic}
\STATE $\textbf{Data\,:\,}$ 
Volume $\CC_i$ with width $\Delta$, volume fractions $\alpha_i^{(k),0}$ $k=1,2$, corresponding vectors of conserved variables $\VU^{(k),0}_i$, a number of subcells $N^{(i)}$;\\
\STATE $\textbf{Output\,:\,}$ Two-phase distribution $X^{(k)}_0(x;\omega)$ and corresponding initial condition $\VU_{0,i}(x)$;

\STATE $\bullet$ Subdivide $\CC_i$ into $N^{(i)}$ (sub-)cells, i.e. generate a (local) sub-mesh $\Big\lbrace \xi_j^{(i)}\Big\rbrace_{j=0}^{N^{(i)}}$ such that
\[
x_{i-\frac{1}{2}}=\xi_0<\xi_1<\ldots<\xi_{N^{(i)}-1}<\xi_{N^{(i)}}=x_{i+\frac{1}{2}}
\]
\STATE $\bullet$ Generate a sample $\VX(\omega) = (X_1(\omega),\ldots,X_{N^{(i)}}(\omega)) \in \lbrace 1,2\rbrace^{N^{(i)}}$ as to ensure that the (path-wise) two-phase distribution
\[
X^{(k)}(x;\omega) := \begin{cases}
1 & x\in [\xi_j,\xi_{j+1})\textit{ and } X_j = k\textit{  for some } j=0,\ldots,N^{(i)}-1\\
0 & \textit{otherwise}
\end{cases}
\]
verifies
\begin{equation}
\label{eq:AI:ConsistencyCond-initialLength}
\lim_{
\Delta x\rightarrow 0
}
\mathbb{E}
\Bigg [
\sum_{j=0}^{N^{(i)}-1} X^{(k)}\left( \frac{\xi^{(i)}_{j+1}+\xi^{(i)}_j}{2} ,0;\omega_l\right)\frac{\xi^{(i)}_{j+1}-\xi^{(i)}_j}{\Delta} 
\Bigg ]
= \alpha_i^{(k),0}
\end{equation}
\STATE $\bullet$ Define $\VU_{0,i}(x) = X^{(1)}_0(x;\omega)\VU^{(1),0}_{i} + X^{(2)}_0(x;\omega)\VU^{(2),0}_{i}$.
\end{algorithmic}
\end{algorithm}
For each $\omega\in\Omega$, after performing Alg.\ref{al:AI:MGP} the space-integral average (\ref{eq:AI:FV-approximation}) over $\CC_i$ results in 
\begin{equation}
\label{eq:AI:Estimator-VF-single}
I_i^{(k)}\left[X^{(k)}\right](0;\omega) = \sum_{j=0}^{N-1} X^{(k)}\left( \frac{\xi_{j+1}+\xi_j}{2} ,0;\omega\right)\frac{\xi_{j+1}-\xi_j}{\Delta x}.
\end{equation}
which demonstrates that the condition (\ref{eq:AI:ConsistencyCond-initialLength}) is a necessary, consistency requirement imposed by the (space integral average of the) initial condition.

Notice that Alg.\ref{al:AI:MGP} has at least two degrees of freedom: the size of each subcell and how to generate the realization $X(\omega)$, which are user-defined hyperparameters that may vary from case to case.\\ 
In our simulations we used a simple approximation, namely an equi-spaced sub-discretization, whose corresponding algorithm is provided in Alg.\ref{al:AI:eMGP}.
\begin{algorithm}
\caption{Equispaced (volume-based) micro-scale Generation algorithm}\label{al:AI:eMGP}
\begin{algorithmic}
\STATE $\textbf{Data\,:\,}$ 
Volume $\CC_i$ with width $\Delta$, volume fractions $\alpha_i^{(k),0}$ $k=1,2$, corresponding vectors of conserved variables $\VU^{(k),0}_i$, a number of subcells $N^{(i)}$;\\
\STATE $\textbf{Output\,:\,}$ Two-phase distribution $X^{(k)}_0(x;\omega)$ and corresponding initial condition $\VU_{0,i}(x)$;

\FOR{$i=1,\ldots, M$}
\STATE $\bullet$ Subdivide $\CC_i$ into  into $N^{(i)}$ equi-spaced subcells, i.e. 
\[
x_{i-\frac{1}{2}}=\xi_0^{(i)}<\xi_1^{(i)}<\ldots<\xi_{N^{(i)}-1}^{(i)}<\xi_{N^{(i)}}^{(i)}=x_{i+\frac{1}{2}}
\] 
where $\xi_j^{(i)} = x_{i-\frac{1}{2}} + j\frac{\Delta x}{N^{(i)}}$ for any $j\in\lbrace 0, \ldots, N^{(i)}\rbrace$.
\STATE $\bullet$ Compute the number $N_{k} := \lceil\alpha_i^{(k),0}N^{(i)}\rfloor$ of subcells affected by phase $k$, 
where $\lceil\cdot\rfloor$ denotes the rounding to the nearest integer function;
\STATE $\bullet$ Generate $X(\omega) = [X_1(\omega), \ldots, X_{N^{(i)}}(\omega)]\in \lbrace 1,2\rbrace^{N^{(i)}}$ by repeatedly sampling the discrete random variable
$Y\sim \mathrm{Unif}( [0,N^{(i)}]\cap\bbN )$ up to the point of generating indexes $\Vj^{(k)} := \left[j_1^{(k)},\ldots, j^{(k)}_{N_k}\right]$ such that 
\[
X_{j^{(k)}_s}(\omega) = k,\qquad j_{s}^{(k)}\in\lbrace 1, \ldots, N^{(i)} \rbrace
\]
and $j^{(p)}_s \neq j^{(q)}_r$, $\forall s\neq r \in\lbrace 1, \ldots, N_k\rbrace$, $\forall p, q \in\lbrace 1, 2\rbrace$, see Fig.\ref{Fig:Assign-Merge};
\STATE $\bullet$ Merge adjacent subcells that contain the same indexes;
\STATE $\bullet$ Define the initial condition $\VU_0$ and the corresponding two-phase distribution like in Alg.\ref{al:AI:MGP}.
\ENDFOR
\end{algorithmic}
\end{algorithm}
\begin{figure}[h!]
	\begin{center}
		\begin{tikzpicture}[scale=0.8]
		
		\pgfmathsetmacro{\Dx}{7};
		\pgfmathsetmacro{\hei}{3};
		\pgfmathsetmacro{\top}{1};
		\pgfmathsetmacro{\sp}{1};		
		
		\pgfmathsetmacro{\xi}{0};
		\pgfmathsetmacro{\xL}{\xi -\Dx};
		\pgfmathsetmacro{\xLL}{\xL - \sp};
		\pgfmathsetmacro{\xR}{\xi + \Dx};
		\pgfmathsetmacro{\xRR}{\xR + \sp};
		
		\pgfmathsetmacro{\yL}{0};
		
		\tkzDefPoint(\xLL, \yL){xim};
		\tkzLabelPoint[below, xshift = 0.3cm](xim){$x_{i-\frac{1}{2}}$};
		\tkzDefPoint(\xRR, \yL){xip};
		\tkzLabelPoint[below, xshift = -0.3cm](xip){$x_{i+\frac{1}{2}}$};
		\tkzDefPoint(\xi, \yL){xi};
		\tkzLabelPoint[below](xi){$x_{i}$};
		
		\draw [thick] (xim) -- (xip)
					  (\xL, \yL -\top) -- (\xL, \yL+ \hei + \top)
					  (\xR, \yL -\top) -- (\xR, \yL + \hei + \top);		
				
		\pgfmathsetmacro{\N}{12};
		\pgfmathsetmacro{\D}{2*\Dx/\N};
				
		\foreach[count=\i] \ind in {1,1,2,1,2,2,1,2,2,2,1,2}
		{	
			\ifthenelse{\ind = 1}{
				\draw [fill=blue!20] (\xL + \i*\D - \D,\yL) rectangle (\xL + \i*\D,\yL+\hei);
				\node [color=blue!50!black] at (\xL + \i*\D -0.5*\D,\yL+0.5*\hei){$U^{(1)}_i$};
			}{
				\draw [fill=green!20] (\xL + \i*\D - \D,\yL) rectangle (\xL + \i*\D,\hei);
				\node [color=green!50!black] at (\xL + \i*\D -0.5*\D,\yL+0.5*\hei){$U^{(2)}_i$};
			};			
		}		

	\pgfmathsetmacro{\dsp}{1};
	
	\draw [thick, -stealth] (\xi,\yL-\top) -- (\xi, \yL-\top -\dsp); 	
	
	\pgfmathsetmacro{\yL}{\yL - 2*\top -\dsp -\hei};

		\foreach \Dn/\ind [evaluate=\Dn as \ss using \Dn + \s,
		remember=\ss as \s (initially 0)] in {2.3/1,1.2/2,1.2/1,2.3/2,1.2/1,3.4/2,1.2/1,1.2/2}
		{	
			\ifthenelse{\ind=1}{
				\draw [fill=blue!20] (\xL + \s,\yL) rectangle (\xL+\ss,\yL+\hei);	
				\node [color=blue!50!black] at (\xL + \s + 0.5*\Dn,\yL+0.5*\hei){$U^{(1)}_i$};
			}{
				\draw [fill=green!20] (\xL + \s,\yL) rectangle (\xL+\ss,\yL + \hei);
				\node [color=green!50!black] at (\xL + \s + 0.5*\Dn,\yL + 0.5*\hei){$U^{(2)}_i$};
			};			
		}

		\tkzDefPoint(\xLL, \yL){xim};
		\tkzLabelPoint[below, xshift = 0.3cm](xim){$x_{i-\frac{1}{2}}$};
		\tkzDefPoint(\xRR, \yL){xip};
		\tkzLabelPoint[below, xshift = -0.3cm](xip){$x_{i+\frac{1}{2}}$};
		\tkzDefPoint(\xi, \yL){xi};
		\tkzLabelPoint[below](xi){$x_{i}$};
		
		\draw [thick] (xim) -- (xip)
					  (\xL, \yL -\top) -- (\xL, \yL+ \hei + \top)
					  (\xR, \yL -\top) -- (\xR, \yL + \hei + \top);
		
		\end{tikzpicture}
	\end{center}
\caption{Schematic representation of an equi-spaced subdivision, random assignment of states $\VU^{(k),0}_i$ and merge. 
In the depicted example the total number of subcells is $N^{(i)} = 12$, with $N_1 = 5$, $N_2 = 7$. 
Randomly generated indexes are $\textbf{j}^{(1)}= \left[1,2,4,7,11\right]$ and $\textbf{j}^{(2)}=\left[3,5,6,8,9,10,12\right]$.}\label{Fig:Assign-Merge}
\end{figure}
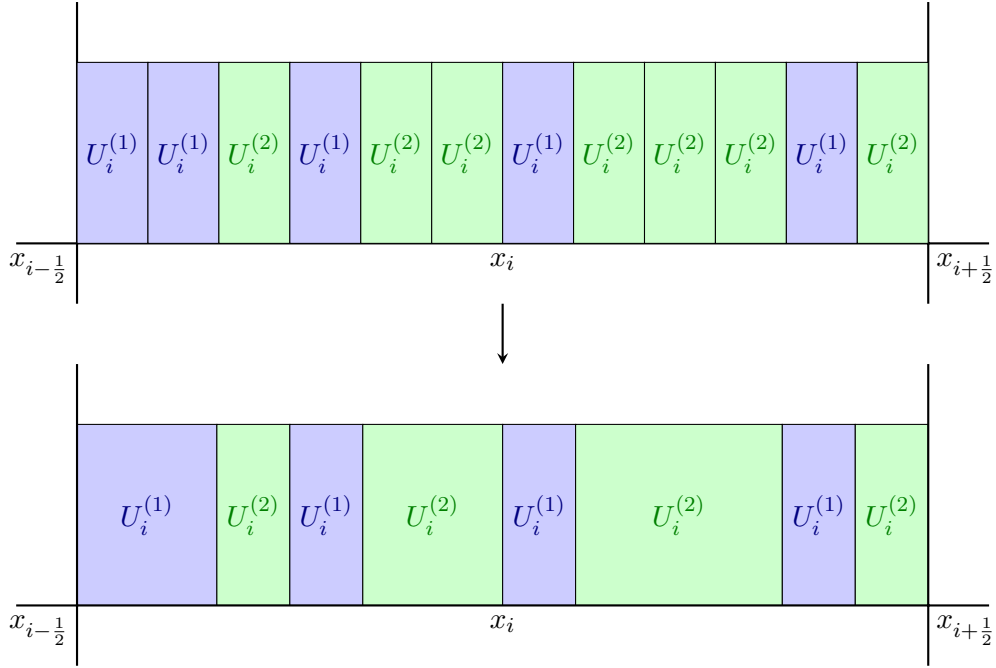

The following remarks are in order.
\begin{enumerate}
\item By performing an equispaced sub-discretization, the value $\Delta = \frac{\Delta x}{N^{(i)}}$ corresponds to choosing $\Delta$ as the finest resolution of dispersed matter inside $\CC_i$, initially;
\item Alg. \ref{al:AI:eMGP} is well-posed: notice that, due to the Law of Large Numbers, one is always able to produce the (distinct) vectors $\Vj^{(k)}$ for $k=1,2$ for sufficiently large sampling of the variable $Y$. This ensures that the algorithm comes to end, or equivalently, rules out the possibility of looping endlessly in the iterative process;
\item By (\ref{eq:AI:FV-approximation}) and (\ref{eq:AI:Estimator-VF-single}), one computes
\begin{equation}
\label{eq:AI:rounding}
I_i^{(k)}\left[X^{(k)}\right](0,\omega) = \frac{N_{k}}{N^{(i)}}
=
\alpha^{(k),0}_i \pm \frac{\delta_i^{(k)}}{N^{(i)}}
\end{equation}
for some rounding-off error $\delta_i^{(k)}\in [0,1)$.
Hence, the sub-discretization generated by Alg. \ref{al:AI:eMGP} verifies (\ref{eq:AI:ConsistencyCond-initialLength}) in the limit of $N^{(i)}\rightarrow \infty$. This can be understood as the Alg.\ref{al:AI:eMGP} to be only asymptotically unbiased, in general.
\item Notice that, due to the freedom in choosing the sub-discretization in Alg. \ref{al:AI:MGP}, the number of generated two-phase distributions is infinite.
The underlying idea that motivated us to consider Alg.\ref{al:AI:eMGP} is the underlying idea typical in Finite Element Methods (FEM): in order to solver a problem on a infinite dimensional set, we localize the problem on a sub-space of finite dimension, and aim at computing the solution on the former set via the limit of the latter.
Indeed, algorithm \ref{al:AI:MGP} defines all possible two-phase distributions with positive finite length of dispersed matter. For any of such realization, we aim at reproducing the same realization with a sufficiently large $N^{(i)}$ in Alg. \ref{al:AI:eMGP}. In particular, one is defining solutions with possibly infinitely many interfaces as the limit of processes having a finite number of interfaces. 
\end{enumerate}

\subsection{Choosing the number of sub-volumes}
\label{sec:AI:MG:volumes}

Alg. \ref{al:AI:eMGP} constructs two-phase distributions that are also in the range of the more general (meta-)algorithm Alg.\ref{al:AI:MGP}.
The former, allows for arbitrary choices of $N^{(i)}$ and in this section we would like to discuss the implication of choosing such hyper-parameter.
Two natural choices are possible : either to fix the number of sub-volumes for each computational cell $\CC_i$ (\emph{uniform} case) or to select it randomly (\emph{random} case) inside each volume.\\ 
In order to appreciate the difference we discuss the probability space both strategies are sampling from. 
Let us first consider the uniform case, where the number of sub-volumes $N^{(i)} = N$ for each volume $\CC_i$ with $i=1,\ldots, M$.
Inside $\mathcal{C}_i$, the (vector-valued) random variable $X = [X_1,\ldots,X_{N}]$ defined in the algorithm takes values in the discrete set 
\begin{equation}
\label{eq:vol_prob_space_uniform}
\begin{split}
\Omega^{(i)}_{N} &= 
\Bigg\lbrace 
\VX^{(i)} = \left(X^{(i)}_1, \ldots, X^{(i)}_{N}\right) \in \lbrace 1,2\rbrace^{N}\, \Big\vert \, 
\forall k\in\lbrace 1,2\rbrace,
\,
\#\lbrace j\in [0,N]\cap\bbN\,|\, X_j^{(i)} = k \rbrace
= N_k
\Bigg\rbrace.
\end{split}
\end{equation}
Notice that the inclusion $\Omega_{N}^{(i)} \subset\lbrace 1,2\rbrace$ is proper: depending on the volume fraction $\alpha^{(k),0}_i$, the element $\VX = (1,\ldots, 1) \in \lbrace 1,2\rbrace^{N} \setminus \Omega_{N}^{(i)}$. 
In particular, the space $\Omega^{(i)}_{N}$ consists of all the possible combinations (with permutation) of $N_k$ times $k$ in $N$ places; it then follows that $\vert \Omega_{N}^{(i)} \vert = \binom{N_k + N - 1}{N_k} = \binom{N_k + N - 1}{N -1}$.
In such case, the global probability space associated to the uniform sub-discretization using $N$ cells reads
\begin{equation}
\label{eq:AI:sampling_prob_space}
\Omega_{N} := \bigoplus_{i=0}^M \Omega_{N}^{(i)} := 
\Big\lbrace
\bm{\omega}^{(N)} := (\VX^{(1)}, \ldots, \VX^{(M)}) \in \Omega^{(1)}_{N}\times \ldots \times \Omega^{(M)}_{N}
\Big\rbrace
\subset
\lbrace 1,2\rbrace^{M\cdot N}
\end{equation}
The dimension of $\Omega_{N}$ is readily given by $|\Omega_{N}| = \vert\Omega_{N}^{(1)}\vert^M$. 

Conversely, a much bigger space is necessary when dealing with a random choice of the number of subcells: this may typically be provided by fixing an upper bound $N$ on the number of sub-volumes and then choosing $N^{(i)} \in [1,N]\cap\bbN$.
If so, inside the cell $\CC_i$, the random variable $\VX$ takes values in

\begin{equation}
\label{eq:vol-sampling-prob-space}
\Omega^{(i)}_{\leq N} := \bigcup_{ 1\leq N^{(i)}\leq N } \Omega_{N^{(i}}^{(i)},
\end{equation}
where $\Omega^{(i)}_{N^{(i)}}$ is defined in (\ref{eq:vol_prob_space_uniform}). 
The corresponding global probability space is given by $\Omega_{\leq N} = \bigoplus_{i=0}^M \Omega_{\leq N}^{(i)}$.

By monotonicity of the counting measure, it follows that, for fixed $N$, the dimension of $\Omega_{\leq N}$ is bigger than $\Omega_{N}$, showing that selecting the number of sub-volumes randomly needs to integrate over a much bigger space than the one associated to the uniform sub-discretization.
Since MC averages can also be interpreted as integration techniques with a fixed volume over the domain of interest, one then concludes that the uniform sub-discretization would offer a faster converging integration technique as opposed to the random one, when using the same number of samples.

We conclude this section by stressing that over a computational mesh of size $M\in\bbN$, where each volume is further sub-divided using $N\in\bbN$ volumes, there exists finitely many possible two-phase distribution over $D$ generated by Alg. \ref{al:AI:MGP}.

\section{Numerical Evolution Operator}
\label{sec:FT_res}

It is now established that high-resolution finite volume schemes can numerically approximate the single-phase Euler equations of gas dynamics in a robust and efficient manner. However, translating these schemes to simulate multi-phase flow has been unsuccessful, on account of oscillations of large amplitude as well as high frequency near phase interfaces \cite{Saurel&Abgrall} and reference therein. 

Given this, one needs to innovate in order to obtain a suitable numerical evolution operator. 
In this context, existing approaches can be divided into two categories i.e., diffuse interface methods and sharp interface methods, see \cite{Saurel1999} and references therein for comparison. 
Diffuse interface methods include traditional finite volume type schemes. 
However, the inherent numerical viscosity of these schemes leads to smearing of sharp interfaces and results in artificial mixing zones. 
It is unclear which equation of state should hold in these artificial mixing zones leading to possible conceptual failure of the algorithm. 
Typically, one circumvent such difficulty by making use of non-equilibrium models, which are though unable to track interface without assuming closure conditions \cite{Cocchi}. 
However, even if assuming this procedure to be viable, one would still need to fulfill a CFL-like condition which drastically reduces the time-step when the number of interfaces increases.\\
Consequently, sharp interface methods appear to be more promising alternative, particularly in one space dimension. 
These methods include Lagrangian based schemes \cite{Daude, Donea}, which are well-known to provide limited accuracy on relatively coarse meshes. 
These are some of the motivations that led us to explore another avenue \cite{Petrella22FT}, that of the \emph{front tracking} (FT) scheme \cite{Risebro} and references therein. 
Front tracking has been widely used in proving theoretical results for hyperbolic systems of conservation laws such as existence and stability (and thus uniqueness). 
Moreover, they constitute a powerful and efficient numerical simulator for hyperbolic systems, particularly in one space dimension \cite{Risebro}. 
Given their inherent lack of numerical viscosity, we modify the front tracking scheme to be used as the numerical evolution operator.

\begin{figure}[h!]
	\begin{center}
		\begin{tikzpicture}[scale=0.8]
		
		\pgfmathsetmacro{\xL}{-7};
		\pgfmathsetmacro{\Dx}{15};
		\pgfmathsetmacro{\xR}{\xL + \Dx};
		
		\pgfmathsetmacro{\yL}{-1};
		\pgfmathsetmacro{\Dy}{5};
		\pgfmathsetmacro{\yR}{\yL + \Dy};
		
		\tkzDefPoint(\xL, 0){xl};
		\tkzDefPoint(\xR, 0){xr};
		\tkzDefPoint(\xL+1, \yL){yl};
		\tkzDefPoint(\xL+1, \yR){yr};
		
		\tkzLabelPoint[below](xr){$x$};
		\tkzLabelPoint[left](yr){$t$};

		\draw[->] (xl) -- (xr);
		\draw[->] (yl) -- (yr);	
		
		\pgfmathsetmacro{\dt}{1.5};
		
		\foreach \s in {1.5,0.5,-0.4,-0.5,-0.6,-0.7,-0.8,-0.9}
		{
			\pgfmathparse{\s > 1.0 ? 1 : 0}; 
			\ifthenelse{ \pgfmathresult = 0 }{
				\draw (-2.75, 0) -- (-2.75 + \s * 2 * \dt, 2 * \dt);
			}{
				\draw (-2.75, 0) -- (-2.75 + \s * \dt, \dt);
			};
		}
		
		\foreach \s in {-1,0.1,0.5}
		{
			\pgfmathparse{\s > 0.0 ? 1 : 0}; 
			\ifthenelse{ \pgfmathresult = 1 }{
				\draw (1, 0) -- (1 + \s * 2 * \dt, 2 * \dt);
			}{
				\draw (1, 0) -- (1 + \s * \dt, \dt);
			};
		}
		
		\foreach \s in {-1.2,-0.3,0.4,0.5,0.6,0.7,0.8,0.9}
		{
			\pgfmathparse{\s < -0.5 ? 1 : 0}; 
			\ifthenelse{ \pgfmathresult = 0 }{
				\draw (5, 0) -- (5 + \s * 2 * \dt, 2 * \dt);
			}{
				\draw (5, 0) -- (5 + \s * 1.57* \dt, 1.57*\dt);
			};
		}
		
		
		\pgfmathsetmacro{\col}{1 - \dt};

		\draw (\col, \dt) -- (0.4,\dt + \dt);
		\draw (\col, \dt) -- (0,\dt + \dt);
		\draw (\col, \dt) -- (-0.3,\dt + \dt);
		
		\pgfmathsetmacro{\Col}{1 + 0.5* 1.57* \dt};

		\draw (\Col, 1.57* \dt) -- (2.,2*\dt);
		\draw (\Col, 1.57* \dt) -- (3.5,2*\dt);		
			
		\end{tikzpicture}
	\end{center}
\caption{Schematic representation of a typical FT approximation.}\label{Fig:FT}
\end{figure}
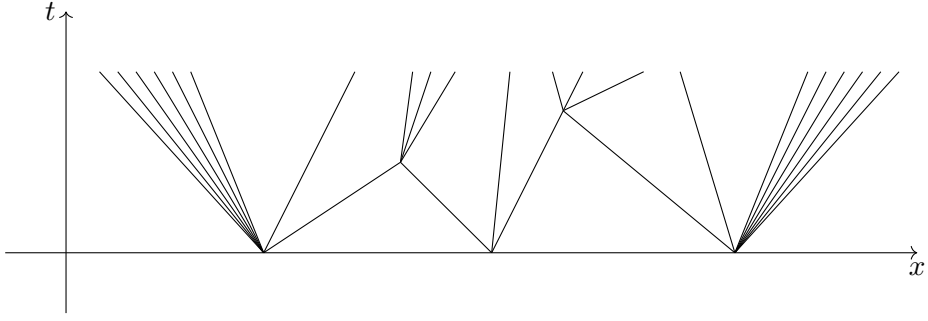

Classical FT approximations stems from the representation of wave interactions in characteristics space: the evolution of discontinuities in space-time is visualized by means of straight lines emanating from the initial discontinuity position. 
As it is well-known, in the paradigm of entropy solutions, only three types of solutions are involved: shocks, contact discontinuities and rarefactions. 
This latter is the only continuous type of solution and, as such, it gets approximated by a stair-like function, where the number of steps is controlled by an accuracy parameter $\delta>0$.
Each of the resulting discontinuity lines is called a \emph{front}, and FT starts with the assumptions that the initial condition is projected onto a piece-wise constant function, such that this latter converges to the initial condition as $\delta\rightarrow 0$, in suitable norm.
The thereby generated discontinuities are then evolved in time until two fronts collide. 
At this stage the simulation is stopped, the resulting Riemann Problem resolved and the newly generated fronts are inserted in the flow FT approximation. 
Iterating over all possible collisions yields the FT approximation of the solution at the output time $T>0$, see Fig. \ref{Fig:FT}.

Although intuitive, this procedure involve a non-negligible amount of approximations and its well-posedness is far from being trivial.
Nevertheless, in \cite{Petrella22FT} we discuss the well-posedness and stability of the methodology when applied to system of non-strictly hyperbolic conservation laws, providing a numerical algorithm for the FT approximation of corresponding solutions.

Concerning the time step $\Delta t$, due to the step-free property of the FT algorithm, no stability constraint must be enforced and arbitrary choices can be made, in principle.
This shall be treated as an hyper-parameter of the algorithm which is case dependent.
On the other hand, the concrete value for such parameter should be made as to balance efficiency of the computations and accuracy. 
Indeed, when projecting the solution onto a fixed Eulerian mesh, one introduces numerical viscosity in the simulation, which increases with the number of re-sampling steps. On the other hand large time steps allow for large number of collisions before re-sampling, thus reducing the overall speed.\\
Popular choices for a time step is provided by a simple equi-spaced rule
$
t^n = n\Delta t$, where $
\Delta t = \frac{T}{N}
$
for a user defined $N\in\bbN$, or the celebrated CFL rule (\ref{eq:CFL_tp}).
We detial the algorithmic procedure to carry out the FT method with resampling in Alg.\ref{al:AI:FT}.
\begin{algorithm}
\caption{Front-Tracking algorithm with re-sampling}\label{al:AI:FT}
\begin{algorithmic}
\STATE $\textbf{Data}\,:\,$ Initial datum $\VU^{(k)}_0\in L^{\infty}(D)$ $k=1,2$ with known material discontinuities, mesh width $\Delta>0$, a FT numerical evolution operator $S^\Delta$ and a end time $T>0$.
\STATE $\textbf{Output}\,:\,$ Evolved data $\VU^{(k)}(x,T)$.
\STATE $\bullet$ Generate a numerical mesh $\mathcal{T}^\Delta$ composed of volumes with width $\Delta$;
\STATE $\bullet$ Compute the volume averages $(\textbf{U}_0^{(k),h})_{h\in\mathcal{T}^\Delta}$ of the IC $\textbf{U}^{(k)}_0$ on $\mathcal{T}^\Delta$;
\STATE $\bullet$ Initialize the input time $t_{in}$, a time step $\Delta t$ and the output time: $t_{in} \leftarrow 0$, $\Delta t \leftarrow 0$, $t_{out} \leftarrow t_{in} + \Delta t$;
\WHILE {$t_{out} < T$}
\STATE $\bullet$ Generate a positive time step $0 < \Delta t = \mathrm{Stepper}\Big(t_{in},(\textbf{U}_0^h)_{h\in\mathcal{T}^\Delta},T\Big)$;
\STATE $\bullet$ $t_{out} \leftarrow t_{in} + \Delta t$;
\IF { $t_{out} \geq T$ }
\STATE \textbf{break};
\ENDIF
\IF { $t_{out}> T$ }
\STATE $\Delta t \leftarrow T - t_{out}$;
\ENDIF
\STATE $\bullet$ Evolve till the output time $t_{out}$ the volume averages $(\textbf{U}_0^h)_{h\in\mathcal{T}^\Delta}$ using the FT evolution operator $S^\Delta$ of Alg. \ref{al:FT:FT}, namely 
$$
(\textbf{U}_0^h)_{h\in\mathcal{T}^\Delta} = S^\Delta_{t_{out}}\Bigg((\textbf{U}_0^h)_{h\in\mathcal{T}^\Delta}\Bigg)
$$
\STATE $\bullet$ $t_{in} \leftarrow t_{out}$
\ENDWHILE
\end{algorithmic}
\end{algorithm}

Given the well-posedness of the evolution operator, the aim of the forthcoming sections is to provide an approximation of two-phase flow solutions via Monte-Carlo (MC) approximations.
As it is classical in two-phase flow, solutions are sought in statistical sense, where expectancy is take over all possible regimes.
This presents the restrictive property of requiring an infinite amount of possible regimes to be computed, in principle, whose number of interfaces grow towards infinity.
We aim at capturing such solutions by taking the limit of FT-approximations involving a finite number of interfaces.
Such a procedure essentially corresponds to the extension of what done in \cite{Petrella22FT} to the closure taken as $N\rightarrow\infty$:
let us consider an initial random field $\VU_{0,N}(x;\omega)$ generated by Alg.\ref{al:AI:eMGP} using $N\in\bbN$ number of sub-volumes. 
By denoting as $S^\delta_t$ the FT operator applied until time $t$ using an accuracy parameter $\delta$, then \cite{Petrella22FT} grants well-posedness of $\VU_N(x,t;\omega) = \lim_{\delta\rightarrow 0} S^\delta_t(\VU_{0,N}(x;\omega))$. 
If there exists the limit $\VU(x,t;\omega) = \lim_{N\rightarrow \infty} \VU_N(x,t;\omega)$, then one can write
\begin{align*}
\Vert \bbE\left[\VU(x,t;\cdot)\right] - \bbE\left[S^\delta_t(\VU_{0,N}(x;\cdot))\right] \Vert_{L^2(D)} 
&\leq 
\Vert \bbE\left[\VU(x,t;\cdot)\right] - \bbE\left[\VU_N(x,t;\cdot)\right] \Vert_{L^2(D)}\\
&\qquad
+
\Vert \bbE\left[\VU_N(x,t;\cdot)\right] - \bbE\left[S^\delta_t(\VU_{0,N}(x;\cdot))\right]\Vert_{L^2(D)}.
\end{align*}
where the first terms corresponds to the error introduced by taking a finite sub-scale resolution approximation (i.e. error introduced by Alg.\ref{al:AI:eMGP}) and the second corresponds to the error associated to the $\delta$-approximation of the numerical FT approximation.
To this extent, based on the forthcoming numerical tests we claim that the limit $\VU$ verifies
\begin{equation}
\label{eq:AI:FT_bubbles_approx_ansatz}
\Vert \bbE\left[\VU(x,t;\cdot)\right] - \bbE\left[\VU_N(x,t;\cdot)\right] \Vert_{L^2(D)} \leq C N^{-q}.
\end{equation}
where $C$ is a positive constant, and $q\in\bbN$ is the rate of sub-scale refinement.

The success of MC-based algorithm has been acknowledged in many recent works including \cite{Mishra10} for the first formulation (and corresponding convergence) of the Finite Volume (FV) Monte Carlo (MC) and FV Multi-Level Monte Carlo (FV-MLMC) method (see also \cite{Mishra2012}, and references therein), \cite{Risebro12} for the convergence of Front-Tracking (FT) MC (FT-MC) and FT MLMC  (FT-MLMC) methods.
Furthermore, FV-MC and FV-MLMC methods have been proven to be well-posed also for conservation laws with discontinuous fluxes \cite{Ruf21}.
We refer to these for technical details and convergence results.

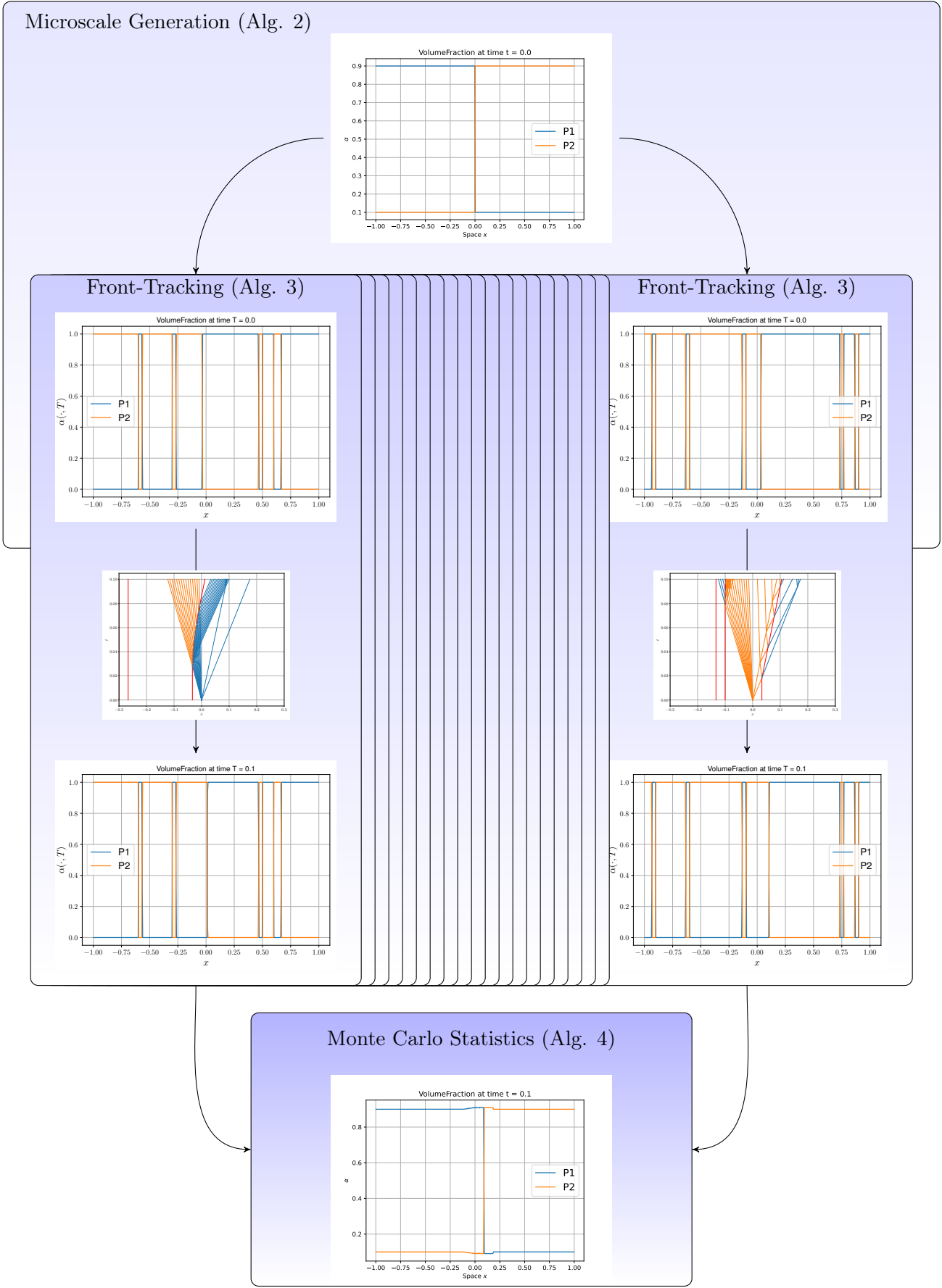
\begin{figure}[htbp!]
\begin{center}
\begin{tikzpicture}[->,>=stealth',auto,node distance=2cm]

\pgfmathsetmacro{\figDim}{0.3};
\pgfmathsetmacro{\frontDim}{0.2};
\pgfmathsetmacro{\xshift}{5cm};
\pgfmathsetmacro{\yshift}{3.1cm};

\tikzstyle{block} = [rectangle, rounded corners, minimum width=3cm, minimum height=1cm, draw=black, shade]

\node (mgp) [block, top color=blue!10, minimum width = 17cm, minimum height = 10cm] {};

\node (IC) [below of=mgp, yshift = +4.5cm] {\includegraphics[width=\figDim\textwidth]{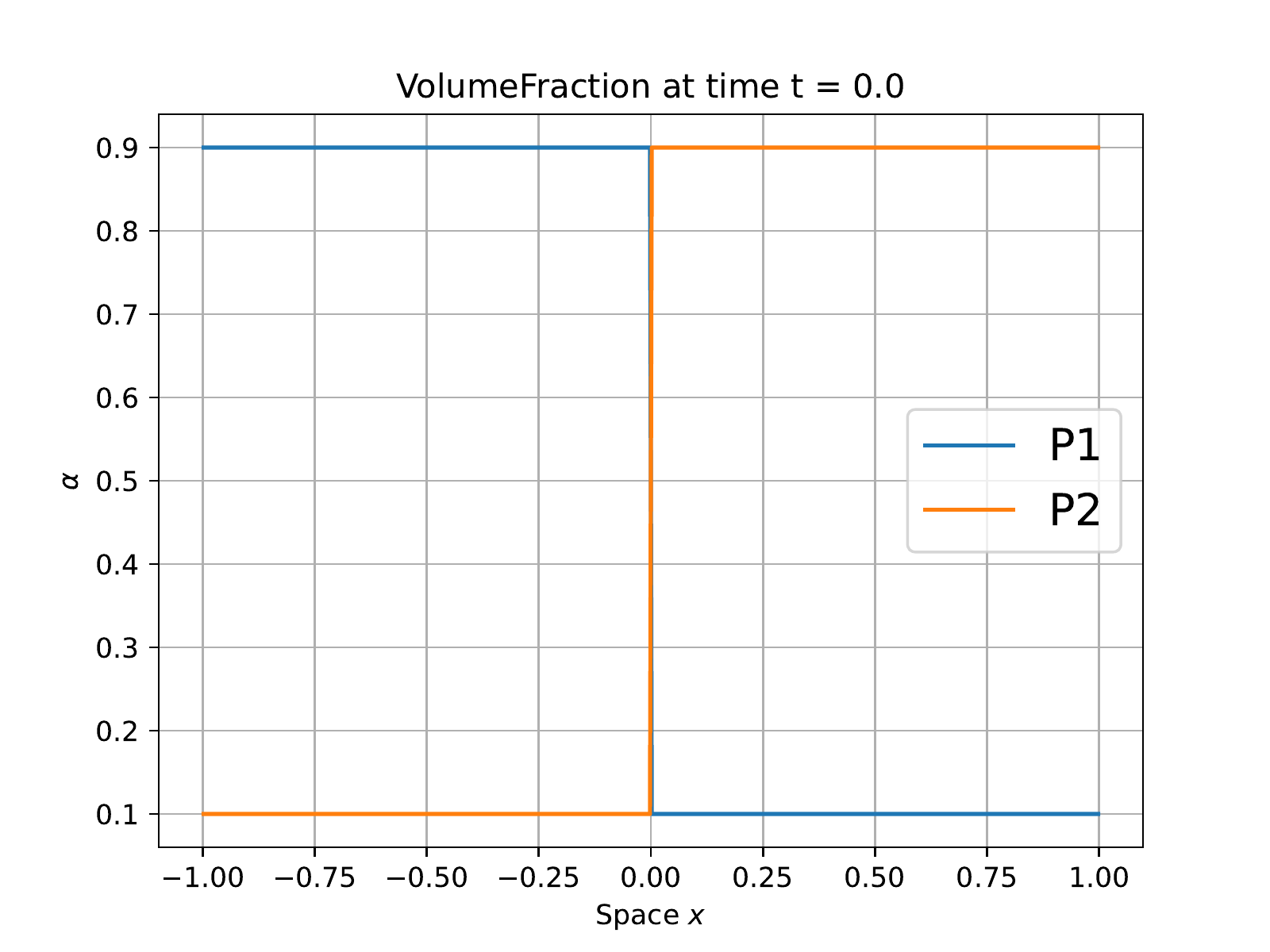}};

\node (S0) [block, top color=blue!20, minimum width = 6cm, minimum height = 13cm, shade,
	below of=mgp, yshift=-4.5cm, xshift=5cm] {};
	
\node [below of=IC, yshift = -\yshift, xshift=\xshift] (IC1) {\includegraphics[width=\figDim\textwidth]{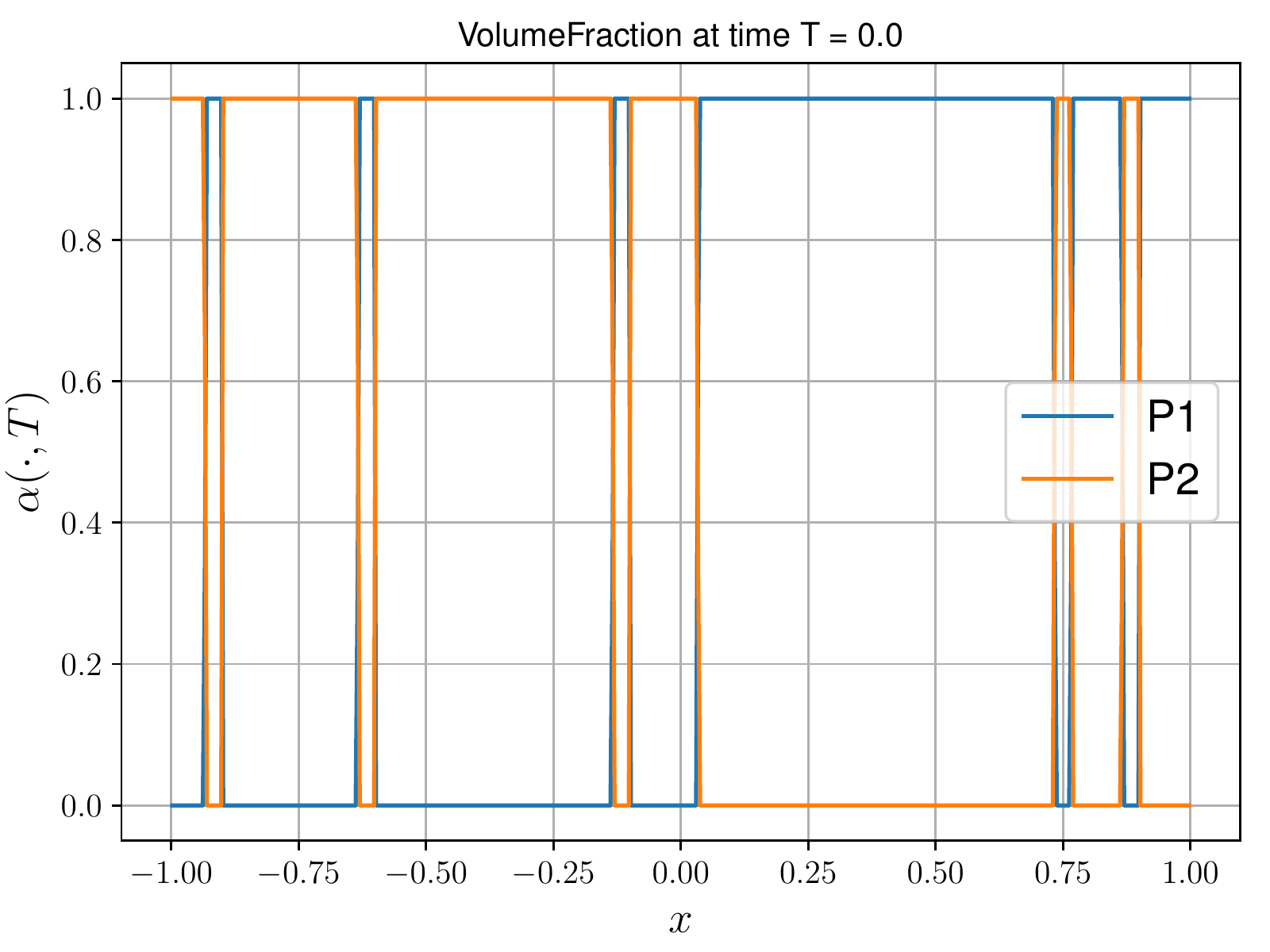}};
\node [below of=IC1, yshift = -2*\yshift] (T1) {\includegraphics[width=\figDim\textwidth]{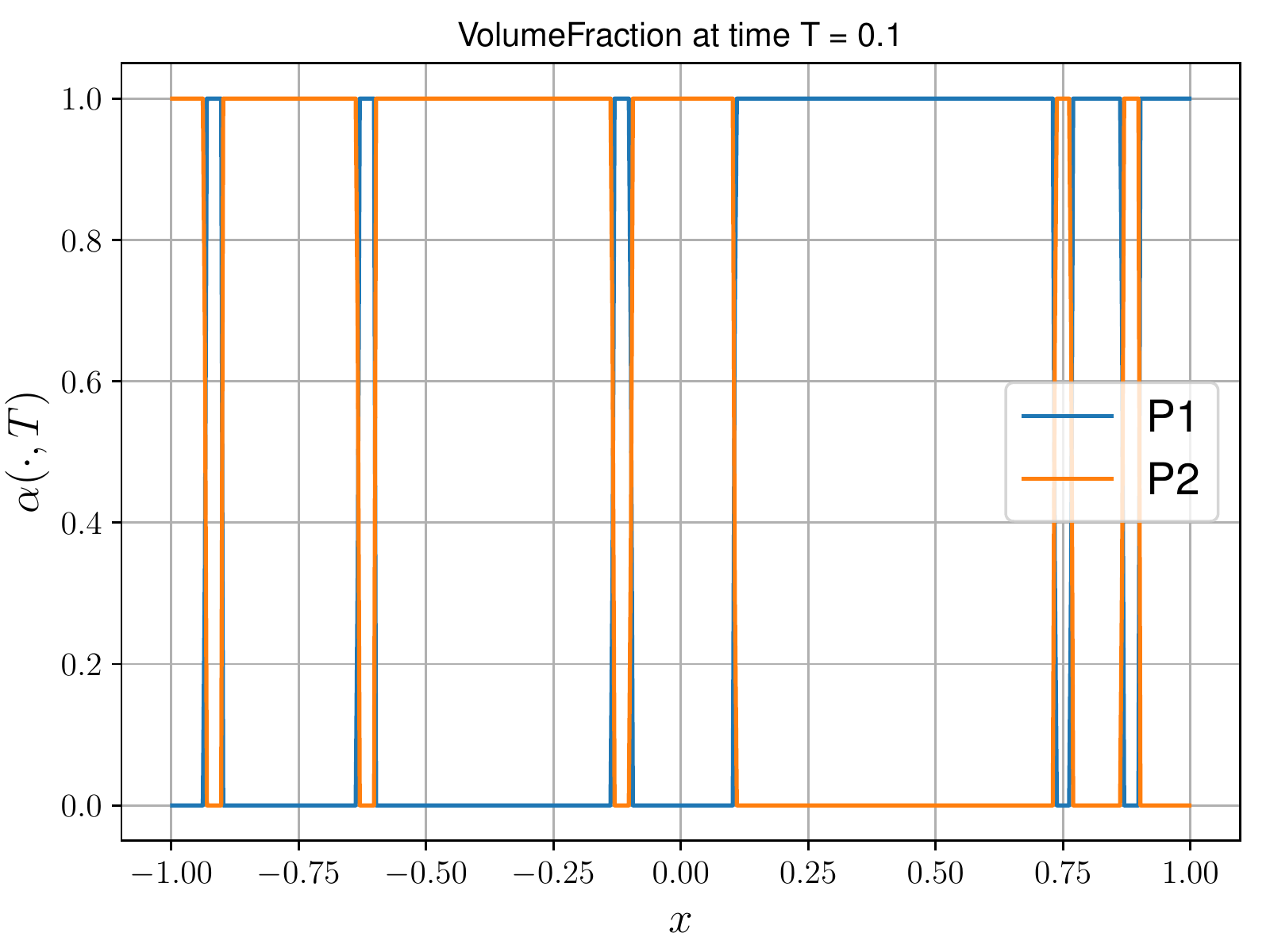}};

\foreach \n [count=\xi] in {0.5,0.75,1,1.25,1.5,1.75,2,2.25,2.5,2.75,3,3.25,3.5,3.75,4,4.25,4.5,4.75,5}
{
\node (S\xi) [block, top color=blue!20, minimum width = 6cm, minimum height = 13cm, 
	below of=mgp, yshift=-4.5cm, xshift=-\n*1cm] {};
}

\node [below of=IC, yshift = -\yshift, xshift=-\xshift] (IC0) {\includegraphics[width=\figDim\textwidth]{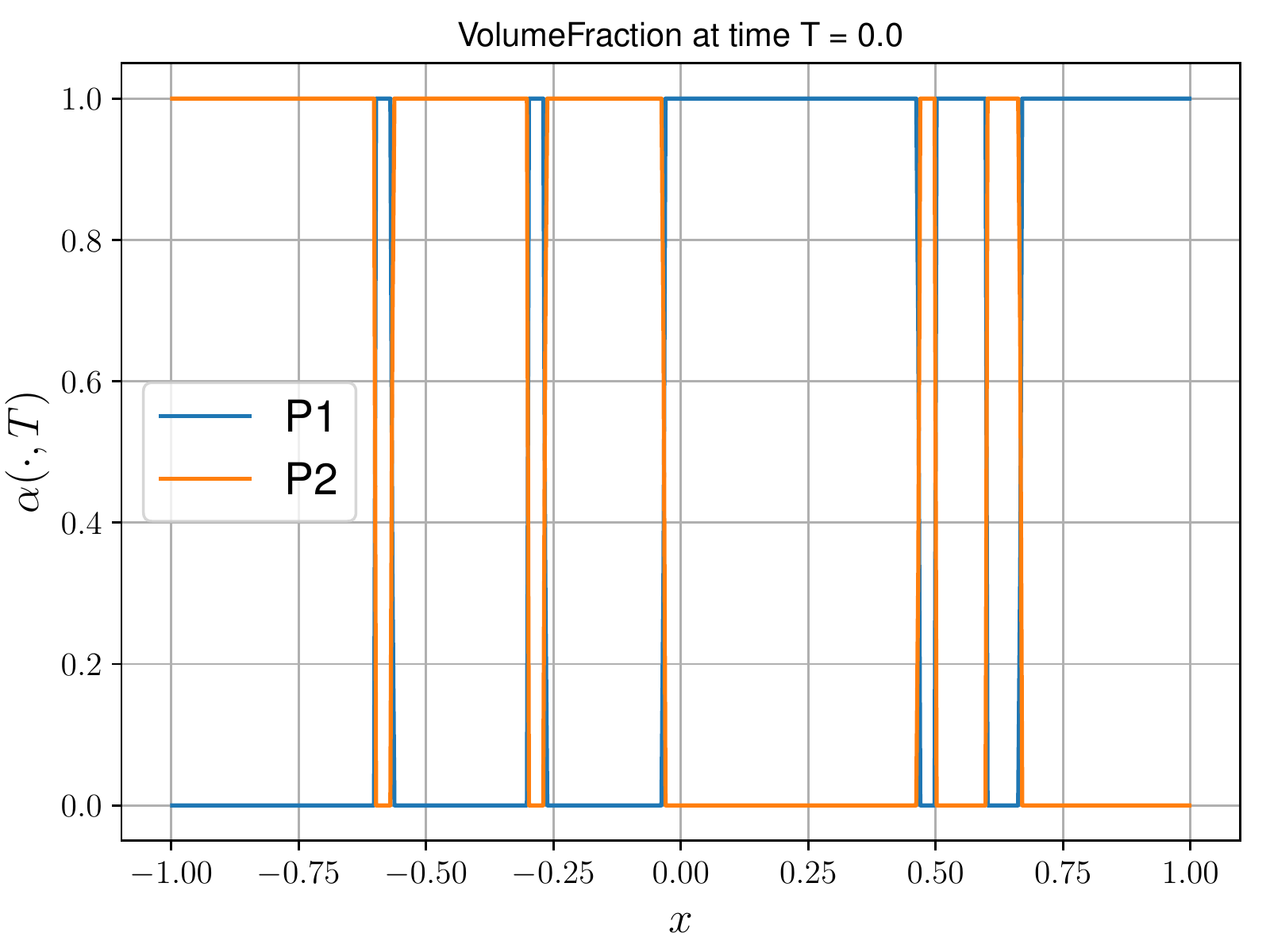}};

\node [below of=IC0, yshift = -2*\yshift] (T0) {\includegraphics[width=\figDim\textwidth]{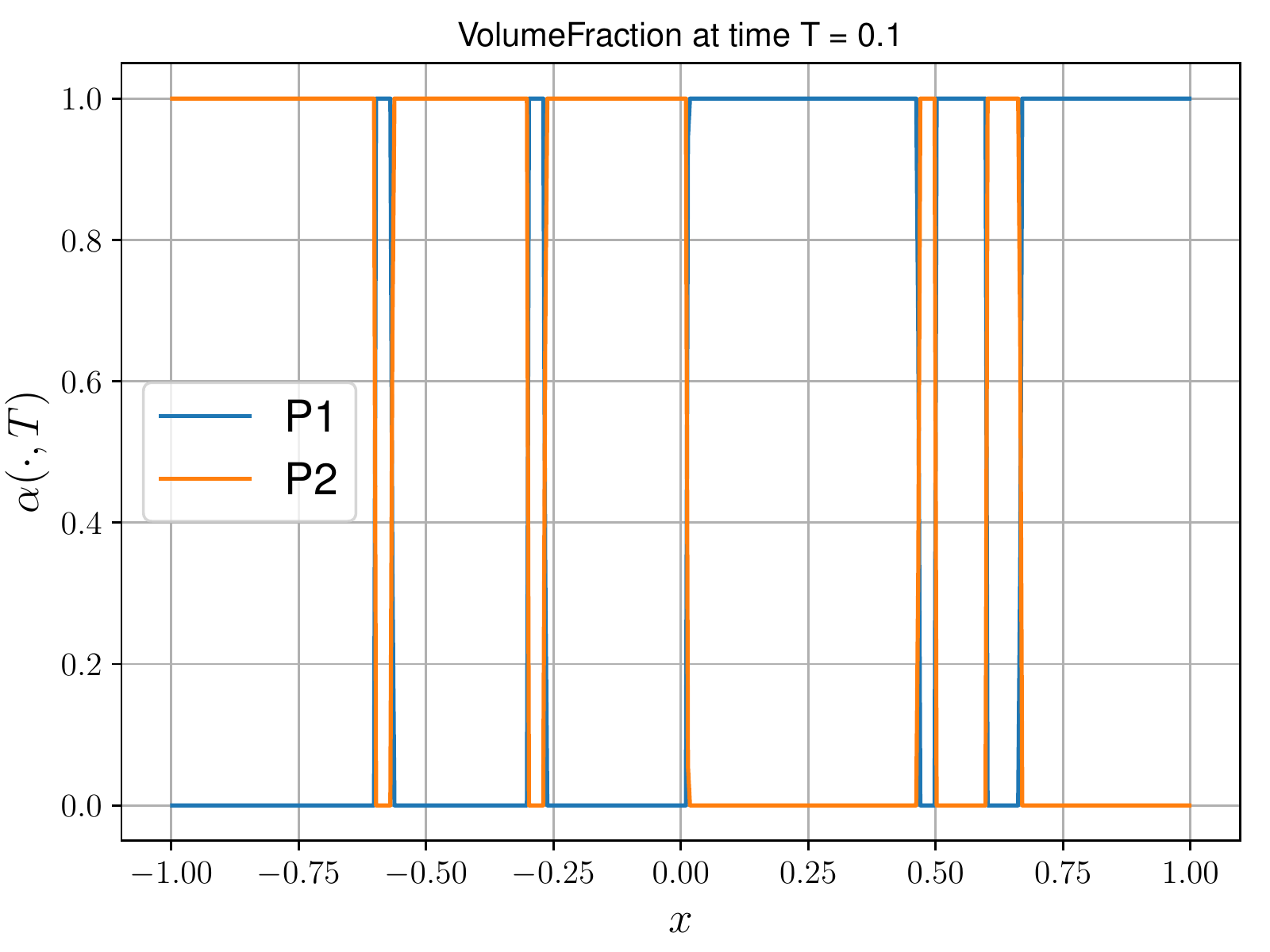}};

\draw [->] (IC0) -- (T0);
\draw [->] (IC1) -- (T1); 

\node [below of=IC0, yshift=-0.7*\yshift] (FT0)
{\includegraphics[width=\frontDim\textwidth]{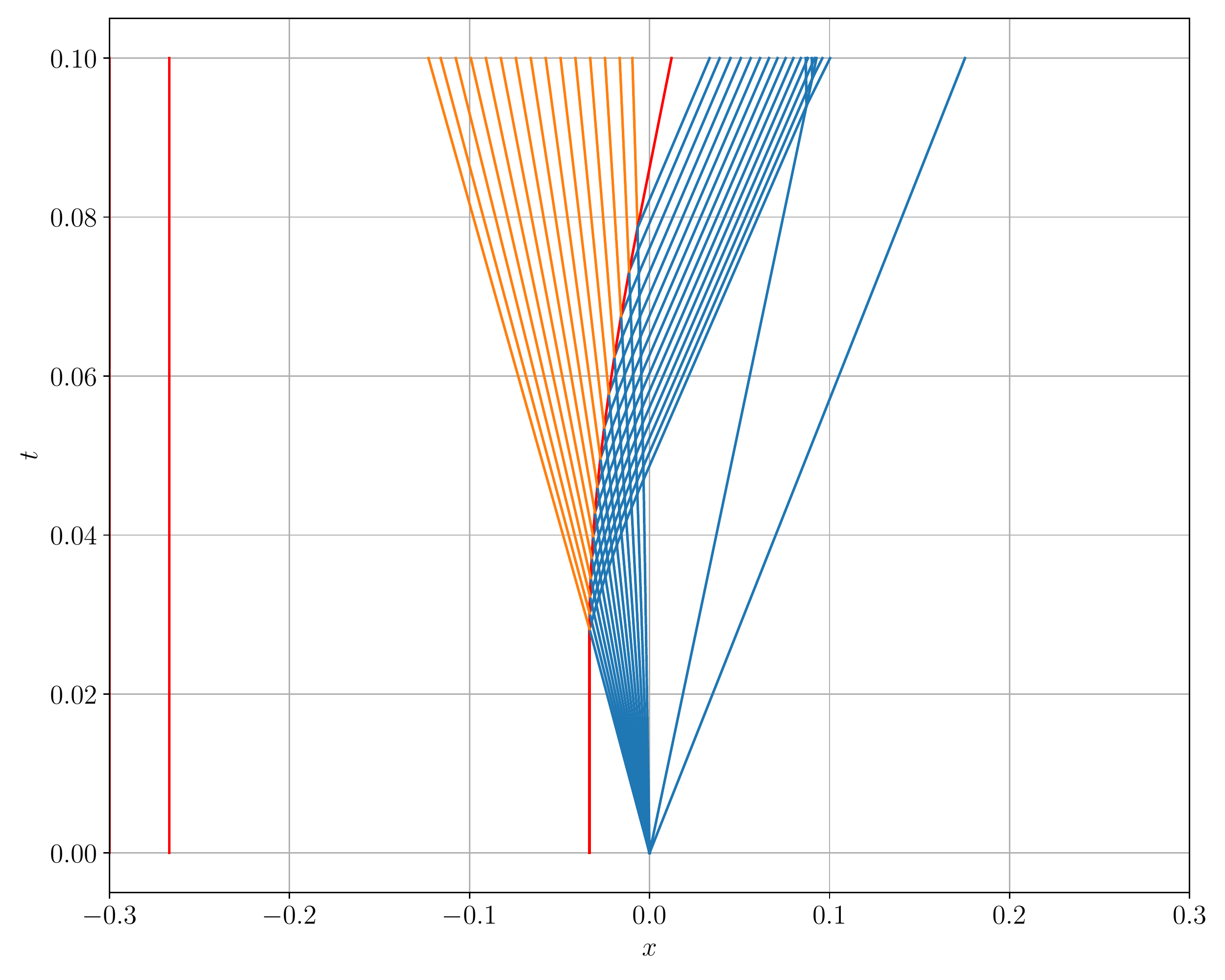}};
\node [below of=IC1, yshift=-0.7*\yshift] (FT1)
{\includegraphics[width=\frontDim\textwidth]{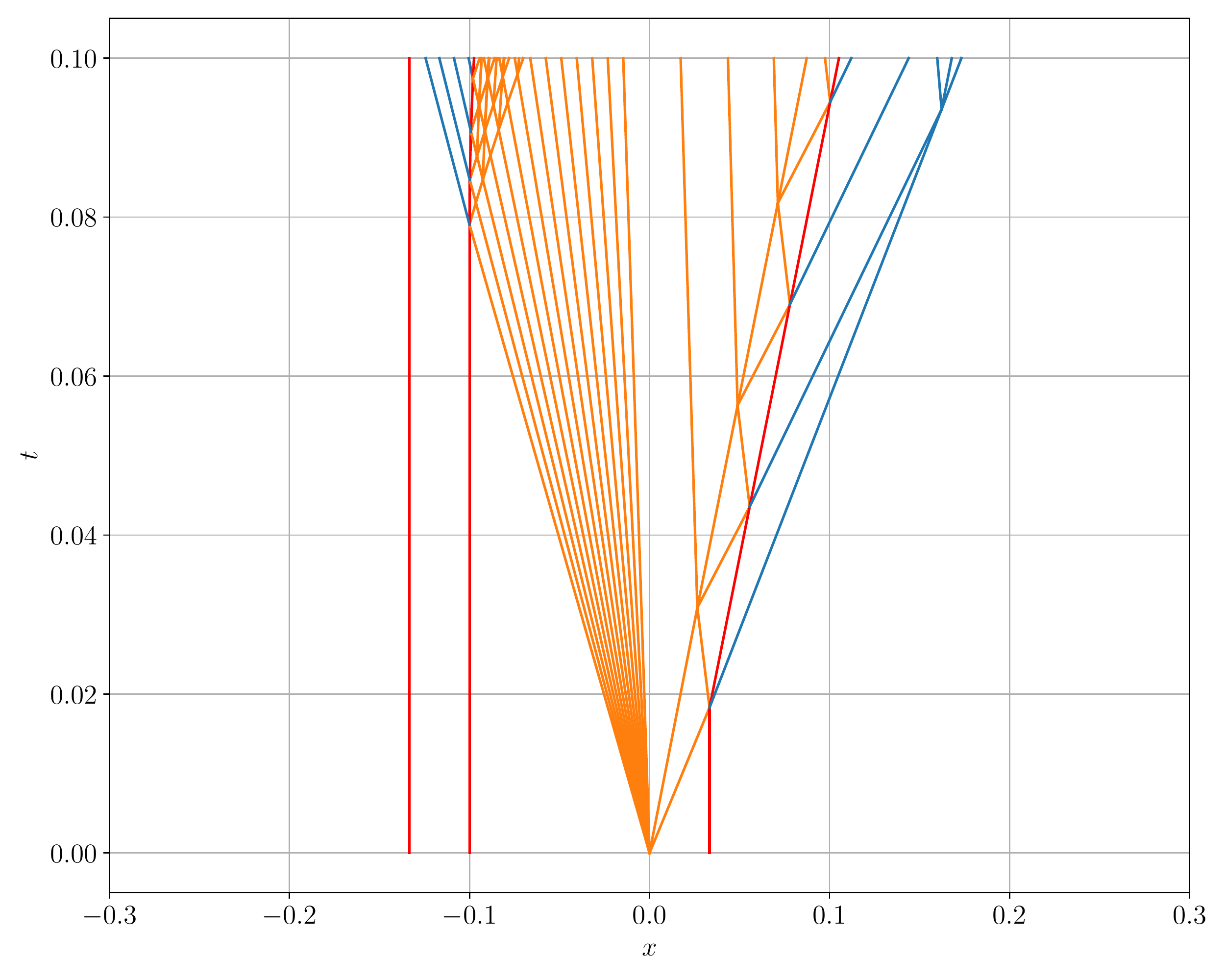}};

\node (MC) [block, top color=blue!30, minimum width = 8cm, minimum height = 5cm, 
	below of=mgp, yshift=-14cm] {};

\node[below of=IC, yshift =-5.5*\yshift] (T)
{\includegraphics[width = \figDim\textwidth]{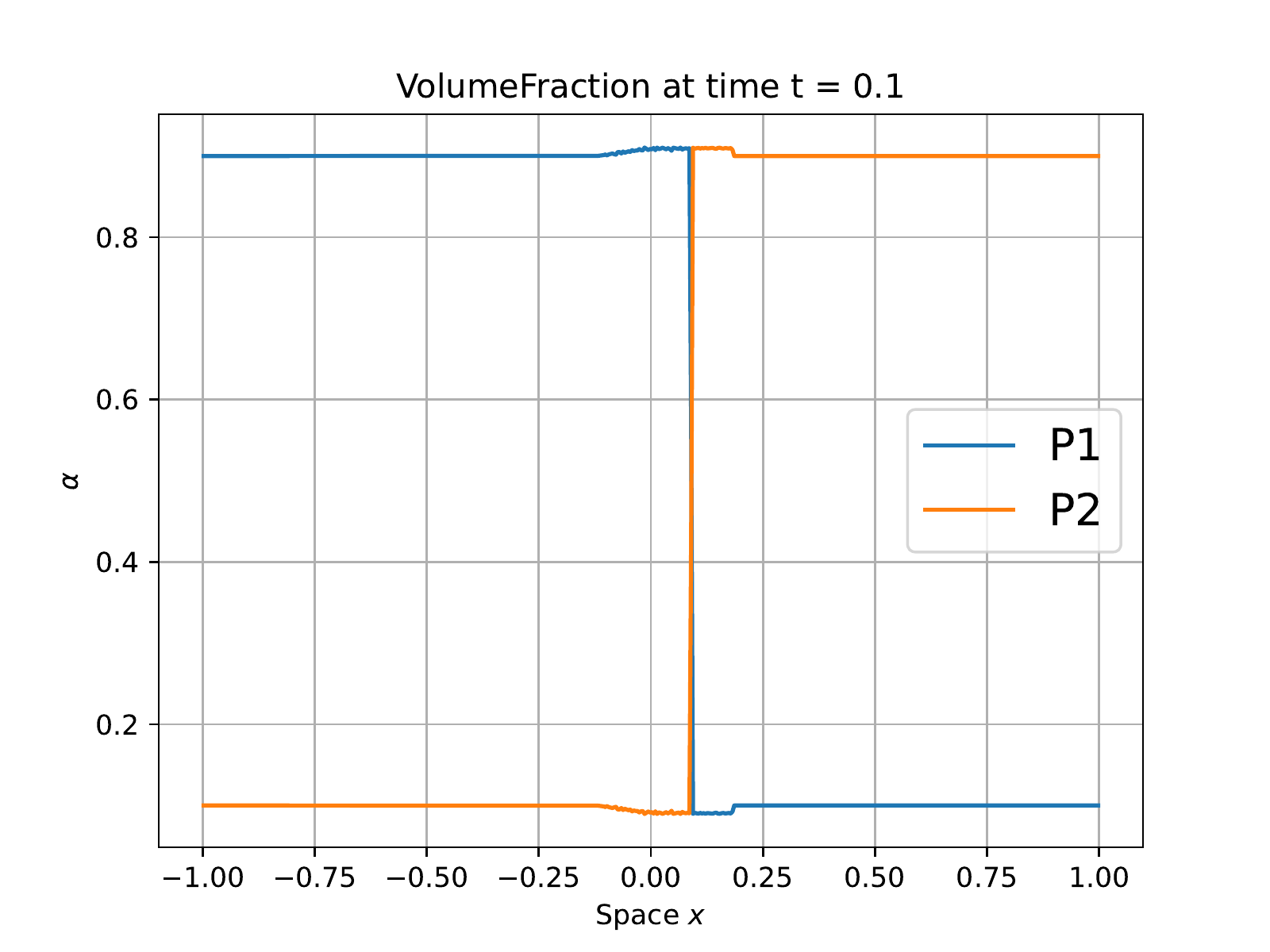}};

\node (mgp-tit) at (-5.5, 4.6) {Microscale Generation (Alg. \ref{al:AI:eMGP})};
\node (FT-tit-start) at (-5,-0.25) {Front-Tracking (Alg. \ref{al:AI:FT})};
\node (FT-tit-end) at (5,-0.25) {Front-Tracking (Alg. \ref{al:AI:FT})};
\node (MC-tit) at (0.0,-14) {Monte Carlo Statistics (Alg. \ref{al:AI:MC_es})};

\path[every node/.style={font=\sffamily\small}]
(IC) edge [out=180, in=90] node [right] {} (S19)
(IC) edge [out=0, in=90] node [right] {} (S0)

(S19) edge [out=270, in=180] node [right] {} (MC)
(S0) edge [out=270, in=0] node [right] {} (MC);

\end{tikzpicture}
\end{center}
\caption{Schematic representation of the Alg. \ref{al:AI:MC}. }
\label{fig:abinitio}
\end{figure}

\section{Monte-Carlo algorithm}
\label{sec:MC}

Starting from the composition of the procedure detailed in Section \ref{sec:MSR} followed by the one of Section \ref{sec:FT_res}, one computes the evolution of a specific realization of the random two-phase distribution $D_k(0;\omega)$ with $\omega\in\Omega$, and aims at computing (ensemble averages) of the generated solutions.

Given $L\in\mathbb{N}$ independent and identically distributed (i.i.d.) realizations $\omega_1, \ldots, \omega_L \in\Omega$, the empirical mean $\bbE_L$ and empirical variance $\bbV_L$ operators of the random field $q^{(k)}\, : \, \Omega\rightarrow C([0,T];L^\infty\cap BV(\mathbb{R}))$ are defined as (the unbiased estimators) 
\begin{subequations}
\label{eq:AI:MC-estimators}
\begin{align}
\mathbb{E}_{L}\left[q^{(k)}\right] &:= \frac{1}{L}\sum_{l=1}^L q^{(k)}(\omega_l)
\label{eq:AI:MC-estimators:mean}
\\
\mathbb{V}_{L}\left[q^{(k)}\right] &:= \frac{1}{L-1}\sum_{l=1}^L \left(q^{(k)}(\omega_l) - \mathbb{E}_{L}\left[q^{(k)}\right]\right)^2 
= \frac{L}{L-1}
\left(
\mathbb{E}_L\left[\left(q^{(k)}\right)^2\right]
- 
\mathbb{E}^2_{L}\left[q^{(k)}\right]
\right)
\label{eq:AI:MC-estimators:var}
\end{align}
\end{subequations}
In practice, it is well-known that these formulas lead to the so called subtractive cancellation phenomenon \cite{Sukys2014}, therefore, we will make use of the (more stable) Welford's on-line algorithm \ref{al:AI:MC_es}.
\begin{algorithm}
\caption{Welford's on-line algorithm}\label{al:AI:MC_es}
\begin{algorithmic}
\STATE $\textbf{Data\,:\,}$ I.i.d. samples $q_1,\ldots, q_L$;\\
\STATE $\textbf{Output\,:\,}$ Empirical mean and variance $\bbE_L$ $\bbV_L$;\\
\STATE $\bullet$ Set $\overline{q}_0 = 0$ and $M_0^{(2)} = 0$;
\FOR{$l=1,\ldots, L$}
\STATE $\bullet$ Update on-line mean
$
\overline{q}_l = \overline{q}_{l-1} + \frac{q_l - \overline{q}_l}{l}
$
\STATE $\bullet$ Update on-line (un-normalized) variance
$
M^{(2)}_l = M^{(2)}_{l-1} + (q_{l} - \overline{q}_l)(q_{l} - \overline{q}_{l-1})
$
\ENDFOR

\STATE $\bullet$ Set $\bbE_L = \overline{q}_L$ and $\bbV = M^{(2)}/(L-1)$.
\end{algorithmic}
\end{algorithm}

\subsection{Relevant quantities of interest and algorithm}
\label{sec:AI:relevant_quants}
\noindent

For a phasic variable $q^{(k)}$ we are interested in the following quantity of interest
\[
\begin{split}
\overline{q}^{(k)}(x_i, t^n) &:= \mathbb{E}\left[X^{(k)}q^{(k)}\right](x_i,t^n) = \int_\Omega X^{(k)}(x_i, t^n;\omega)q^{(k)}(x_i, t^n;\omega) \frac{d\bbP(\mathrm{d}\omega)}{\alpha^{(k)}(x_i, t^n)}
\end{split}
\]
where $\alpha^{(k)}(x_i,t^n) := \mathbb{E}\left[X^{(k)}(x_i,t^n;\cdot)\right]$ is the volume fraction, and represents the probability of finding phase $k$ in volume $\mathcal{C}_i$ at time $t=t^n$.
The aforementioned quantity is well-defined away from $\alpha^{(k)}(\cdot,t^{n})=0$, and we extend it by fixing its value at $\alpha^{(k)}=0$ as to be $0$.
The rational for doing so is that any random variable $X(\omega)\in [0,1]$ such that $\bbE[X] = 0$ implies $X(\omega) = 0$ $\bbP$-a.s..

A second quantity of interest is provided based on the second moment of the quantity $q$, namely the variance.
Unfortunately, it is not immediate to define variance when it comes to Favre-averaged quantities: for any variable $y\in\lbrace \rho, u, p \rbrace$, one defines the corresponding perturbed variable $\phi^{(k)} := X^{(k)}y^{(k)}$ so that the Reynolds and Favre decomposition (respectively) would read
\begin{align*}
(Reynolds)
\qquad\qquad\qquad
\phi^{(k)} & = \overline{\phi}^{(k)} + \phi^{(k),*}
\qquad
\qquad
\overline{\phi}^{(k)} := \bbE[\phi^{(k)}]\\
(Favre)
\qquad\qquad\qquad
\phi^{(k)} & = \tilde{\phi}^{(k)} + \phi^{(k),**}
\qquad
\qquad
\tilde{\phi}^{(k)} := \frac{\overline{\phi}^{(k)}}{\overline{X}^{(k)}}.
\end{align*}
In turn, variance can be defined in both cases as a measure of the quadratic deviation from a given value, namely
\begin{align*}
\bbV [\phi^{(k)}] 
&= 
\overline{ \left( \phi^{(k)} - \overline{\phi}^{(k)} \right)^2 }
= \overline{ \phi^{(k),*}\phi^{(k),*} }\\
\tilde{\bbV} [\phi^{(k)}] 
&= 
\overline{ \left( \phi^{(k)} - \tilde{\phi}^{(k)} \right)^2 }
= \overline{ \phi^{(k),**}\phi^{(k),**} }\\
&= \bbV[\phi^{(k)}] + \left(\overline{\phi}^{(k)} - \tilde{\phi}^{(k)}\right)^2
\end{align*}
where the last equality follows by $\overline{\tilde{\phi}^{(k)}} = \tilde{\phi}^{(k)}$. The last addendum in the Favre-averaged variance $\tilde{\bbV}$ is constant and constitutes a measure of the distance between Reynolds and Favre averaging, which coincide in the pure phase case, i.e. $\overline{X}^{(k)} = 1$ for some $k=1,2$. 
Since stochastic behaviors enter in flow simulations only though the different dispersion of phases, in the following we will consider only the first term as a measure of variance.

At the numerical level, we are computing finite volume approximations of the hereby introduced quantities. 
In particular, for any time level $t^n\in [0,T]$ we have that
\begin{equation}
\label{eq:AI:mean_var_MC}
\begin{split}
\alpha^{(k)}(x_i,t^n) &\approx \mathbb{E}_{MC}\left[ I^{(k)}_i\left[X^{(k)}\right](t^n,\cdot) \right] =: \alpha^{(k), n}_i\\
\hat{\alpha}^{(k)}(x_i,t^n) &\approx \mathbb{V}_{MC}\left[ I^{(k)}_i\left[X^{(k)}\right](t^n,\cdot) \right] =: \hat{\alpha}^{(k), n}_i\\
\overline{q}^{(k)}(x_i, t^n) &\approx \frac{\mathbb{E}_{MC}\left[X^{(k)}q^{(k)}\right](x_i,t^n)}{\alpha^{(k),n}_i} =: \overline{q}^{(k),n}_i \\
\hat{q}^{(k)}(x_i, t^n) &\approx \mathbb{V}_{MC}\left[X^{(k)}q^{(k)}\right](x_i,t^n) =: \hat{q}^{(k),n}_i
\end{split}
\end{equation}
where $\mathbb{E}_{MC}$ and $\mathbb{V}_{MC}$ are the Monte-Carlo mean and variance defined in \eqref{eq:AI:MC-estimators}. 

The resulting Monte-Carlo algorithm for the ab-initio method is then provided in Alg. \ref{al:AI:MC}, and schematically illustrated in Fig.\ref{fig:abinitio}.

\begin{algorithm}
\caption{Monte-Carlo Ab-initio algorithm}\label{al:AI:MC}
\begin{algorithmic}
\STATE $\textbf{Data}\,:\,$ Initial datum $\textbf{U}_0\in L^{\infty}(D)$, mesh width $\Delta>0$, end-time $T>0$ and number of samples $L\in \mathbb{N}$;
\STATE $\textbf{Output}\,:\,$ Averages $\bbE[\VU^{(k)}](x,T)$;
\STATE $\bullet$ Compute initial volume averages $\VU^{(k),0}_i$;
\STATE $\bullet$ Merge volumes affected by the same states;
\STATE $\bullet$ Generate $L$ i.i.d samples
$\textbf{U}^{l}_0 \in L^\infty(D);$
using Alg.\ref{al:AI:eMGP}.
\FOR{$l=1,\ldots, L$}
\STATE Evolve the sample $\VU^{l}_0$ in time, $\VU_l^\Delta = S_T^\Delta(\VU^{l}_0)$, using Alg.\ref{al:AI:FT};
\ENDFOR
\STATE $\bullet$ Compute statistical mean $\bbE_L\left[\VU^\Delta\right]$ and variance $\bbV_L\left[\VU^\Delta\right]$ via Alg. \ref{al:AI:MC_es};\\
\STATE $\bullet$ Compute (Favre-averaged) statistical quantities $\overline{\VU}^\Delta$ as in (\ref{eq:AI:mean_var_MC}).
\end{algorithmic}
\end{algorithm}

\section{Numerical Experiments}
\label{sec:AI:MC:NE}

We consider here some numerical experiment for the MC-based ab-initio method.
The physical domain is $D = [-1,1]$ subdivided into $M=500$ volumes, over which the FT-approximation is run using accuracy parameters $\bm{\delta} = [\delta^{(1)},\delta^{(2)}]$.
In this section we will consider two phases associated to the ideal-gas EOS with $\gamma^{(1)} = 1.4$ and $\gamma^{(2)}=1.6$.\\
Initial conditions will be provided in terms of the primitive variables $\VV = [\alpha,\VW] = [\alpha, \rho, u, p, \gamma]$.
For each simulation, we run the FT-algorithm presented in \cite{Petrella22FT} to produce $L\in\bbN$ realizations of the flow field, and corresponding mean and variance are then calculated via Alg.\ref{al:AI:MC_es}.
For the sake of clarity, results are shown by normalizing (Favre-averaging) corresponding moments as discussed in Section \ref{sec:AI:relevant_quants}, while un-normalized variables are used for computing the Cauchy rates during the convergence studies.

As our aim is to compare the two-phase flows simulated at the \emph{microscopic level} with the ab-initio algorithm \ref{al:AI:MC} against the corresponding state-of-the-art macroscopic simulations. Specifically, we choose the \emph{generalized Discrete Equation Model} (DEM) of the recent paper \cite{Petrella2022} as the macroscopic simulator. 
This scheme is recalled and summarized in Appendix \ref{app:r-model}. 
In particular, it involves a key parameter, denoted by $r \ in [0,1]$, which models the underlying probability coefficients for each phase. 
For most of the test cases we shall plot the ab-initio results with that generated by the two extreme values $r=0,1$ for the DEM. 

\subsubsection{Phases in mechanical equilibrium}
\label{sec:AI:MC:NE:Un}

We consider two phases initiated at mechanical equilibrium. As postulated by Abgrall's criterion, the evolution of the mixture is expected to maintain uniform conditions throughout time, so that it can also understood as to be a relaxation-free test case. 
Therefore, it represents a suitable test case to investigate sub-scale impacts on macroscopic quantities, and a necessary check for Abgrall's fulfillment.
The associated initial condition reads
\[
\VV_0(x) = 
\begin{cases}
\begin{bmatrix}
\alpha^{(1)}_L = 0.9\\
\VW_L^{(1)}\\
\alpha^{(2)}_L = 0.1\\
\VW_L^{(2)}
\end{bmatrix}
& x< 0\\
\begin{bmatrix}
\alpha^{(1)}_R = 0.1\\
\VW_R^{(1)}\\
\alpha^{(2)}_R = 0.9\\
\VW_R^{(2)}
\end{bmatrix}
& x>0
\end{cases}
\qquad
\qquad
\VW^{(k)}_L
=
\begin{bmatrix}
\rho_L\\
0.9\\
0.3
\end{bmatrix}
\quad
\VW^{(k)}_R
=
\begin{bmatrix}
\rho_R\\
0.9\\
0.3
\end{bmatrix}
\]
where $\rho_L = 1$ and $\rho_R = 0.125$. The end time $T=0.1$ and the parameters used to produce results displayed in Fig.\ref{Fig:AI:T1} are listed in Table \ref{Tab:AI:T1}.
\begin{table}[!h]
\begin{center}
\begin{tabular}{c||c|c|c|c|c|c|}
&
$M$ & $\tilde{N}$ & $\mathrm{CFL}$ & $\delta^{(1)}$ & $\delta^{(2)}$ & $L$\\
\hline
\hline
\textbf{Ab-initio} & $1000$ & $16384$ & $-$ & $0.01$ & $0.01$ & $1024$
\end{tabular}
\end{center}
\caption{Parameters of computed solutions reported in Fig. \ref{Fig:AI:T1} }\label{Tab:AI:T1}
\end{table}
This test case considers the motion of a macroscopic discontinuity in the mixture composition, rigidly moving to the right of the computational domain. 
Such behavior is clearly visible in the motion of the initial discontinuity from (the initial location) $x=0$ to $x=0.9\cdot 1.1 = 0.99$ in the volume fraction and density plots. 
Preservation of (the initial) uniform mechanical conditions is instead visible in pressure and velocity plots, demonstrating the intrinsic ability of the proposed methodology to maintain uniform conditions (i.e. it fulfills Abgrall's criterion).\\
\begin{figure}[!htbp]
\begin{center}
\includegraphics[scale=\figsize]{\main/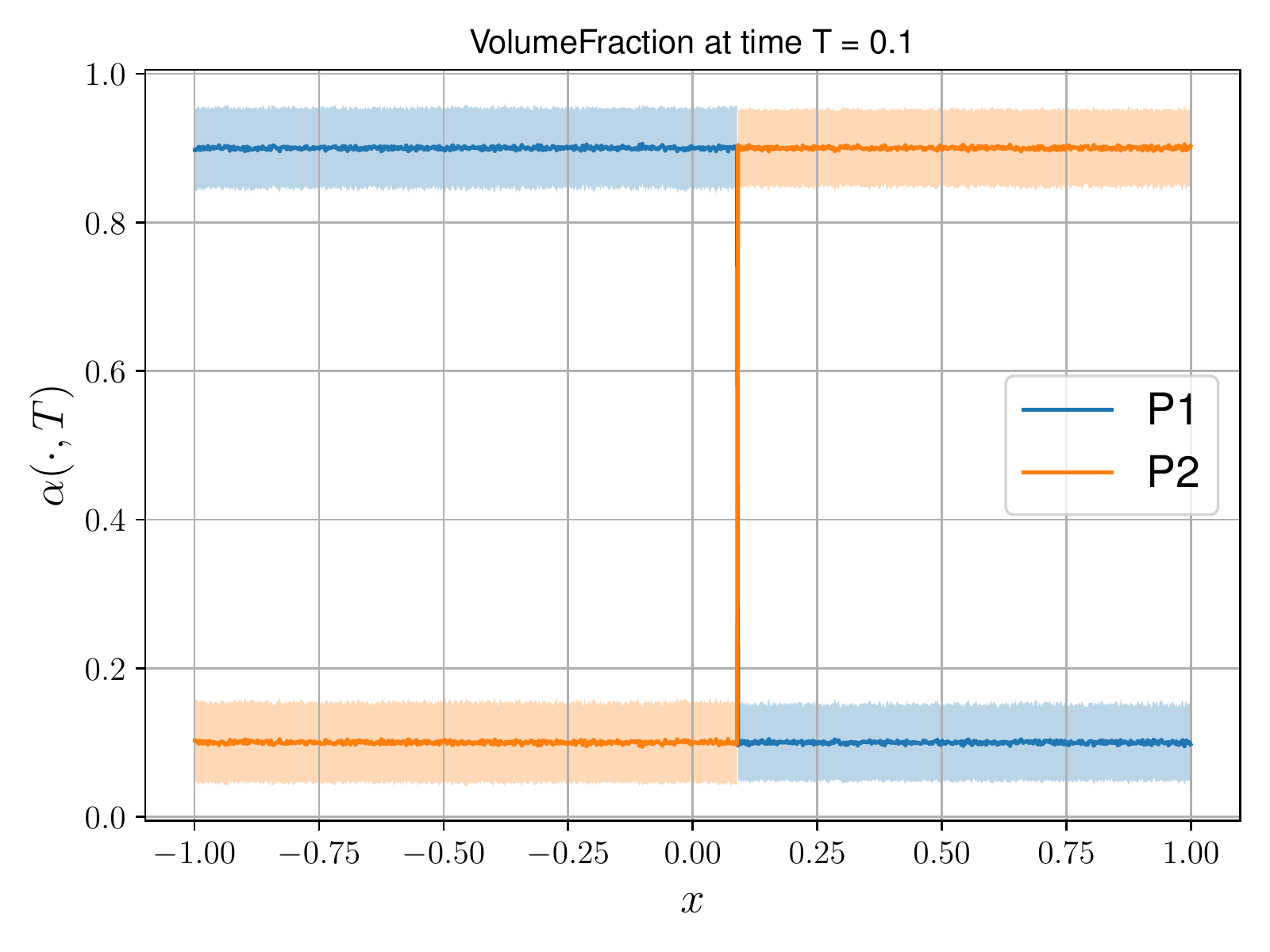}\,
\includegraphics[scale=\figsize]{\main/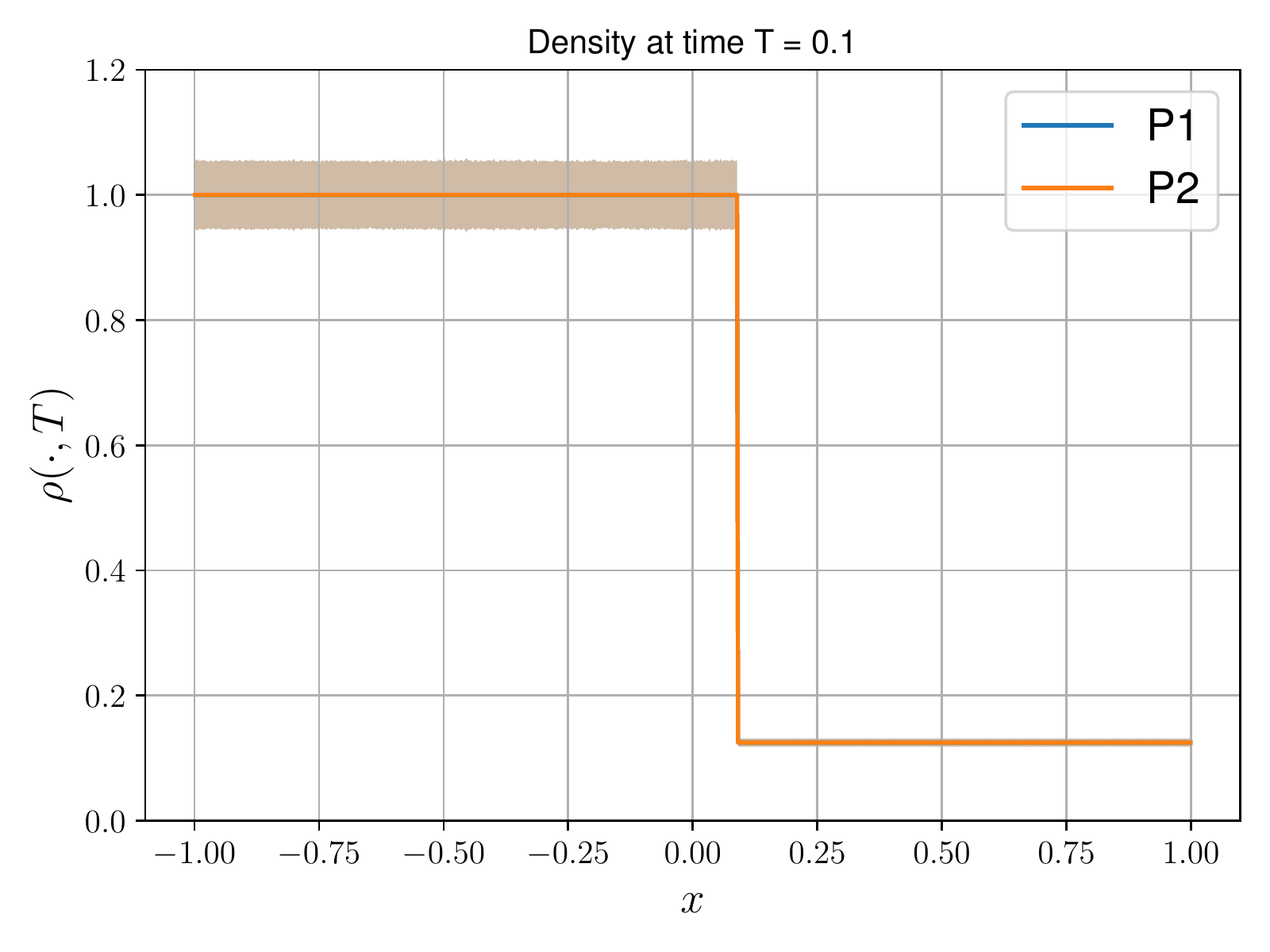}\\
\includegraphics[scale=\figsize]{\main/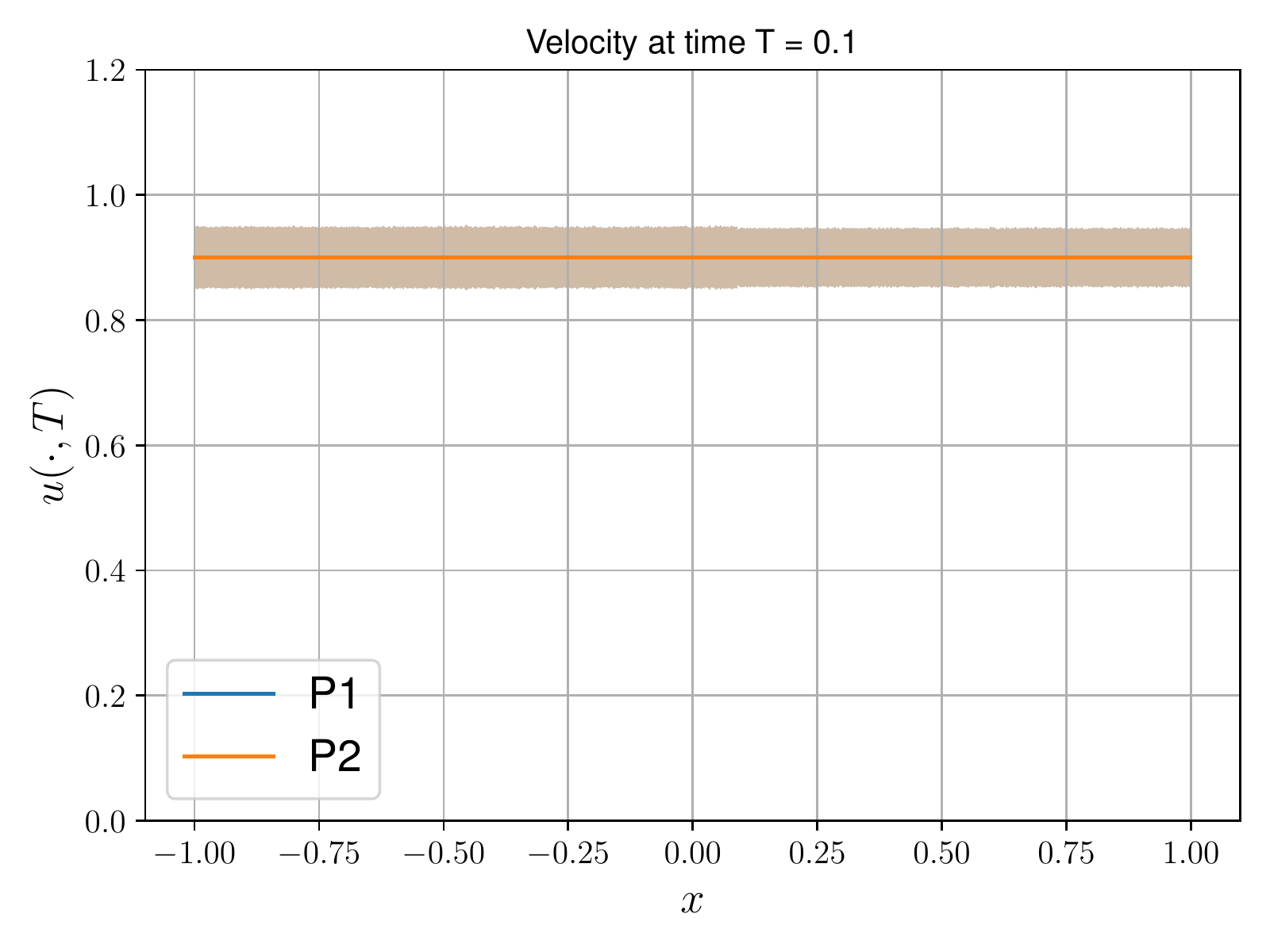}\,
\includegraphics[scale=\figsize]{\main/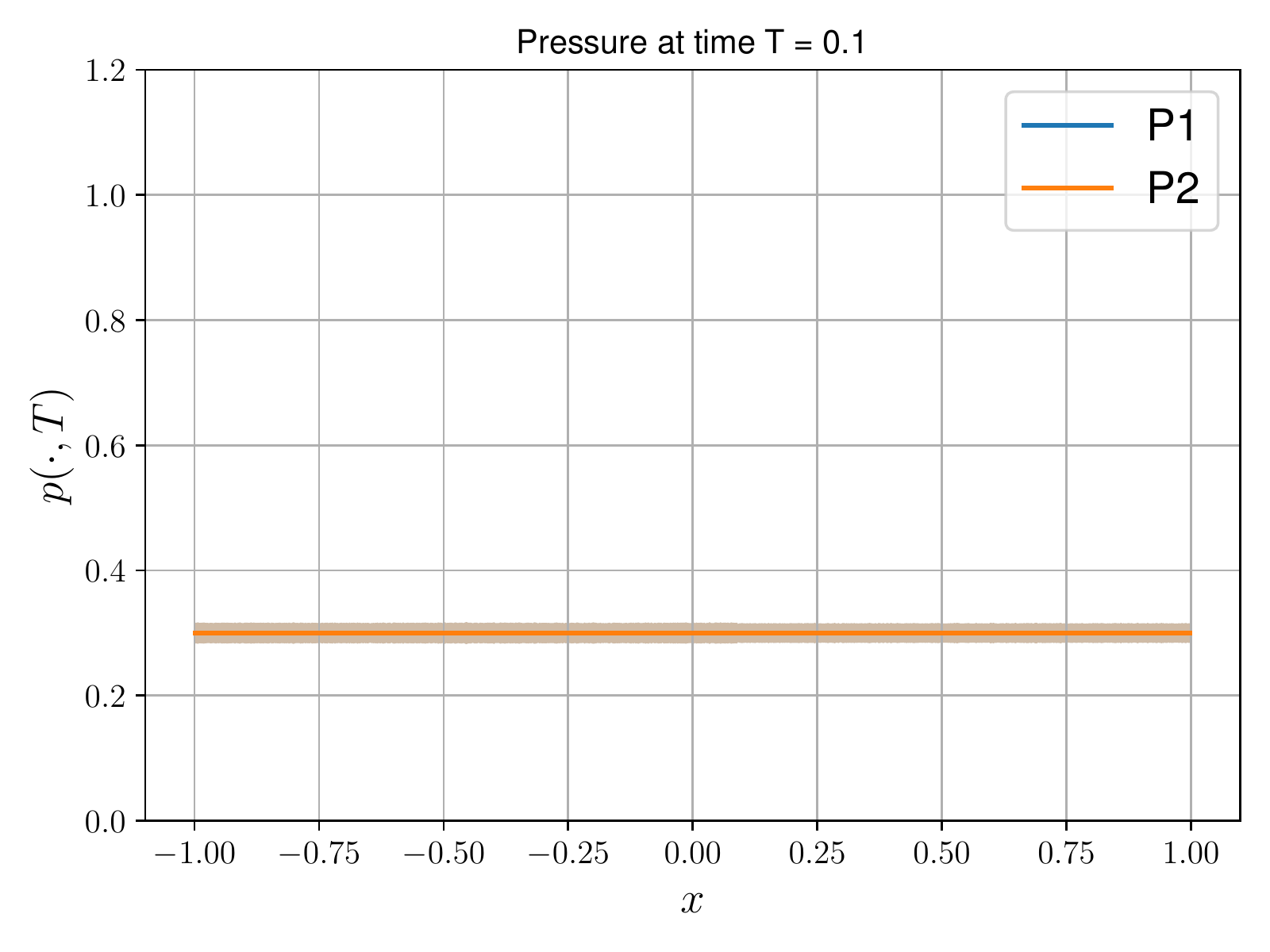}
\caption{Results for the mixtures under mechanical equilibrium. Confidence intervals are shown as colored areas.}
\label{Fig:AI:T1}
\end{center}
\end{figure}

Notice that, by the mechanical equilibrium initially present at each interface, evolved solutions do not involve interactions through the interface, and thus implying that this test case is relaxation-free. Therefore, we use the present test case to study sub-discretization properties.

First, we are interested in investigating the impact that the strategy used to define the maximum number of volumes $N$ has on macroscopic quantities, when applying Alg.\ref{al:AI:eMGP}. In Section \ref{sec:AI:MG:volumes} we distinguish between two major strategies: fixing the same number across all the volumes of the physical discretization (termed uniform case), and a random choice for each volume.
We run the present test case for an increasing number of sub-volumes $N_j = 2^{7+j}$ $j=0,\ldots, 8$, and compute the Cauchy rates $e^{(k)}_j(y)$ associated to the phase $k$ and variable $y\in \lbrace\alpha, \rho, u,p \rbrace$ based on $L=250$ samples
\[
e_j^{(k)}(y) 
:= 
\Vert 
\bbE_L
\left[
y_{N_{j+1}}
(T)
\right]
-
\bbE_L
\left[
y_{N_{j}}
(T)
\right]
\Vert_{L^1(D)}.
\]
Results are shown in Fig.\ref{Fig:AI:T1_conv} in the log-log scale.
\begin{figure}[!htbp]
\begin{center}
\includegraphics[scale=\figsize]{\main/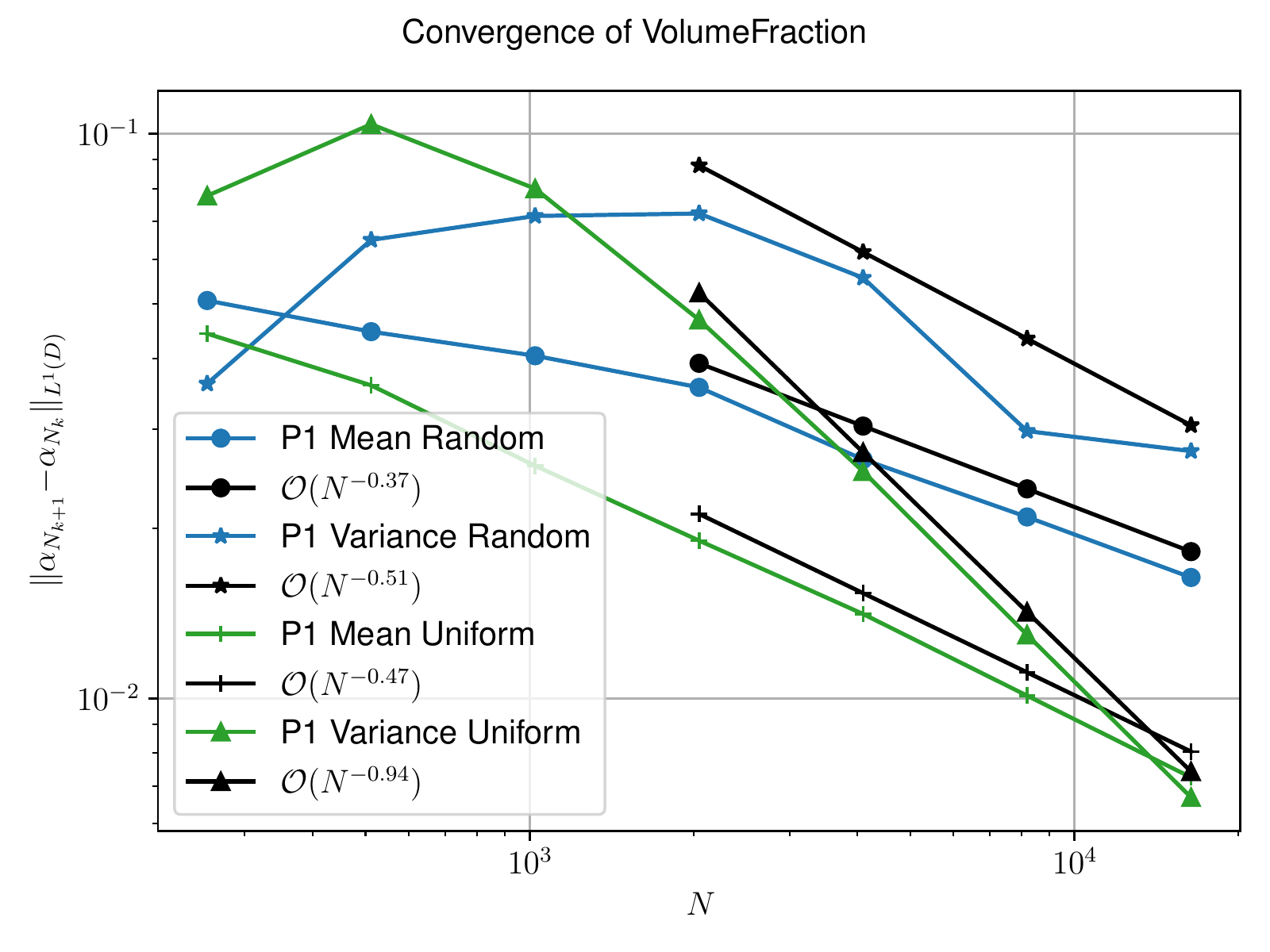}\,
\includegraphics[scale=\figsize]{\main/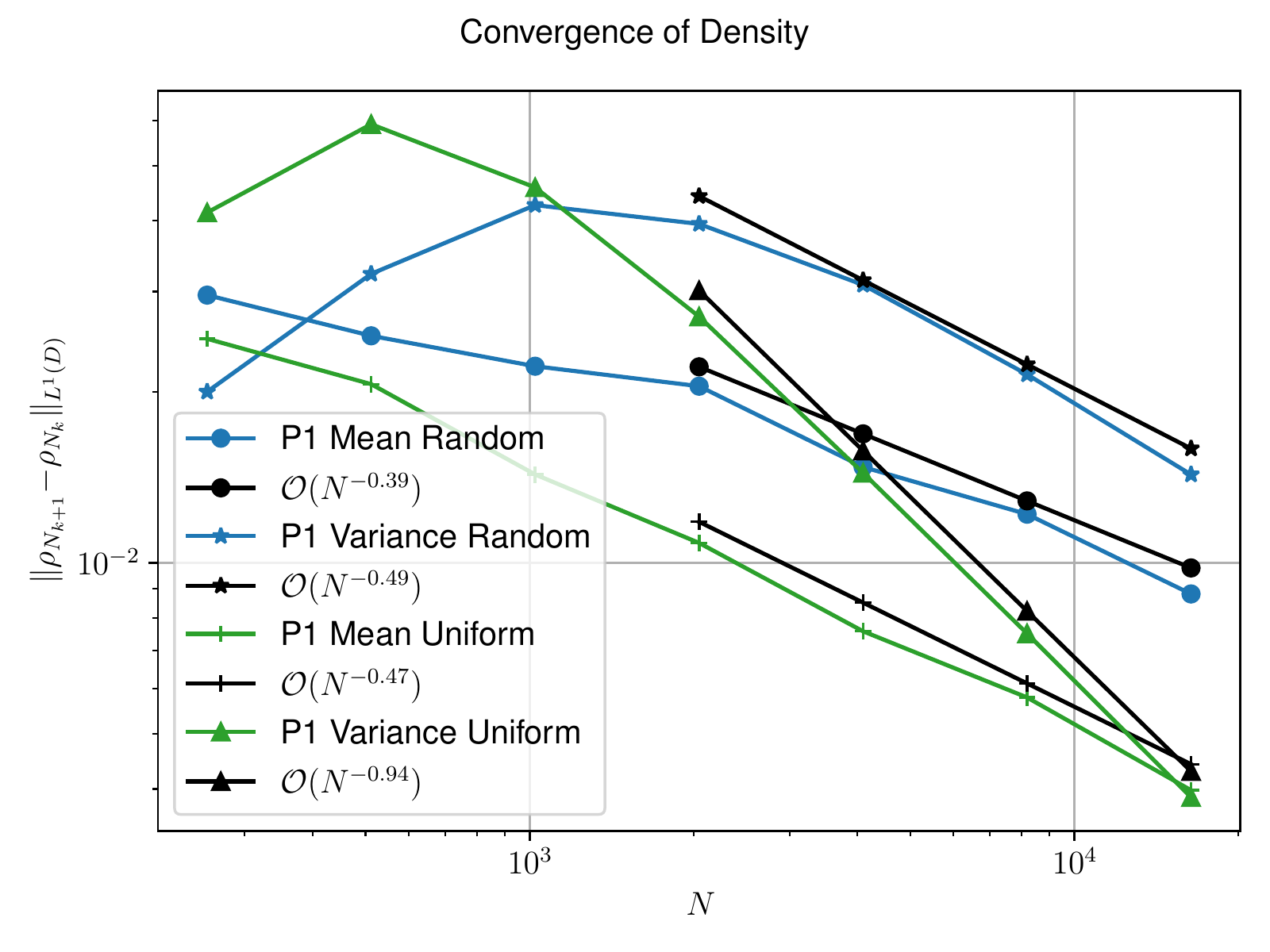}\\
\includegraphics[scale=\figsize]{\main/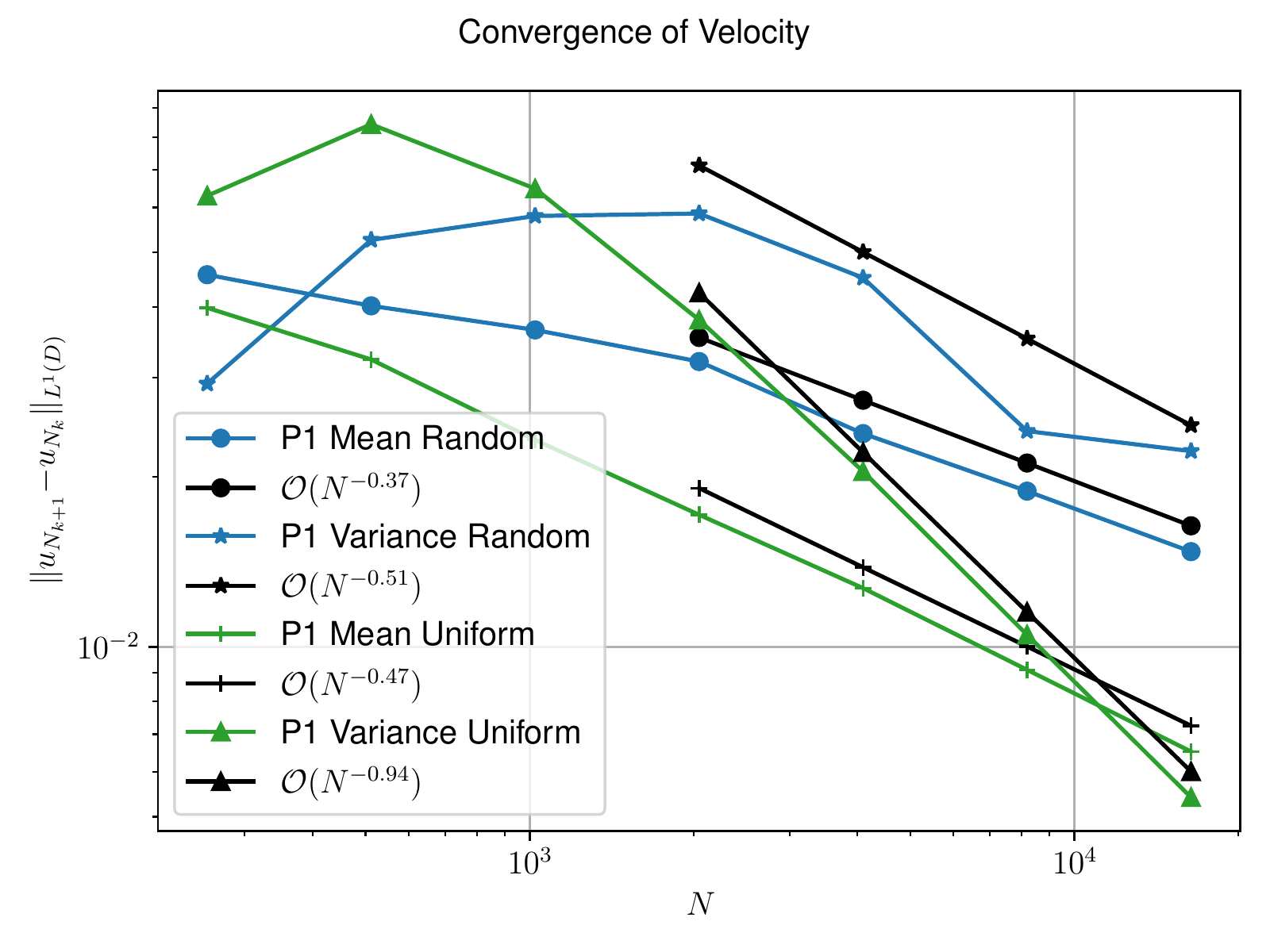}\,
\includegraphics[scale=\figsize]{\main/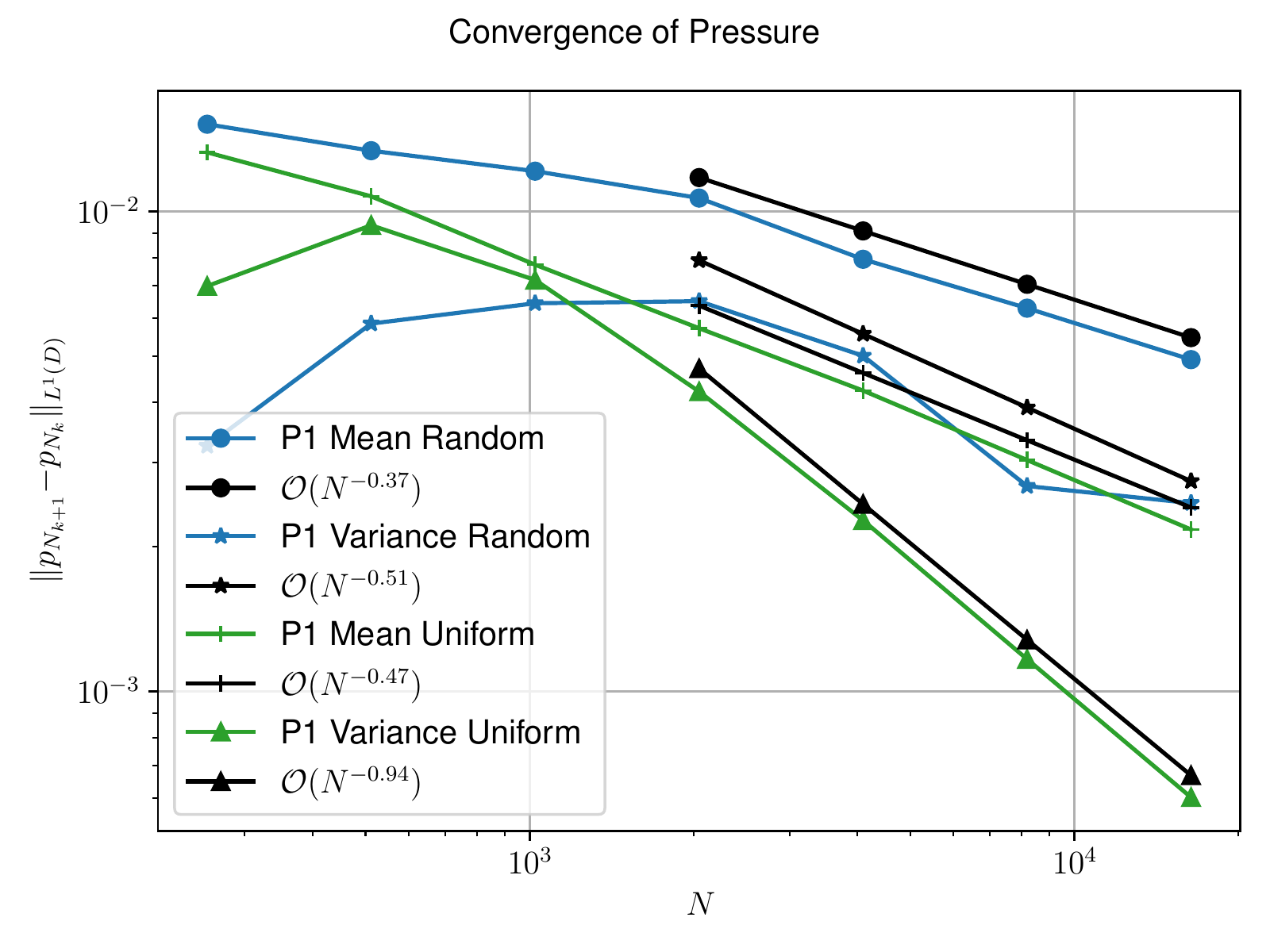}
\caption{Empirical convergence study for the mixtures in mechanical equilibrium under number of sub-volumes refinement. 
Only results of phase $1$ are presented to improve readability of the pictures; corresponding results for phase $2$ are analogous.}\label{Fig:AI:T1_conv}
\end{center}
\end{figure}
Plotted results suggests convergence of all the macroscopic quantities of interest under number of sub-volumes refinement for each strategy. Remarkably, the Random strategy seems to provide a slower converging sequence for both mean and variance, where this latter presents a slow-down of a factor $2$ as compared to the Uniform one. This is in complete accordance with the discussion performed in Section \ref{sec:AI:MG:volumes}.\\

Second, we investigate convergence of both strategies as the number of samples is increased: we run the same test using an increasingly higher number of samples $L_j = 2^{3+j}$ $j=0,\ldots,8$, all computed with $N=1000$ sub-volumes.
Results of corresponding Cauchy rates for all the variables in the log-log scale are presented in Fig.\ref{Fig:AI:T1_conv_sample}.
Expected order of convergence $\frac{1}{2}$ are recovered for both mean and variance for any sub-discretization strategy.
\begin{figure}[!htbp]
\begin{center}
\includegraphics[scale=\figsize]{\main/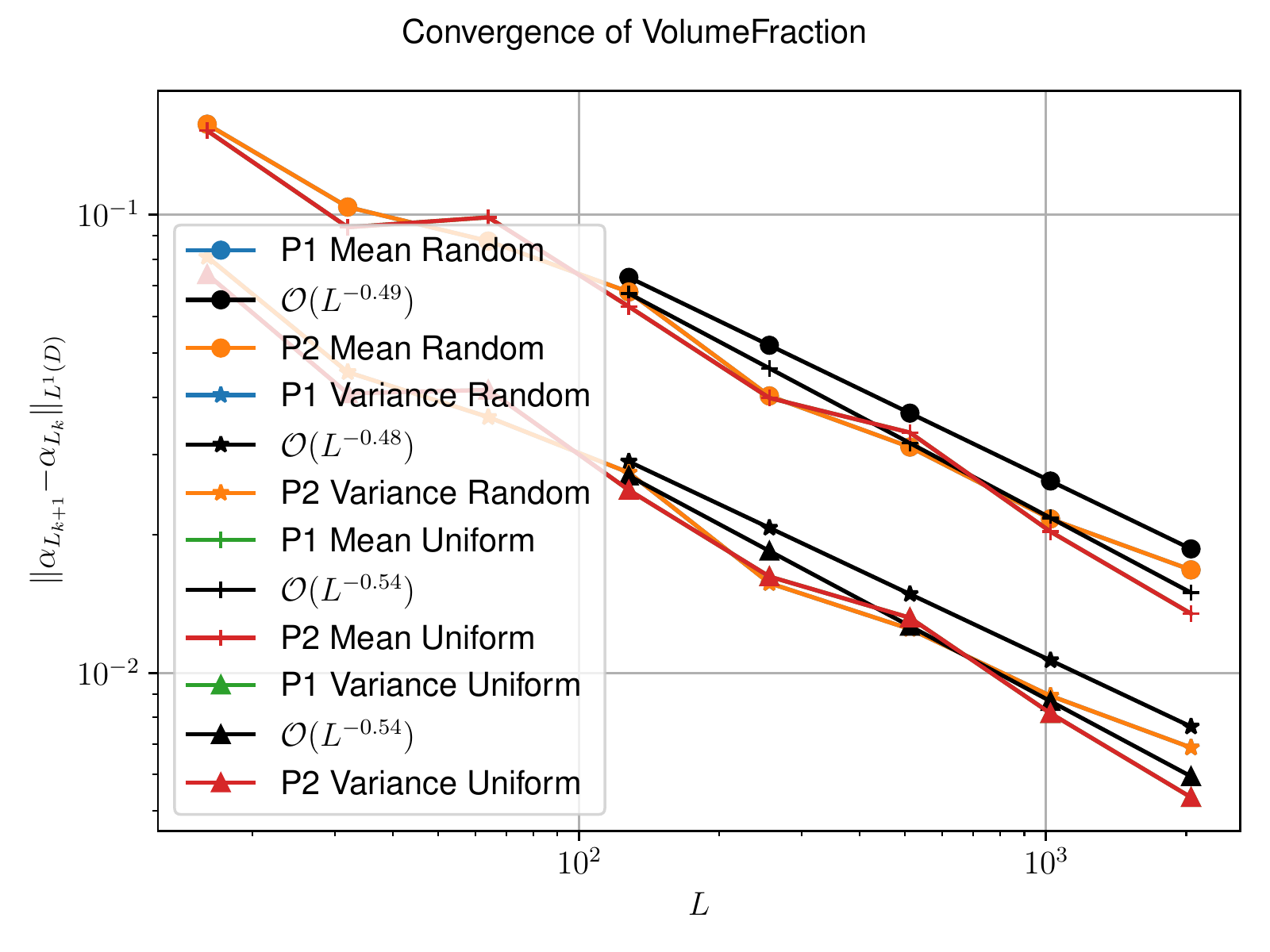}\,
\includegraphics[scale=\figsize]{\main/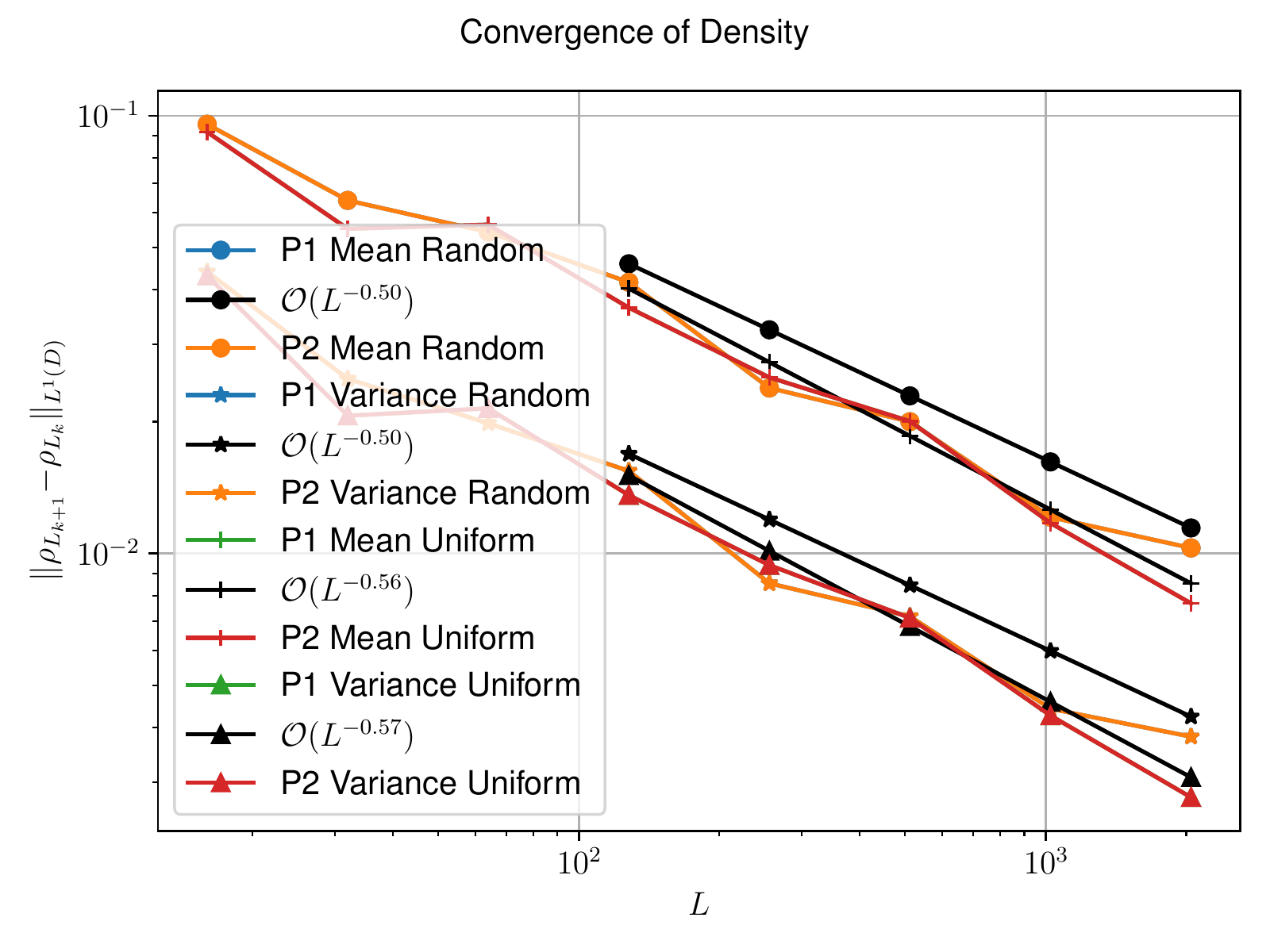}\\
\includegraphics[scale=\figsize]{\main/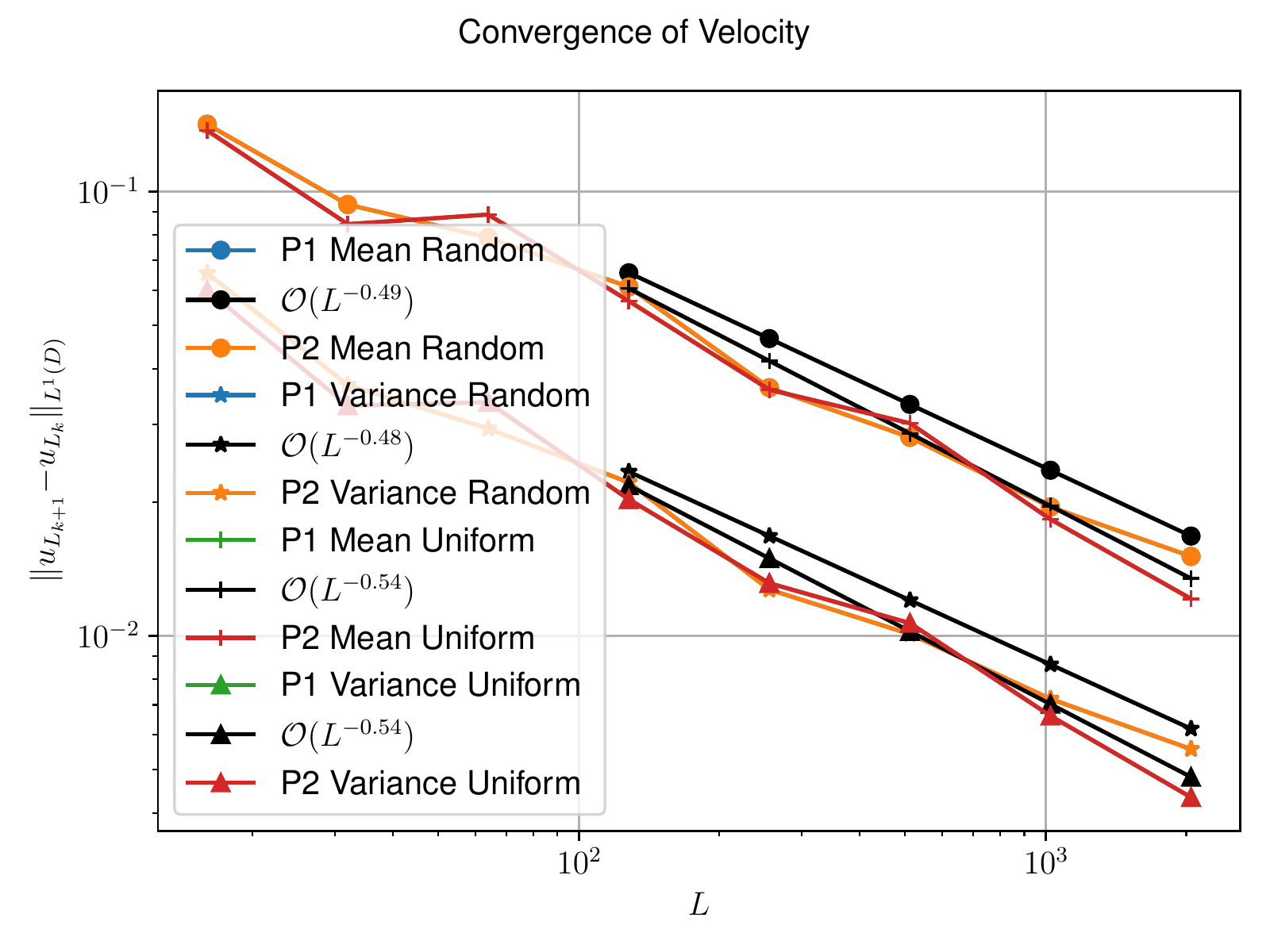}\,
\includegraphics[scale=\figsize]{\main/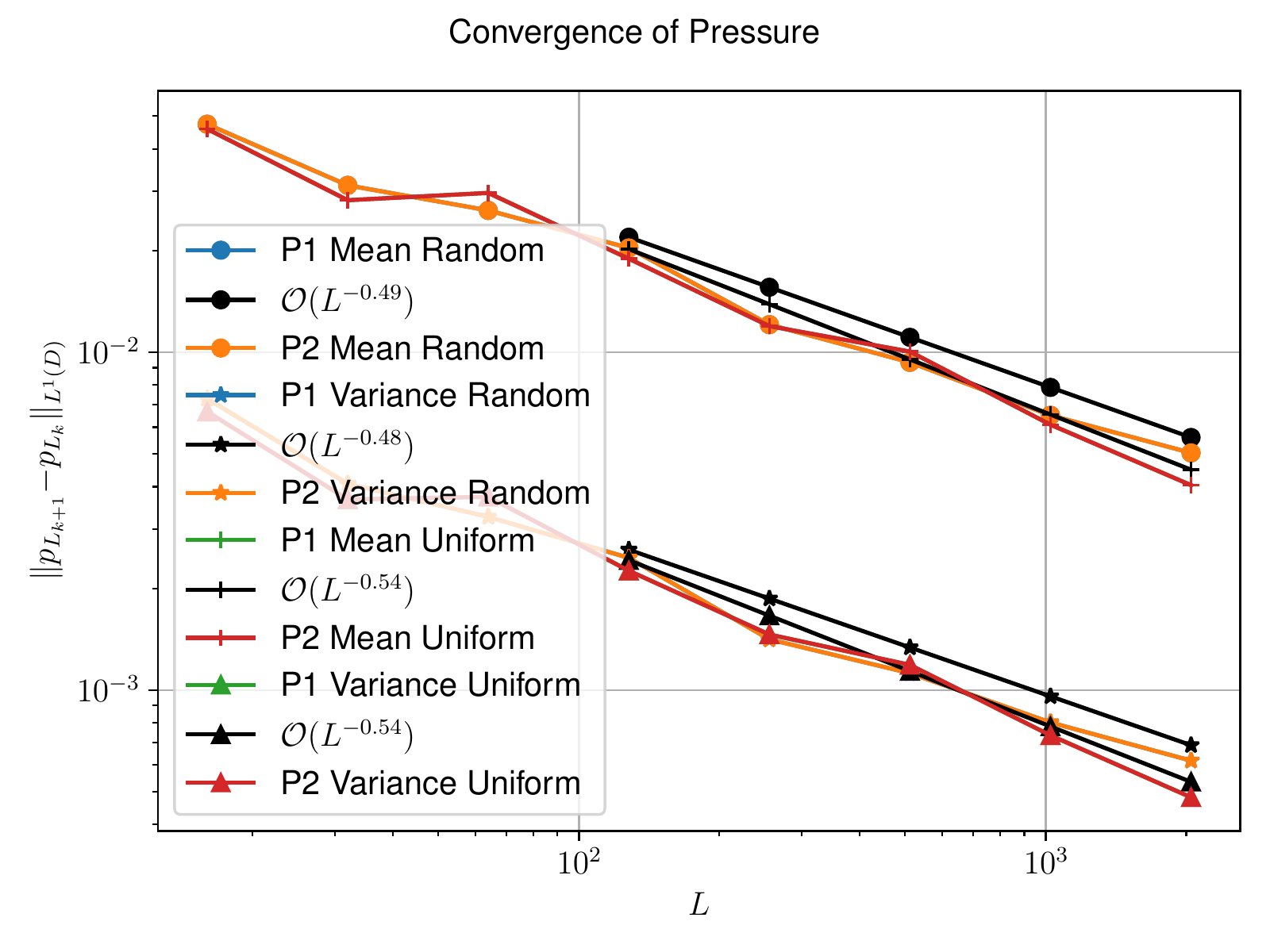}
\caption{Empirical convergence study for the mixtures in mechanical equilibrium under number of samples refinement.}\label{Fig:AI:T1_conv_sample}
\end{center}
\end{figure}
Notice that, again, computed results suggests that for given number of samples and macroscopic resolution, choosing the number of samples in a uniform manner is achieving better asymptotic properties and lower errors. This is the reason why in the following only the Uniform strategy will be considered.\\

Based on the current discussion, we observe that convergence under sub-scale refinement and number of samples is obtained, thus constituting an evidence that the strategy is stable. 
This, however, it does not provide any information about how to choose the corresponding (hyperparameter) number of sub-volumes $N\in\bbN$. 
Indeed, at the numerical level, the smaller the width of a phase component, the higher the number of interactions that the FT evolution operator needs to resolve, and thus the higher the computational demand. 
Hence, for computational efficiency it would be desirable to fix a possibly large number of sub-volumes. 
Unfortunately, this seems, even for (such) simple test cases, not possible: to demonstrate such claim we show via a convergence study that solutions computed on low sub-sale resolution cannot compute solutions generated using a fine sub-scale resolution, independently of the number of samples.
First, we fix the results generated using the parameters listen in Table \ref{Tab:AI:T1} as our target solution. 
Subsequently, we compute, for several sub-scale resolutions, predictions on a sequence of increasingly higher number of samples. 
Corresponding results are compared (in norm) to the first solution (computed on a fine sub-scale resolution), and convergence rates are reported in the loglog scale in Fig.\ref{Fig:AI:T1:err_vs_L_sub-scale}.
\begin{figure}[!htbp]
\begin{center}
\subfloat[$N=200$]{\includegraphics[scale=0.27]{\main/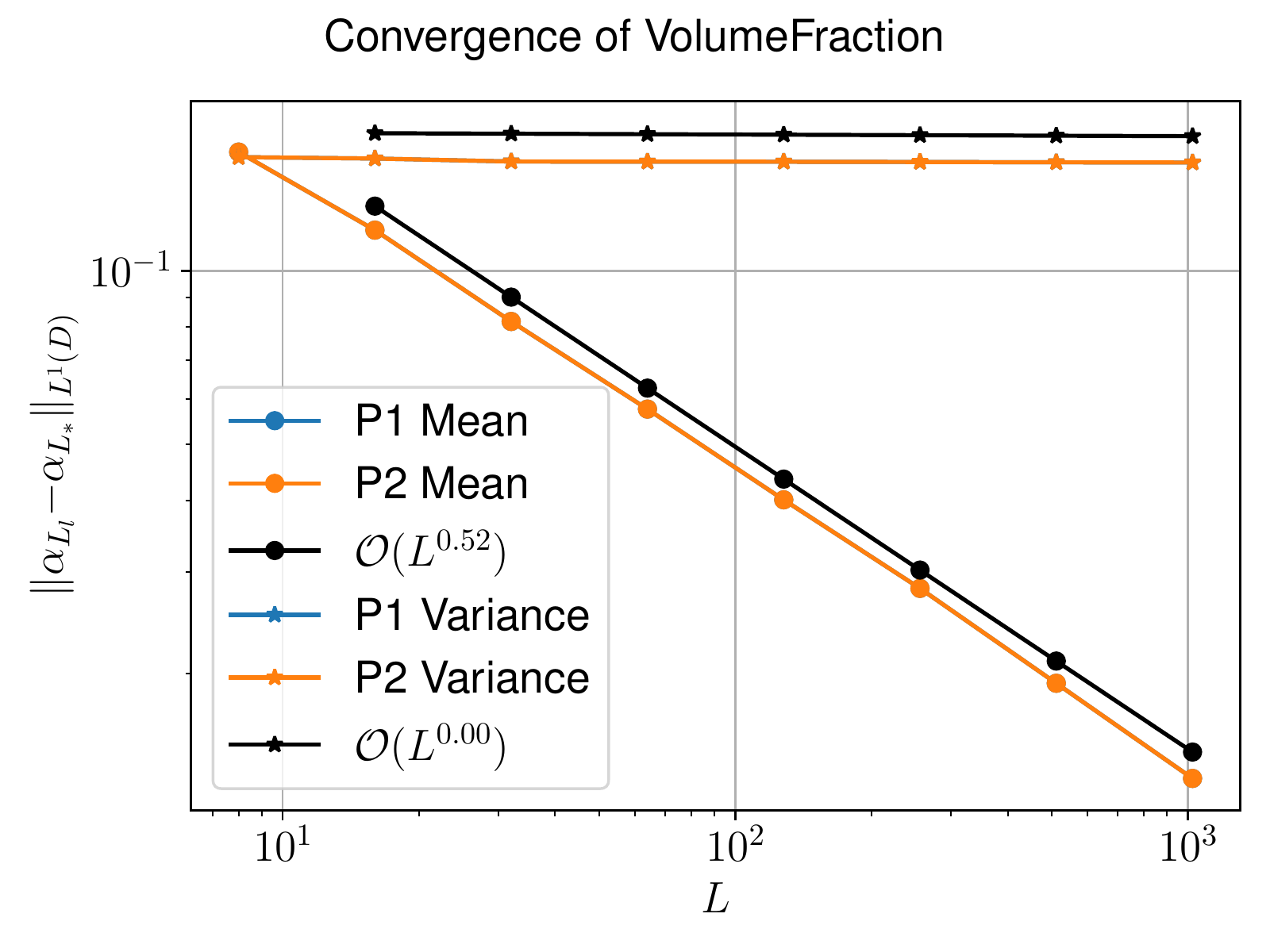}}\,
\subfloat[$N=800$]{\includegraphics[scale=0.27]{\main/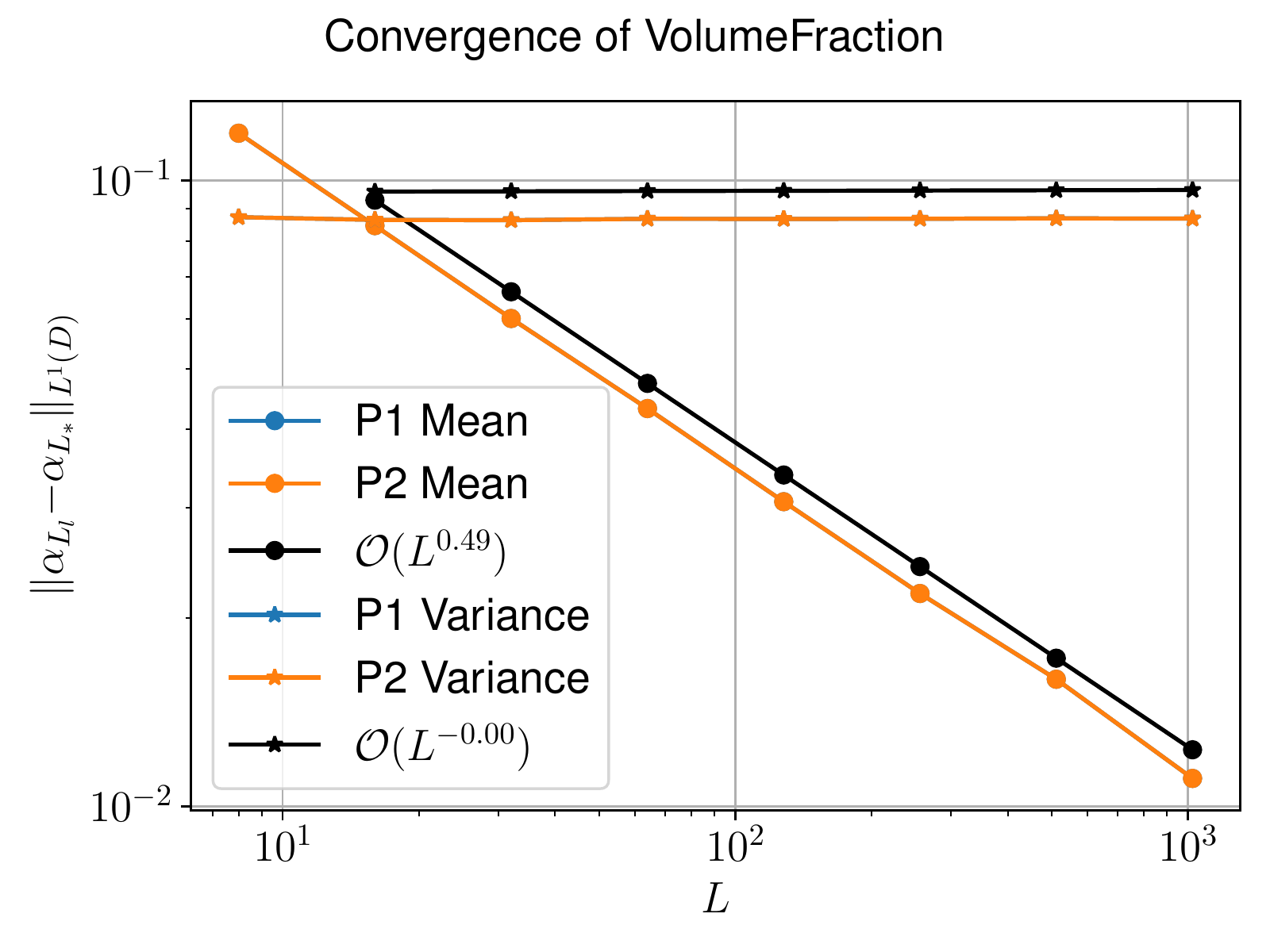}}\,
\subfloat[$N=3200$]{\includegraphics[scale=0.27]{\main/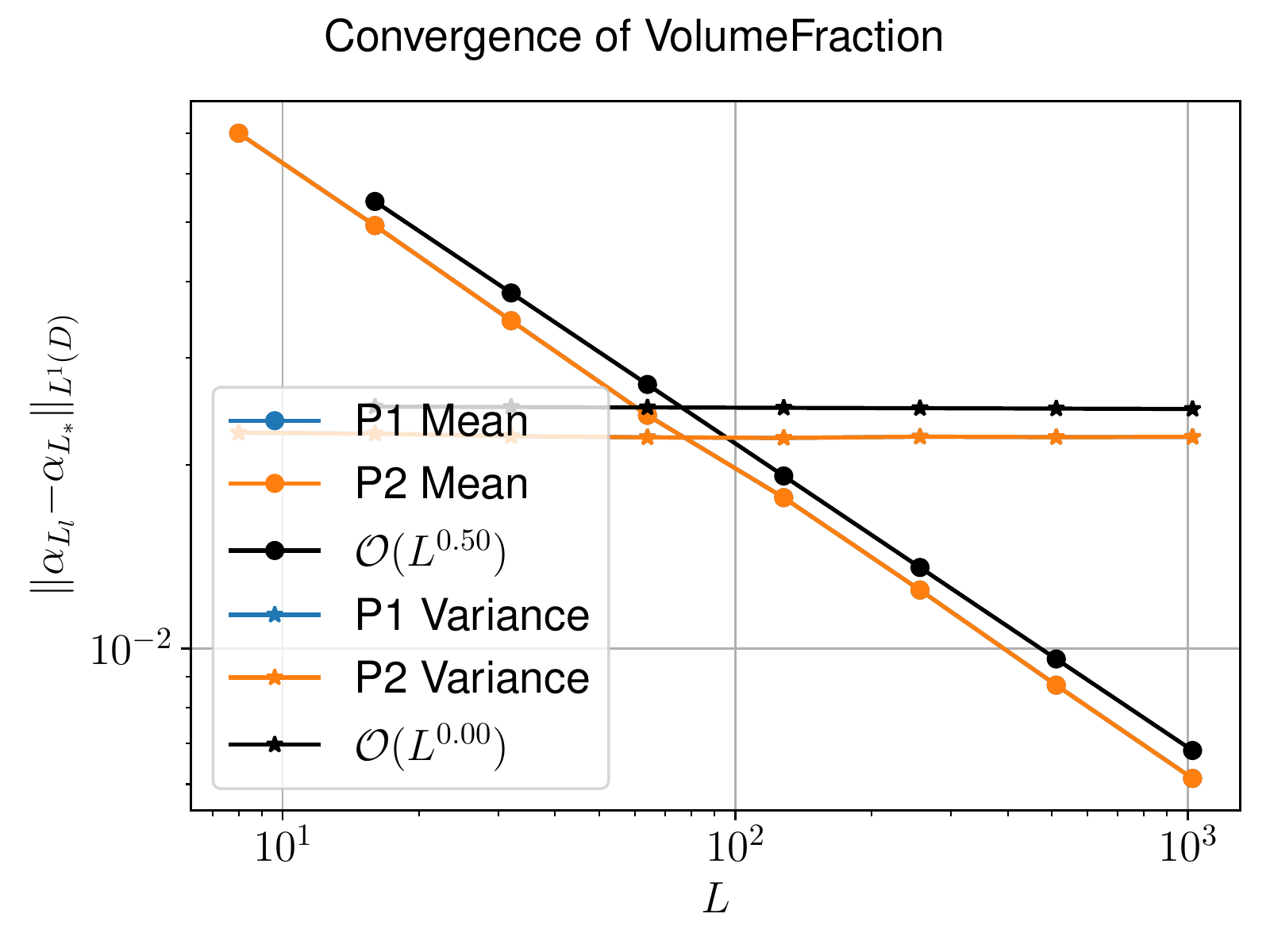}}
\caption{Empirical convergence studies for the mixture under uniform conditions: error-vs-number of samples ($L$), for several sub-scale resolutions.}\label{Fig:AI:T1:err_vs_L_sub-scale}
\end{center}
\end{figure}
Results show that expectations computed on different sub-scale refinements agree with each-other in the limit of an infinite number of samples.
Interestingly, the same conclusion does not hold true for the second moment. 
Indeed, a low level description of the microstructure introduces events that are not realizable via high level resolutions, and thus convergence does not take place. 
Conversely, refinement in the sub-scale resolution produces a decrease in variance, thus identifying as target solutions those computed in the limit of vanishing sub-scale resolution.

\begin{figure}
\begin{center}
\subfloat[$L=1,\,\,M=100$]{
\includegraphics[scale=0.25]{\main/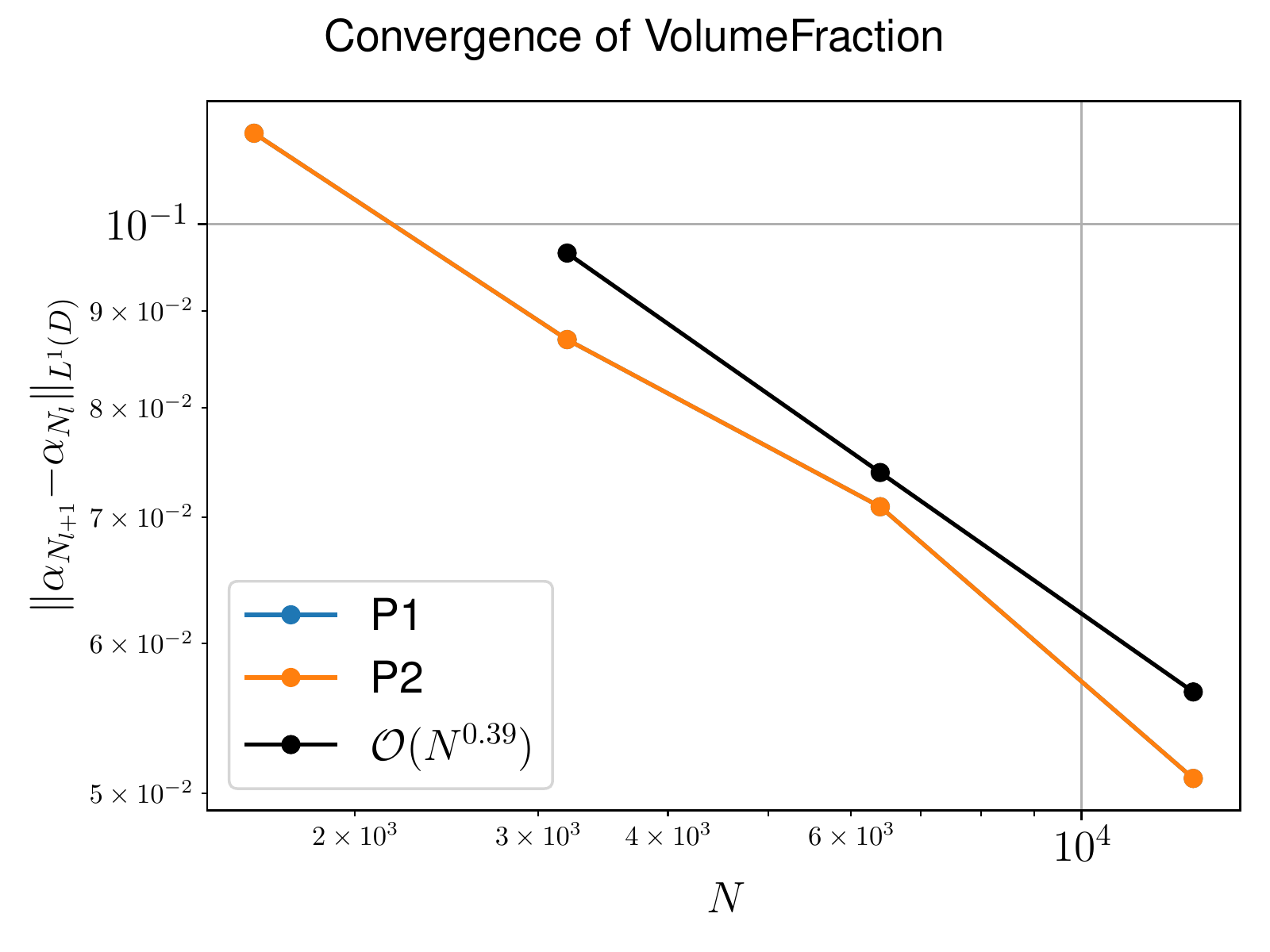}
}
\,
\subfloat[$L=100,\,\,M=100$]{
\includegraphics[scale=0.25]{\main/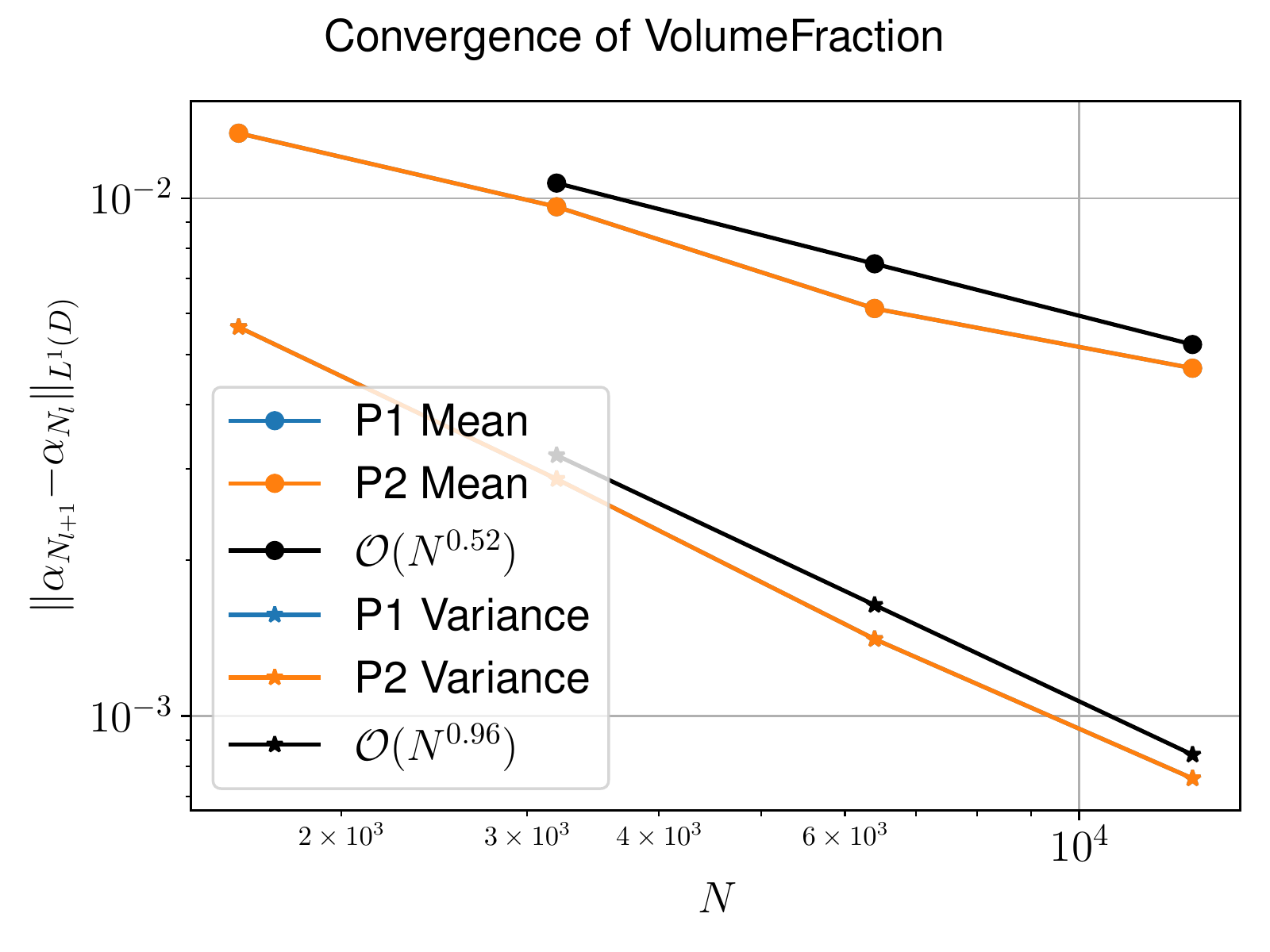}
}
\,
\subfloat[$L=1000,\,\,M=100$]{
\includegraphics[scale=0.25]{\main/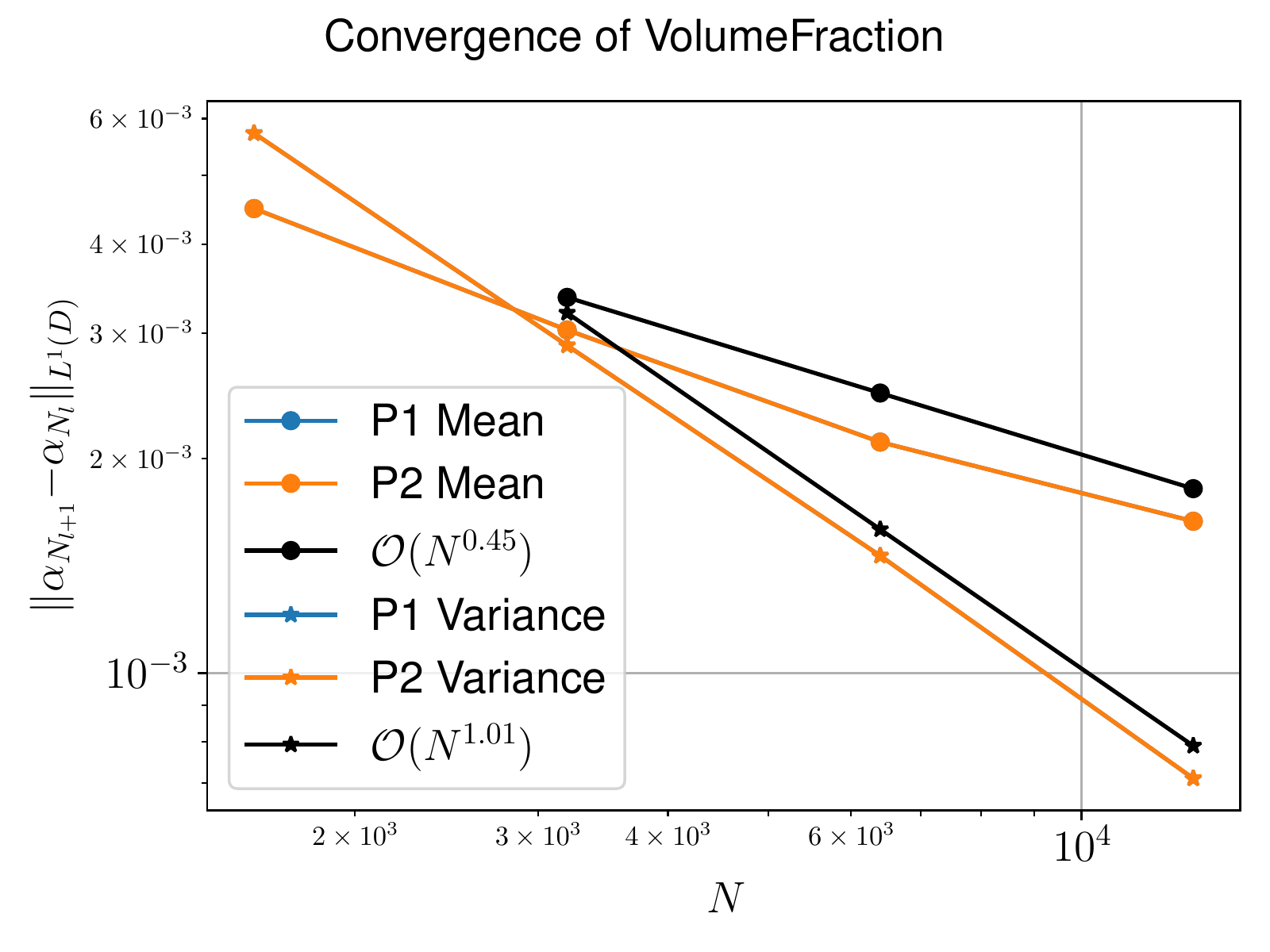}
}
\\
\subfloat[$L=1,\,\,M=500$]{
\includegraphics[scale=0.25]{\main/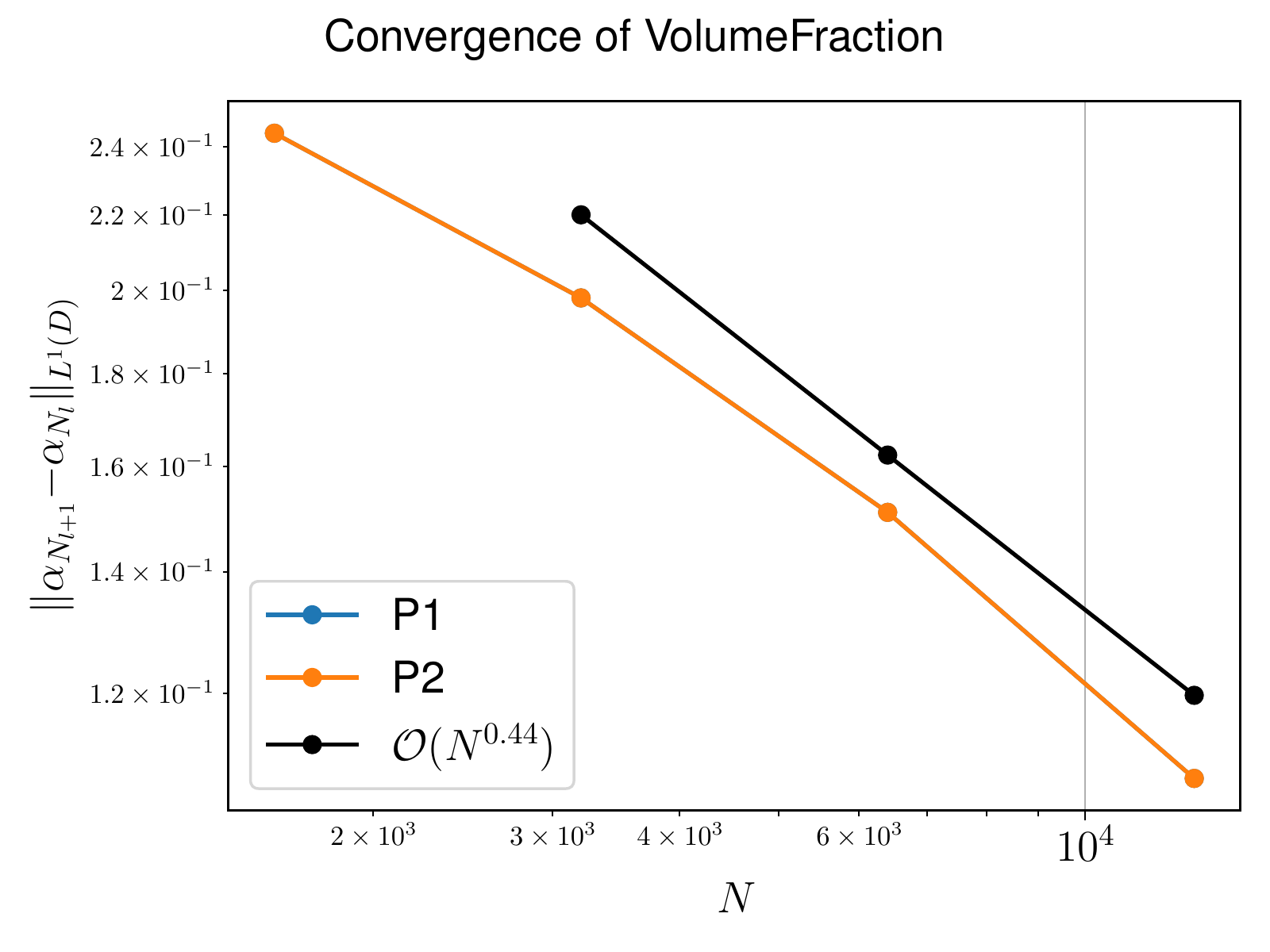}
}
\,
\subfloat[$L=100,\,\,M=500$]{
\includegraphics[scale=0.25]{\main/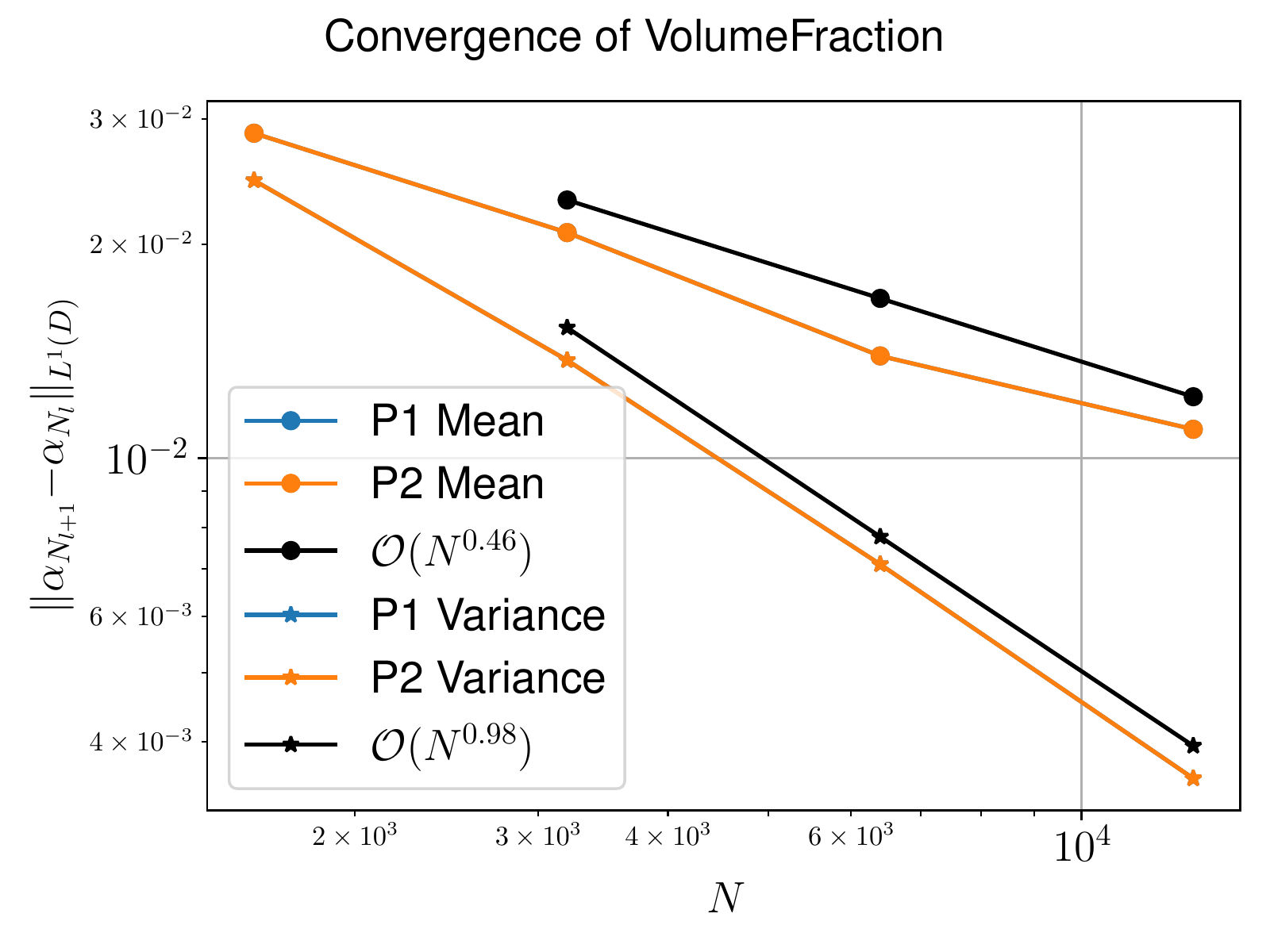}
}
\,
\subfloat[$L=1000,\,\,M=500$]{
\includegraphics[scale=0.25]{\main/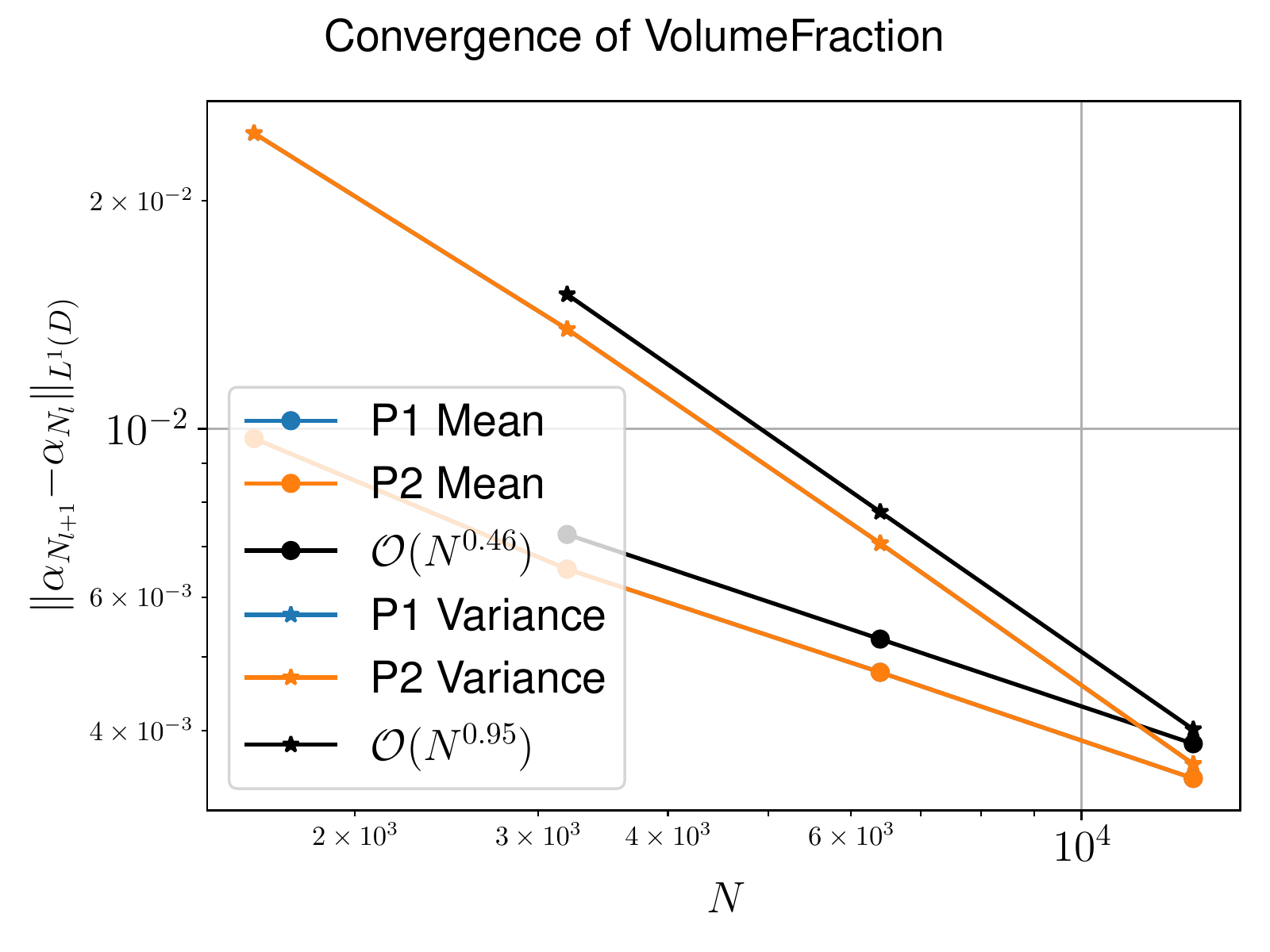}
}
\\
\subfloat[$L=1,\,\,M=1000$]{
\includegraphics[scale=0.25]{\main/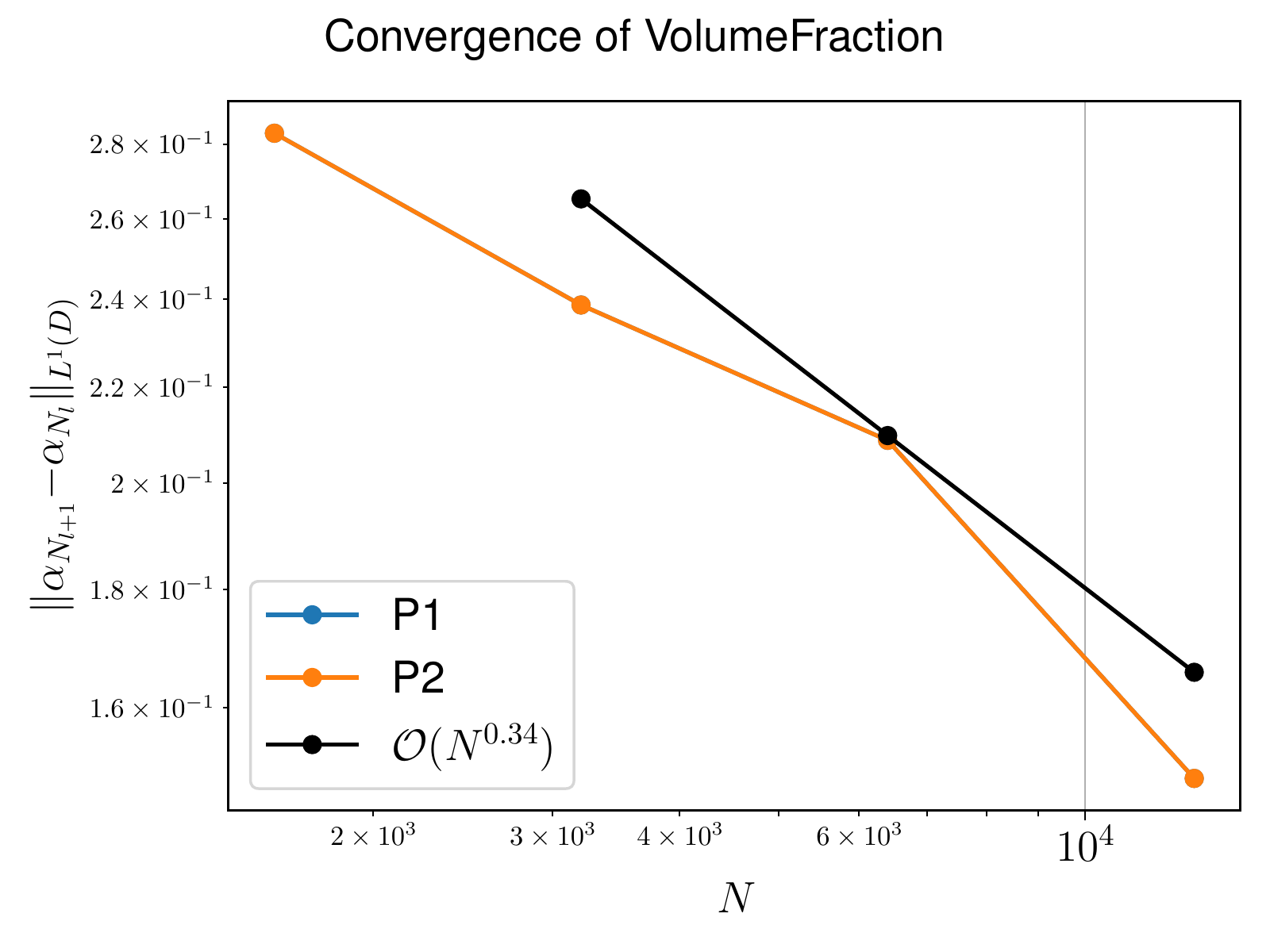}
}
\,
\subfloat[$L=100,\,\,M=1000$]{
\includegraphics[scale=0.25]{\main/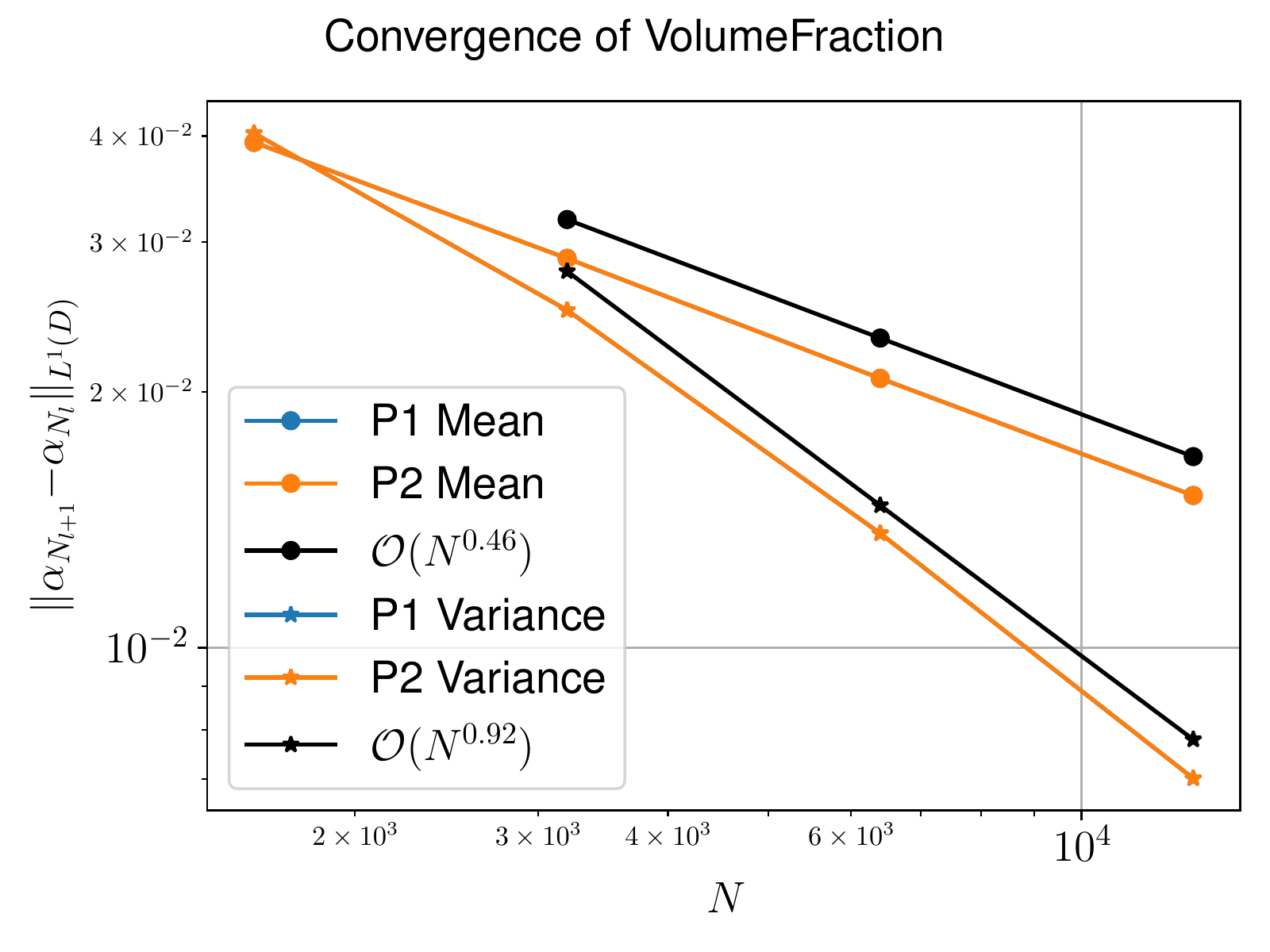}
}
\,
\subfloat[$L=1000,\,\,M=1000$]{
\includegraphics[scale=0.25]{\main/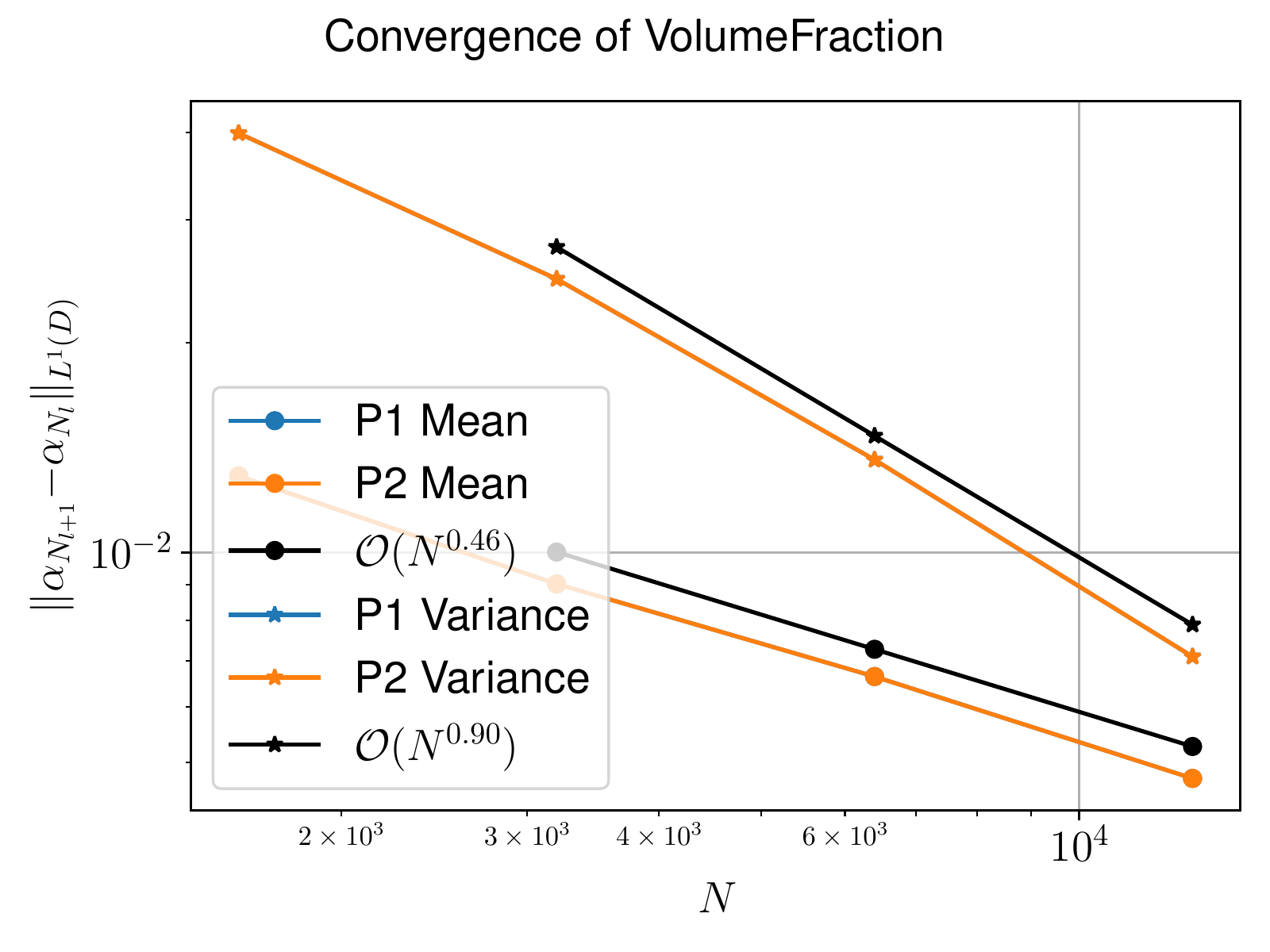}
}
\\
\subfloat[$L=1,\,\,M=5000$]{
\includegraphics[scale=0.25]{\main/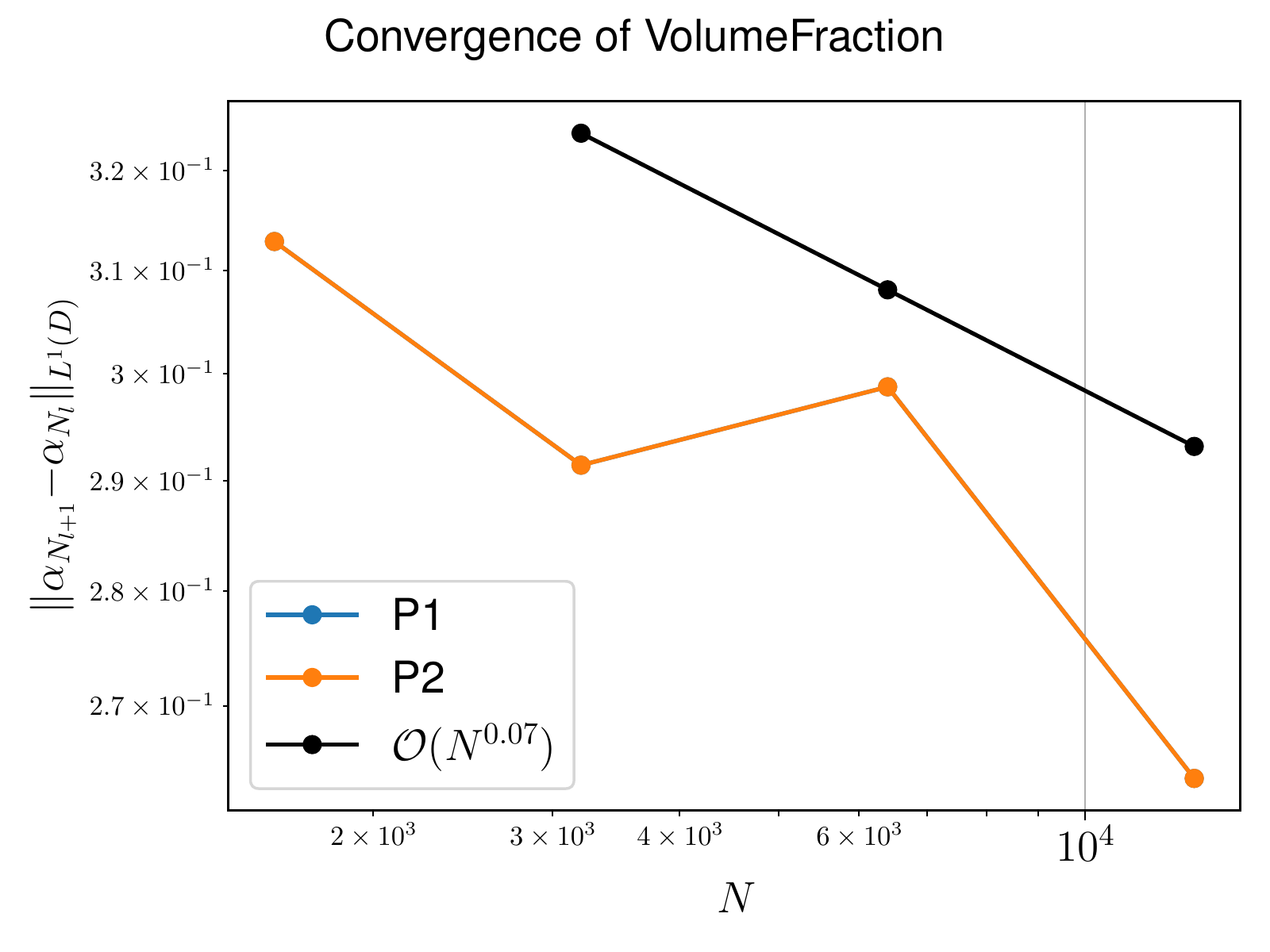}
}
\,
\subfloat[$L=100,\,\,M=5000$]{
\includegraphics[scale=0.25]{\main/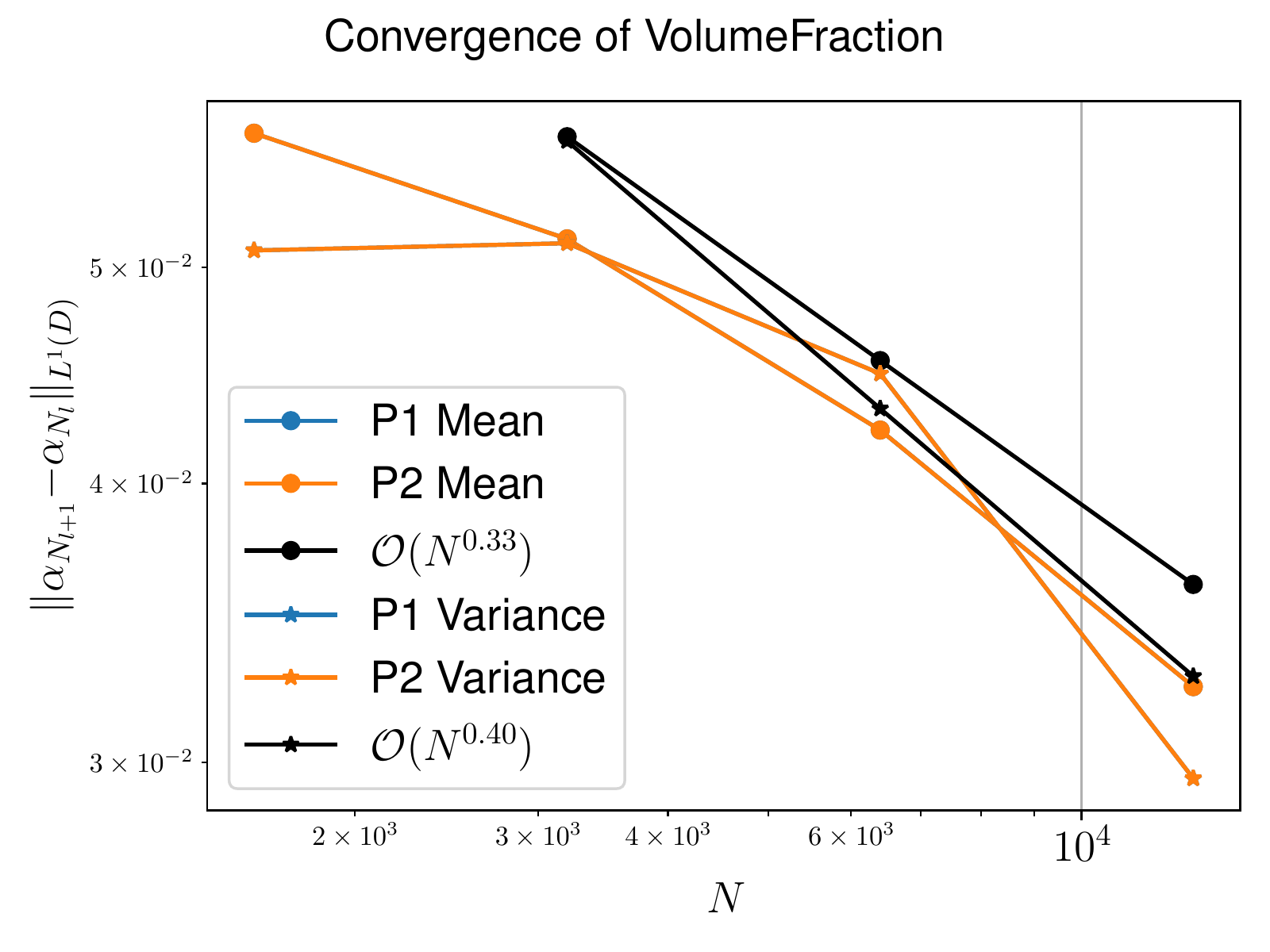}
}
\,
\subfloat[$L=1000,\,\,M=5000$]{
\includegraphics[scale=0.25]{\main/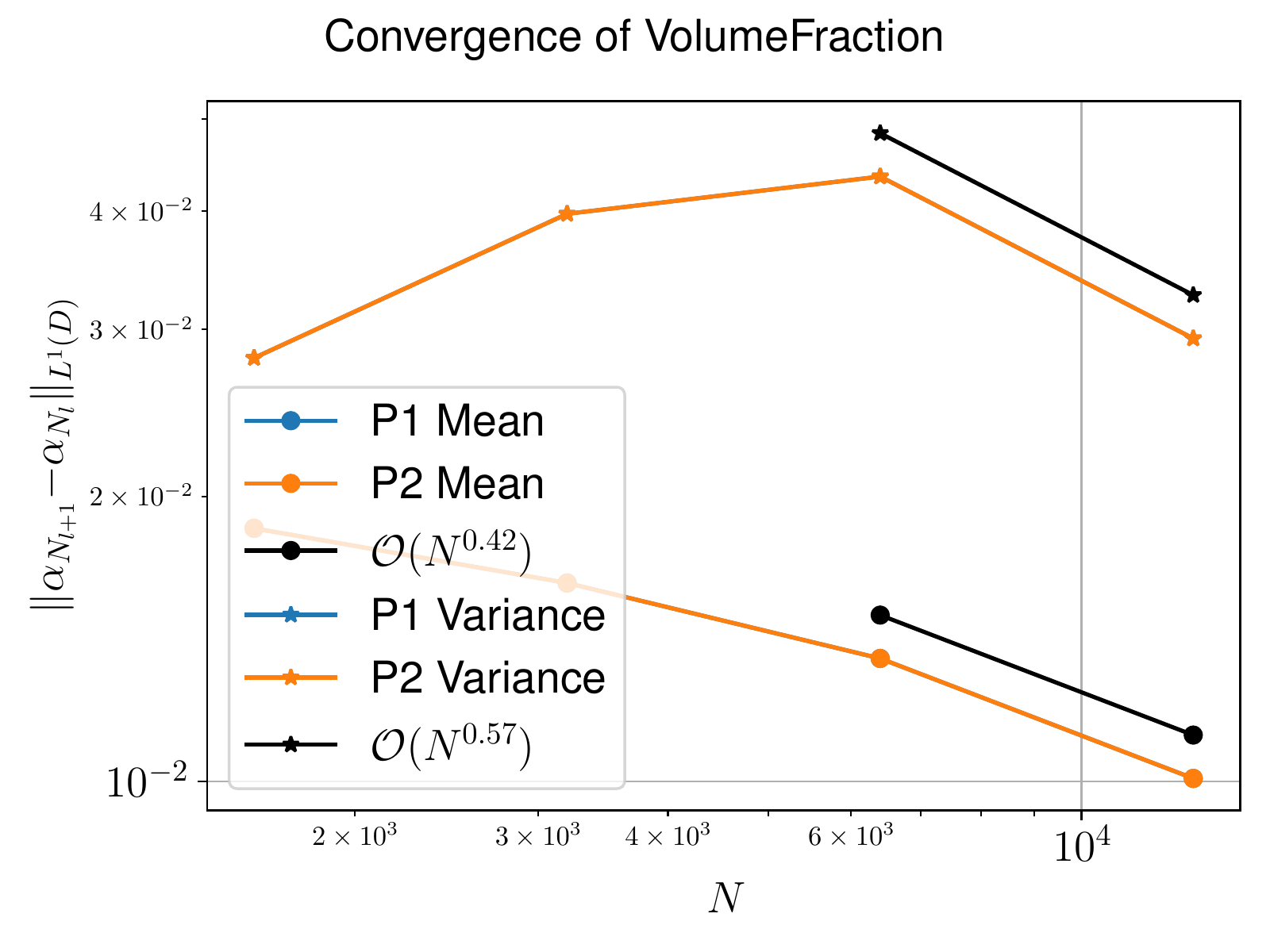}
}
\caption{Empirical convergence study for the test case under uniform conditions: bi-variate analysis for physical mesh resolution ($M$) and number of samples ($L$), as the sub-scale resolution is refined ($N\rightarrow\infty$).}\label{Fig:AI:T1_conv_bivariate}
\end{center}
\end{figure}

We conclude this test case by investigating the influence that the variation of physical mesh and sub-scale resolution have on macroscopic quantities. 
Indeed, based on the above discussion, it is clear that, for fixed physical mesh and number of samples, convergence is taking place as the sub-scale resolution is refined. \\
As discussed in \cite{Petrella22FT}, for a given two-phase distribution, convergence under physical mesh refinement is obtained, and thus for a single realization generated using Alg.\ref{al:AI:eMGP}.\\
Conversely, it is not clear what happens when one refines both physical and sub-scale resolutions. To this extent, we run the present test case for several physical mesh resolutions and using several number of samples. 
Mesh convergence study are then performed in the limit of vanishing sub-scale resolution. Corresponding plots in the loglog scale are reported in Fig.\ref{Fig:AI:T1_conv_bivariate}.
First let us consider the first column (on the left), about convergence for a single sample for several physical resolutions as the sub-scale is refined. 
Notice that, for low number of physical volumes $M\in\bbN$, convergence under sub-scale refinement is preserved at a rate approximately $1/2$, which doubles for variance.
Interestingly, as the physical mesh increases, the convergence rate gets deteriorated up to the point of showing no convergence (see rate in Fig.\ref{Fig:AI:T1_conv_bivariate}j). 
This is reasonable as samples generated using a fixed sub-scale resolution do not allocate dispersed matter in the same locations of those constructed using finer resolutions: samples generated using $N$ sub-volumes distribute phases in different places as compared to those generated with $2N$ sub-volumes. 
Thus, doubling the sub-scale resolution \emph{does not} provide more accurate results on the same sample.\\
For sufficiently fine physical resolution, one capture such discrepancy across sub-scale resolutions and convergence is clearly not recovered, as expected.
Nevertheless, such lack of convergence seems to disappear as the number of samples increases (see Fig.\ref{Fig:AI:T1_conv_bivariate}k and l):
this is essentially due to the fact that any event generated using $N$ sub-volumes can be generated using $2N$ volumes, thus making ensemble contain similar realizations at different sub-scale resolutions.
In turn, the distance in norm between the two levels of description reduces.\\
Interestingly, convergence with a fixed rate can be appreciated independent of the number of samples for any resolution (first three rows of Fig.\ref{Fig:AI:T1_conv_bivariate}). 
Such an outcome seems to be due to the submersion of the sub-scale grid (defined by interfaces) into the physical mesh, which clearly introduces an homogenization at the macroscopic level. This is further confirmed by the scarce convergence visible in the last row of Fig.\ref{Fig:AI:T1_conv_bivariate} as long as the physical mesh is finer than the sub-scale one (i.e. $N>M$).
This should not be regarded as a failure of the procedure, but rather as the result of an inversion in the hierarchy of the two scales: each sample is resolved in great detail, and more and more details are added as the number of samples increases, leading to an increase in variance, which then causes lack of convergence. 
This is in complete contrast with what happens once the proper hierarchy of scales is reestablished by increasing the sub-scale resolution.
Note how this "order" of hierarchies is actually imposed by the initial condition, which defines an expectation over the entire domain. 
By discretizing the volume fraction on a very fine grid, one defines volume-wise an expected value, which can be met only by further refining the (already) fine volumes. This necessarily introduces the need to generate regimes with a finer resolution than that used for the physical mesh.\\
Such observations suggest that the ab-initio method will produce convergent results only if the sub-scale is submerged in the physical one (as required by physics), and the use of very fine physical meshes imposes the need of a very large number of sub-volumes. 
Furthermore, by taking into account the discussion on using only sub-volumes with big size, one concludes that in order to obtain reliable results, dispersion at small-scale is unavoidable.

\subsubsection{Relaxation towards equilibrium}

As it is well-known, a primary characteristic of two-phase flow phenomena is the fact that both phases moves macroscopically with a single pressure and single velocity. Such phenomenon is the outcome of the smaller time scale at which the two phases are exchanging energy through the interface as compared to the system time scale. 
Typically, at the numerical level this is simulated via the use of stiff-mechanical relaxation terms which force the two phases to achieve the desired uniform mechanical conditions. 
In \cite{Abgrall&Saurel}, it was firstly constructed an explicit relation between the number of interfaces and the parameters that controls such relaxation process.
In this section we aim at constructing a test case for the analysis of the relaxation between phases: we consider a domain filled with two gas with uniform volume fractions and densities, but at different pressures. 
The corresponding initial condition reads
\[
\VV(x,0)
=
\begin{bmatrix}
\alpha^{(1)}_0 = 0.9\\
\VW^{(1)}_0\\
\alpha^{(2)}_0 = 0.1\\
\VW^{(2)}_0
\end{bmatrix}
\qquad\qquad
\VW^{(1)} =
\begin{bmatrix}
1\\
0\\
1
\end{bmatrix}
\qquad
\VW^{(2)} =
\begin{bmatrix}
0.125\\
0\\
0.1
\end{bmatrix}
\]
Triggered by the pressure difference, both phases start to exchange energy through the interfaces as to equilibrate the respective states.
Moreover, due to the uniform density, no variation but the one resulting from interfacial exchanges is involved.
Thus, the corresponding solutions for this problem are the time parametrized, (space-)constant states
$\VV^{(k)}(t) = [\alpha^{(k)}(t),\rho^{(k)}(t),u^{(k)}(t),p^{(k)}(t)]$
with $k=1,2$.
\begin{table}
\begin{center}
\begin{tabular}{c||c|c|c|c|c|c|}
&
$M$ & $\tilde{N}$ & $\mathrm{CFL}$ & $\delta^{(1)}$ & $\delta^{(2)}$ & $L$\\
\hline
\hline
\textbf{Ab-initio} & $1000$ & $2000$ & $0.9$ & $0.1$ & $0.1$ & $1000$\\
\hline
$\textbf{r}\equiv \bm{0}$ & $10000$ & $-$ & $0.9$ & $-$ & $-$ & $-$\\
\hline
$\textbf{r}\equiv \bm{1}$ & $10000$ & $-$ & $0.9$ & $-$ & $-$ & $-$\\
\hline
\end{tabular}
\end{center}
\caption{Parameters of computed solutions reported in Fig. \ref{Fig:AI:T2} }\label{Tab:AI:T2}
\end{table}
We first run the present test with the parameters summarized in Table \ref{Tab:AI:T2}, using a resampling strategy based on $N=100$ steps, and, for comparison, we also run the DEM scheme of \ref{Petrella22FT}. 
Notice that, for the DEM, single-pressure and single-velocity is achieved after the first step (due to the application of the relaxation step), as opposed to the ab-initio formulation.
Results for each scheme are presented in Fig.\ref{Fig:AI:T2}.\\
\begin{figure}[!htbp]
\begin{center}
\includegraphics[scale=\figsize]{\main/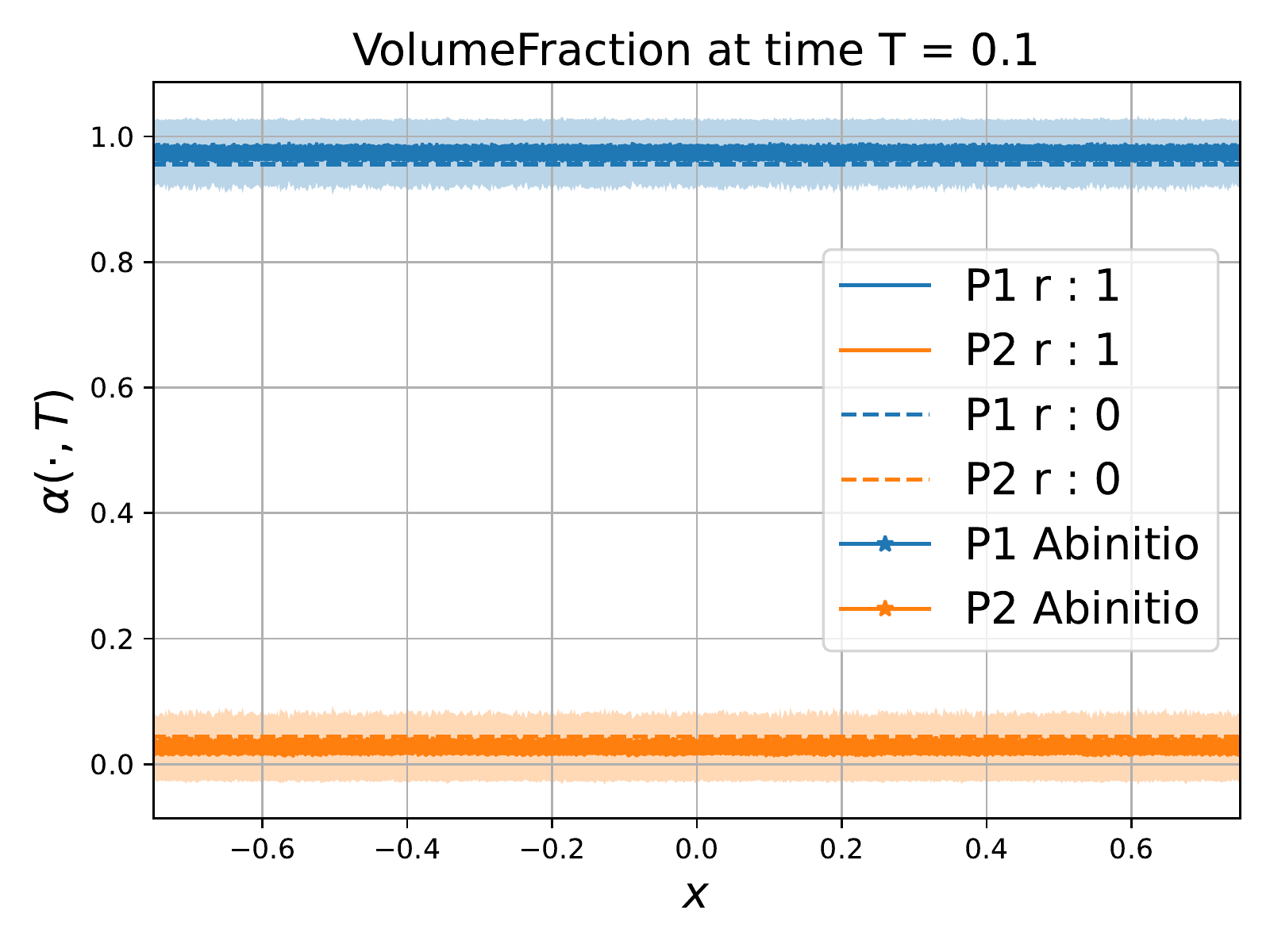}\,
\includegraphics[scale=\figsize]{\main/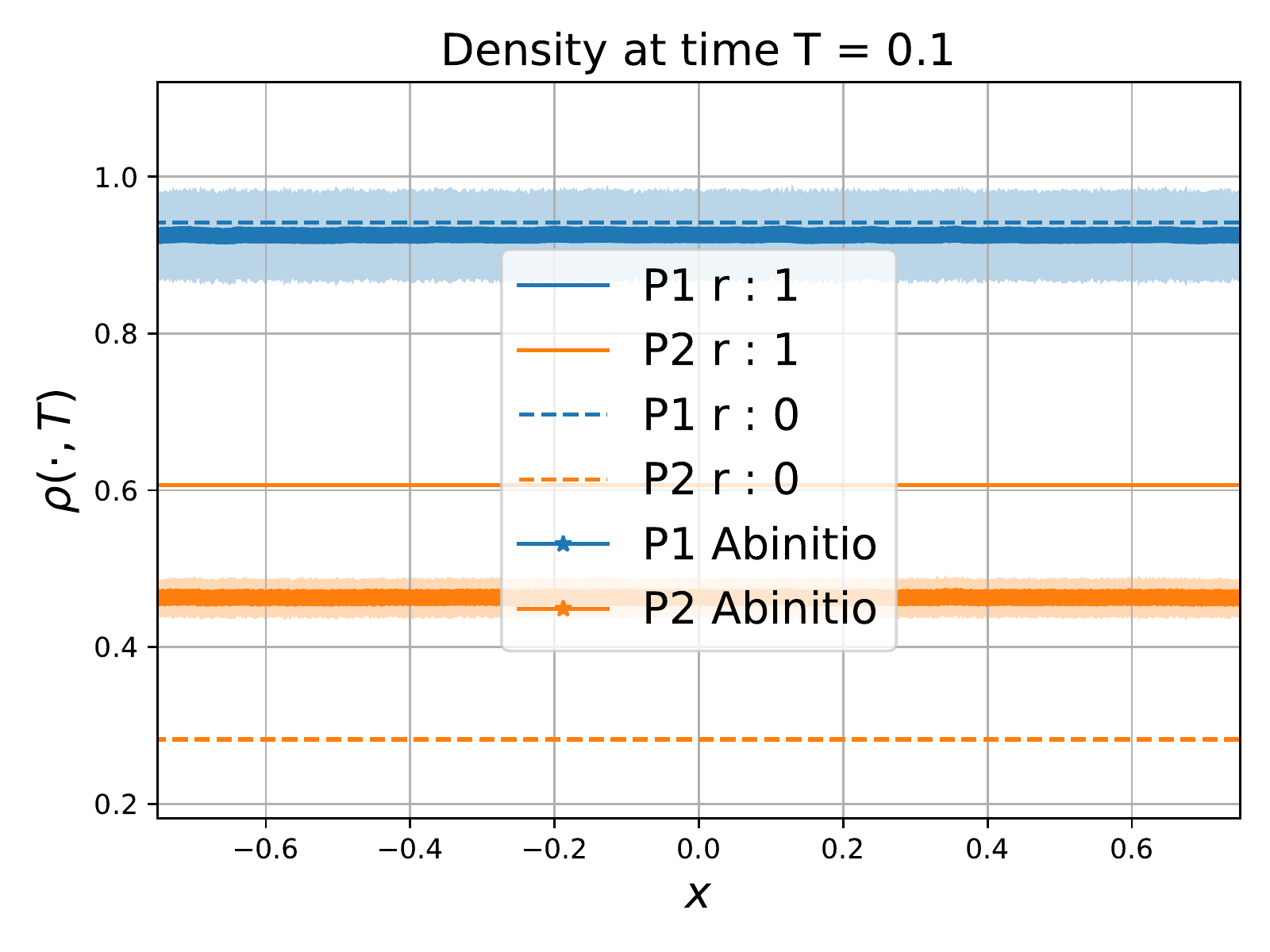}\\
\includegraphics[scale=\figsize]{\main/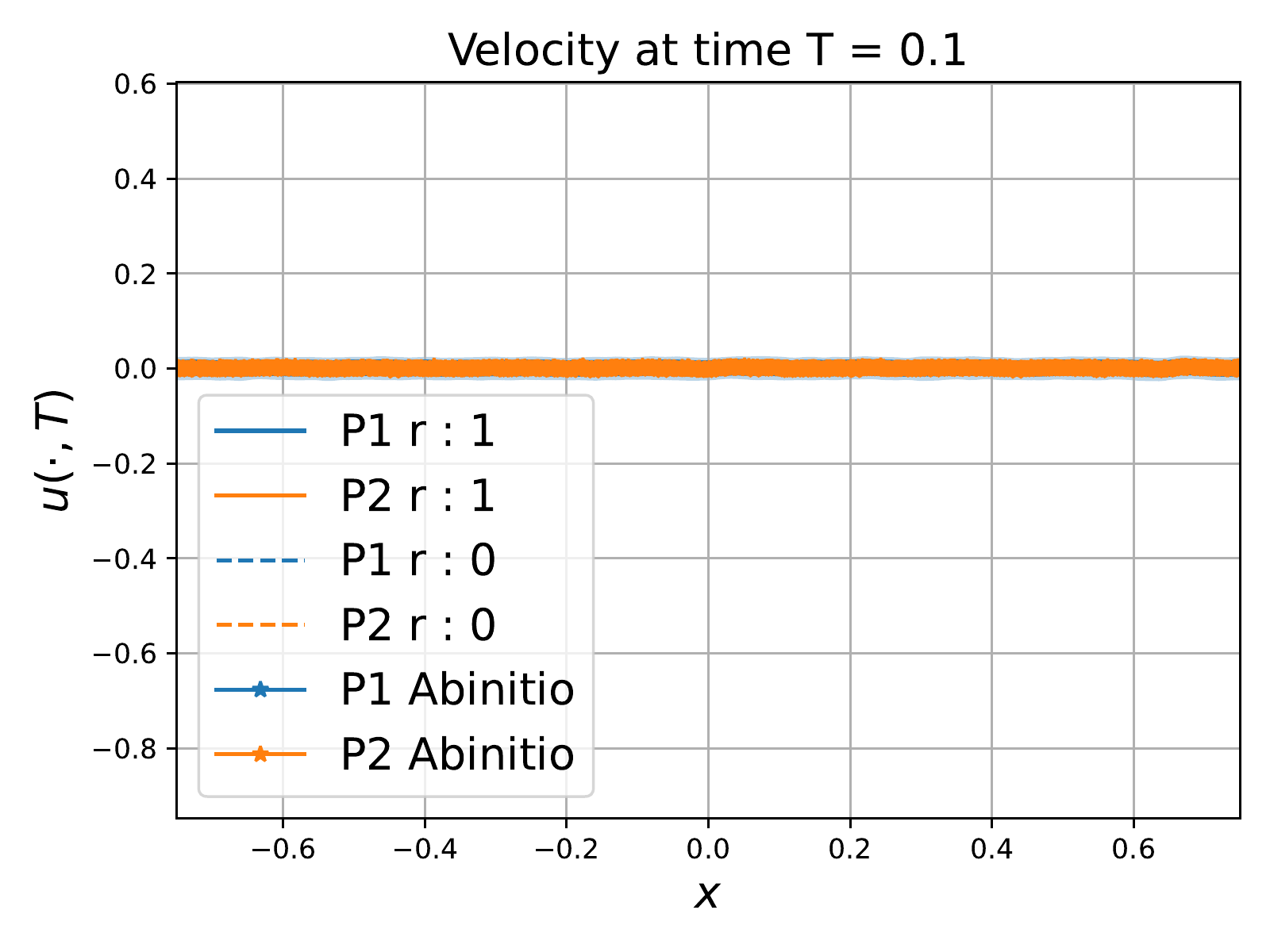}\,
\includegraphics[scale=\figsize]{\main/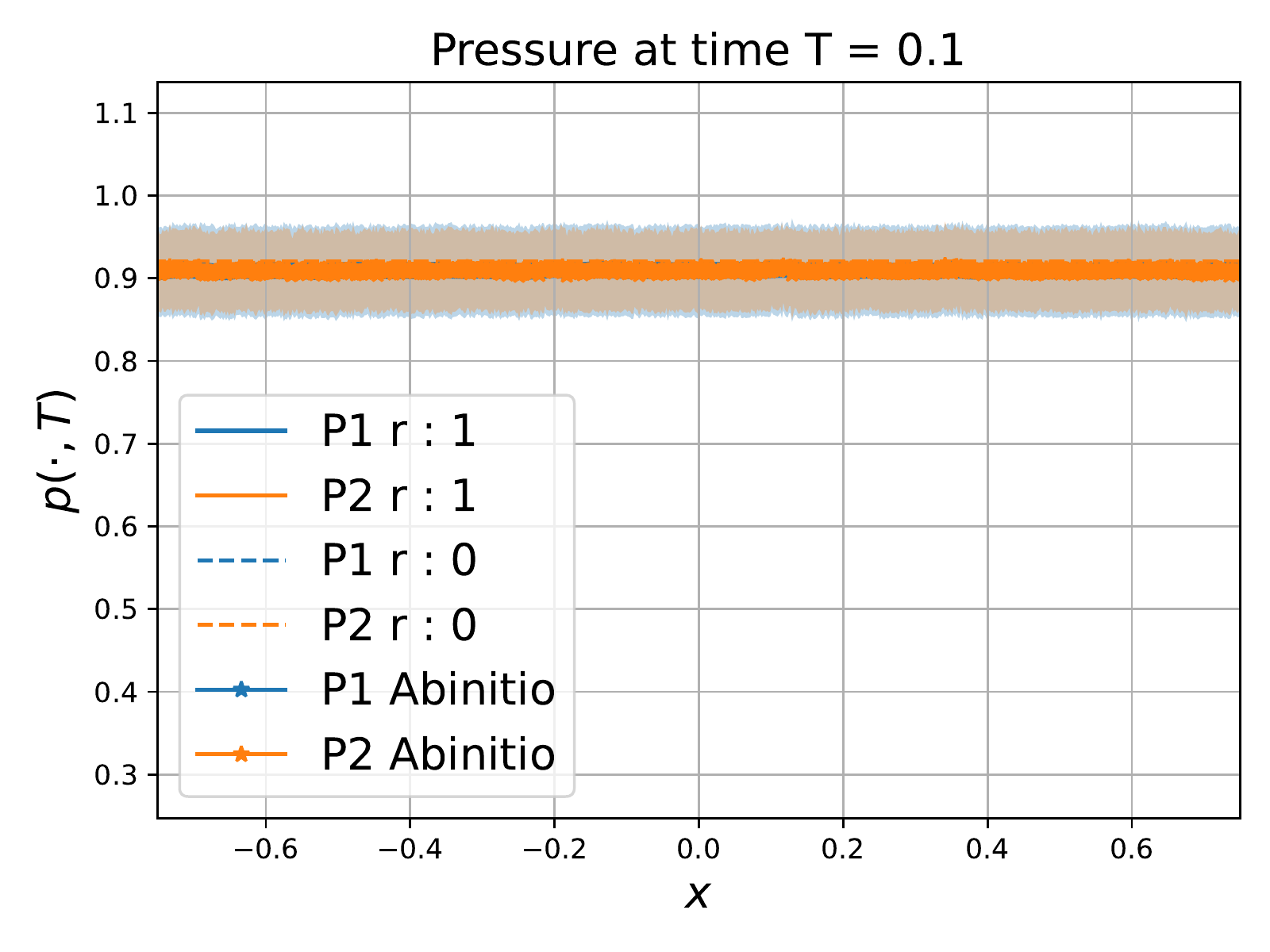}
\end{center}
\caption{Computed results for the test case about relaxation phenomena.
Solutions for the MC-based (Abinitio) strategy are presented along with the DEM for two constant choices of the parameter $r$ ($r\,:\,$). Different phases ($P\,:\,$) are highlighted using different colors.}
\label{Fig:AI:T2}
\end{figure}
First, notice that uniform conditions for each quantity of interest is recovered in the ab-initio simulation: this is in principle not trivial, since no relaxation strategy is employed by the MC-based method. 
Indeed, each sample is evolved independently, and so are predictions of interface location and corresponding macroscopic states. 
It then becomes presumable to impute the achievement of uniform conditions to the averaging procedure, emphasizing how it induces a process of homogenization over the ensemble.
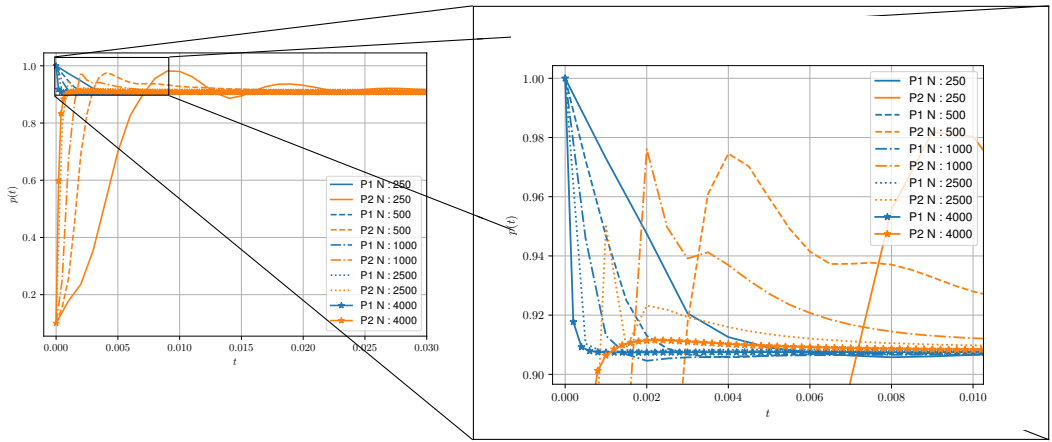
\begin{figure}[!htbp]
\begin{center}
\begin{tikzpicture}
\node (FT) {
\includegraphics[scale=0.4]{\main/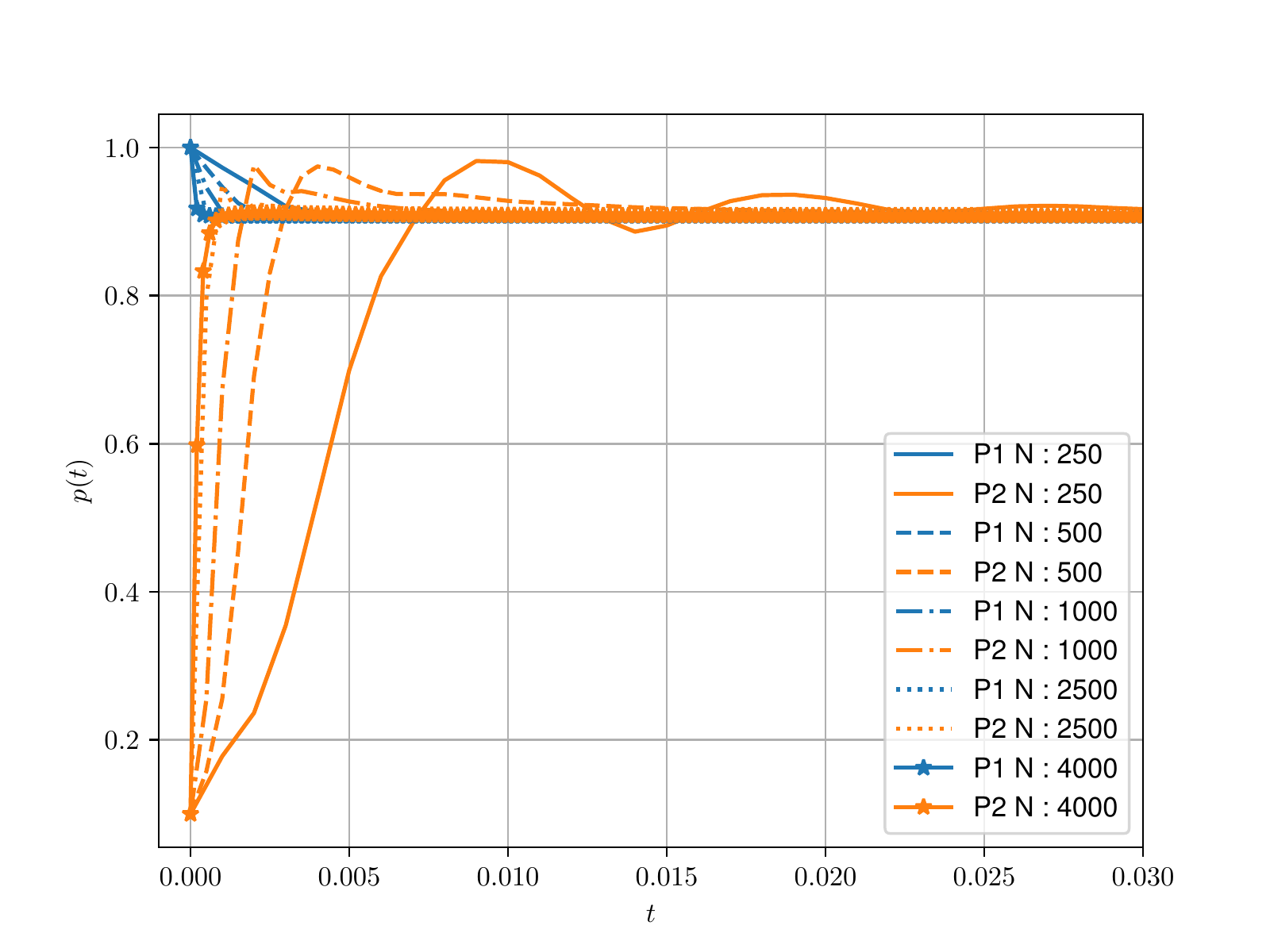}
};

\node (FTz) [rectangle, draw=black, right of=FT, xshift = 6cm, yshift=-0.4cm] {
\includegraphics[scale=0.45]{\main/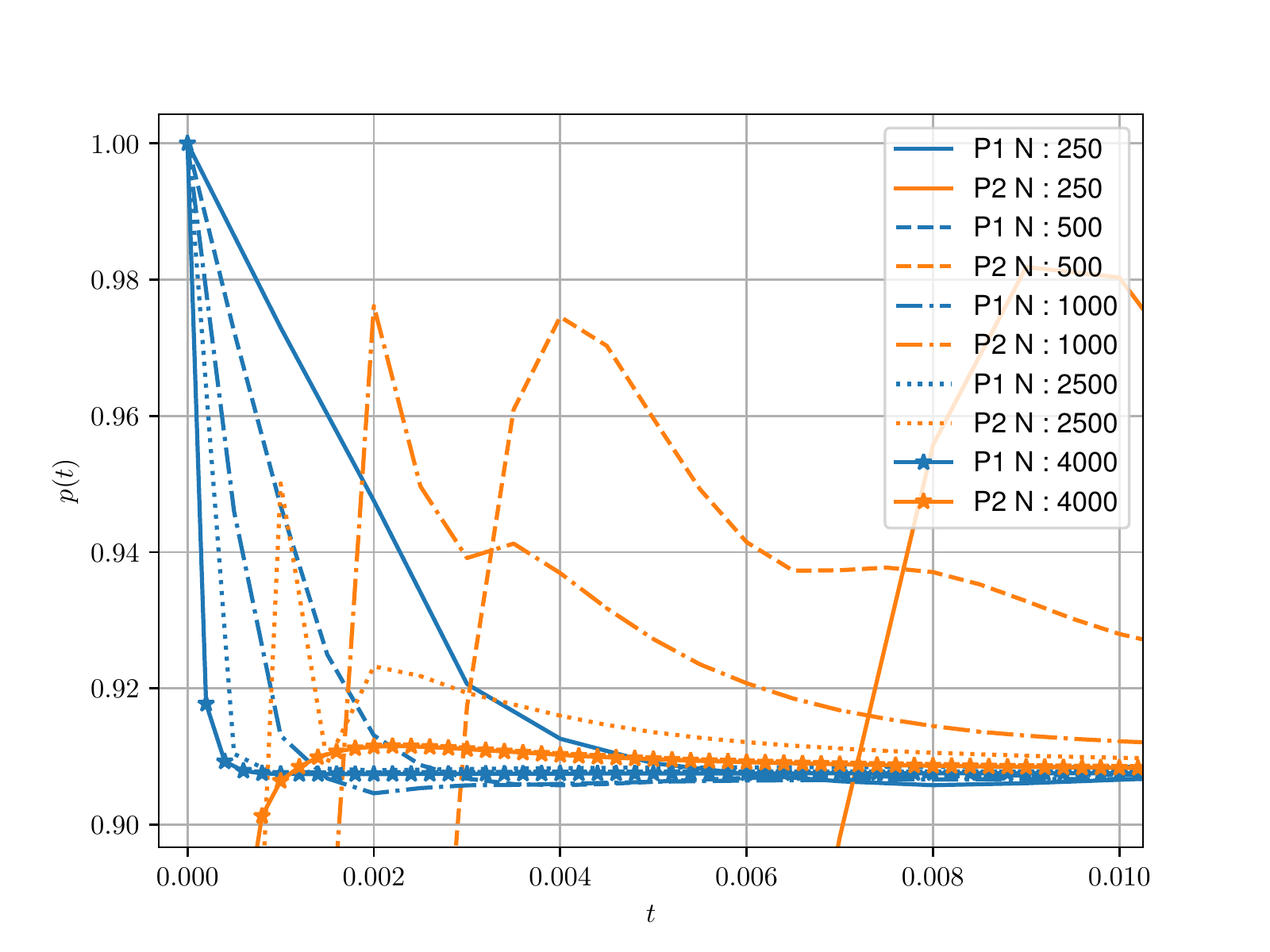}
};

\node [rectangle, draw=black, anchor=north west, xshift=-2.3cm, yshift = 1.8cm, minimum width=1.5cm, minimum height=0.5cm] (Z_on_FT) {};

\draw (Z_on_FT.south west) -- (FTz.south west);
\draw (Z_on_FT.north west) -- (FTz.north west);

\begin{scope}[on background layer]
\draw (Z_on_FT.north east) -- node [pos=0.4] (Mup){} (FTz.north east);
\draw (Z_on_FT.south east) -- node [pos=0.4] (Mdown){} (FTz.south east);
\end{scope}

\draw (Z_on_FT.north east) -- (Mup);
\draw (Z_on_FT.south east) -- (Mdown);
\end{tikzpicture}
\end{center}
\caption{Pressure plots against time for the test case about relaxation strategies.}
\label{Fig:AI:T2_p}
\end{figure}
Second, a strong discrepancy in the values attained by densities of phase $2$ can be seen between different choices of the parameter $r$ in the DEM scheme and the ab-initio simulation. 
The rational for such a difference resides in the speed of relaxation, which is \emph{prescribed} in the DEM, without any information about the actual regime of
the flow under consideration. 
Indeed, one typically let relaxation parameters run into infinity, without knowing the concrete values for such parameters. 
The present results highlight how crude such approximation may be.
Furthermore, in the DEM predictions, uniform mechanical equilibrium is achieved after the first time step, so that no variation afterwards is involved.
This highlights a qualitative and quantitative discrepancy between the DEM and the ab-iinitio: the former not only computes relaxed values incorrectly, but also flattens any time oscillation, thus over-simplifying the mechanics of the process.
In order to highlight the complexity of such process we compute the constant states $\VV^{(k)}(t)$ as the mean value over $D$, for each time step, and plot the values of pressure over time in Fig.\ref{Fig:AI:T2_p}.
\begin{figure}[!htbp]
\begin{center}
\includegraphics[scale=\figsize]{\main/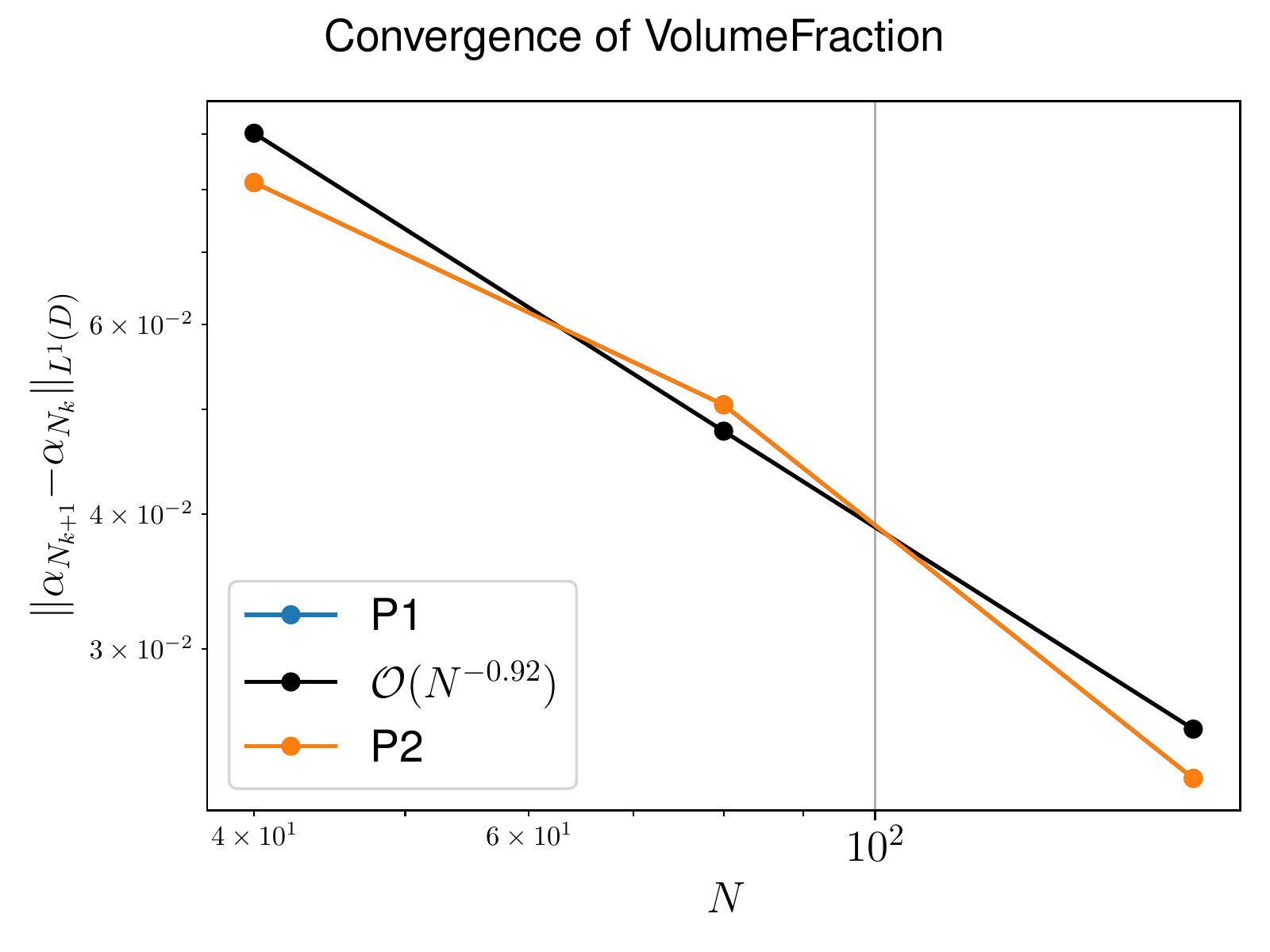}\,
\includegraphics[scale=\figsize]{\main/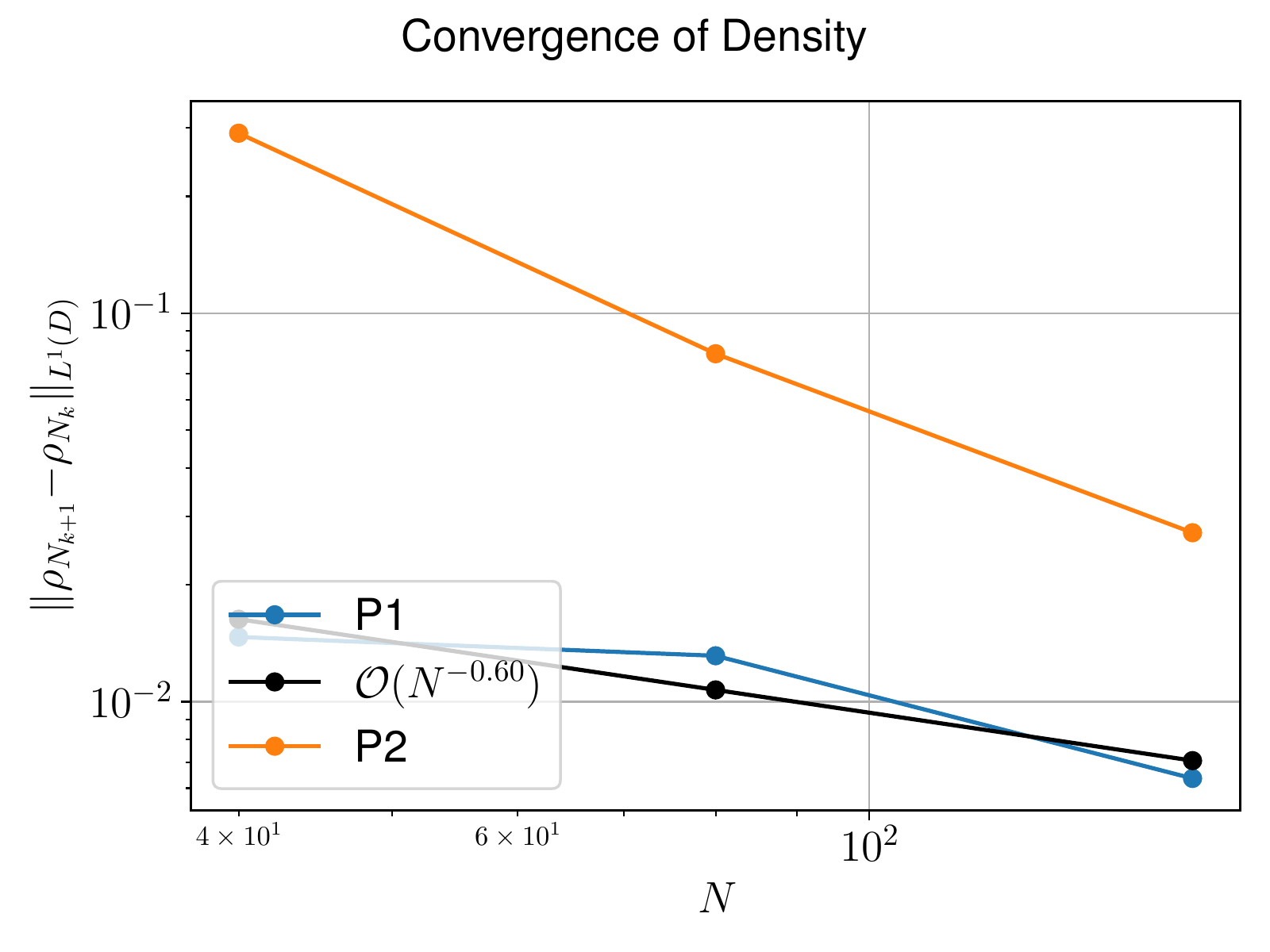}\\
\includegraphics[scale=\figsize]{\main/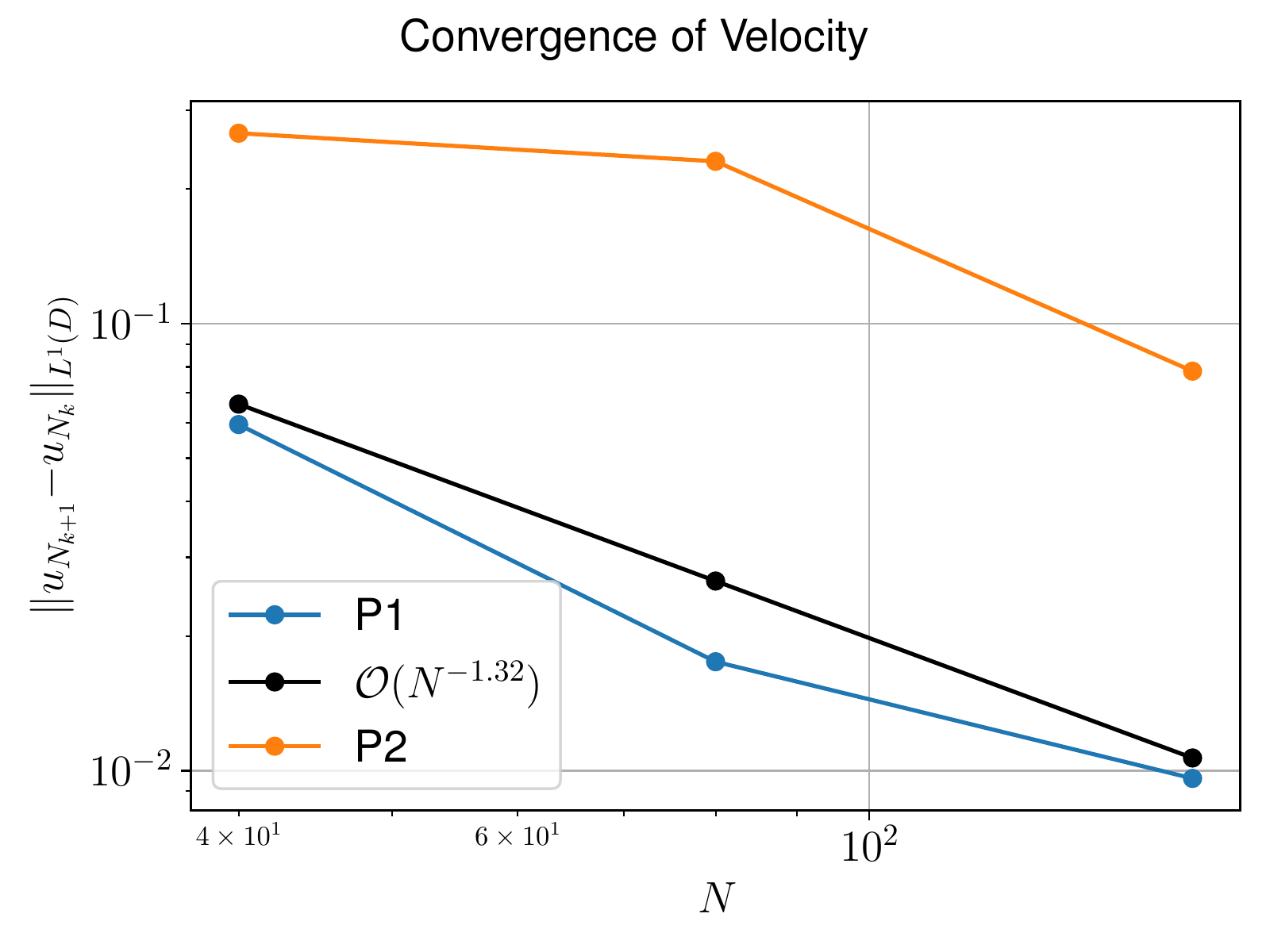}\,
\includegraphics[scale=\figsize]{\main/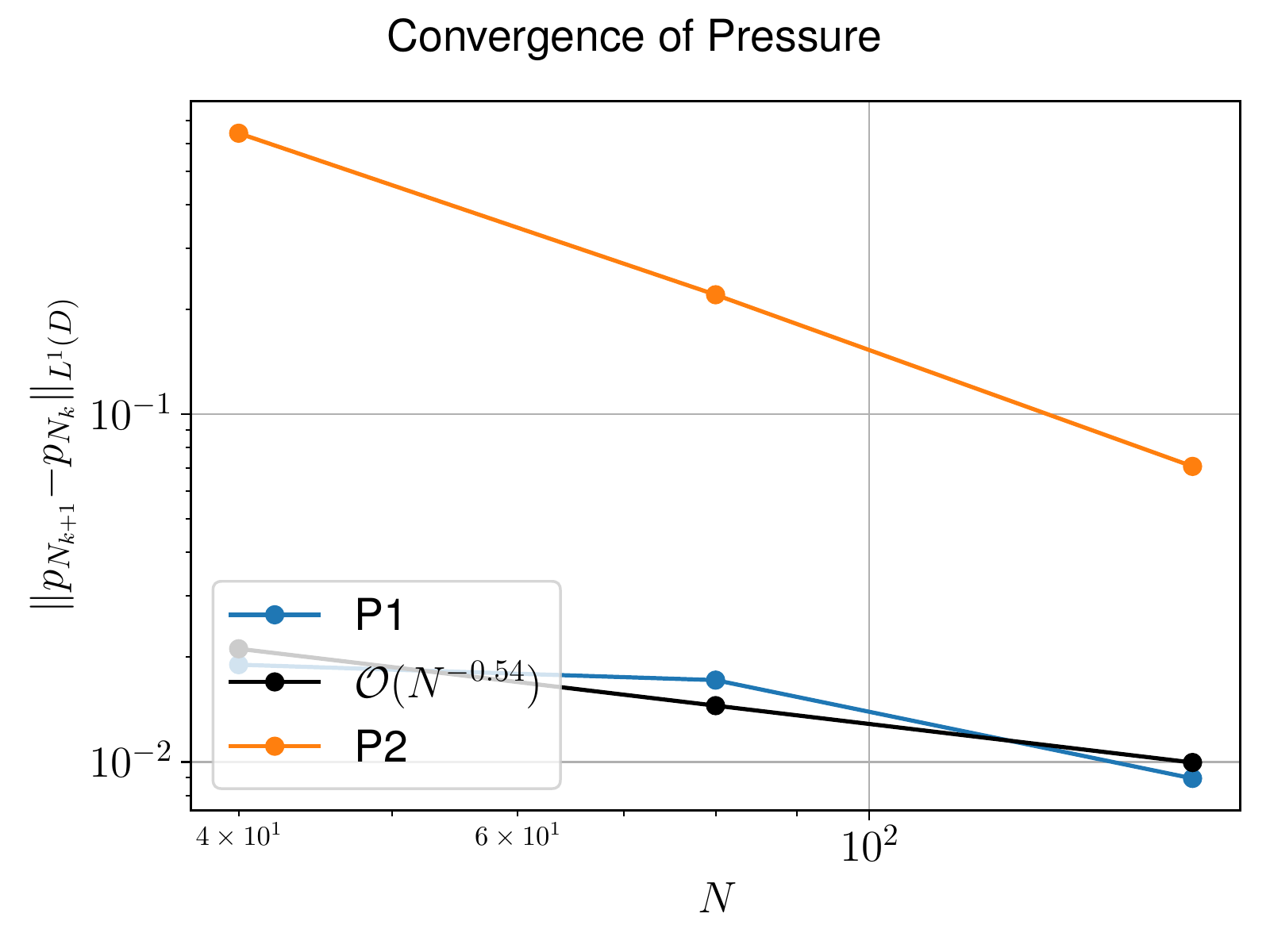}
\end{center}
\caption{Empirical convergence study for the ab-initio method on the test case about relaxation : error-vs-number of sub-volumes. Different phases ($P$) are highlighted with different colors.}
\label{Fig:AI:T2_conv}
\end{figure}
Notice how, for fixed number of volumes, pressure values for phase $2$ tend to oscillate around the equilibrium, and stabilize with the passing of time.
As the number of sub-volumes is increased (and thus the one of interfaces), such oscillations seems to disappear and reduction of the time at which the two-phases run into equilibrium is shifted towards zero.
We though point out that such oscillation process seems to be at the heart of the relaxation process, meaning that the continuous exchange happening at the interface is responsible for shaping macroscopic quantities, as well as the number of interfaces.
This, in turn, translates into the necessity of mapping any space-time control volume with its correct number of interfaces when dealing with relaxation phenomena.
The present formulation allows for tracking interfaces and can provide additional information to capture such parameters.
For this reason, the ab-initio perspective constitutes a more fundamental point of view than the DEM.\\ 
For the sake of completeness, we analyzed asymptotic properties of the present test case:
based on our numerical approximations, no improvement of this discrepancy is achieved under mesh-refinement in the DEM scheme nor under increase of sub-volumes number for the ab-initio.

We conclude this test case by performing an empirical convergence study for the ab-initio method under refinement of subvolumes number.
Results for the Cauchy rates of each quantity of interest computed using a sequence $N_j = 20\cdot 2^j$ $j=0,\ldots,4$ volumes is reported in Fig.\ref{Fig:AI:T2_conv}.

\subsubsection{A two-phase Sod's shock-tube problem}
\label{sec:AI:MLML:NE:Sod}

We consider a two-phase variant of the classical Sod's shock tube problem: 
\[
\VV_0(x) = 
\begin{cases}
\begin{bmatrix}
\alpha^{(1)}_L = 0.9\\
\VW_L^{(1)}\\
\alpha^{(2)}_L = 0.1\\
\VW_L^{(2)}
\end{bmatrix}
& x< 0\\
\begin{bmatrix}
\alpha^{(1)}_R = 0.1\\
\VW_R^{(1)}\\
\alpha^{(2)}_R = 0.9\\
\VW_R^{(2)}
\end{bmatrix}
& x>0
\end{cases}
\qquad
\qquad
\VW^{(k)}_L
=
\begin{bmatrix}
1\\
0\\
1
\end{bmatrix}
\quad
\VW^{(k)}_R
=
\begin{bmatrix}
0.125\\
0\\
0.1
\end{bmatrix}
\]

Results for the MC-version of the ab-initio method are plotted in Fig.\ref{Fig:AI:T3} against the two (limiting) choices of the hyper-parameter $r$ in the (first-order version of the) DEM. 
\begin{table}
\begin{center}
\begin{tabular}{c||c|c|c|c|c|c|c|}
&
$M$ & $\tilde{N}$ & $\mathrm{CFL}$ & $L$ & $\delta^{(1)}$ & $\delta^{(2)}$ & $L$\\
\hline
\hline
\textbf{Ab-initio} & $500$ & $6400$ & $-$ & $4$ & $0.05$ & $0.05$ & $1000$\\
\hline
$\textbf{r}\equiv \bm{0}$ & $10000$ & $-$ & $0.9$ & $-$ & $-$ & $-$ & $-$\\
\hline
$\textbf{r}\equiv \bm{1}$ & $10000$ & $-$ & $0.9$ & $-$ & $-$ & $-$ & $-$\\
\hline
\end{tabular}
\end{center}
\caption{Parameters of computed solutions reported in Fig. \ref{Fig:AI:T3} }\label{Tab:AI:T3}
\end{table}
The details of our simulations are summarized in Table \ref{Tab:AI:T3}. 
\begin{figure}[!htbp]
\begin{center}
\includegraphics[scale=\figsize]{\main/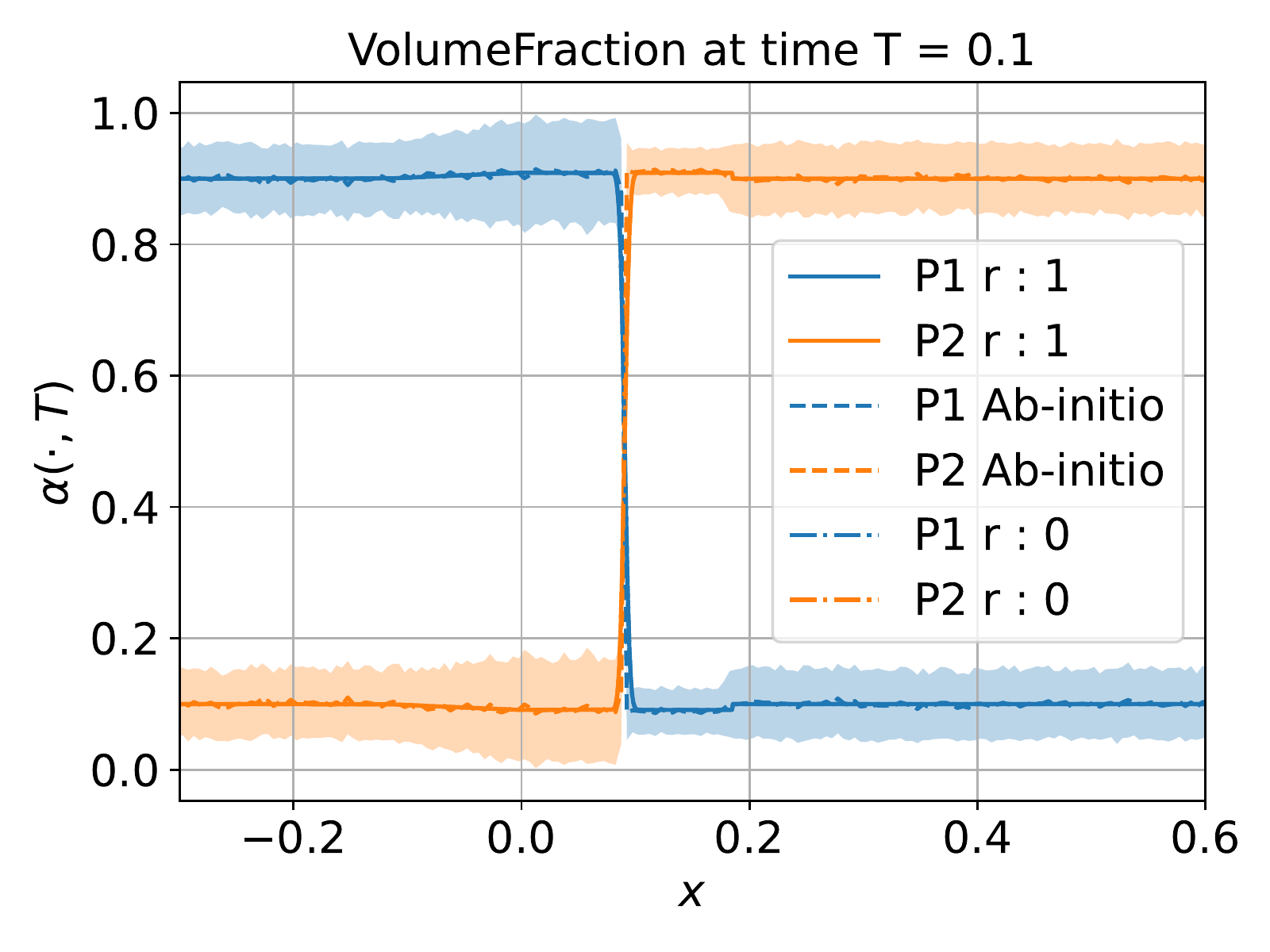}\,
\includegraphics[scale=\figsize]{\main/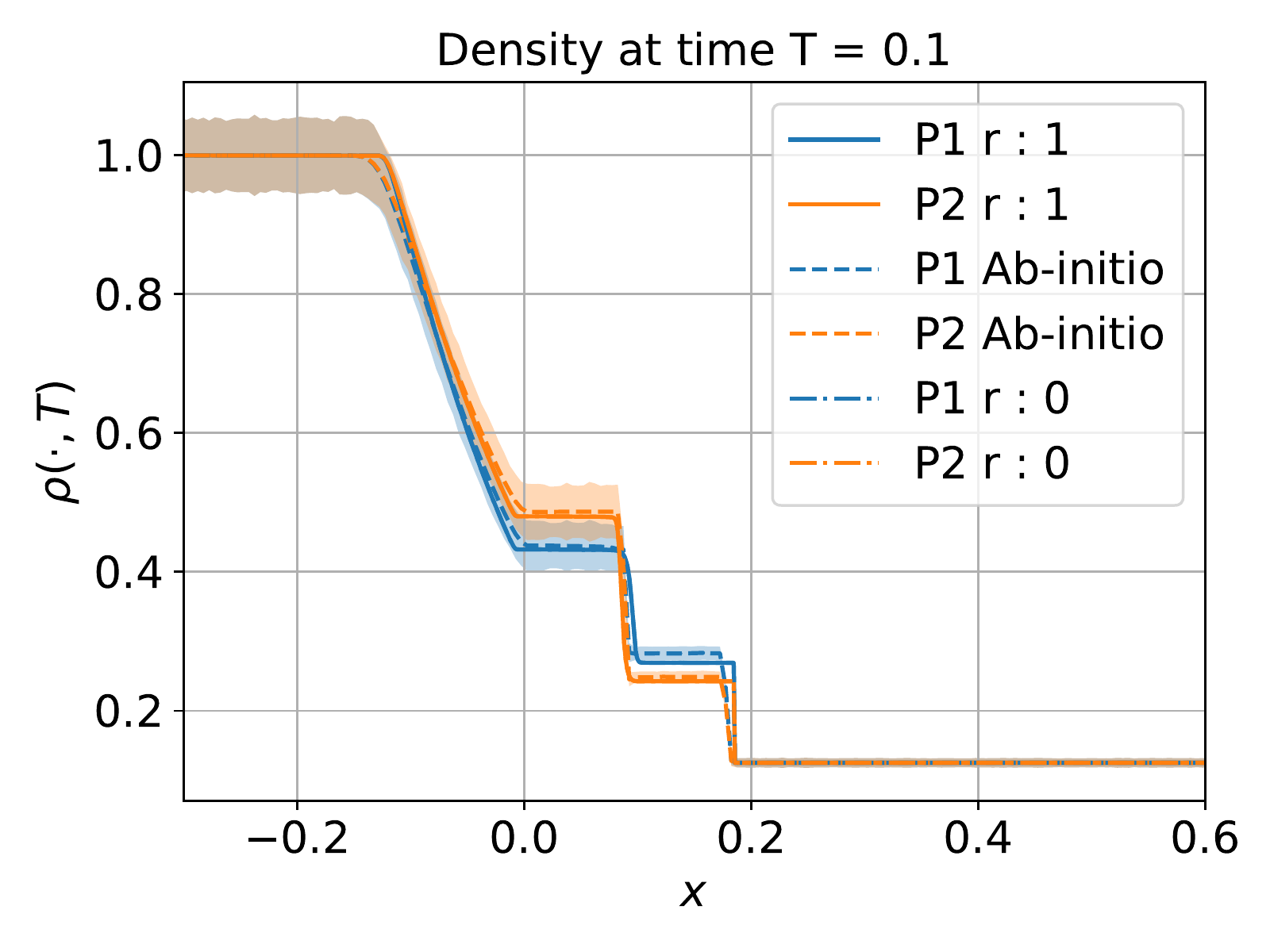}\\
\includegraphics[scale=\figsize]{\main/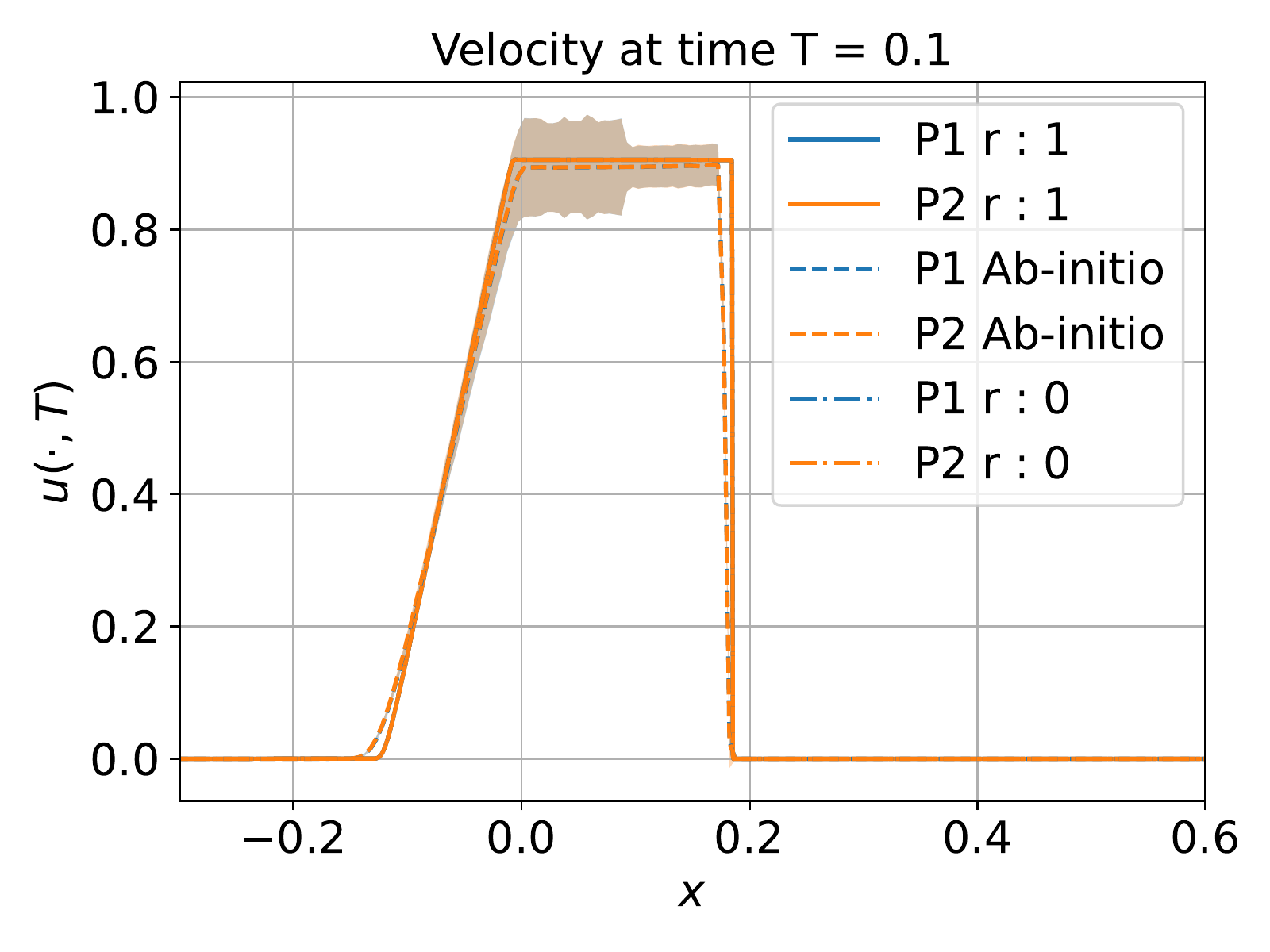}\,
\includegraphics[scale=\figsize]{\main/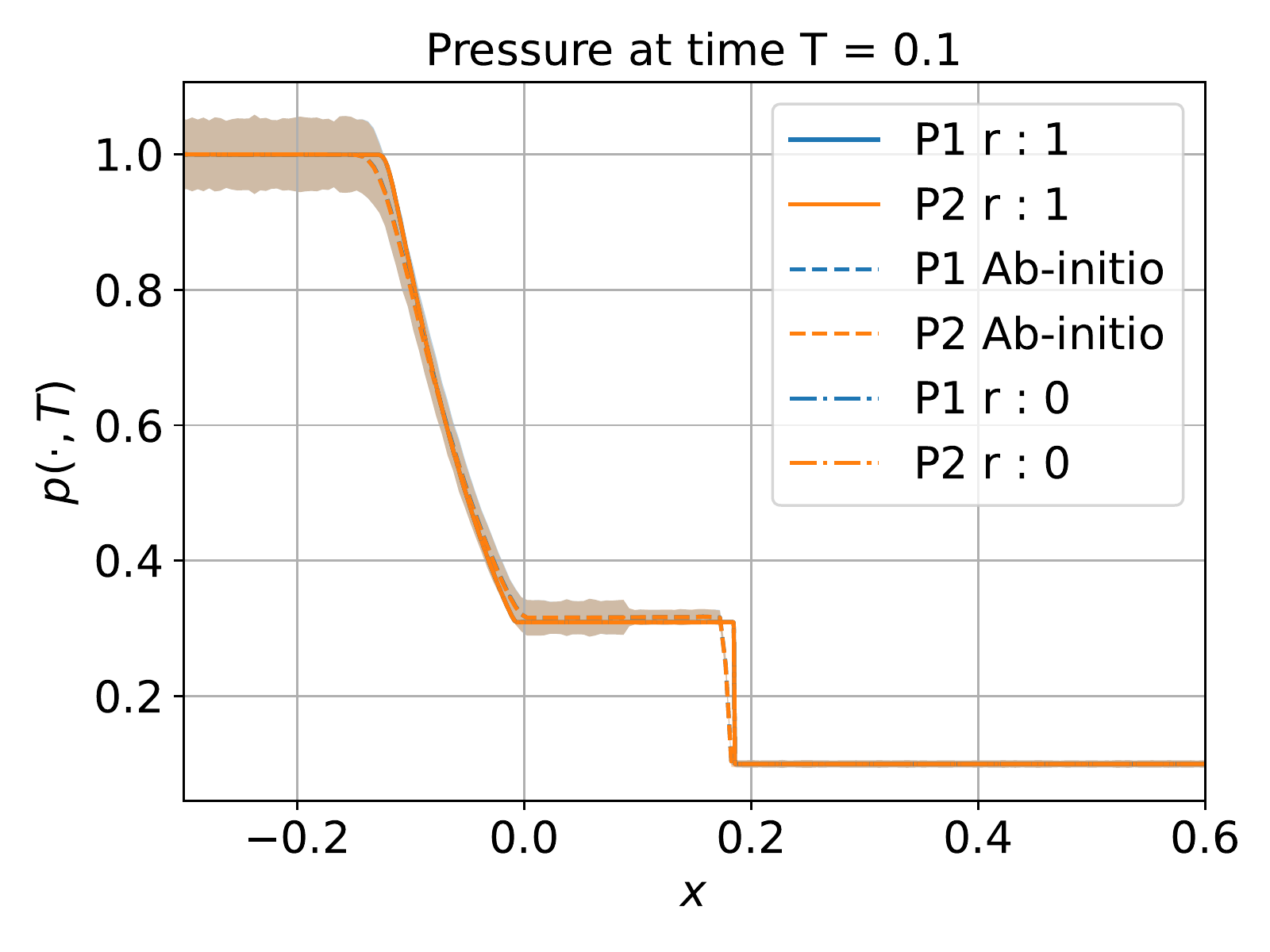}
\end{center}
\caption{Results of the Ab-initio method against two solutions of the DEM ($r\,:$) for the two-phase Sod's shock tube problem.
}
\label{Fig:AI:T3}
\end{figure}

First, notice the 'uniform mechanical conditions' (i.e. single-pressure and single-velocity) behavior in Ab-initio simulations.
As pointed out for the previous test case, this is quite striking as no relaxation is involved and equilibrium across phases is achieved only as an outcome of the averaging.
Second, we observe very good agreement between the Ab-initio and the DEM results, even if minor discrepancies can be seen in the plateau of velocity, and in the shock location of velocity and pressure predictions. 
On the other hand, we observe the (virtual) coalescence between results produced with the DEM using several choices of the hyper-parameter $r$.\\
By carefully analyzing the DEM scheme, one can notice that solutions produced with different choices of the hyper-parameter $r$ are approximately non-distinguishable. 
To further motivate such conclusion (and the following ones) we conduct a mesh convergence study of the DEM: for each resolution $M_j = 100\cdot2^{j}$ $j=1,\ldots,6$, we generate solutions for the DEM using constant $r\equiv 0$ and $r\equiv 1$, and compute the $L^1$-distance between the thus generated solutions
\[
e_j^{(k)}(y;T) := 
\Vert 
y^{(k)}(\cdot,T;r\equiv 0) - y^{(k)}(\cdot,T;r\equiv 1)
\Vert_{L^1(D)},
\qquad
y\in
\lbrace
\alpha, \rho, u, p
\rbrace,
\quad
k=1,2.
\]
Results in the log-log scale are shown in Fig.\ref{Fig:AI:T3_conv_r}. 
\begin{figure}[!htbp]
\begin{center}
\includegraphics[scale=\figsize]{\main/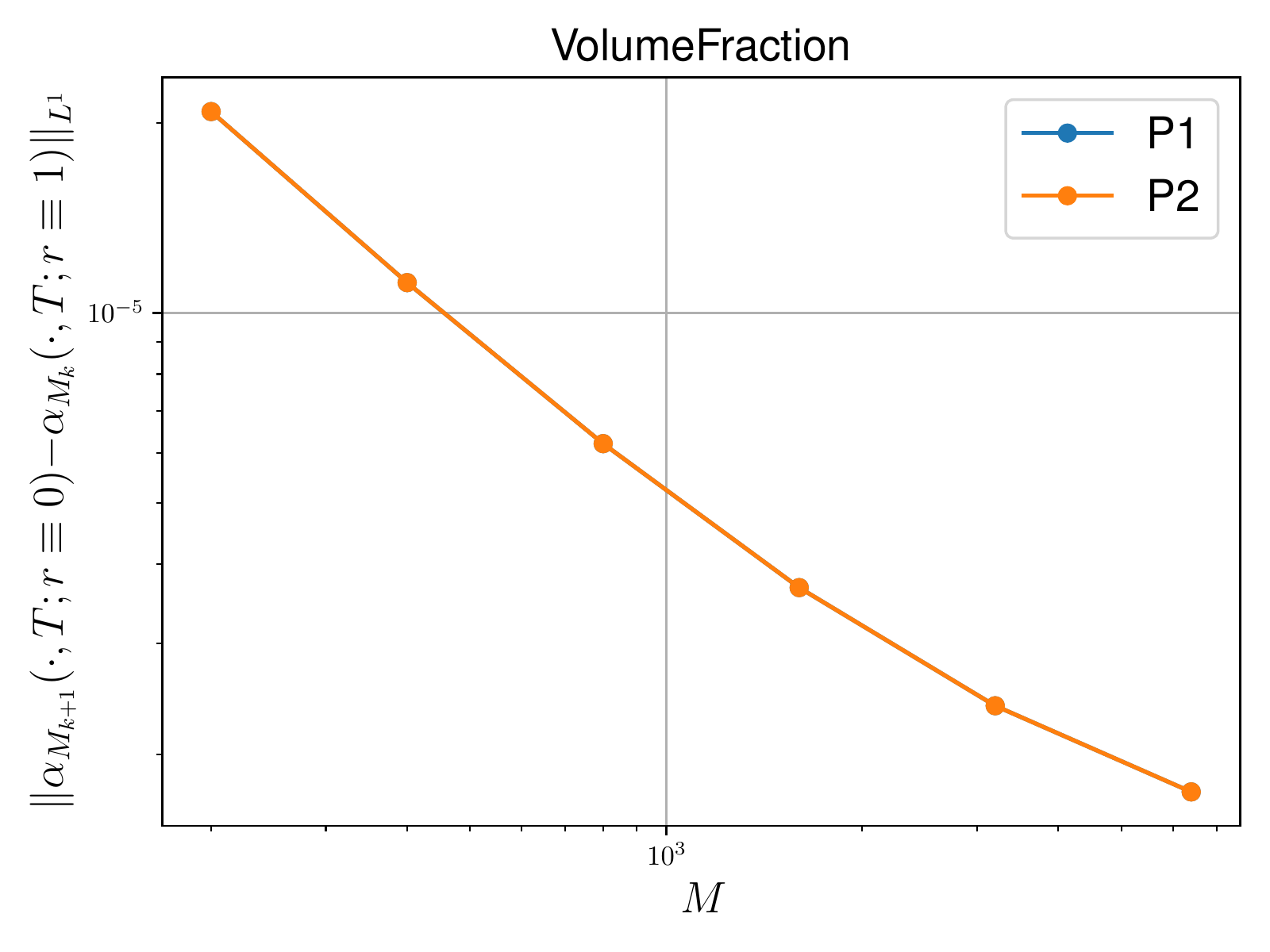}\,
\includegraphics[scale=\figsize]{\main/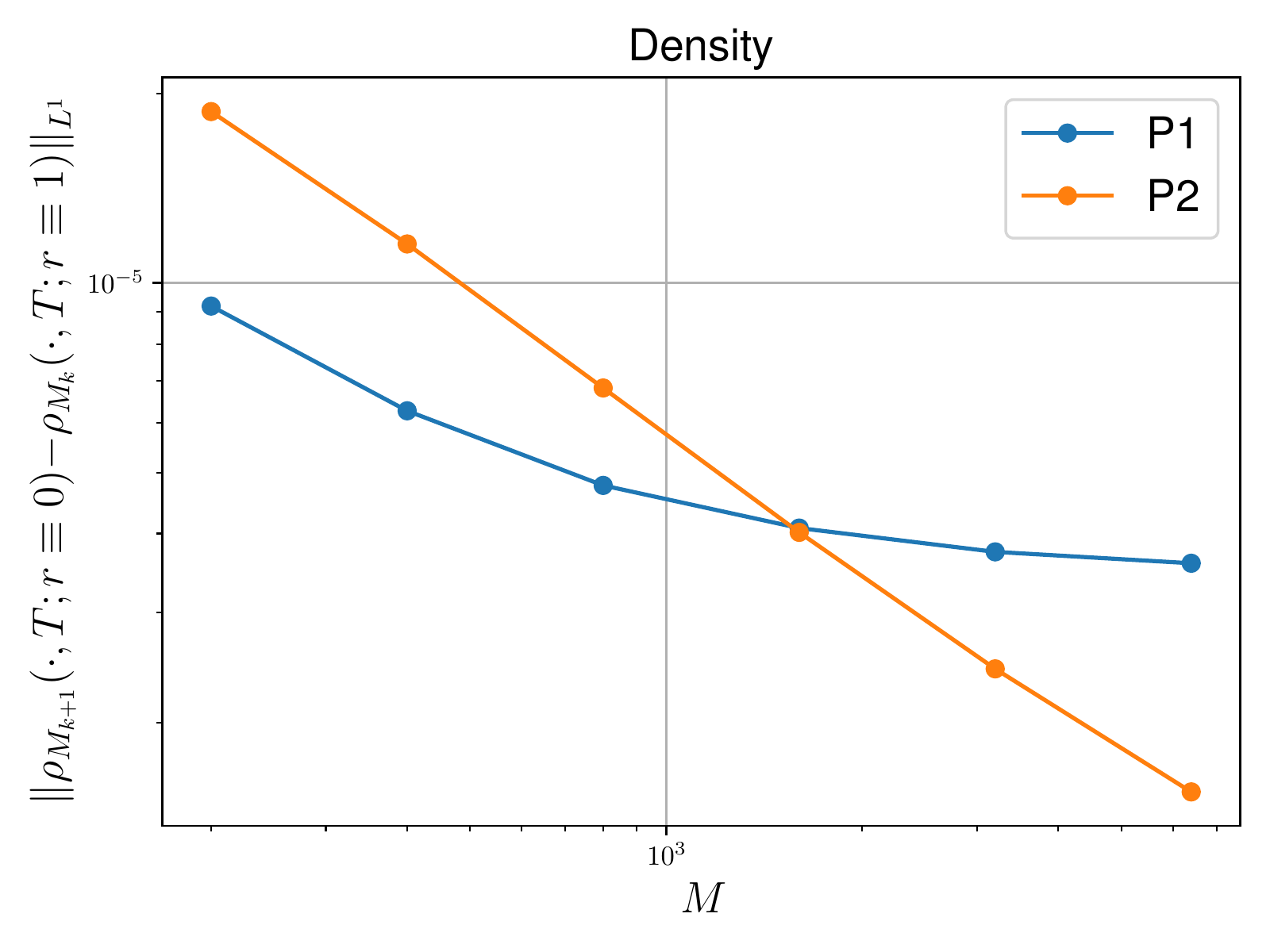}\\
\includegraphics[scale=\figsize]{\main/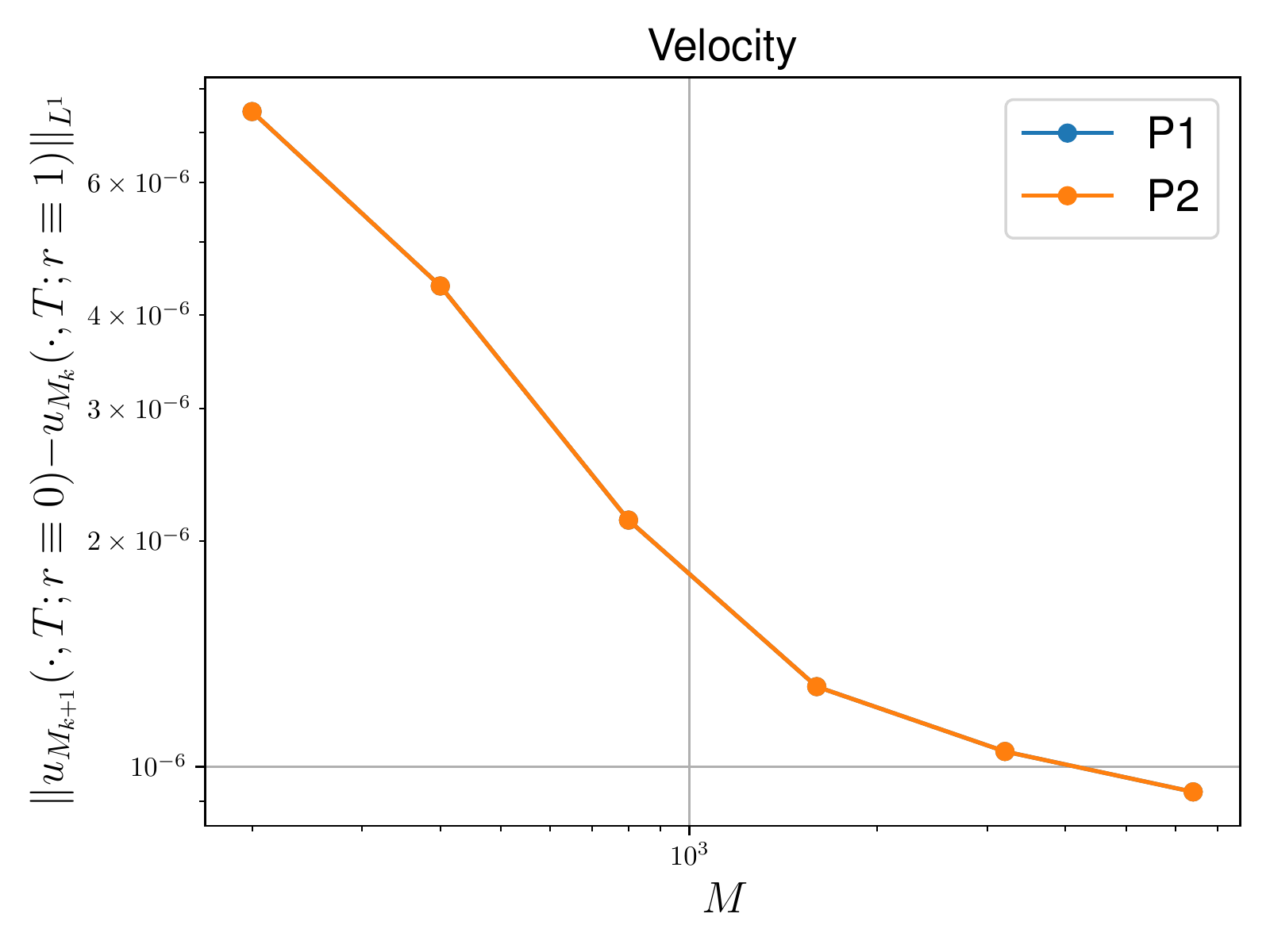}\,
\includegraphics[scale=\figsize]{\main/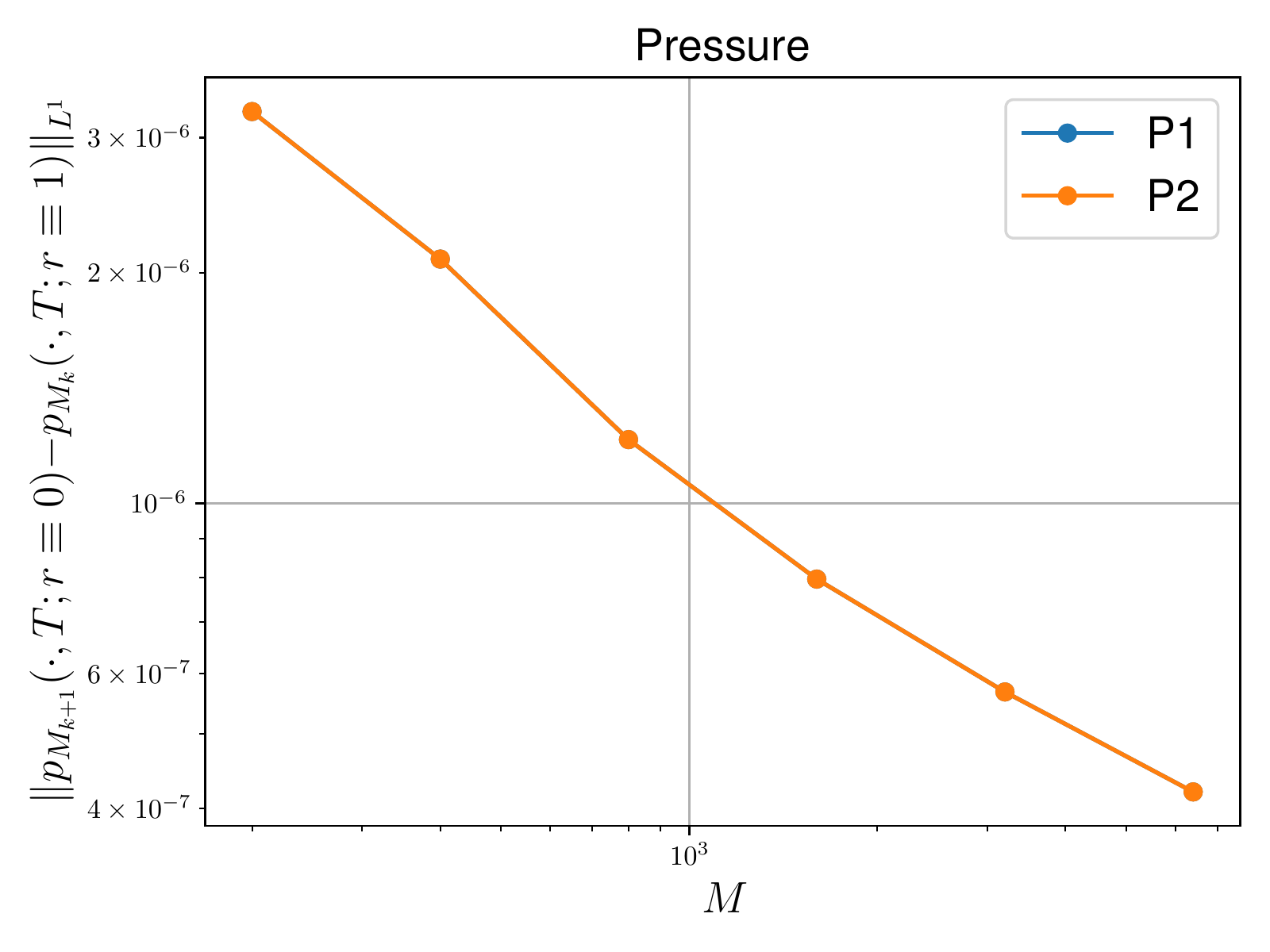}
\end{center}
\caption{Empirical convergence study for the distance between solutions generated using the DEM and different choices of the parameter $r$ for the two-phase Sod's shock tube problem.
}
\label{Fig:AI:T3_conv_r}
\end{figure}
Computed distance between macroscopic quantities of the DEM show steady convergence of all quantities of interest, except for the density value of phase $1$. This latter can be explained in the difference between the plateaus of solutions associated to different choices of the parameter $r$, which, however, corresponds to an error below the $0.001\%$.
Hence, up to such precision, one can consider predictions of the DEM essentially independent of the choice of the hyper-parameter $r$ for this test case.
In turn, the (virtually) unique solution prescribed with the DEM agrees with the one produced with the ab-initio.
It seems then reasonable to conclude that this test case is supporting a weak-uniqueness principle: if solutions generated with DEM are independent of the choice of the parameter $r$, then the uniquely defined solution should converge (up to some precision) to the (limit of) the ab-initio method.
This clearly highlights how the ab-initio method constitutes a generalization of the DEM.
\begin{figure}[!htbp]
\begin{center}
\includegraphics[scale=\figsize]{\main/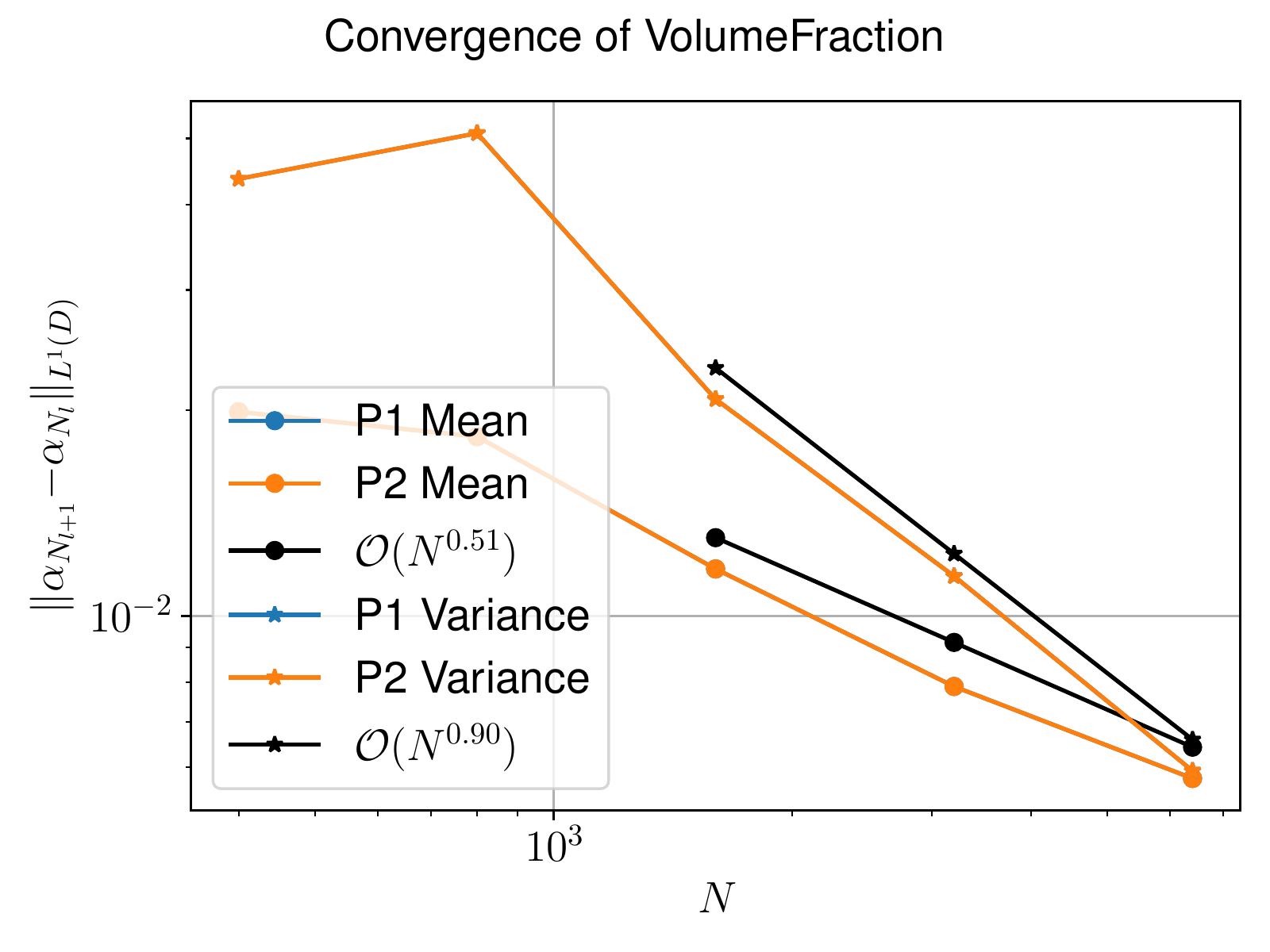}\,
\includegraphics[scale=\figsize]{\main/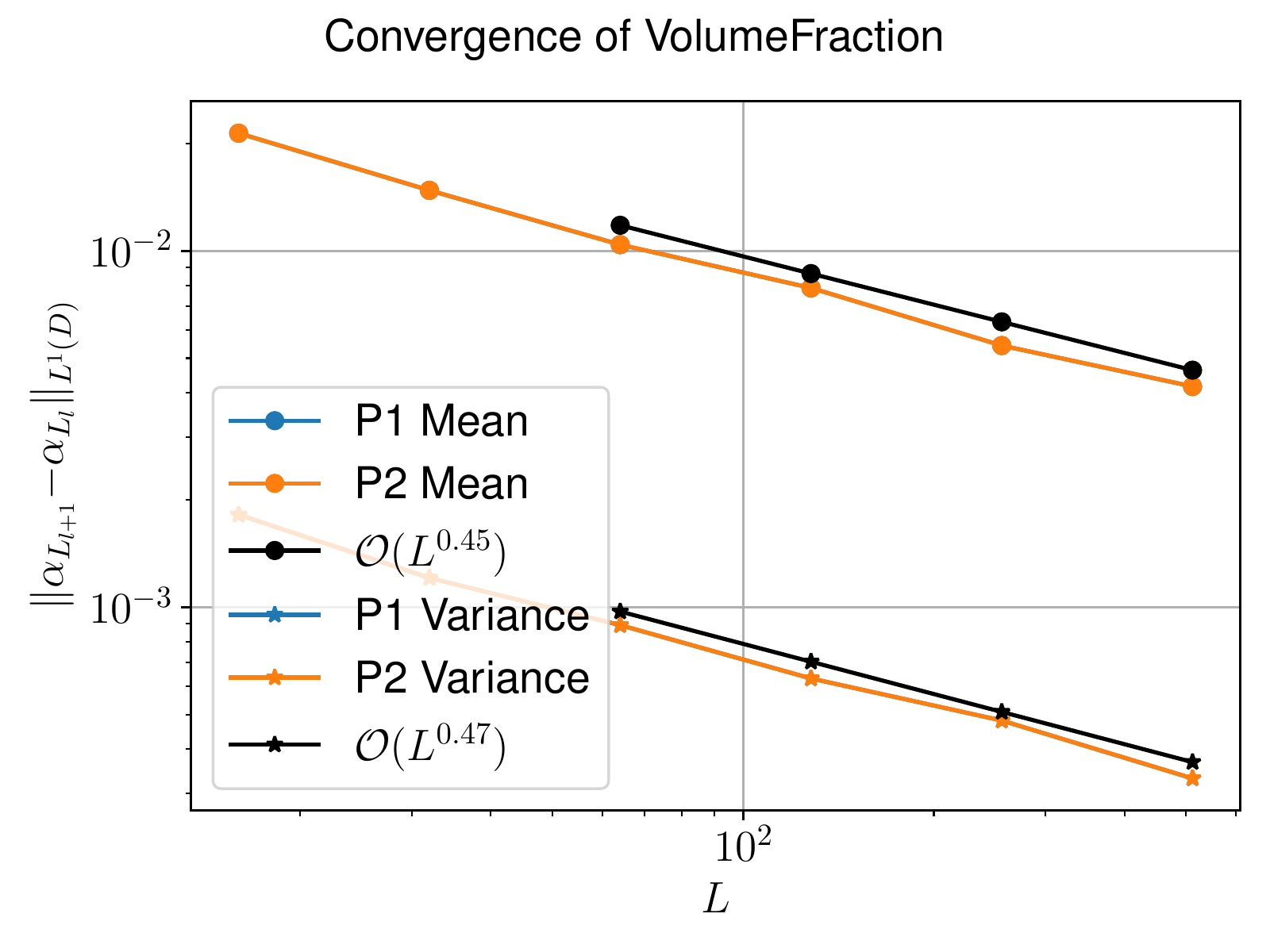}
\end{center}
\caption{Empirical convergence study for the two-phase Sod's shock tube problem for the Ab-initio method under sub-scale refinement (left) and number of samples (right).}
\label{Fig:AI:T3_conv}
\end{figure}
We conclude this test case by carrying out a mesh convergence study of the Ab-initio method for the present test case under number of sub-scale refinement and number of samples increment.
Results in the loglog scale for the volume fraction are reported in Fig.\ref{Fig:AI:T3_conv}, suggesting convergence in both refinement directions.

\subsubsection{Lax's shock tube problem}

We now consider a two-phase variant of the Lax's shock tube problem: 
\[
\VV_0(x) = 
\begin{cases}
\begin{bmatrix}
\alpha^{(1)}_L = 0.9\\
\VW_L^{(1)}\\
\alpha^{(2)}_L = 0.1\\
\VW_L^{(2)}
\end{bmatrix}
& x< 0\\
\begin{bmatrix}
\alpha^{(1)}_R = 0.1\\
\VW_R^{(1)}\\
\alpha^{(2)}_R = 0.9\\
\VW_R^{(2)}
\end{bmatrix}
& x>0
\end{cases}
\qquad
\qquad
\VW^{(k)}_L
=
\begin{bmatrix}
\rho_k\\
0.7\\
3.5
\end{bmatrix}
\quad
\VW^{(k)}_R
=
\begin{bmatrix}
\rho_k\\
0\\
0.1
\end{bmatrix}
\]
where $\rho_1 = 0.2$ and $\rho_2 = 1$. 
Phase $1$ is assumed to be governed by the IG-EOS, while the phase $2$ is associated to the SG-EOS and the parameters read 
\[
\gamma^{(1)} = 1.4,
\qquad
\qquad
\gamma^{(2)} = 1.6,
\quad
\pi^{(2)} = 2.5.
\]
\begin{table}[!h]
\begin{center}
\begin{tabular}{c||c|c|c|c|c|c|c|}
&
$M$ & $\tilde{N}$ & $\mathrm{CFL}$ & $L$ & $\delta^{(1)}$ & $\delta^{(2)}$ & $L$\\
\hline
\hline
\textbf{Ab-initio} & $500$ & $5000$ & $-$ & $5$ & $0.05$ & $0.1$ & $1000$\\
\hline
$\textbf{r}\equiv \bm{0}$ & $10000$ & $-$ & $0.9$ & $-$ & $-$ & $-$ & $-$\\
\hline
$\textbf{r}\equiv \bm{1}$ & $10000$ & $-$ & $0.9$ & $-$ & $-$ & $-$ & $-$\\
\hline
\end{tabular}
\end{center}
\caption{Parameters of computed solutions reported in Fig. \ref{Fig:AI:T4} }\label{Tab:AI:T4}
\end{table}
\begin{figure}[!htbp]
\begin{center}
\includegraphics[scale=\figsize]{\main/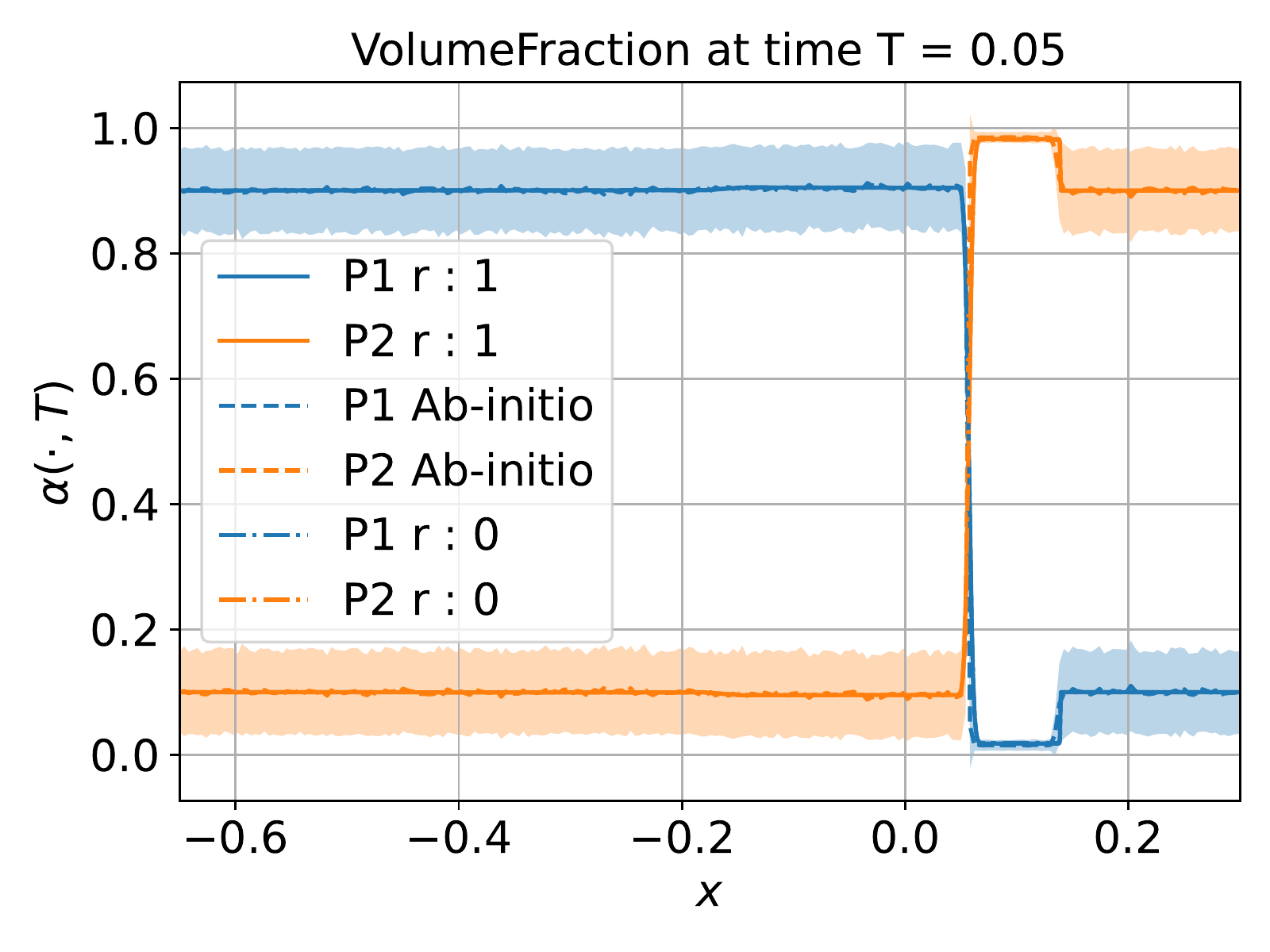}\,
\includegraphics[scale=\figsize]{\main/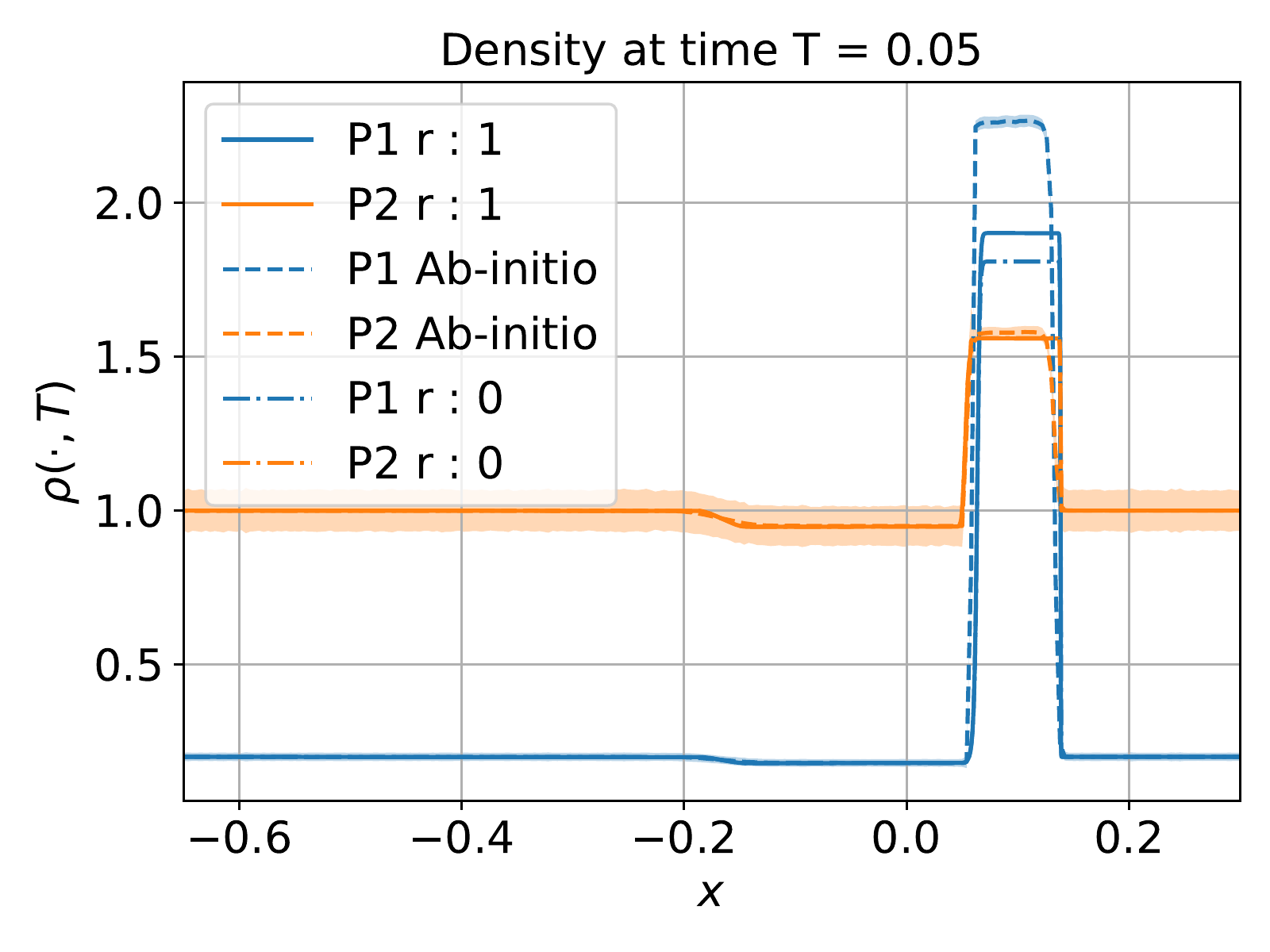}\\
\includegraphics[scale=\figsize]{\main/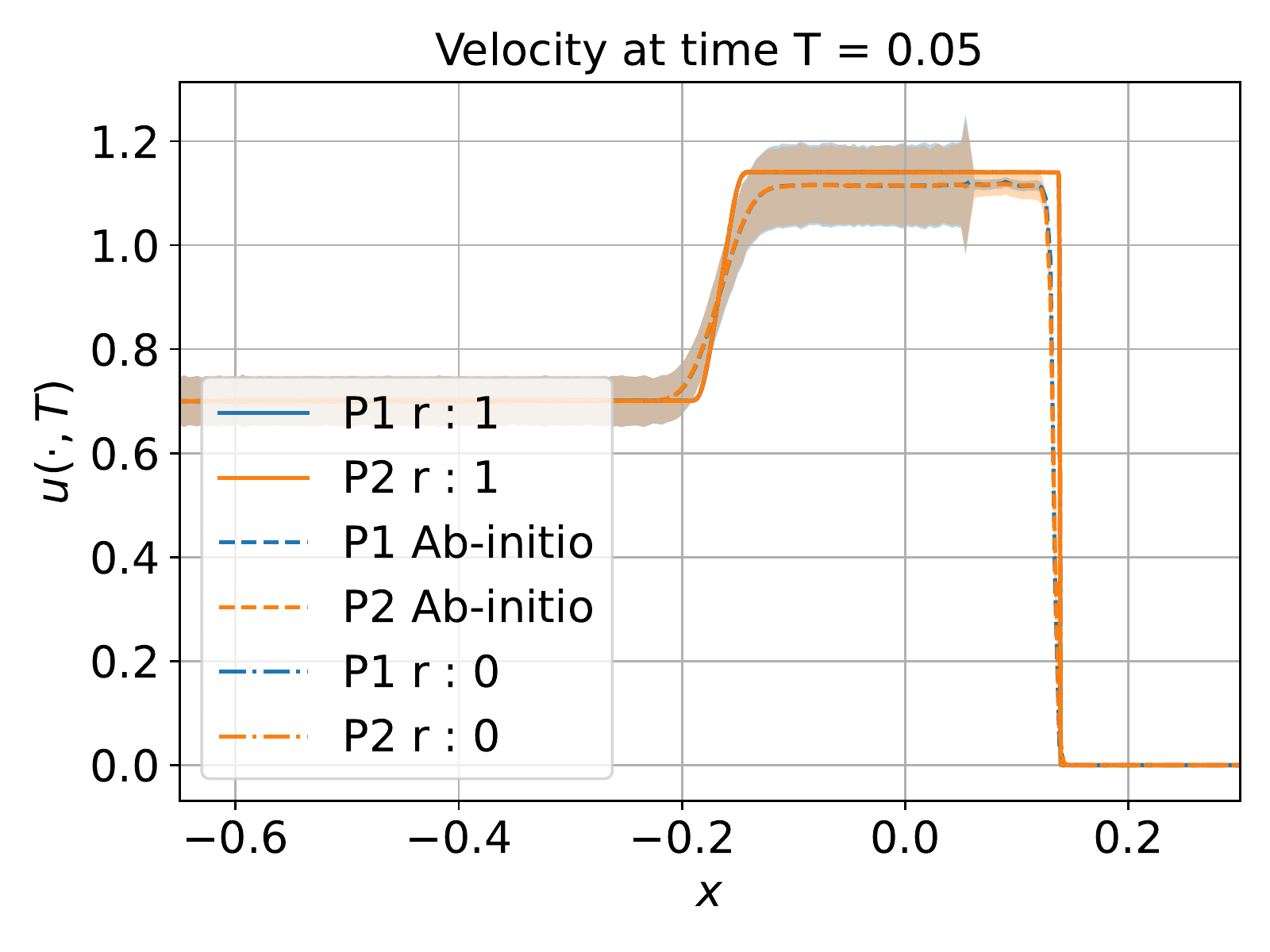}\,
\includegraphics[scale=\figsize]{\main/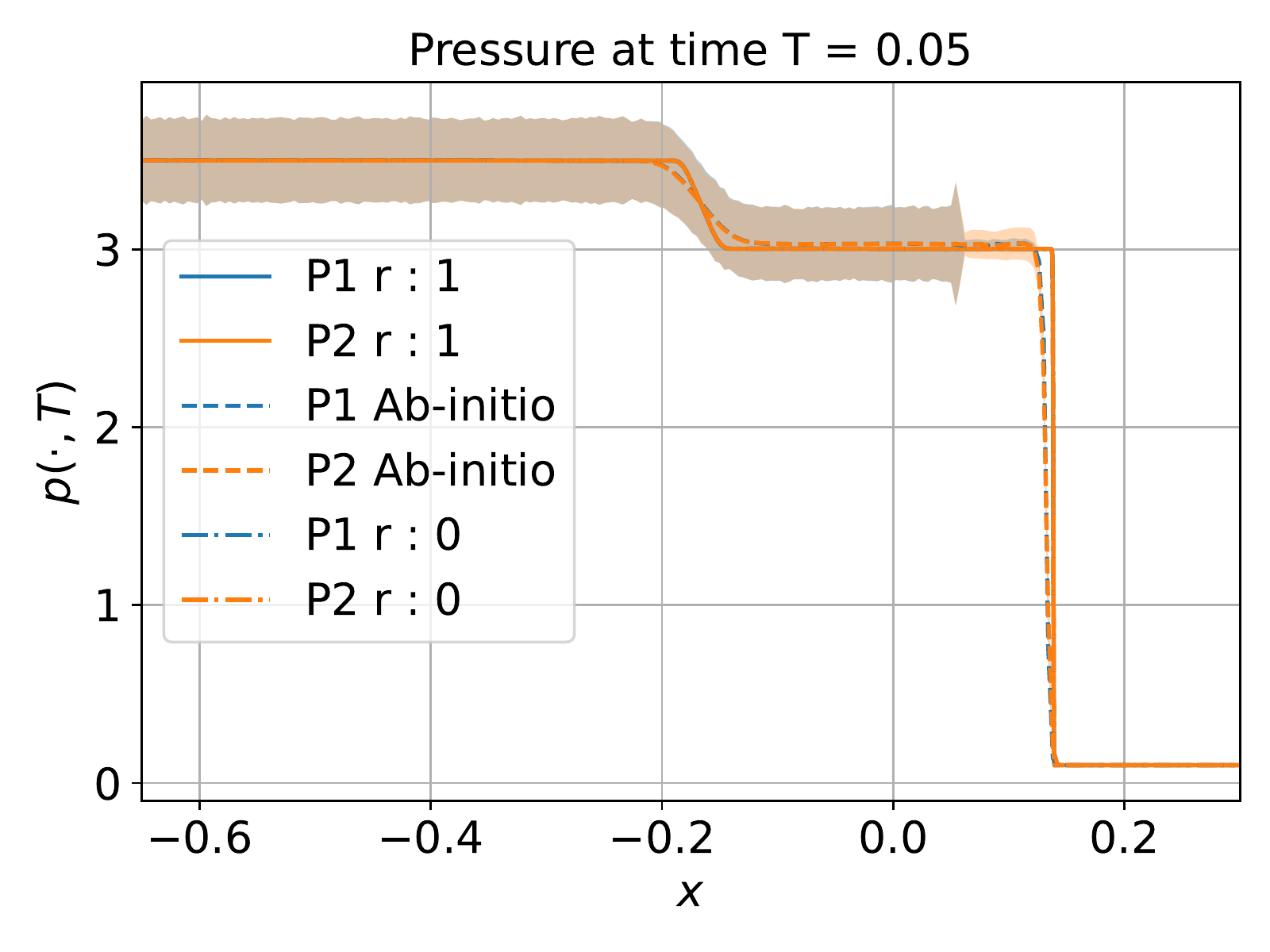}
\end{center}
\caption{Computed solutions for the Lax shock-tube problem. Results for the (MLMC-)ab-initio and two choices of the DEM ($r\,:$) scheme are shown. }
\label{Fig:AI:T4}
\end{figure}
Results for the ab-initio method are plotted in Fig.\ref{Fig:AI:T4} against the two (limiting) choices of the hyper-parameter $r$ in the DEM.
The details of our simulations are summarized in Table \ref{Tab:AI:T4}.\\
Notice the virtually coalescent behavior of results for phase $2$ associated to different choices of the hyperparameter $r$ in every quantity of interest plot. 
The DEM predictions show a good agreement with the ab-initio results, particularly for the volume fraction. Slight discrepancies can be notice in the plateau of the velocity plot.
Conversely, big discrepancies between the different methodologies can be appreciated around the (post-shock) density plateaus.
Interestingly, virtually no discrepancy can be appreciated in the shock speeds across the different methods.
This may be due to the marginal variation of computed results with respect to the hyperparameter $r$, which is affecting the hyperbolic step in the DEM. 
In contrast, the relaxation employed by the DEM is clearly inducing an erroneous value for the plateaus of the densities.
This underlines the necessity of sharply tracking two types of hyperparameters in the case of the DEM: the parameter $r$, controlling the convective part, and $\lambda_i = \bbE[N_{int}(\omega)/ \Dx]$, controlling the value at which the two phases are relaxing one another.
For the sake of completeness, we also carry out the usual empirical convergence study, whose results are reported in Fig.\ref{Fig:AI:T4_conv} in the log-log scale.
\begin{figure}[!htbp]
\begin{center}
\includegraphics[scale=\figsize]{\main/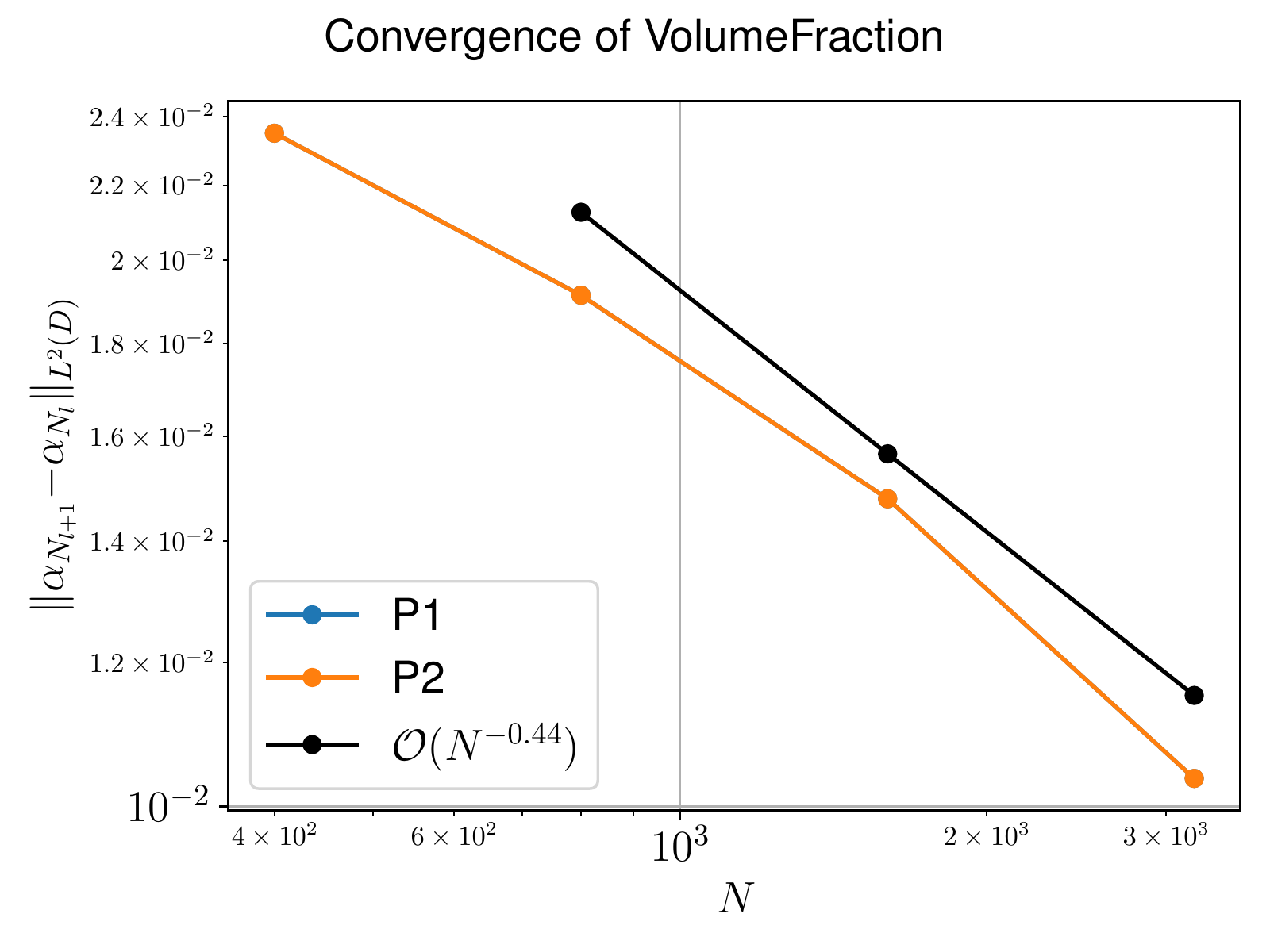}\,
\includegraphics[scale=\figsize]{\main/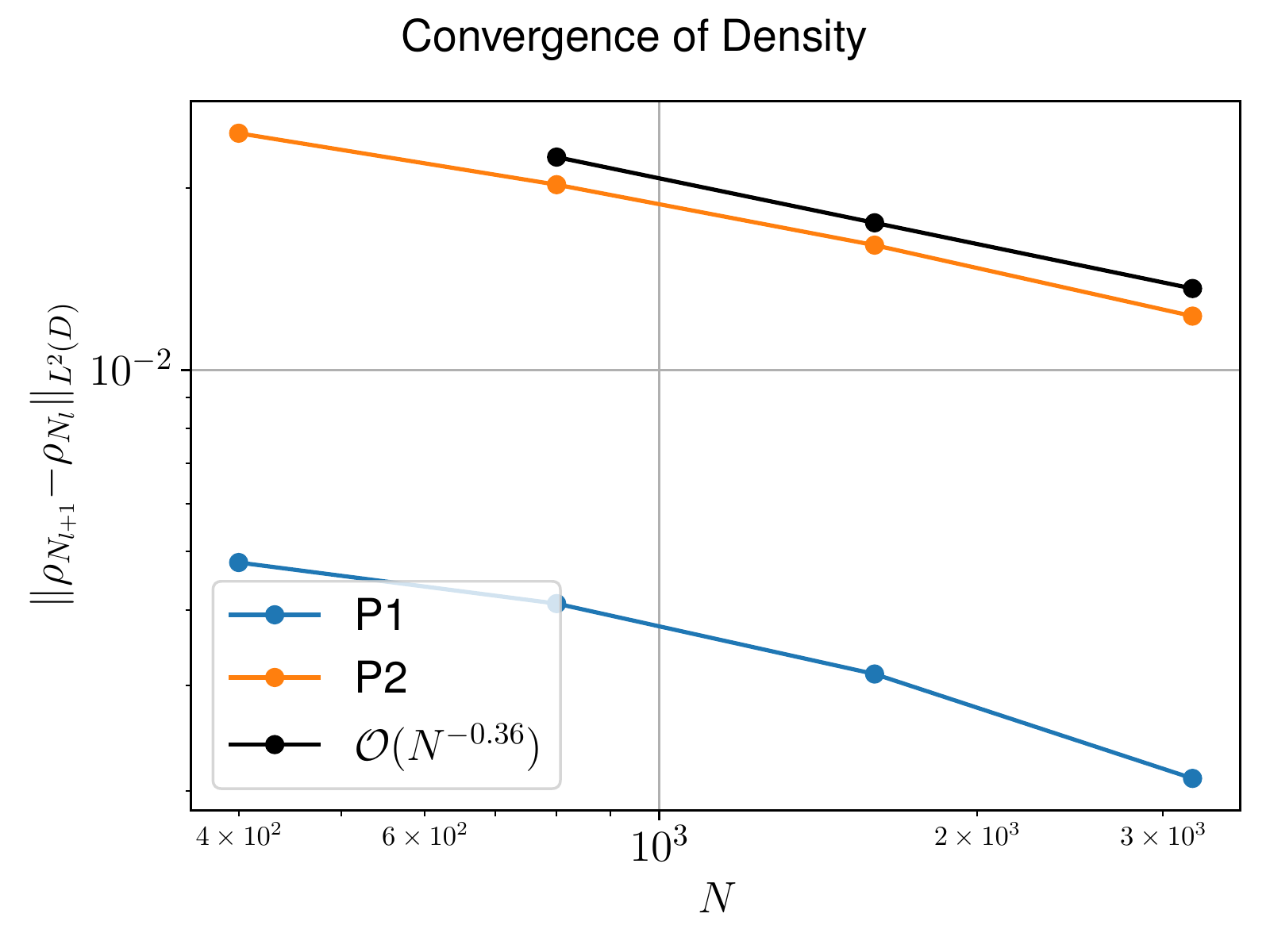}
\end{center}
\caption{Empirical convergence study for the two-phase Lax's shock tube problem for the Ab-initio method in its MLMC mode.}
\label{Fig:AI:T4_conv}
\end{figure}

\section{About the regime-generating strategy}

The evidences provided in the numerical experiments of Sections \ref{sec:AI:MC:NE} show that the methodology is converging with respect to the number of samples and with respect to the number of sub-volumes.
In particular, one establishes convergence for the volume fraction, which, in turn, implies convergence in mean of the characteristic functions $X^{(k)}$. 
Such a results is definitely non-trivial and it also provides evidences that the relative number of interfaces $\lambda_i := \bbE[N_{int} / \Delta x]$ is also converging.\\
Furthermore, numerical experiments showed that the methodology is stable under choice of the maximum number of sub-volumes. 
Such results seems to indicate that stability with respect to the sub-discretization is achieved in all the test-cases under consideration.\\
We stress here that the regime-generating algorithm provided in Alg.\ref{al:AI:eMGP} is based on a uniform distribution. 
Such an assumption is mostly justifiable for pragmatic reasons: if no information about the distribution of $X^{(k)}$ for some $k=1,2$ is given, then the uniform distribution assign equal weight to any event, since there is no reason to prefer any. 
What is more, we have shown that, in some cases, the ab-initio strategy is in good agreement with the DEM predictions. 
One can then reinterpret the DEM as to approximate ab-initio results starting from a uniform distribution.
The following section is devoted to a numerical investigation of similar strategies when starting from a different distribution.

Recently, in \cite{Perrier21}, a connection between bubble size and the image of a Gaussian process was established: the authors assumed that the characteristic function of the initial random two-phase distribution takes the form of a Gaussian Process (GP) \cite{BishopPR}. 
Such assumption is then used to derive an explicit form for the volume fraction, and for the closure problem.\\
More specifically, it is assumed that each initial characteristic function $X^{(k)}(x;\omega)$ can be written as
\begin{equation}
\label{eq:AI:X_gaussian}
X^{(k)}(x;\omega) = \frac{1 + \mathrm{sign}( g^{(k)}(x;\omega) )}{2}
\qquad
g^{(k)}(x;\cdot)
\sim
\mathrm{GP}(\mu^{(k)}(x), \sigma^{(k)}(x,y))
\end{equation}
where the mean $\mu\, :\, D \rightarrow \bbR$ and the positive definite kernel $\sigma\,:,\, D\times D\longrightarrow \bbR$ are two functions to be determined.\\
It is easy to check that for any $l\neq k = 1,2$, the saturation condition (\ref{eq:SaturationCondition}) implies that, for every $\omega\in\Omega$ $g^{(k)}(x;\omega) = - g^{(l)}(x;\omega)$, so that one can consider only one Gaussian process, say $g^{(1)}$.
In addition, simple calculations \cite{Perrier21} imply that
\[
\alpha^{(k)}(x) = \bbE\left[ X^{(k)}(x;\cdot) \right]
= \frac{1}{2}
\left(
1 + \mathrm{erf}
\left(
\frac{\mu^{(k)}(x)}{\sqrt{2}}
\right)
\right)
\]
so that, for a given $\alpha^{(k)}(x)$, one fixes the corresponding mean $\mu^{(k)}(x)$ of the GP $g^{(k)}(x;\cdot)$ by setting
\begin{equation}
\label{eq:AI:GP_mean}
\mu^{(k)}(x) = \sqrt{2}\mathrm{erf}^{-1}
\left(
2\alpha^{(k)}(x) - 1
\right).
\end{equation}
The advantage of considering such an approach is firstly provided by the characteristic property of GPs to lead to closed forms for posterior distributions and moments of any order, which, in turn, may be used to derive explicit formulas to solve the closure problem \cite{Perrier21}.
On the other hand, it remains unclear why considering a GP distribution of phases, in addition to leaving open the problem of defining an adequate (in some sense) kernel function $\sigma^{(k)}$.
In this section we want to comment on the difficulties related to taking such perspective in the context of the ab-initio method.\\
\begin{figure}[!htbp]
\begin{center}
\includegraphics[scale=\figsize]{\main/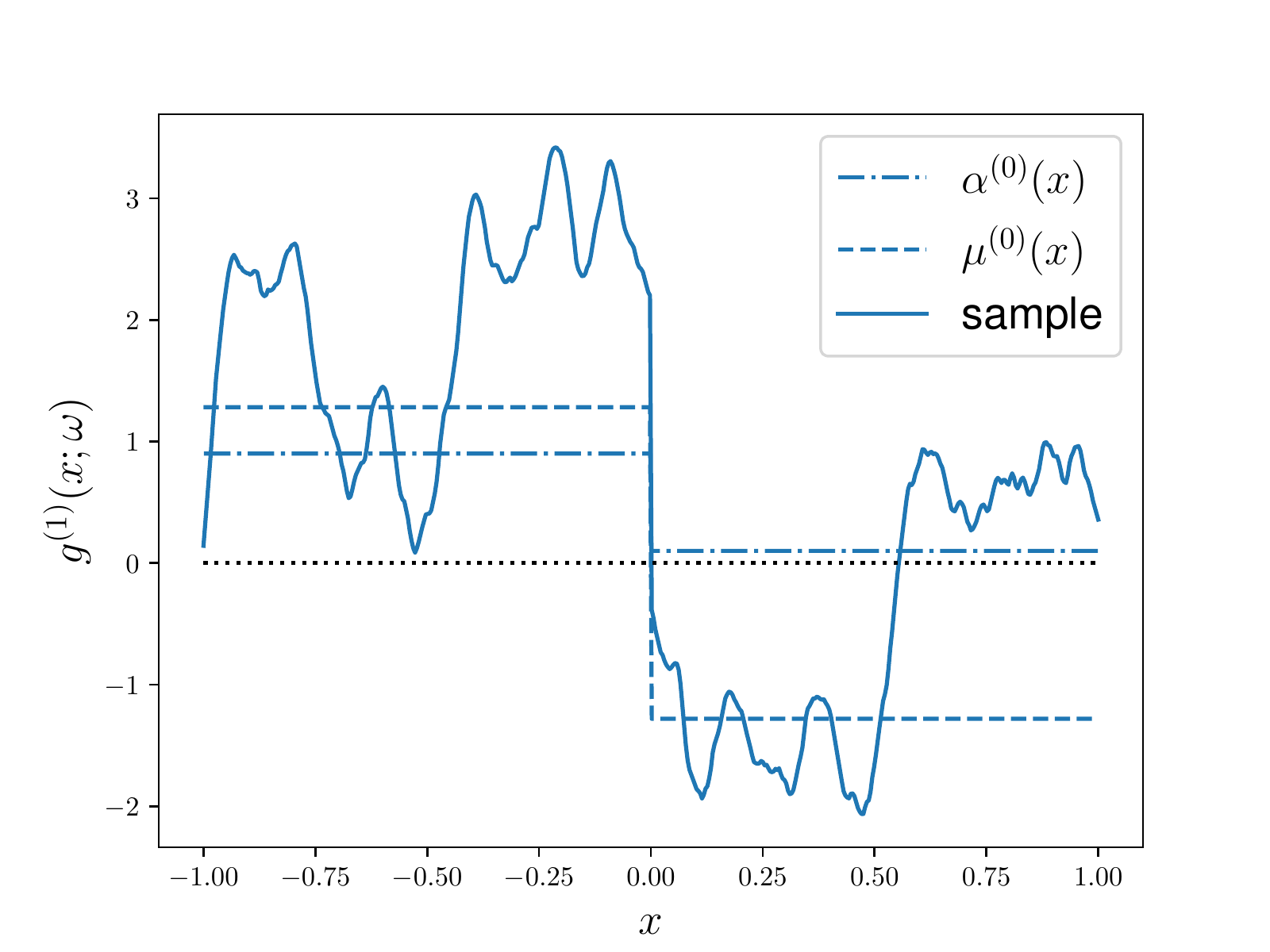}\,
\includegraphics[scale=\figsize]{\main/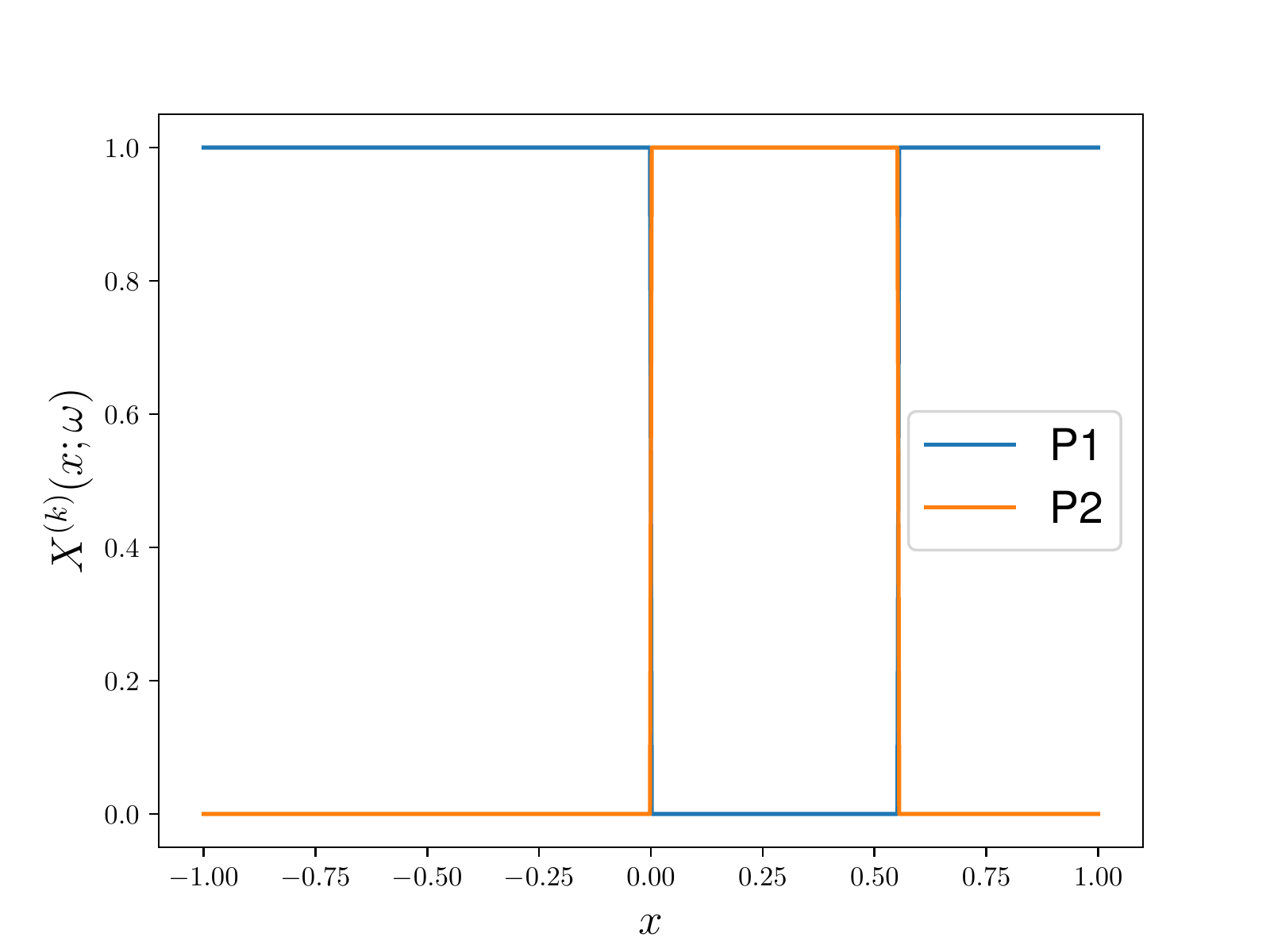}
\end{center}
\caption{Prototypical generation of a random two-phase distribution generated by a GP $g^{(1)}(x;\cdot)$ for a given initial volume fraction $\alpha^{(1)}(x)$. 
Left: volume fraction, mean $\mu^{(1)}(x)$ of $g^{(1)}(x;\cdot)$ and a typical sample;
Right: the two-phase distribution $X^{(k)}(x;\omega)$ induced by the realization plotted on the left.}
\label{Fig:AI:GP}
\end{figure}
Based on the assumed form (\ref{eq:AI:X_gaussian}) of $X^{(k)}(x;\omega)$, one understands this latter in terms of the level-set formulation
\[
X^{(k)}(x;\omega) 
=
1_{ \lbrace
g^{(k)}(x;\omega) \geq 0 
\rbrace
} 
=
\begin{cases}
1 & g^{(k)}(x;\omega) \geq 0\\
0 & \textit{otherwise}
\end{cases}.
\]
Moreover, for a given volume fraction function $\alpha^{(k)}(x)$, the mean value of the GP $g^{(k)}$ is fixed and at the practical level one can take advantage of the following representation
\[
g^{(k)}(x;\omega) = \mu^{(k)}(x) + \tilde{g}^{(k)}(x;\omega)
\qquad
\quad
\tilde{g}^{(k)}(x;\omega)
\sim
\mathrm{GP}(0,\sigma^{(k)}(x,y)).
\]
For notation ease, we will make no distinction between $g^{(k)}$ and $\tilde{g}^{(k)}$.\\
Once the target GP is fixed, an alternative to the regime-generating algorithm detailed in Section \ref{sec:MSR} can be constructed by sampling the GP $g^{(k)}$: one identifies (initial) interface locations as the points where samples change sign, see Fig.\ref{Fig:AI:GP} for an illustration.
Hence, by repeatedly sampling the GP $g^{(1)}$, an ensemble of i.i.d. samples can be constructed (see Alg \ref{al:AI:MGP_gaussian}) such that each of them can be evolve using the FT-operator. Notice that there is no need for an analogous relation to the consistency requirement (\ref{eq:AI:ConsistencyCond-initialLength}), since the mean of the GP is constructed as to comply with the initial volume fraction.

\begin{algorithm}
\caption{Micro-scale Generation algorithm under Gaussian distribution}\label{al:AI:MGP_gaussian}
\begin{algorithmic}
\STATE $\textbf{Data\,:\,}$ 
Mesh-width $\Delta$, initial volume fractions $\alpha^{(k)}_0(x)$ $k=1,2$,
kernel function $\sigma^{(k)}(x,y)$ for one $k\in\lbrace 1,2\rbrace$;\\
\STATE $\textbf{Output\,:\,}$ Two-phase distribution $X^{(k)}_0(x;\omega)$;

\STATE $\bullet$ Define the mesh $(x_i)_{i=1}^M$ such that it is an equispaced discretization of the domain having width $\Delta$;
\STATE $\bullet$ Define the GP $g^{(k)} \sim GP(0,\sigma^{(k)})$;
\STATE $\bullet$ Sample $g^{(k)}$: for a given realization $\omega\in\Omega$, project the sample on the mesh $(x_i)_{i=1}^M$ as to generate $\lbrace
g^{(1)}(x_i,\omega)\rbrace_{i=1}^M$;
\STATE $\bullet$ Update the sample-projection by adding the mean $\mu^{(k)}(x)$, obtained according to (\ref{eq:AI:GP_mean});
\STATE $\bullet$ Define the two-phase distribution $X^{(k)}(x_i,\omega)$ according to (\ref{eq:AI:X_gaussian});
\end{algorithmic}
\end{algorithm}

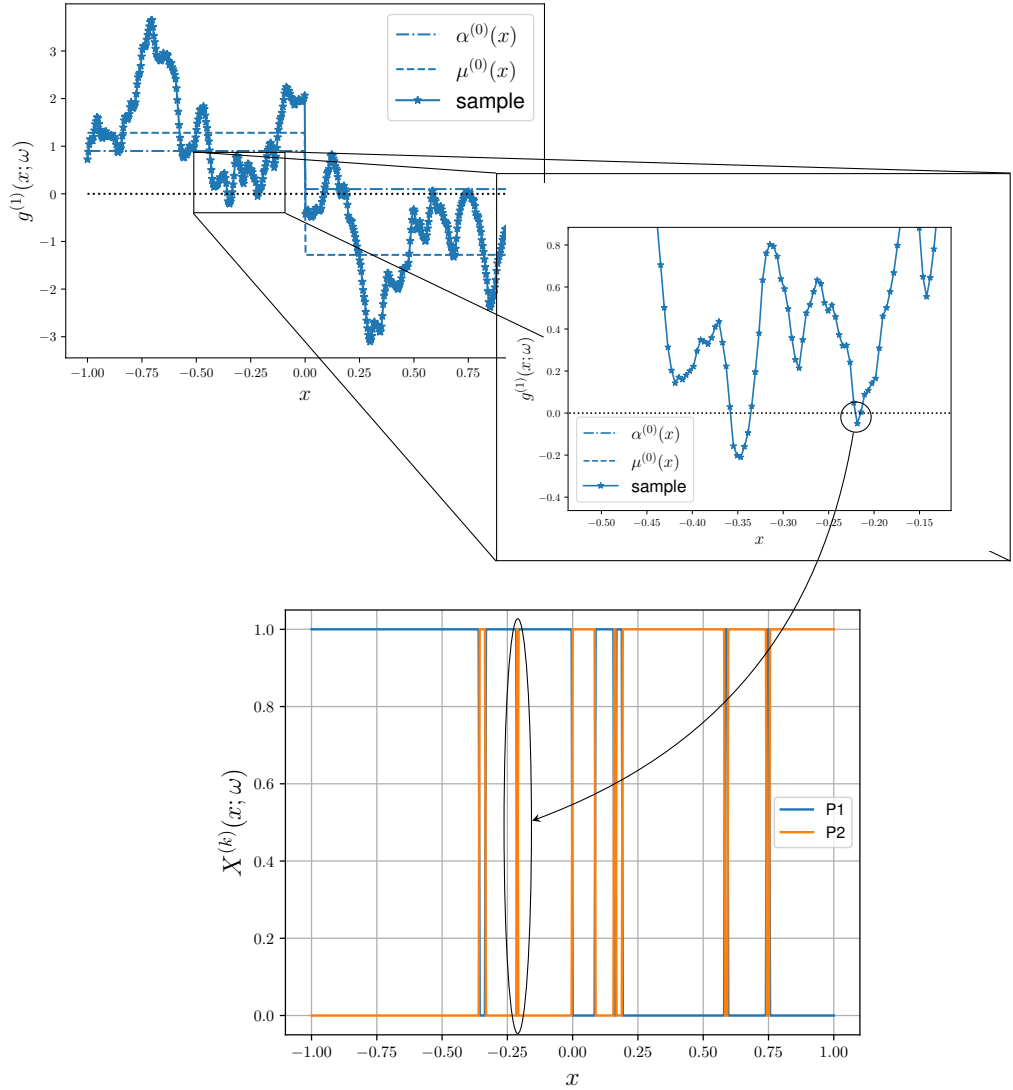
\begin{figure}[!htbp]
\begin{center}
\begin{tikzpicture}
\node (SP) {
\includegraphics[scale=0.5]{\main/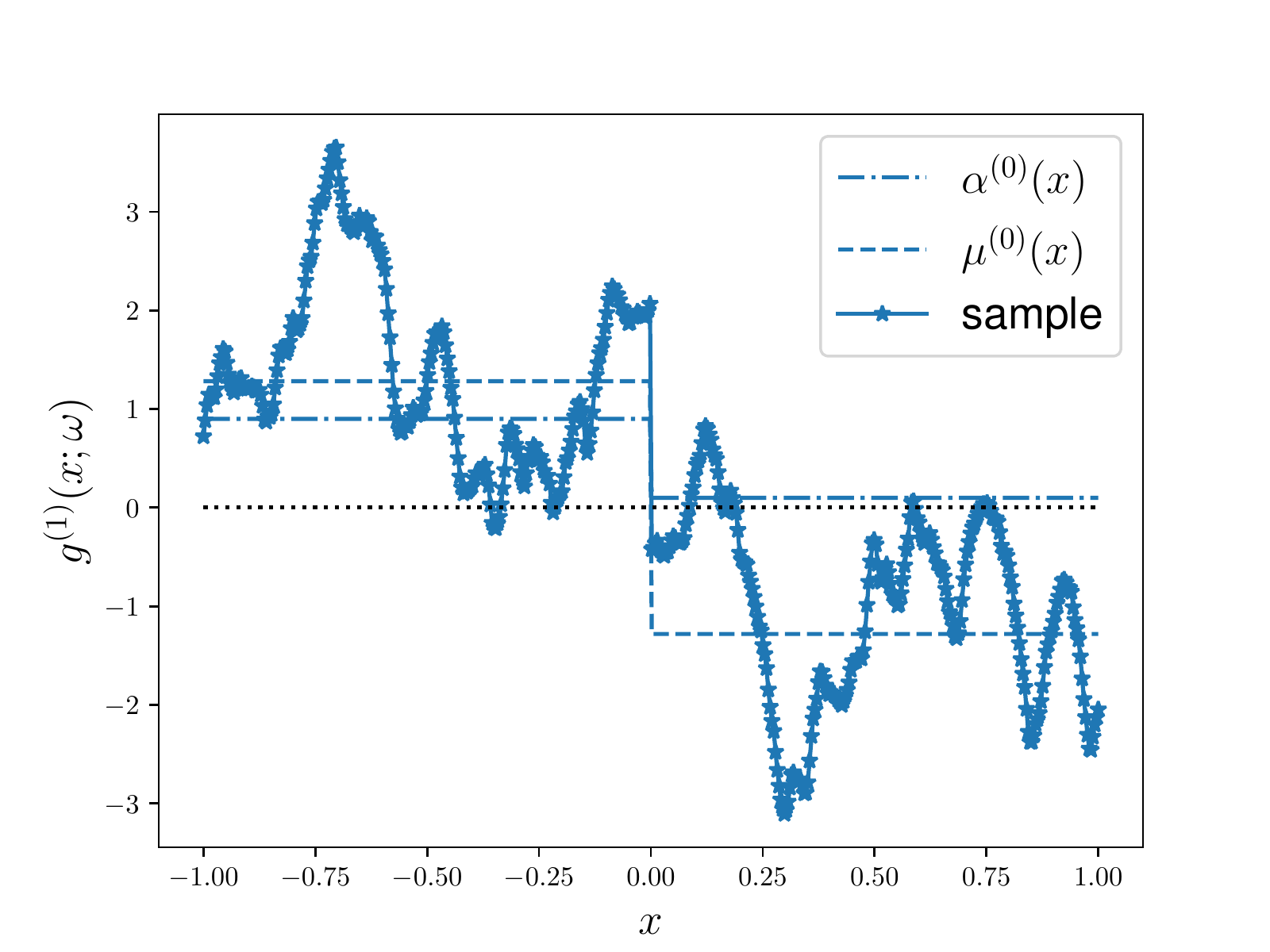}
};

\node (X) [rectangle, below of=SP, xshift=3.5cm, yshift=-7.5cm] {
\includegraphics[scale=0.6]{\main/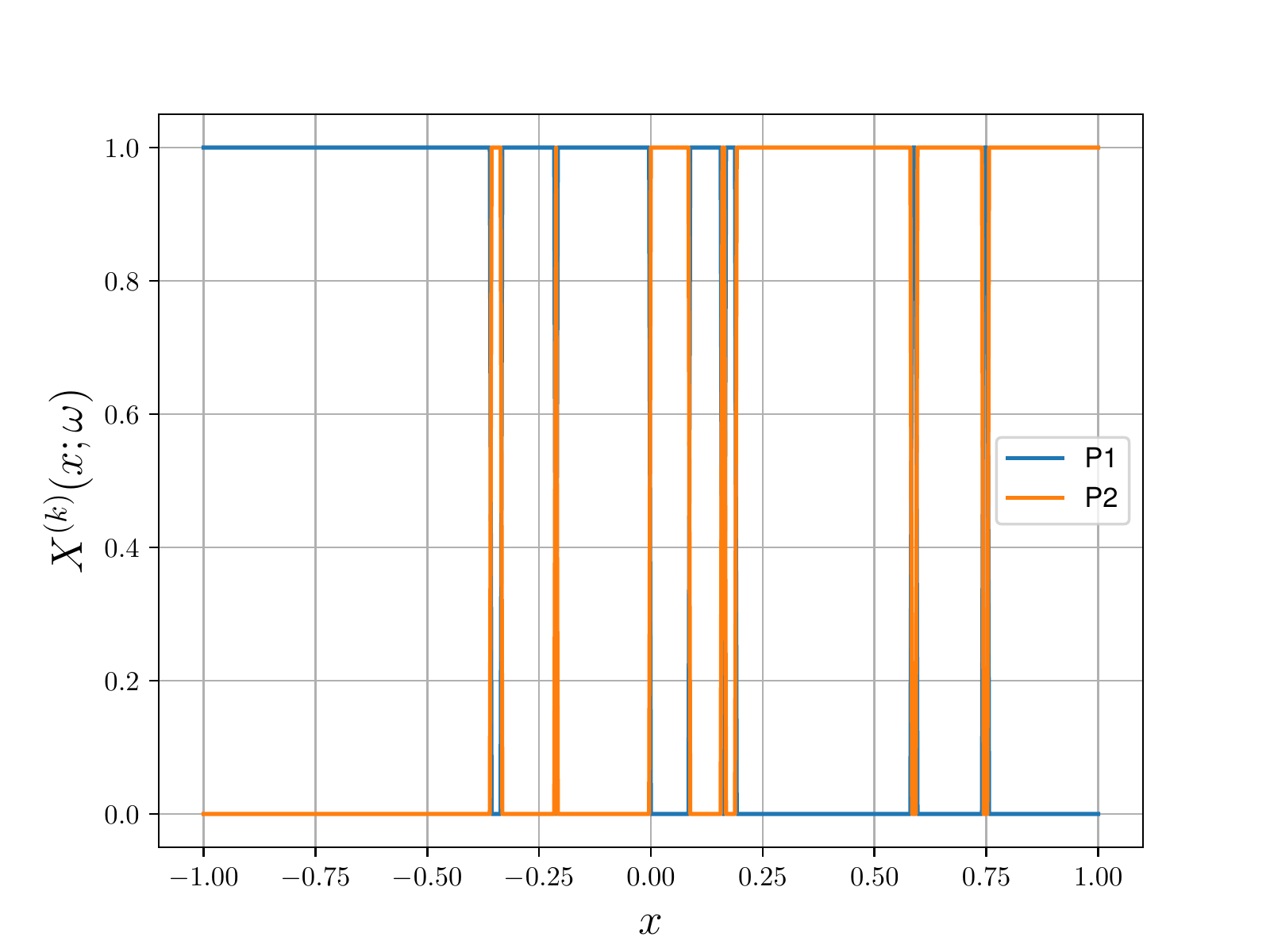}
};

\node [rectangle, draw=black, 
	below right= 4cm of SP.north west,
	xshift = -0.0cm, 
	minimum width=1.2cm, minimum height=0.8cm] 
	(Z_on_SP) {};

\node (SPz) [rectangle, draw=black, right of=SP, xshift = 5cm, yshift=-2.5cm] {
\includegraphics[scale=0.4]{\main/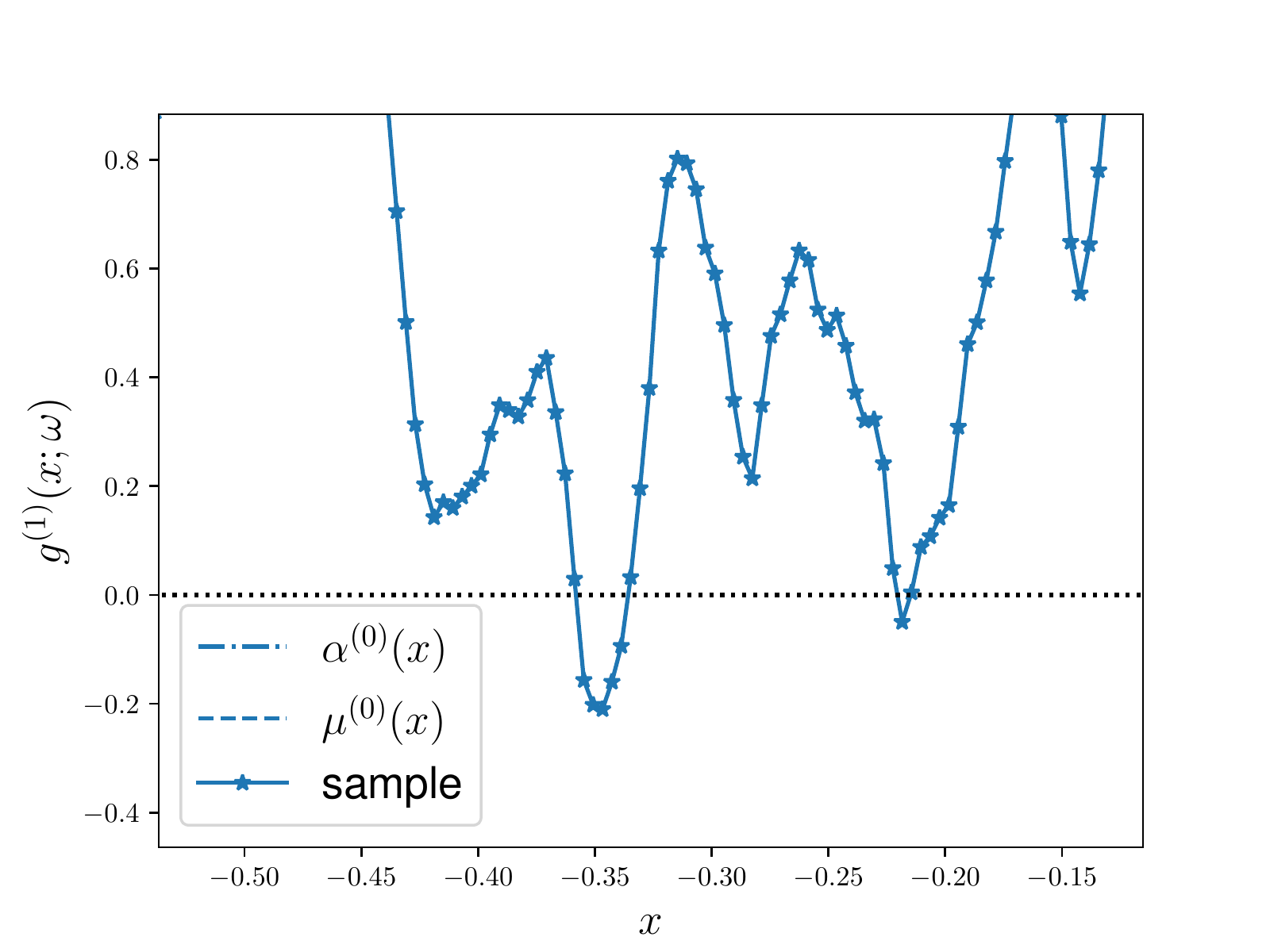}
};

\draw (Z_on_SP.north east) -- (SPz.north east);
\draw (Z_on_SP.south west) -- (SPz.south west);
\draw (Z_on_SP.north west) -- (SPz.north west);

\begin{scope}[on background layer]
\draw (Z_on_SP.south east) -- node [pos=0.37] (Mup){} (SPz.south east);
\end{scope}

\draw (Z_on_SP.south east) -- (Mup);

\node [ellipse,draw=black, 
		below right= 5.5cm of SPz.north west,
		xshift=0.7cm, yshift=0.8cm,		
		minimum width=0.4cm, minimum height=0.4cm] (Z_on_SPz) {};
		
\node [ellipse,draw=black, 
		below right= 4cm of X.north west,
		xshift=1.45cm,	yshift=0.9cm,	
		minimum width=0.15cm, minimum height=5.5cm] (Z_on_X) {};

\path[-stealth] (Z_on_SPz) edge[bend left] (Z_on_X);
\end{tikzpicture}
\caption{Example of a sample of the GP $g^{(1)}$ with an isolated small dispersed particle.}\label{Fig:AI:GP_removed}
\end{center}
\end{figure}
The use of GP in modern applications has motivated extensive efforts in their construction at the numerical level, and most of the coding languages to date provide support for their implementation.
Hence, at the practical level the construction of a sampling strategy is affordable.

Nevertheless,
it should be noted that any sample from a GP takes a finite representation over a mesh at the numerical level.
Indeed, given a gird of points $\VX = (x_i)_{i=1}^M$, the construction of $g^{(1)}(x,\omega)$ for some $\omega\in\Omega$ is achieved by sampling the (multidimensional) normal distribution $\CN(\mu^{(k)}(\VX), \sigma^{(k)}(\VX,\VX))$, where $\mu^{(k)}(\VX)\in\bbR^{M}$ and $\sigma^{(k)}(\VX,\VX)\in\bbR^{M\times M}$ are the vector and matrix (respectively) generated by evaluating the corresponding functions on $\VX$. 
In particular, the generation of one sample requires, on one hand, the construction of the matrix $\sigma^{(k)}(\VX,\VX)$ and on the other its consequent Cholesky decomposition at the cost of order $\CO(M^3)$.
Thus, the dimension of the sampling grid used to evaluated $g^{(1)}$ cubicly increases the cost of sampling a GP as the mesh is refined. 
This introduces already a computational disadvantage as compared to the Alg.\ref{al:AI:eMGP}, due to matrix-decomposition.\\
Furthermore, notice that it is not difficult to incur in situations where the sampled GP changes sign at localized points (see Fig. \ref{Fig:AI:GP_removed}), so that an isolated dispersed portion of matter of width $\Delta = \CO(M^{-1})$ is present and no finer scales can be reproduced using the mentioned algorithm.
Hence, the (arbitrary) sampling mesh width $\Delta$ in Alg.\ref{al:AI:MGP_gaussian} defines a control on the minimum width of dispersed phase, in complete analogy to the choice of the number of sub-volumes in Alg.\ref{al:AI:eMGP}. 
The sub-volumes that are affected by each phase depend on the choice of the kernel function.

In \cite{Perrier21}, big attention is devoted to the choice of the Matern$(\nu,\zeta)$ kernel function
\[
\sigma^{(k)}_{\nu,\zeta}(x,y)
=
\frac{2^{1-\nu}}{\Gamma(\nu)}
\left(
\frac{\sqrt{2\nu}(y-x)}{\zeta}
\right)^\nu
\CK_\nu
\left(
\frac{\sqrt{2\nu}(y-x)}{\zeta}
\right)
\] 
where $\nu,\zeta$ are positive parameters controlling the variation of the GP, $\CK_\nu$ is a modified Bessel function and $\Gamma$ denotes the gamma-function
\[
\Gamma(z) = 
\int_0^\infty
t^{z-1}e^{-t}\,\dt.
\]
Matern kernels are widely used in many applications due to their flexibility under variation of the parameters $\nu$ and $\zeta$, which control smoothness of samples.
\begin{figure}[!htbp]
\begin{center}

\subfloat[Gaussian sample, $\Delta = 5\cdot 10^3$]{
\includegraphics[scale=0.25]{\main/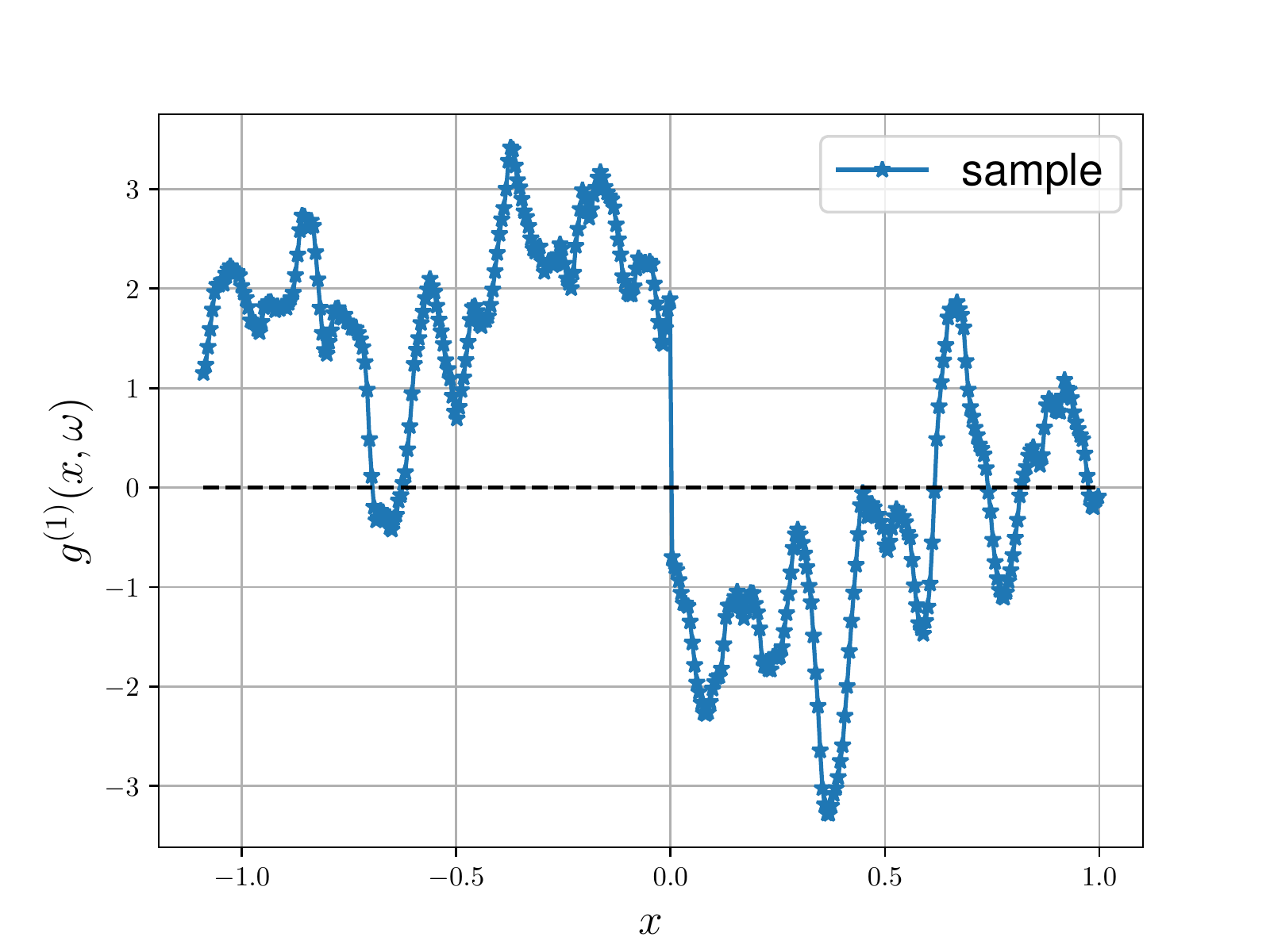}
}
\,
\subfloat[Gaussian, $\Delta = 5\cdot 10^3$]{
\includegraphics[scale=0.25]{\main/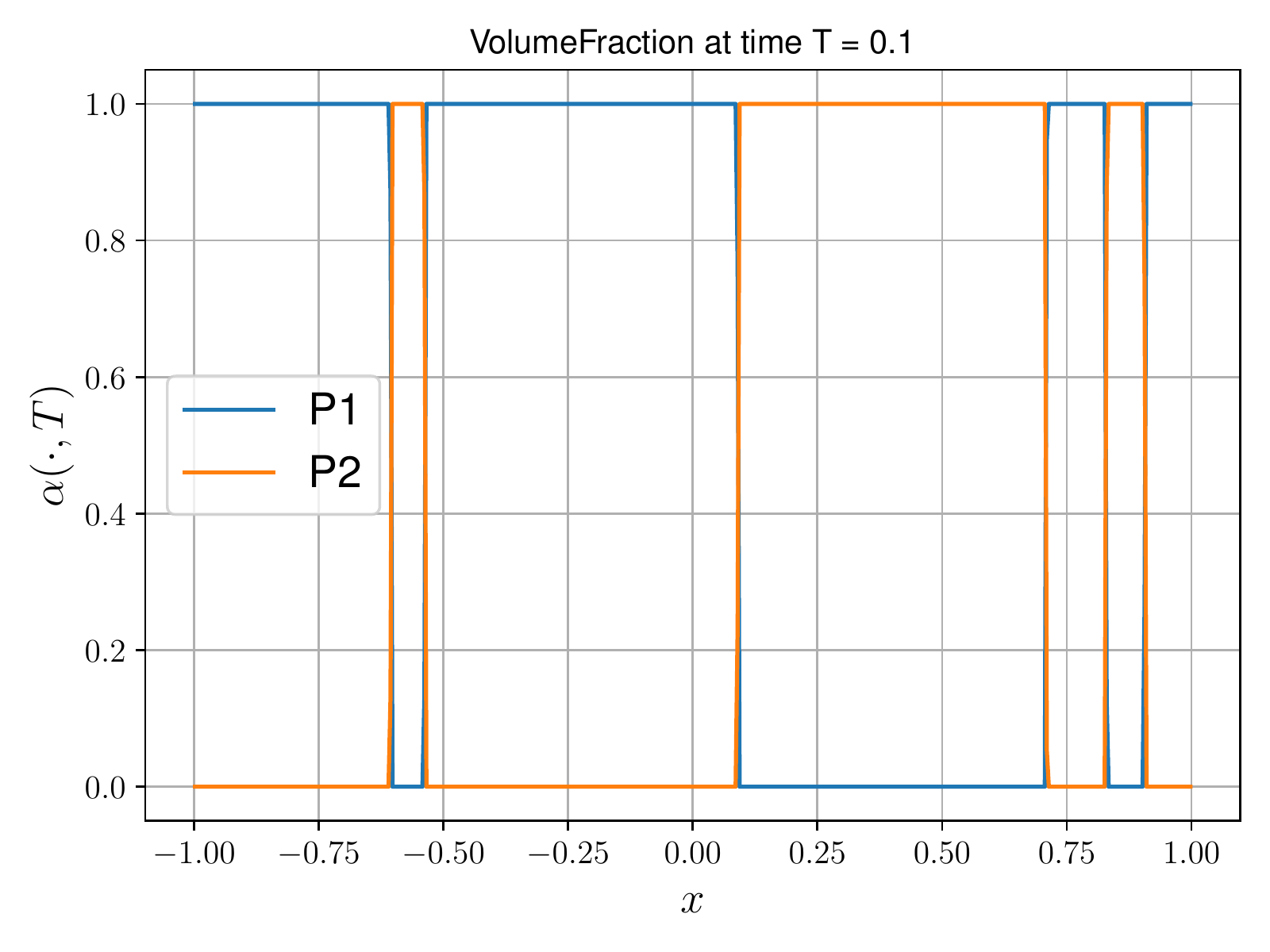}
}
\,
\subfloat[Uniform, $\Delta = 5\cdot 10^3$]{
\includegraphics[scale=0.25]{\main/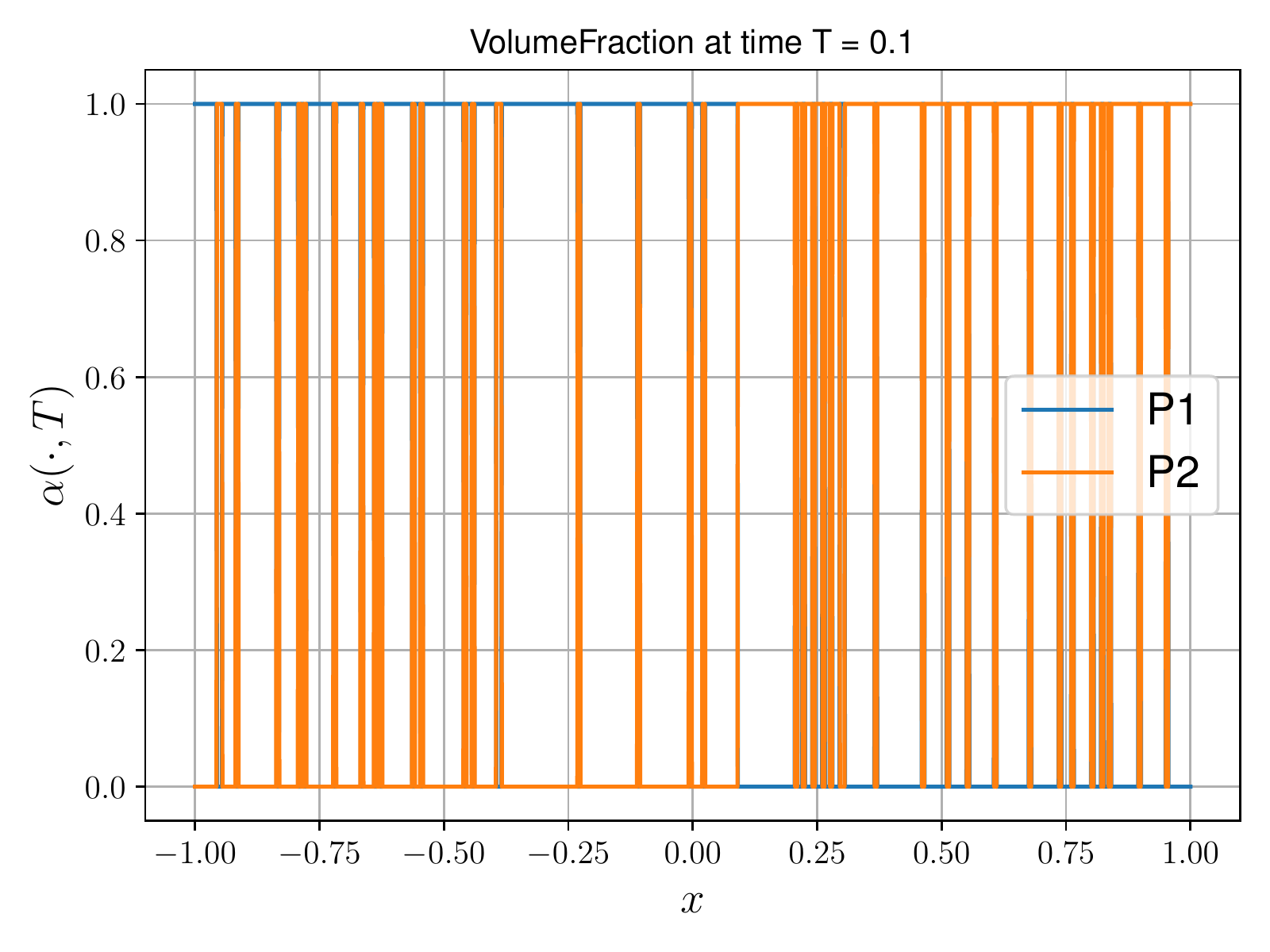}
}\\

\subfloat[Gaussian sample, $\Delta = 1.25\cdot 10^3$]{
\includegraphics[scale=0.25]{\main/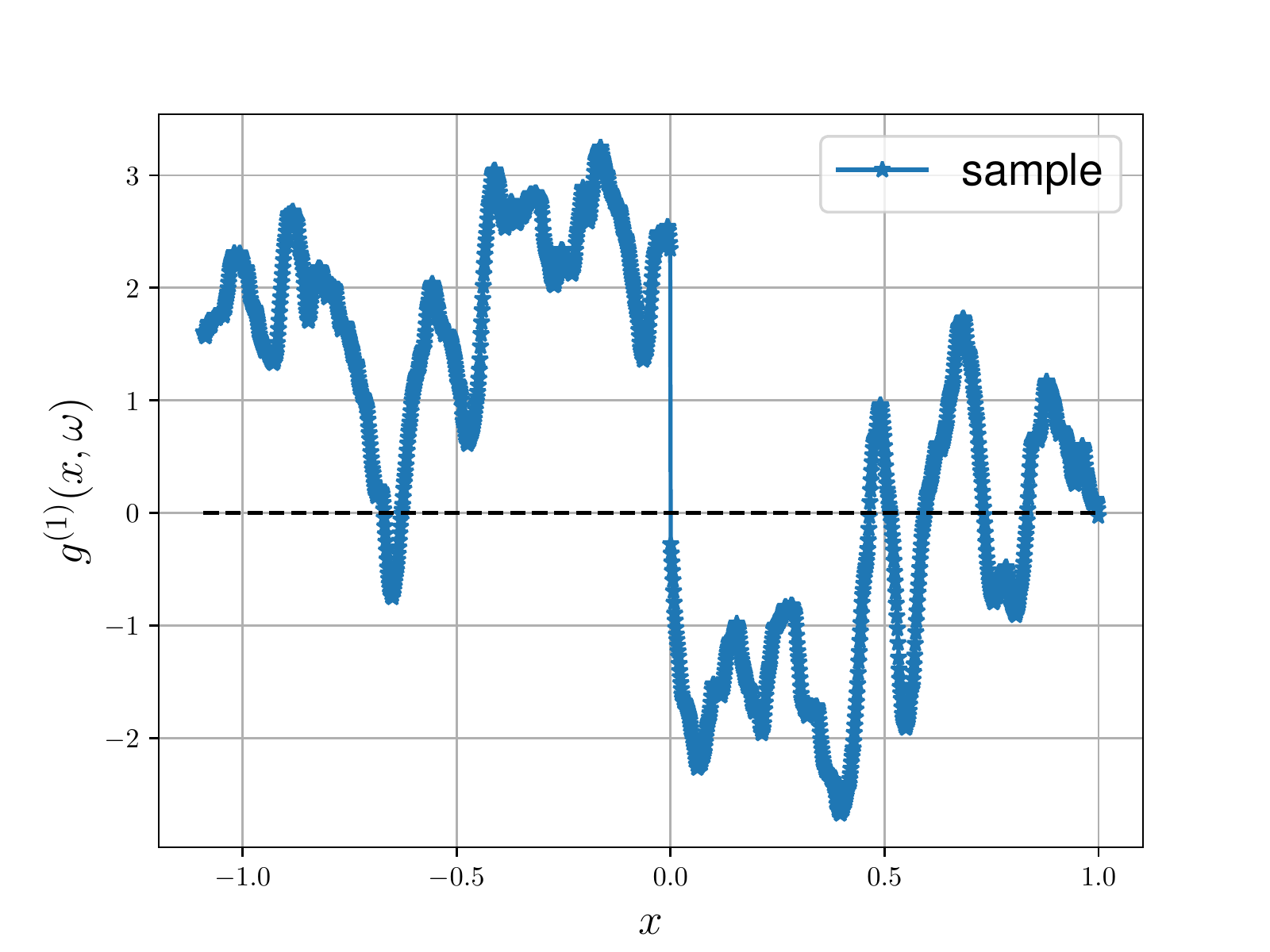}
}
\,
\subfloat[Gaussian, $\Delta = 1.25\cdot 10^3$]{
\includegraphics[scale=0.25]{\main/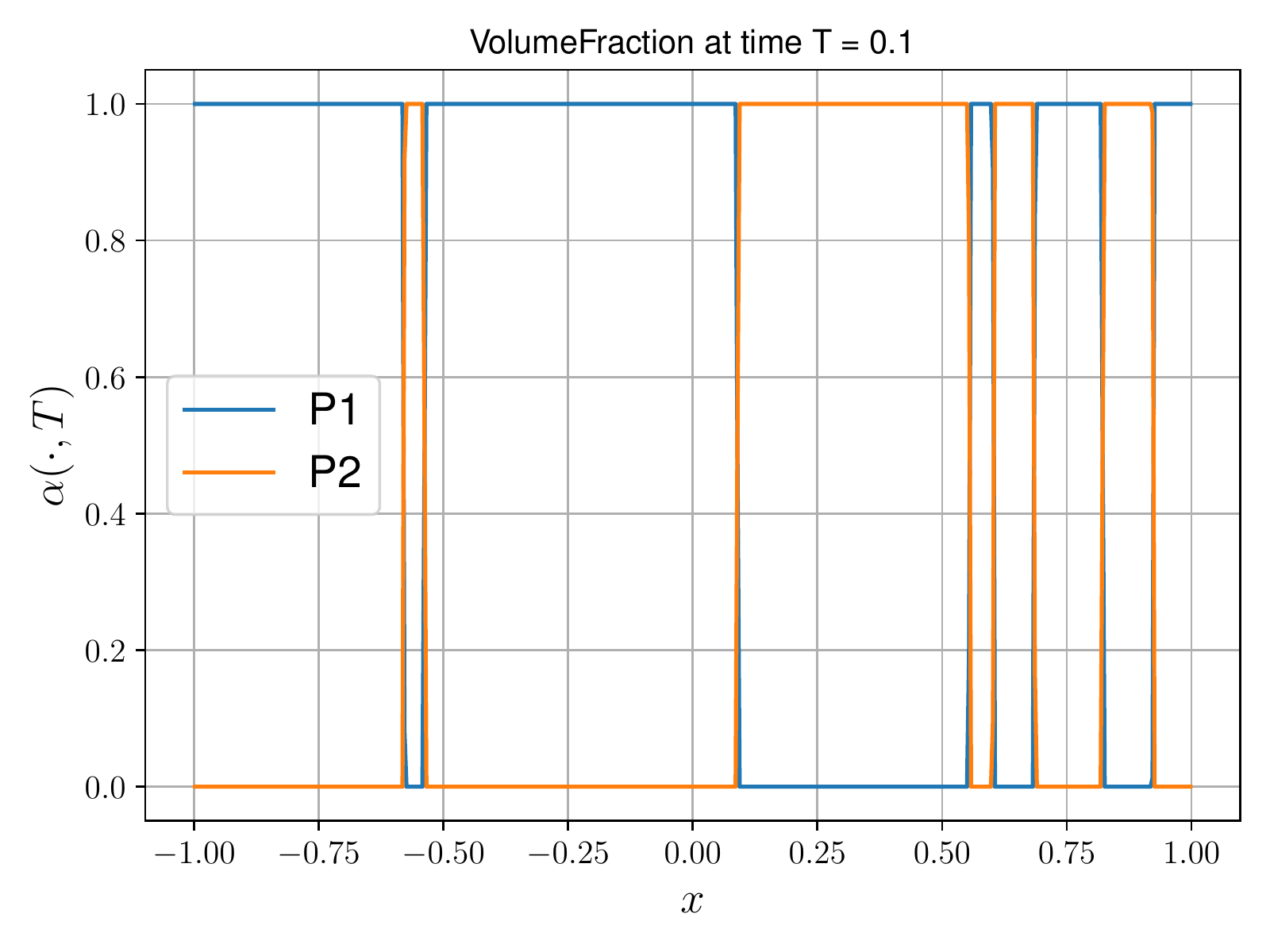}
}
\,
\subfloat[Uniform, $\Delta = 1.25\cdot 10^3$]{
\includegraphics[scale=0.25]{\main/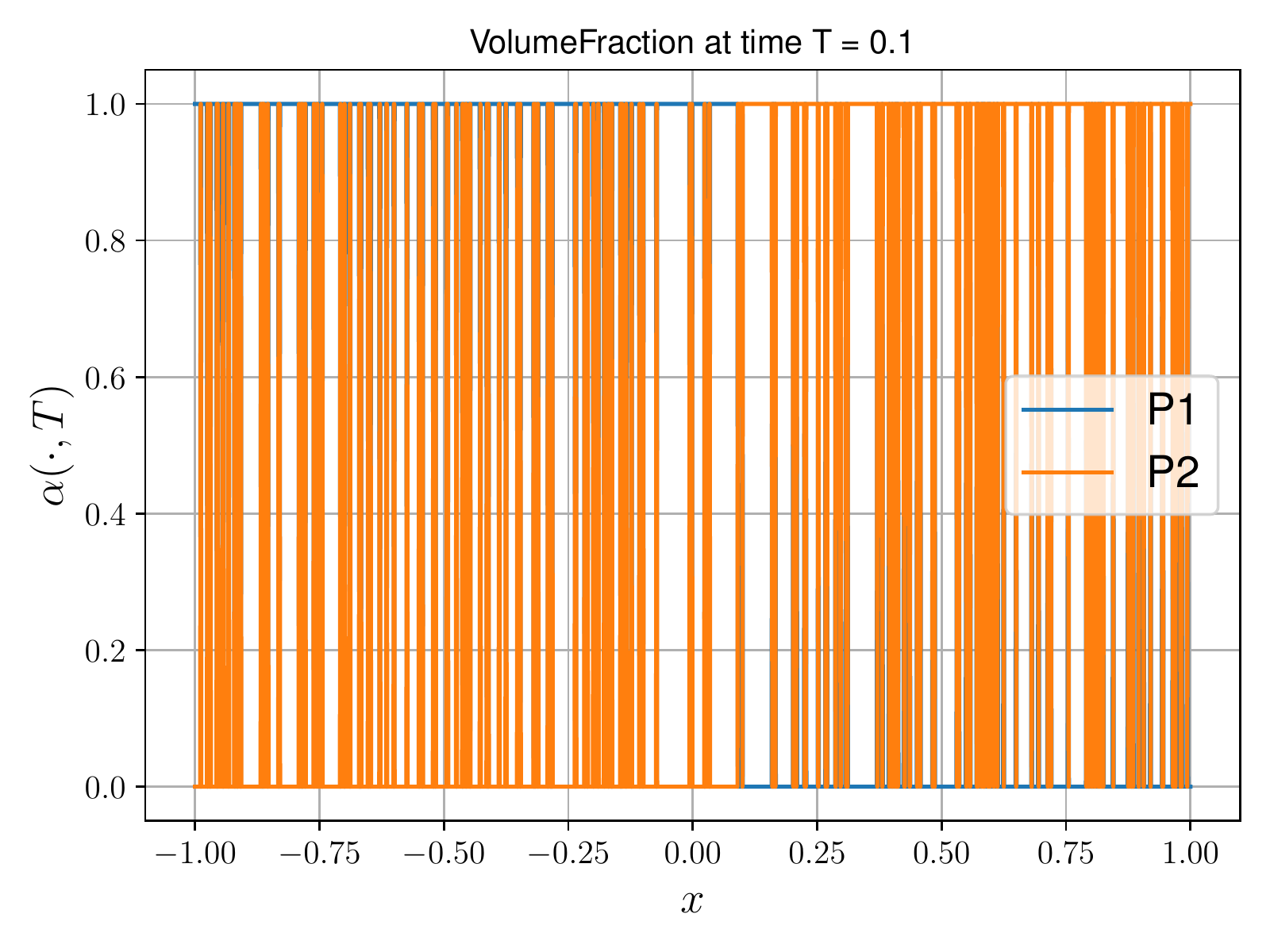}
}\\

\subfloat[Gaussian sample, $\Delta = 0.3125\cdot 10^3$]{
\includegraphics[scale=0.25]{\main/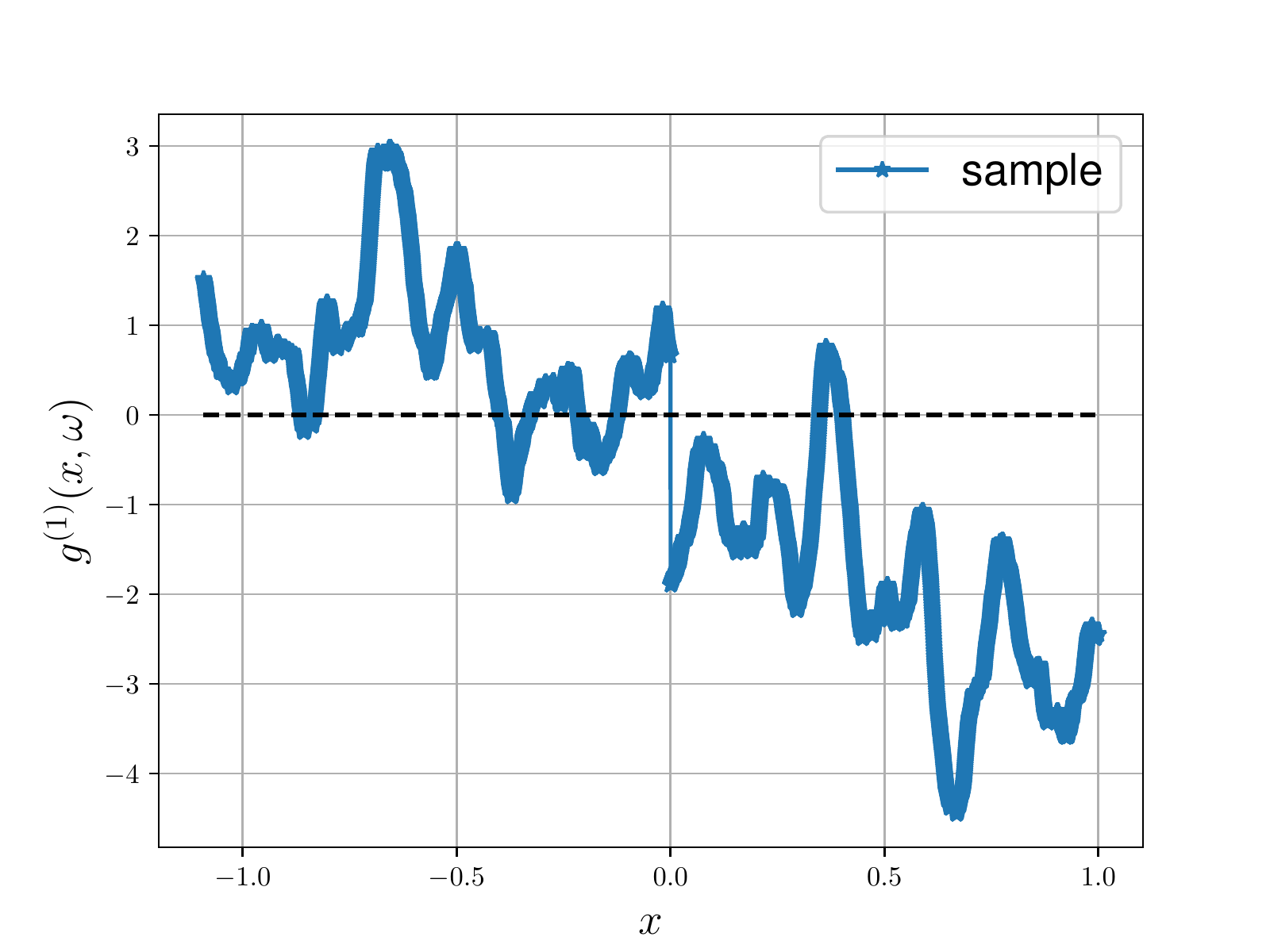}
}
\,
\subfloat[Gaussian, $\Delta = 0.3125\cdot 10^3$]{
\includegraphics[scale=0.25]{\main/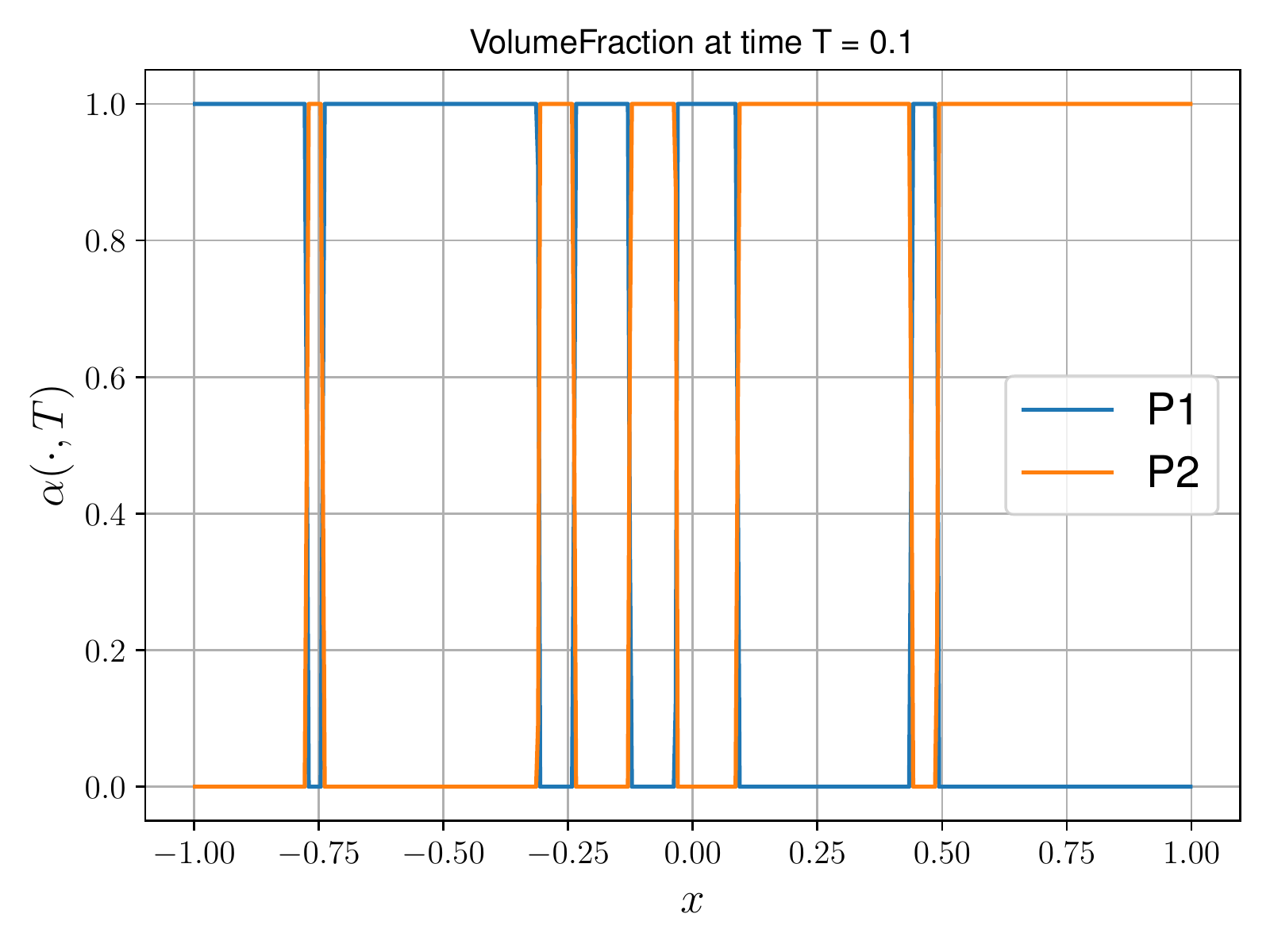}
}
\,
\subfloat[Uniform, $\Delta = 0.3125\cdot 10^3$]{
\includegraphics[scale=0.25]{\main/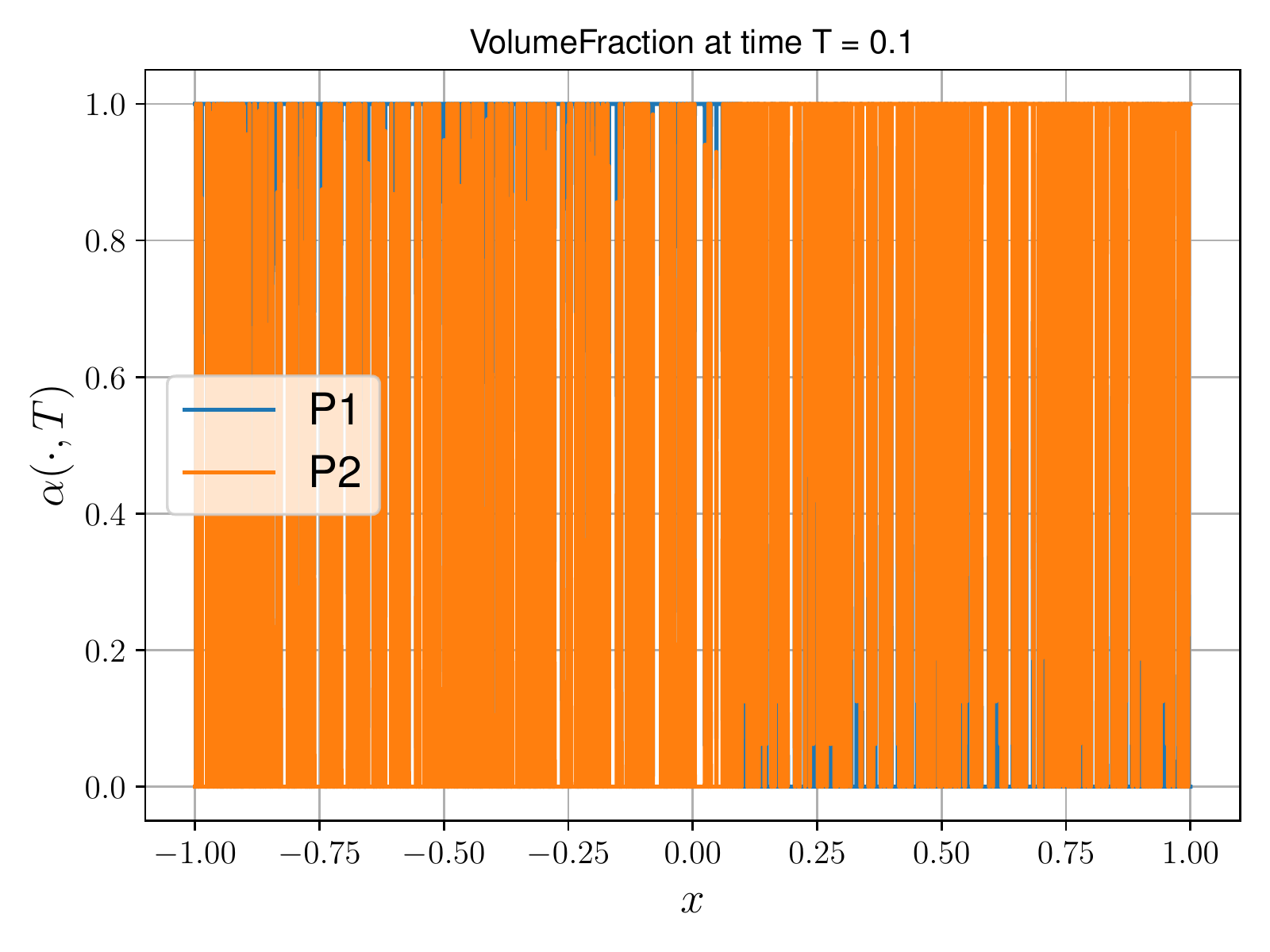}
}\\

\caption{Comparison of two prototypical samples generated by Alg.\ref{al:AI:eMGP} (Uniform) and Alg.\ref{al:AI:MGP_gaussian} (Gaussian).}
\label{Fig:AI:GP_smp_cmp}
\end{center}
\end{figure}

To help appreciate the difference in the sampling strategies provided by the Alg.\ref{al:AI:eMGP} (Uniform) and Alg.\ref{al:AI:MGP_gaussian} (Gaussian), we plot in Fig.\ref{Fig:AI:GP_smp_cmp} two prototypical realizations generated with both algorithms when applied to the test case of Sec. \ref{sec:AI:MC:NE:Un}, using several sub-scale resolutions and under the assumption of considering a Matern$(1.5,0.06)$ kernel function.
Furthermore, we plot also the GP sample generating the initial two-phase distributions in case of Alg.\ref{al:AI:MGP_gaussian}.
For each row (i.e. for a fixed sub-scale resolution), one can observe that the evolved samples generated with the two algorithms differ in regime: uniform distribution generates a dispersed regime constituting of isolated dispersed portion of matter, whereas the Gaussian one induces phase agglomerates. 
Such trend is preserved as the sub-scale resolution is refined, meaning that uniform distribution generates even more dispersed matter, while Gaussian realizations seems to produce (clusterized) larger portion of phases.
Such an observation indicates that, under the assumption of a Gaussian distribution, great care needs to be put in choosing the appropriate kernel function, as it corresponds to choosing a specific form for the regime under consideration. 

Notice that both strategies produce diverse realizations as the sub-scale is refined, meaning that refinement of the sub-scale \emph{does not} define a better resolution on the same sample. 
Indeed, sampling the GP at finer grids requires the evaluation of a different multivariate normal distribution and different Cholesky decompositions, so that each sample ends up being different. 
In turn, the variation of disperse matter location induces lack of sample-wise convergence under sub-scale refinement, in contrast to what happens under uniform distribution generation of samples.

Lack of sample-wise convergence was recently observed to be enjoyed by numerical approximations of turbulent flow simulations, and the starting point for the design of novel solution paradigms for systems of hyperbolic conservation laws \cite{LyePhD, Fjordholm2017}. In the same references, convergence (in adequate norm) was recovered by considering solutions as the mean of some statistical process, suggesting that an homogenization process takes place when passing to the averages.
The present case presents several analogies to this latter, and we would like to comment on the implications of choosing a specific kernel function for the sub-discretization problem.

\begin{figure}[!htbp]
\begin{center}
\subfloat[Gaussian $L = 100$ and $\zeta = 0.06$]{
\includegraphics[scale=0.3]{\main/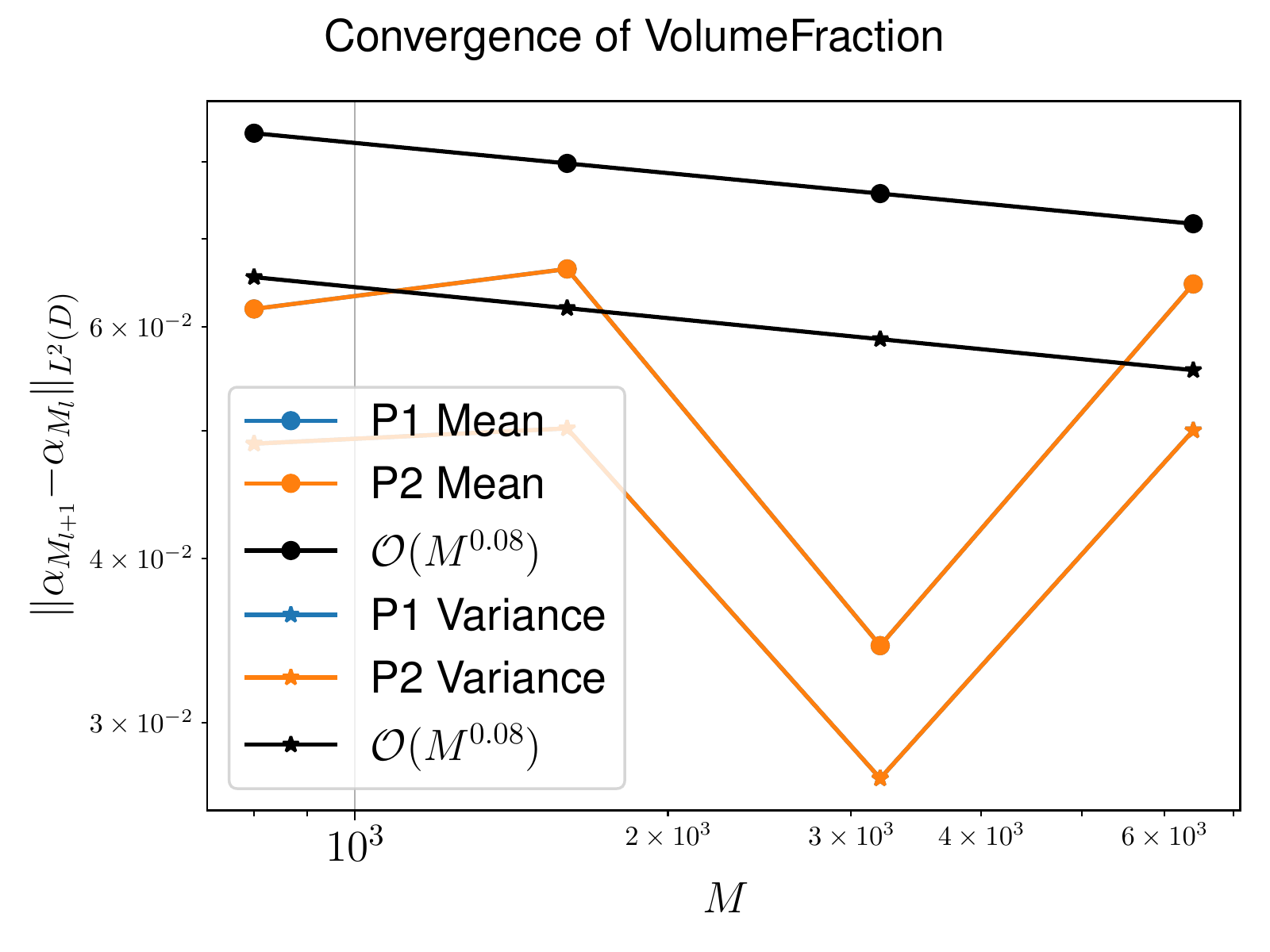}
}\,
\subfloat[Gaussian $L = 800$ and $\zeta = 0.06$]{
\includegraphics[scale=0.3]{\main/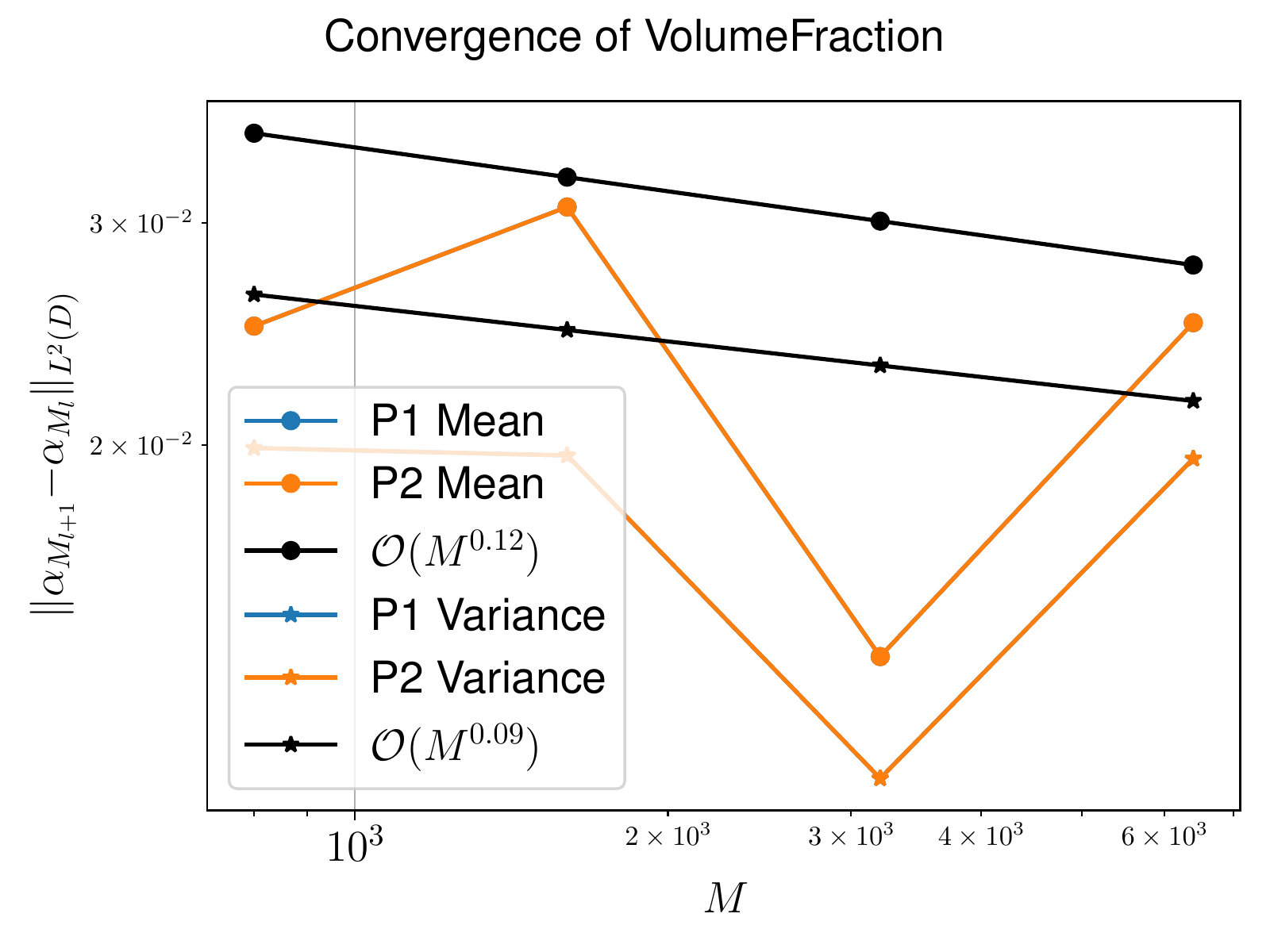}
}
\caption{Empirical convergence study for the regime-generating algorithm \ref{al:AI:MGP_gaussian} under sub-scale refinement.}
\label{Fig:AI:GP:ens_conv}
\end{center}
\end{figure}

As before, we first run the test case about uniform conditions (Sec.\ref{sec:AI:MC:NE:Un}) for two fixed number of samples ($L=200$ and $L=800$) using a Matern$(\frac{3}{2},\zeta)$ with $\zeta = 0.06$ and plot Cauchy rates for the volume fraction in Fig.\ref{Fig:AI:GP:ens_conv}.
Lack of convergence is recovered disrespectfully of the number of samples. 
In particular, we highlight that for a moderate number of samples, no substantial convergence can be appreciated. This implies that the homogenization property introduced by the passage to the average is reduced by the (essentially) non-decreasing trend of variance. 

The above results show that convergence for big sizes of dispersed particle is not taking place, or happening at a very slow rate to be irrelevant for practical usage.
The significance of this results it that the contribution of small dispersed particle is unavoidable to recover convergence in the ab-initio framework at an efficient/affordable manner.
This has the disadvantage of increasing the computational cost to evolve any of such samples, whose bottleneck can be mitigated using Multi-level strategies \cite{Sukys14}.
To this extent, one should note that it is necessary to define a sufficiently large sampling mesh to recover some notion of convergence, which would then increase cubicly the computational cost to produce each sample.
Hence, a deep investigation about the computational advantages for such an approach seems to be necessary.

Conversely, reducing the length scale for the kernel function (and consequently reducing the smoothness of the corresponding GP), do present some favorable advantages.
Indeed, we propose in Fig.\ref{Fig:AI:GP_smp_cmp_zeta_small} a similar comparison to the one exposed in Fig.\ref{Fig:AI:GP_smp_cmp}, where samples are produced by reducing the hyperparameter $\zeta$ up to $0.001$.
Notice the increased complexity of samples as the sub-scale is refined for the Gaussian samples, as opposed to the situation depicted in Fig.\ref{Fig:AI:GP_smp_cmp}.
Thus, reducing the smoothness of the underlying GP (or equivalently, by increasing its total variation), one recovers a similar trend under sub-scale refinement to the one observed for the uniform distribution-based algorithm.

\begin{figure}[!htbp]
\begin{center}
\subfloat[Gaussian sample, $\Delta = 5\cdot 10^3$]{
\includegraphics[scale=0.25]{\main/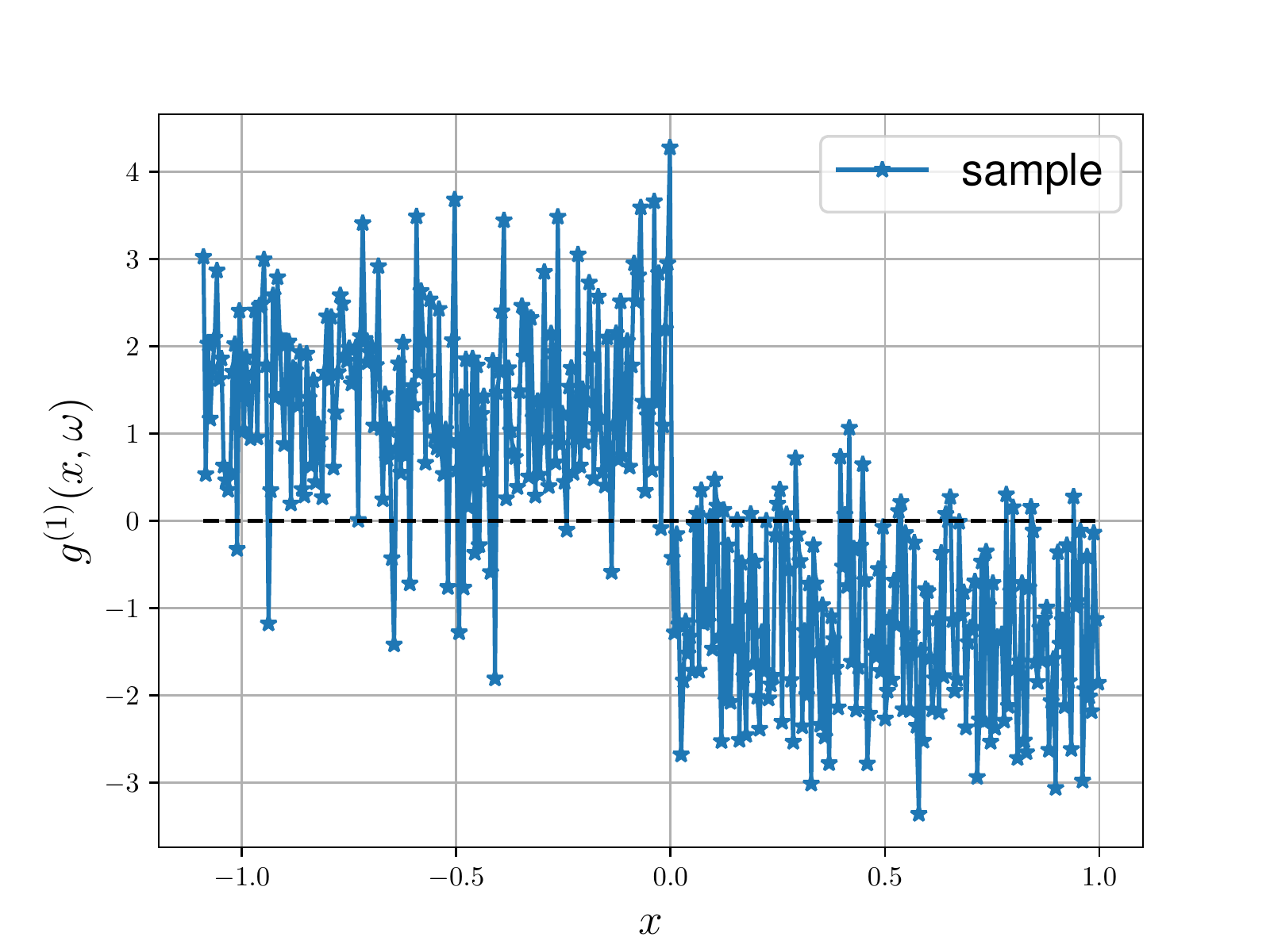}
}
\,
\subfloat[Gaussian, $\Delta = 5\cdot 10^3$]{
\includegraphics[scale=0.25]{\main/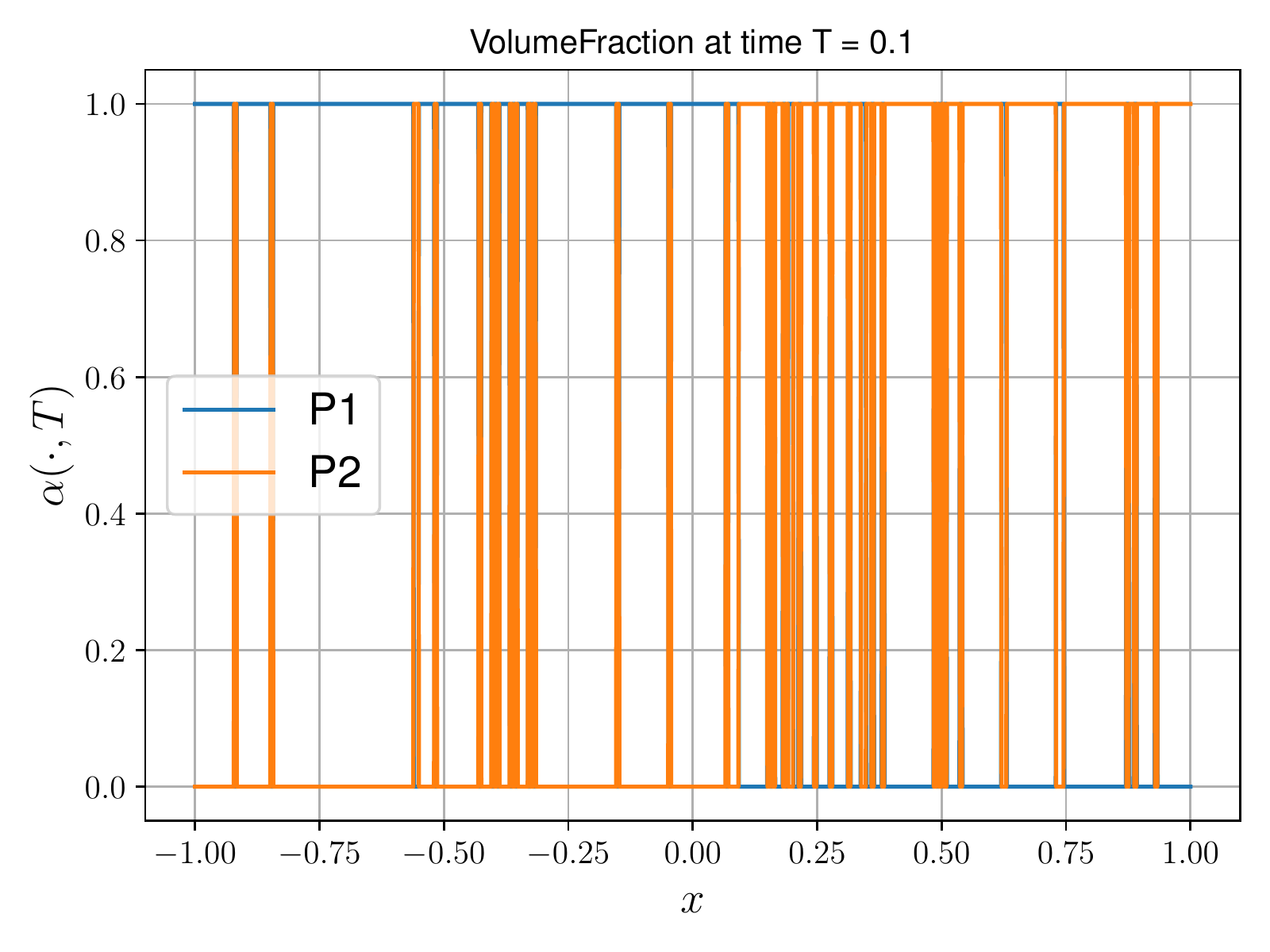}
}
\,
\subfloat[Uniform, $\Delta = 5\cdot 10^3$]{
\includegraphics[scale=0.25]{\main/gaussian/sample_cmp/uniform_200.pdf}
}\\

\subfloat[Gaussian sample, $\Delta = 1.25\cdot 10^3$]{
\includegraphics[scale=0.25]{\main/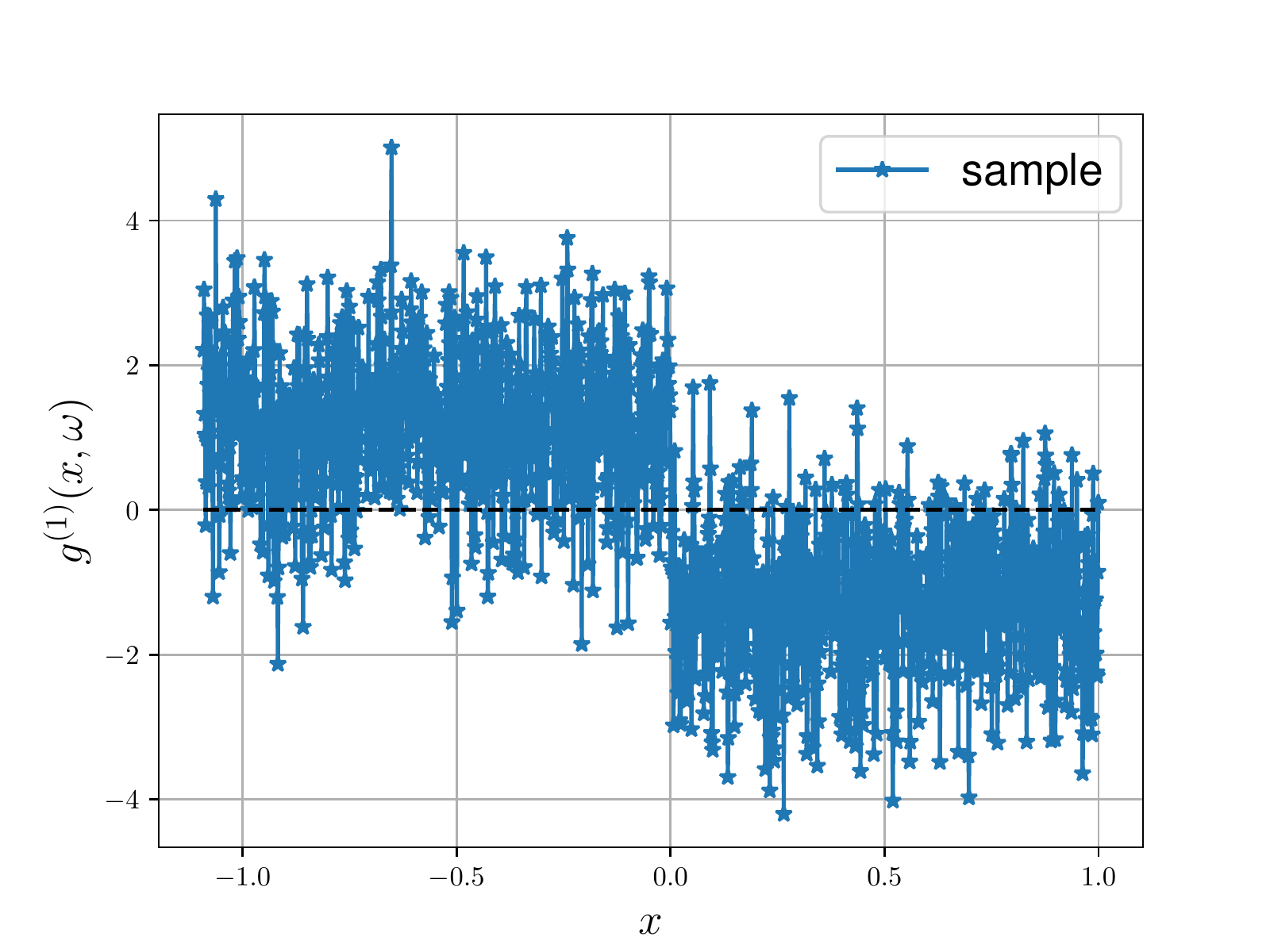}
}
\,
\subfloat[Gaussian, $\Delta = 1.25\cdot 10^3$]{
\includegraphics[scale=0.25]{\main/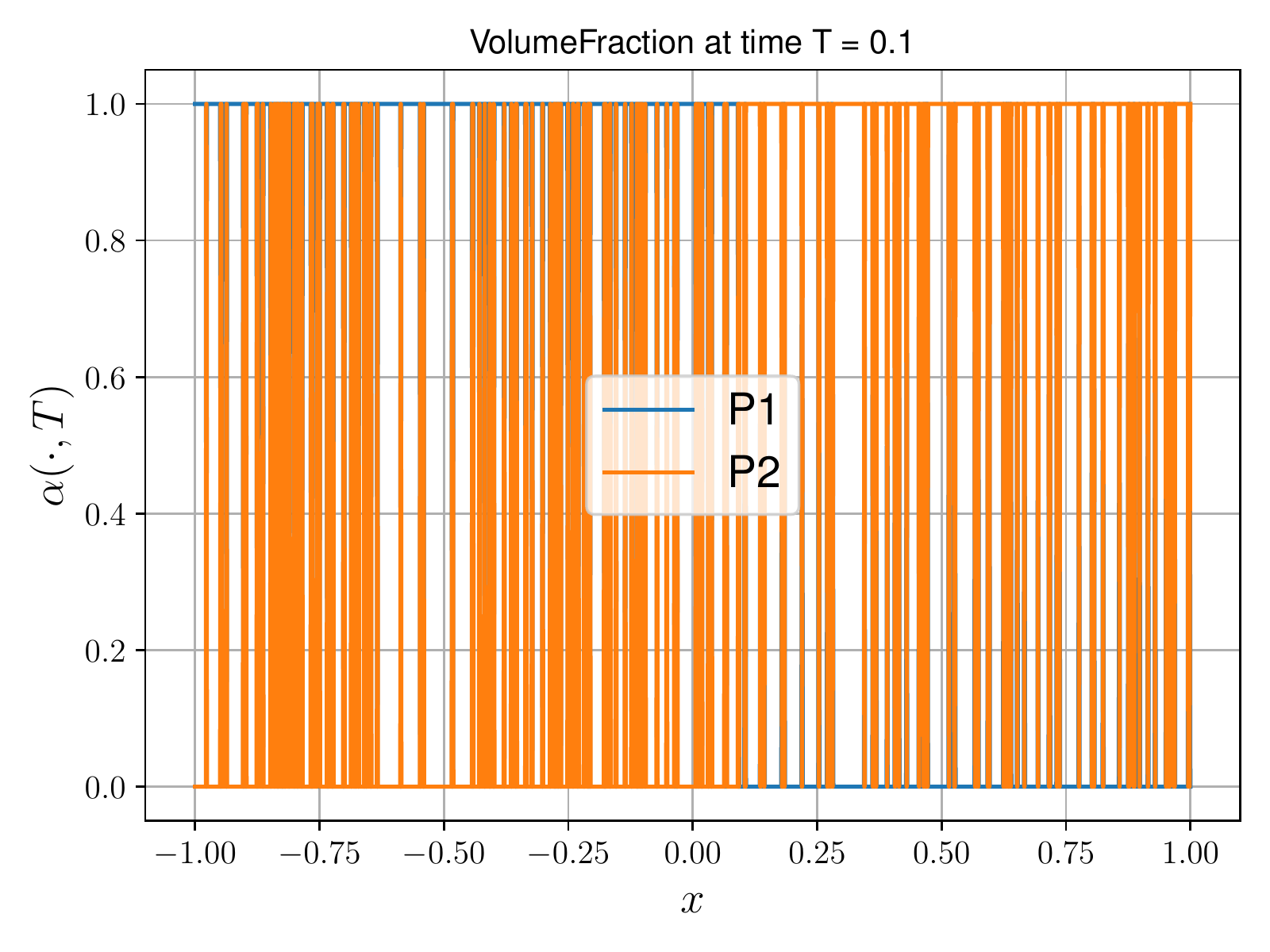}
}
\,
\subfloat[Uniform, $\Delta = 1.25\cdot 10^3$]{
\includegraphics[scale=0.25]{\main/gaussian/sample_cmp/uniform_800.pdf}
}\\

\subfloat[Gaussian sample, $\Delta = 0.3125\cdot 10^3$]{
\includegraphics[scale=0.25]{\main/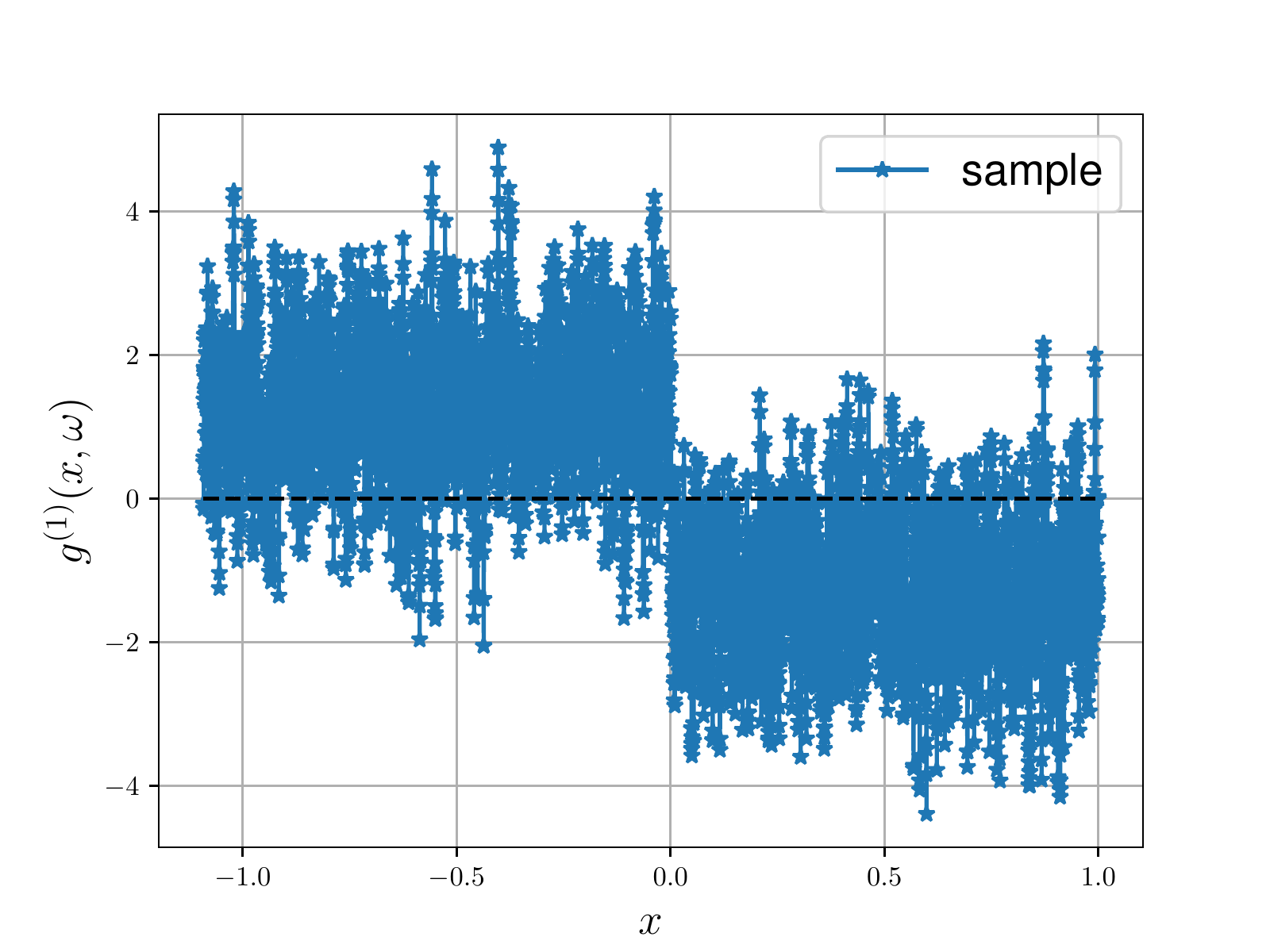}
}
\,
\subfloat[Gaussian, $\Delta = 0.3125\cdot 10^3$]{
\includegraphics[scale=0.25]{\main/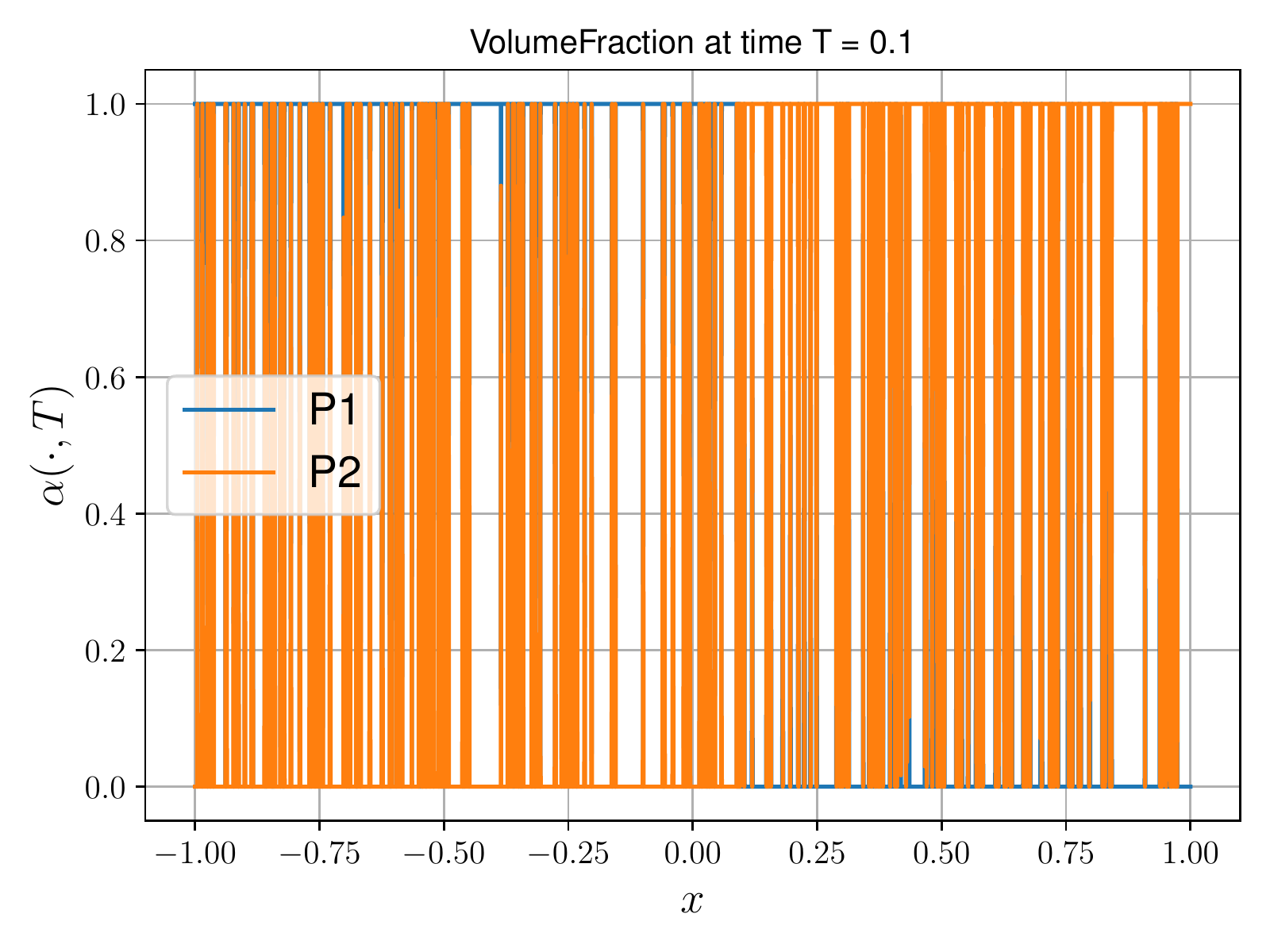}
}
\,
\subfloat[Uniform, $\Delta = 0.3125\cdot 10^3$]{
\includegraphics[scale=0.25]{\main/gaussian/sample_cmp/uniform_3200.pdf}
}

\caption{Samples comparison for Alg. \ref{al:AI:MGP_gaussian} under assumption of Mater$(\frac{3}{2},\zeta)$ with $\zeta = 0.001$. Compare obtained results with those of Fig.\ref{Fig:AI:GP_smp_cmp}.}
\label{Fig:AI:GP_smp_cmp_zeta_small}
\end{center}
\end{figure}

Notice that such an observation encourages two types of analysis: the first is to investigates whether convergence can be recovered for sufficiently small values of the hyperparameter $\zeta$, and, if this is the case, towards what it is converging.
Based on the observations performed in Fig.\ref{Fig:AI:GP_smp_cmp_zeta_small}, if there exists a small value of $\zeta$ such that any sample $j=1,\ldots,J$ produced by exploiting the Gaussian representation is very close (in some norm) initially to a sample $l = 1,\ldots,L$ generated by the uniform-based Alg.\ref{al:AI:eMGP}, then the two ensemble averages shall be very close, due to the continuity and stability of the FT algorithm.
In particular, this would result in concluding that the ab-initio method is in fact computing stable results under variation of the underlying distribution.

\begin{figure}[!htbp]
\begin{center}
\subfloat[Gaussian $L=1$]{
\includegraphics[scale=0.25]{\main/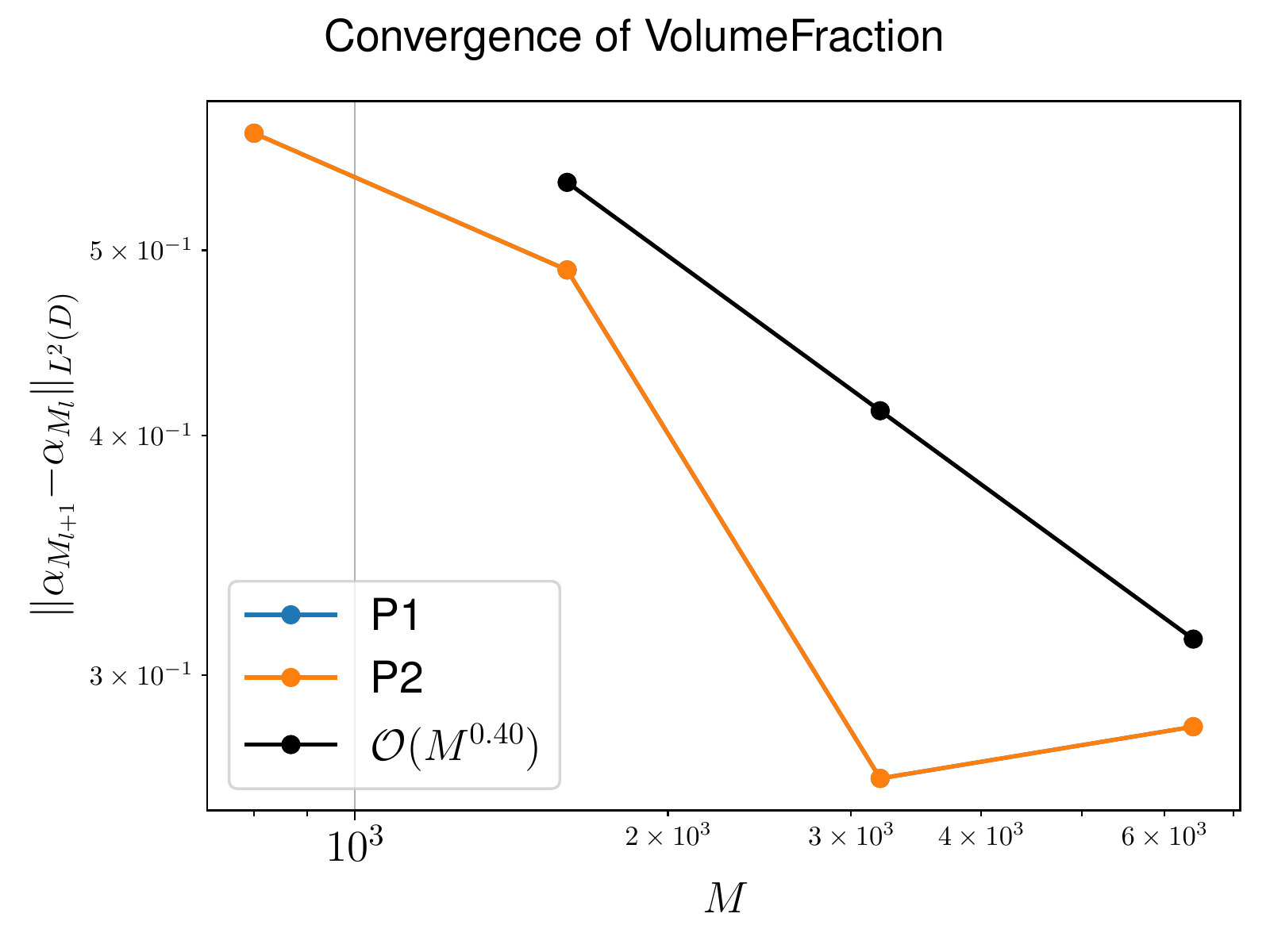}
}\,
\subfloat[Gaussian $L=100$]{
\includegraphics[scale=0.25]{\main/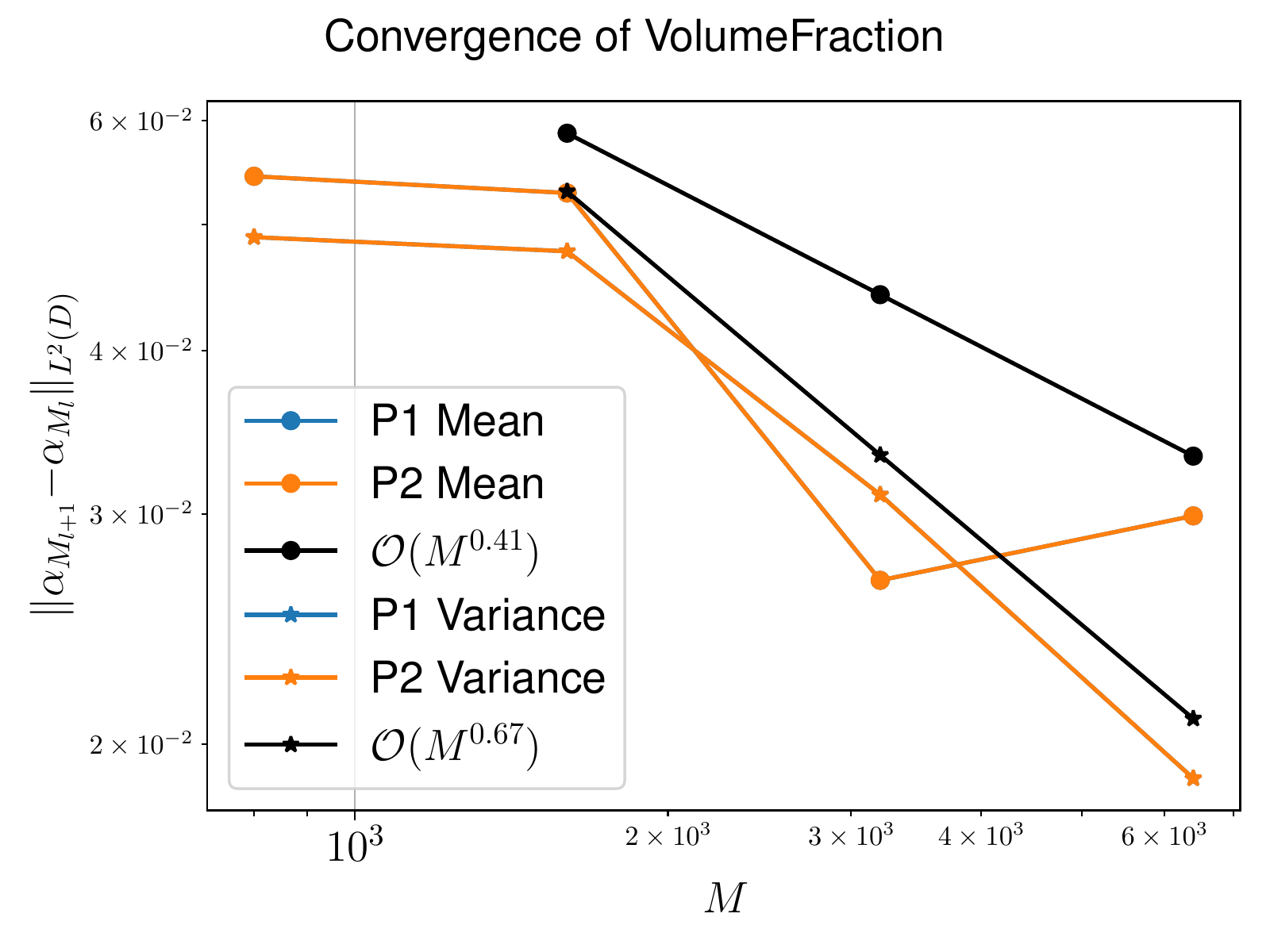}
}\,
\subfloat[Gaussian $L=1000$]{
\includegraphics[scale=0.25]{\main/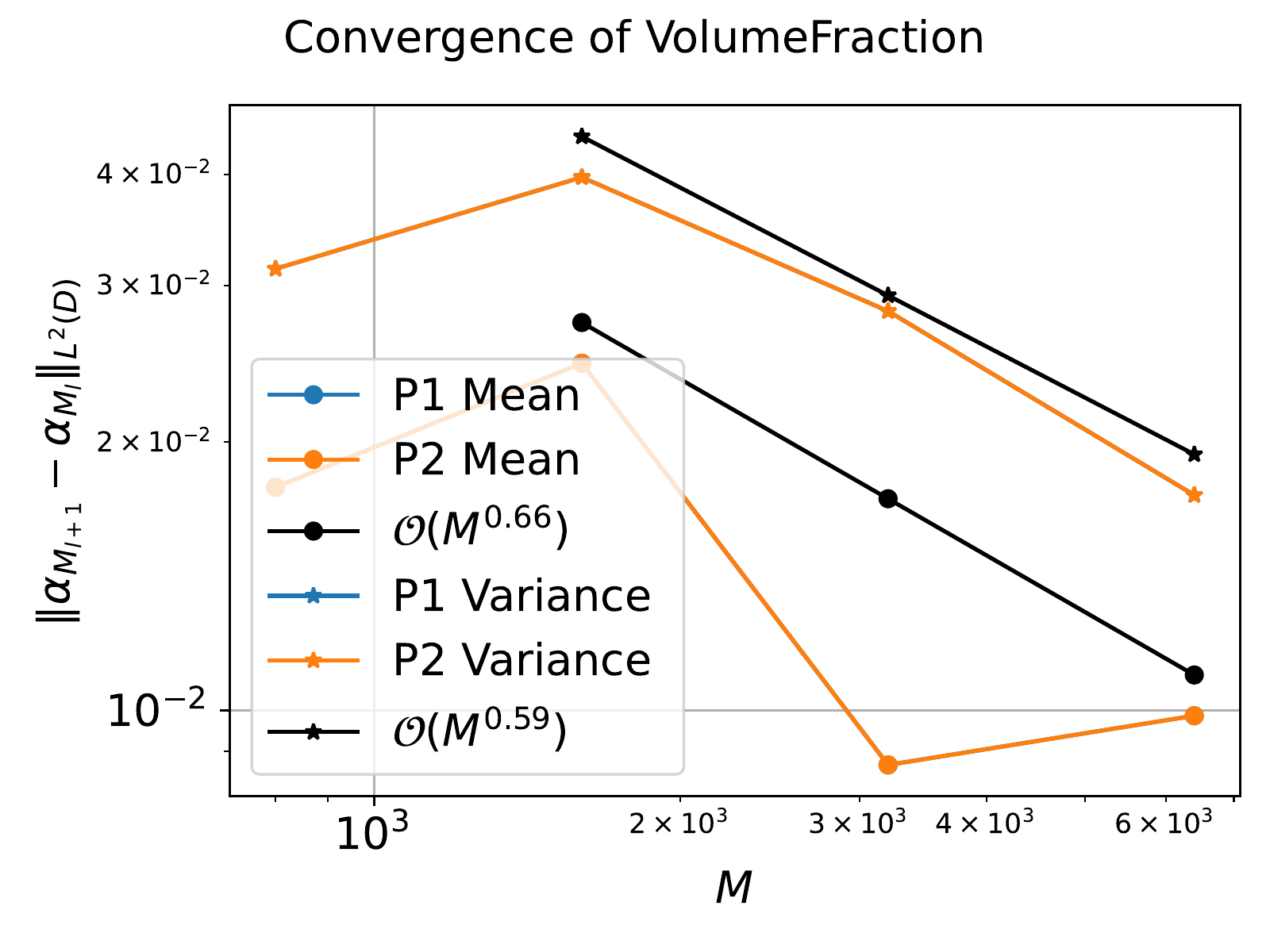}
}
\caption{Empirical convergence study for the regime-generating algorithm \ref{al:AI:MGP_gaussian} under the assumption of a kernel function of type Matern$(\frac{3}{2},\zeta)$ with $\zeta = 0.0001$.}
\label{Fig:AI:GP_conv_zeta_small}
\end{center}
\end{figure}

\begin{figure}
\begin{center}
\includegraphics[scale=0.5]{\main/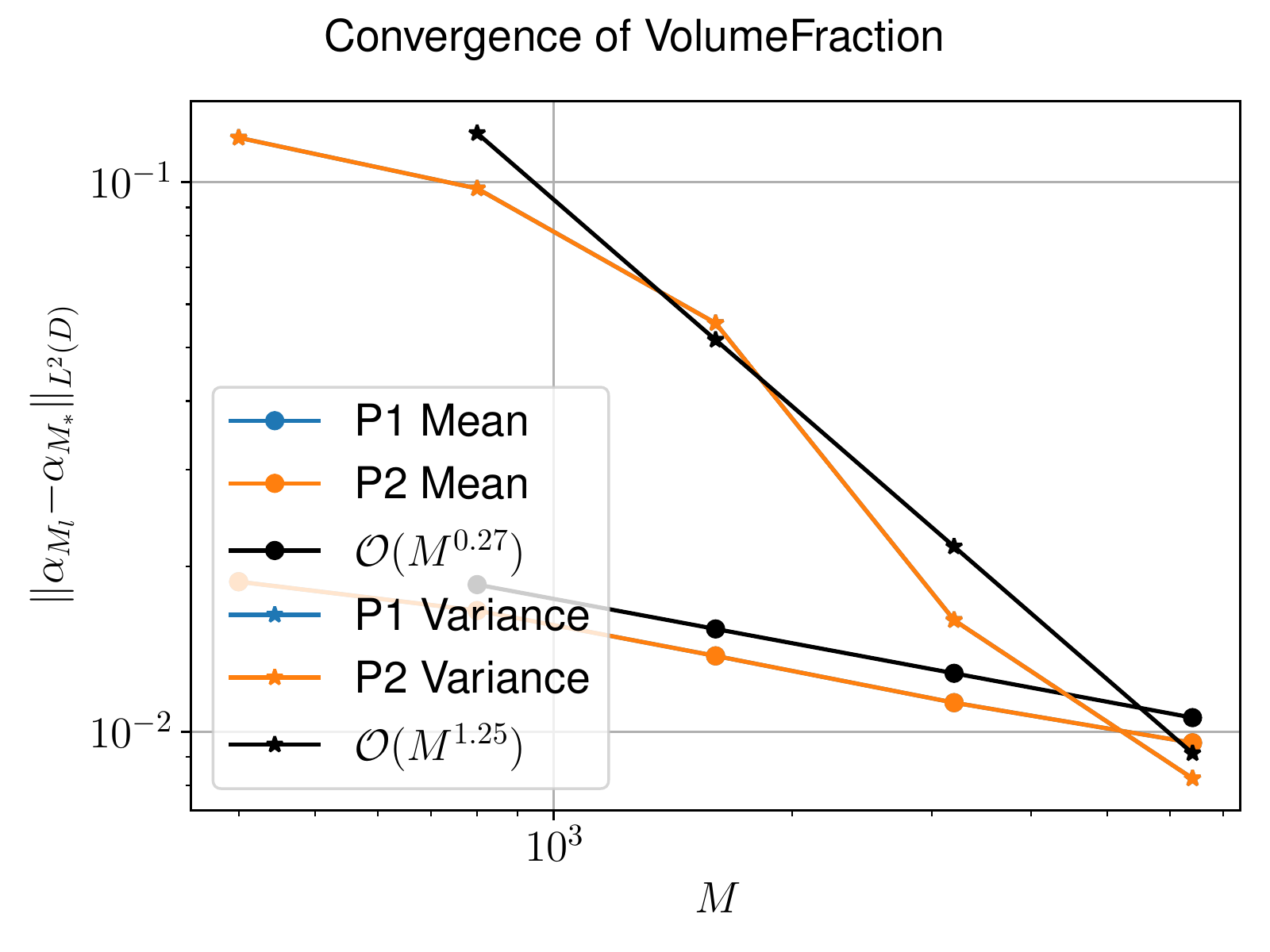}
\caption{Empirical convergence study of ensembles generated with algorithm \ref{al:AI:MGP_gaussian} (Gaussian) towards those generated using algorithm \ref{al:AI:eMGP} (Uniform).}
\label{Fig:AI:GP_towards_Uni}
\end{center}
\end{figure}

In Fig.\ref{Fig:AI:GP_conv_zeta_small} we propose a mesh convergence study for several ensemble sizes $L = 1, 100, 800$ for the volume fraction when samples are generated using a GP with Matern kernel with hyperparameters $\nu = \frac{3}{2}$ and $\zeta = 0.0001$. Finally, we plot in Fig.\ref{Fig:AI:GP_towards_Uni} convergence of computed results using a Gaussian-like distribution towards those generated by using the uniform-distribution-based algorithm.

Depicted results again suggest convergence sample-wise and for ensemble averages. 
Additionally, these latter seems to provide the same approximation that would construct by taking a uniform-based algorithm.

Notice that the above discussion is entirely performed on the test case of Section \ref{sec:AI:MC:NE:Un}, since it is well-suited to investigate the impact of the sub-discretization.
Nevertheless, due to continuity of the FT algorithm, the same conclusions could be extended to different initial conditions since the strategy employed for producing samples is independent of the physical characteristics under consideration.
For the sake of completeness we present in Fig.\ref{Fig:AI:GP:Sod_cmp} a prototypical comparison between the two strategies when applied to a non-trivial test case like the one of Sec. \ref{sec:AI:MLML:NE:Sod}, showing perfect agreement between predictions.

\begin{figure}[!htbp]
\begin{center}
\includegraphics[scale=\figsize]{\main/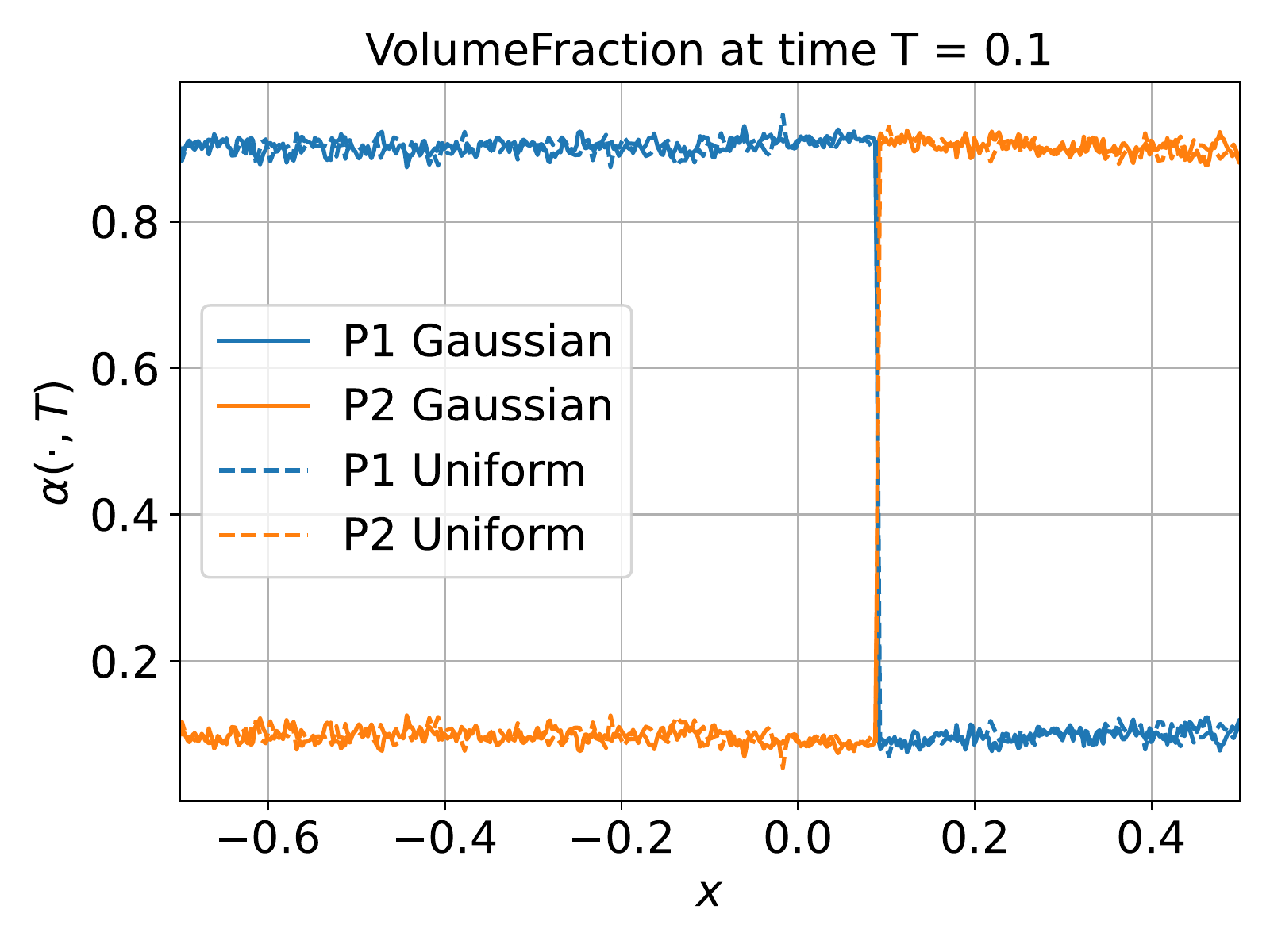}\,
\includegraphics[scale=\figsize]{\main/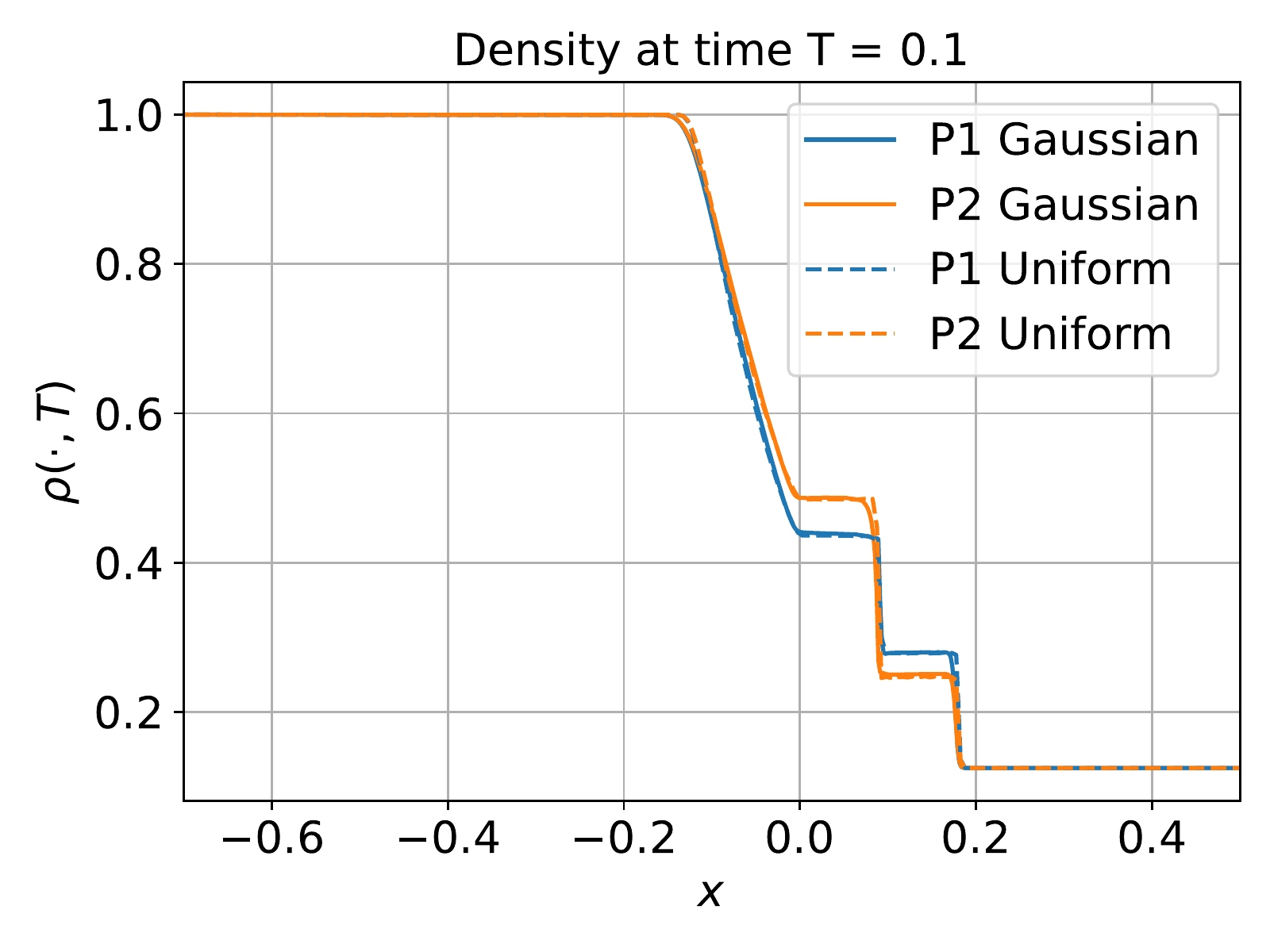}\\
\includegraphics[scale=\figsize]{\main/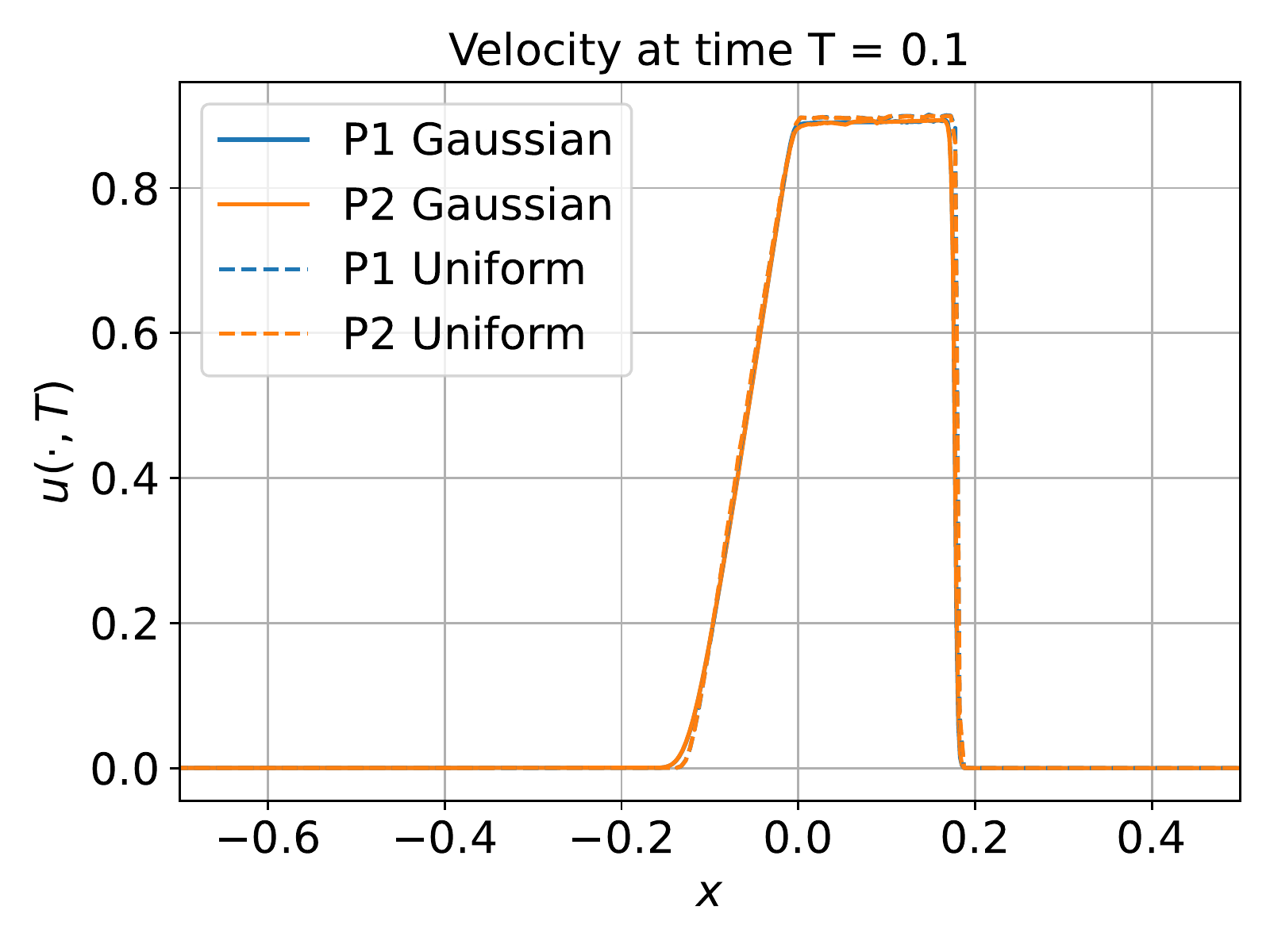}\,
\includegraphics[scale=\figsize]{\main/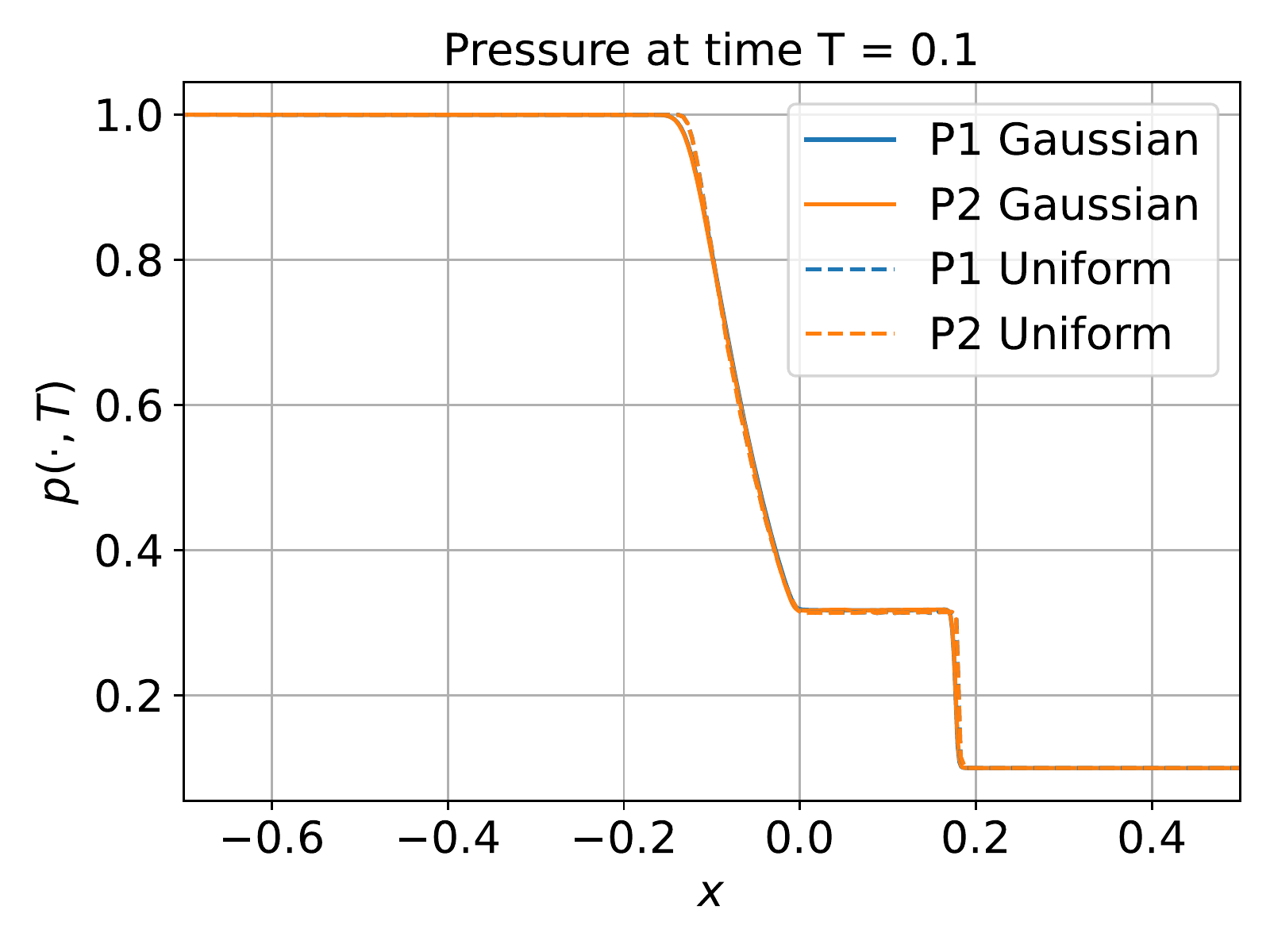}
\caption{Comparison for the test case of Sec.\ref{sec:AI:MLML:NE:Sod} using a uniform distribution and a Gaussian representation.
Results for the Gaussian distribution have been produced using $L=1024$ samples under the assumption of a Matern$(\frac{3}{2},0.005)$ kernel.}
\label{Fig:AI:GP:Sod_cmp}
\end{center}
\end{figure}

\section{Discussion}
\label{sec:disc}
Inspite of their great importance in applications, it is fair to say that a suitable mathematical framework for describing multiphase flows and an efficient numerical methodology for simulating them is still lacking. The inherent uncertainities in the dynamics of multiphase flows necessitate a statistical description. However, the task of deriving equations for the time-evolution of statistical quantities of interest runs into the formidable obstacle of the so-called \emph{closure} problem, leading to macroscopic models that are incomplete, underdetermined and/or inaccurate. These issues arise as information about the underlying microstructure is lacking at the macroscopic scale.

A promising alternative lies in the so-called \emph{discrete equation method} (DEM) \cite{Abgrall&Saurel}, where suitable numerical methods are used to recover microstructure information at the macroscopic level and a continuous description of multiphase flow results from taking a zero-resolution limit. Unfortunately and as exposed in a recent paper \cite{Petrella2022}, this approach also leads to underdetermined models, resulting in an infinite family of possible solutions. In particular, solutions depend on the choice of a parameter $r \in [0,1]$ that models the effect of underlying probability coefficients in a statistical description.  

Given these issues with existing models, we presented an alternative approach in this paper. 
Our starting point where the Eqns. \eqref{eq:Etwo-phase}, which describe the time-evolution of statistical averages for a one-dimensional two-phase flow. 
Instead of trying to model these averages in a self-consistent manner, we explore the alternative avenue of directly \emph{sampling} them using a \emph{Monte-Carlo} approach. 
The resulting algorithm \ref{al:AI:MC} was based on the following key ingredients,
\begin{itemize}
    \item The underlying probability space is sampled and the expectation over it is replaced by empirical (Monte-Carlo) averages \eqref{eq:AI:MC-estimators}. 
    \item A numerical evolution operator is then required to propagate the samples over time. We chose the sharp-interface \emph{front tracking} (see Algorithm \ref{al:AI:FT}) to evolve the samples over time, introduced in \cite{Petrella22FT}. 
    The FT algorithm has the key advantage to keeping the interface sharp and not adding numerical viscosity. Thus, it ameliorates the vexing issues of artificial mixing zones that plague several popular diffuse-interface schemes for multiphase flows. 
    \item A novel \emph{microscale generation} procedure is proposed to generate samples, at initial time, from the (given) macroscopic values of the initial averages.  
\end{itemize}
These ingredients are combined together in a novel \emph{ab-initio} algorithm for directly computing statistical averages of quantities of interest in a two-phase flow. 
To the best of our knowledge, this is the first time that such an algorithm has been proposed in the extensive literature on multiphase flows. 

We test the \emph{ab-initio} algorithm \ref{al:AI:MC} on a suite of test cases to draw the following conclusions,
\begin{itemize}
    \item The ab-initio algorithm is found to be robust at simulating two-phase flow in different underlying flow regimes and provide both the statistical mean of the quantities of interest as well as their variance (and higher moments) which allows one to infer possible uncertainties, implicit in the flow description. 
    \item The algorithm is empirically shown to converge as the number of samples, as well as the number of bubbles within each volume, are increased. Such convergence has been observed to fail when the macroscopic resolution is finer than the microscopic one, due to an inconsistent inversion of the inherent hierarchy of scales take under consideration. 
    In particular, sample-wise convergence under sub-scale refinement is happening only once the  micro-scale gets submerged in the macroscopic one.
    \item We compare the results of the ab-initio algorithm to those generated by a macroscopic DEM scheme of \cite{Petrella2022} to find that as long as the macroscopic DEM scheme leads to a unique solution (for instance, the extreme values of $r=0,1$ lead to the same flow), then it \emph{coincides} with the mean of results generated using the ab-initio method. 
    This indicates conditional uniqueness of the macroscopic models: if variations of the hyperparameter $r\in[0,1]$ do not associate to different results, then the DEM is in fact computing the mean of the ab-initio method. 
   In turn, one concludes that the ab-initio algorithm is a generalization of the macroscopic approach. 
   \item However, in the more generic situations where there is no uniqueness in the macroscopic results (for instance, choosing $r=0$ and $r=1$ in the DEM scheme of \cite{Petrella2022} leads to different solutions), we find significant discrepancies (in the range of $10\%$) in both the values of the intermediate states as well as in locations of shocks, between the ab-initio and macroscopic results. 
   In such cases, it is clear that the macroscopic schemes are unable to recover the ground truth ab-initio flow by just varying the parameter $r$.
   \item Carefully designed test cases have shown how much of the discrepancy resides in the speed at which the two phases are driven towards equilibrium. Such a conclusion represent a concerning bottleneck for the accurate simulation of two-phase flow phenomena as it makes imperative to track over space and time the relative number of interfaces per resolution-volume. Such information is, in many practical situations, difficult to establish, if at all possible, and the ab-initio method presents an opportunity to produce synthetic approximations of it.
\end{itemize}
Summarizing, the proposed ab-initio algorithm provides a viable as well as robust framework for simulating two-phase flows in one space dimension. Given that this algorithm does not require any \emph{closure} assumptions and the microstructure information is implicitly included, such ab-initio simultions can serve as the ground truth for designing macroscopic models for multiphase flows. 

From a practical point of view, the methodology is computationally intensive, with the larger weight being associated to the FT algorithm. Such computational cost has been observed to increase as the sub-scale is refined, due to the larger amount of details needed to be resolved.
In addition, the extension of such an approach to multi-dimensions seems to be hard. These are part of the reasons why it is authors' belief that future work should focus on improving the numerical evolution operator efficiency. Similarly, enhancing convergence with respect to the number of samples could lead to improvements of the overall methodology for practical use.

\begin{figure}
\begin{center}
\begin{tikzpicture}[node distance=2cm]
		\node (p1) {
		\begin{adjustbox}{max totalsize={.2\textwidth}{.5\textheight},center}
		\begin{tikzpicture}
		\pgfmathsetmacro{\Dx}{7};
		\pgfmathsetmacro{\hei}{4};
		\pgfmathsetmacro{\top}{1};
		\pgfmathsetmacro{\sp}{1};		
		
		\pgfmathsetmacro{\xi}{0};
		\pgfmathsetmacro{\xL}{\xi -\Dx};
		\pgfmathsetmacro{\xLL}{\xL - \sp};
		\pgfmathsetmacro{\xR}{\xi + \Dx};
		\pgfmathsetmacro{\xRR}{\xR + \sp};
		
		\tkzDefPoint(\xLL, 0){xim};
		\tkzLabelPoint[below, xshift = 0.5cm](xim){$x_{i-\frac{1}{2}}$};
		\tkzDefPoint(\xRR, 0){xip};
		\tkzLabelPoint[below, xshift = -0.5cm](xip){$x_{i+\frac{1}{2}}$};
		\tkzDefPoint(\xi, 0){xi};
		\tkzLabelPoint[below](xi){$x_{i}$};
		
		\draw [thick] (xim) -- (xip);
		\draw [thick] (\xL, -\top) -- (\xL, \hei + \top);
		\draw [thick] (\xR, -\top) -- (\xR, \hei + \top);		
				
		\pgfmathsetmacro{\N}{12};
		\pgfmathsetmacro{\D}{2*\Dx/\N};
		
		\foreach \Dn/\ind [evaluate=\Dn as \ss using \Dn + \s,
		remember=\ss as \s (initially 0)] in {2.3/1,1.2/2,1.2/1,2.3/2,1.2/1,3.4/2,1.2/1,1.2/2}
		{	
			\ifthenelse{\ind=1}{
				\draw [fill=blue!20] (\xL + \s,0) rectangle (\xL+\ss,\hei);	
				\node [color=blue!50!black] at (\xL + \s + 0.5*\Dn,0.5*\hei){$U^{(1)}_i$};
			}{
				\draw [fill=green!20] (\xL + \s,0) rectangle (\xL+\ss,\hei);
				\node [color=green!50!black] at (\xL + \s + 0.5*\Dn,0.5*\hei){$U^{(2)}_i$};
			};			
		}
		\end{tikzpicture}
\end{adjustbox}		
		};
		\node[right of =p1, yshift = -2cm, xshift=2cm] (p2) {$\mathbb{E}\left[\textbf{U}_0^{(k)}\right]$}
			edge[pre] node[auto, above] {$\mathbb{E}$} (p1);
			
		\node[below of =p2, yshift = -1cm, xshift = -1cm] (p4) {$S_T\Bigg[\mathbb{E}\left[\textbf{U}_0^{(k)}\right]\Bigg]$}
		edge[pre] node[auto, right] {$S_T^\Delta, \, \Delta\rightarrow 0$} (p2);
	
		\path[->] (p1) edge[bend left, color=red] node[auto, color=red, left] {DEM} (p4); 
		
		\node[left of = p1, yshift = -2cm, xshift=-2cm] (p3) {$S_T\left[ \textbf{U}_0^{(k)} \right]$}
			edge[pre] node[auto] {$S_T^\Delta, \, \Delta\rightarrow 0$} (p1);
			
		\node[below of =p3, yshift = -1cm, xshift = 1cm] (p5) {$\mathbb{E}\Bigg[S_T\left[\textbf{U}_0^{(k)}\right]\Bigg]$}
		edge[pre] node[auto, left] {$\mathbb{E}_M, \, M\rightarrow \infty$} (p3);
	
		\path[->] (p1) edge[bend right, color=red] node[auto, color=red] {Ab-initio} (p5);

		\path[<->] (p4) edge[dashed, color=red] node[auto, color=red] {$?$} (p5);	
\end{tikzpicture}
\end{center}
\caption{Schematic representation of the road-map employed by the DEM and the Ab-initio schemes.}\label{fig:schemes_discrepancy}
\end{figure}
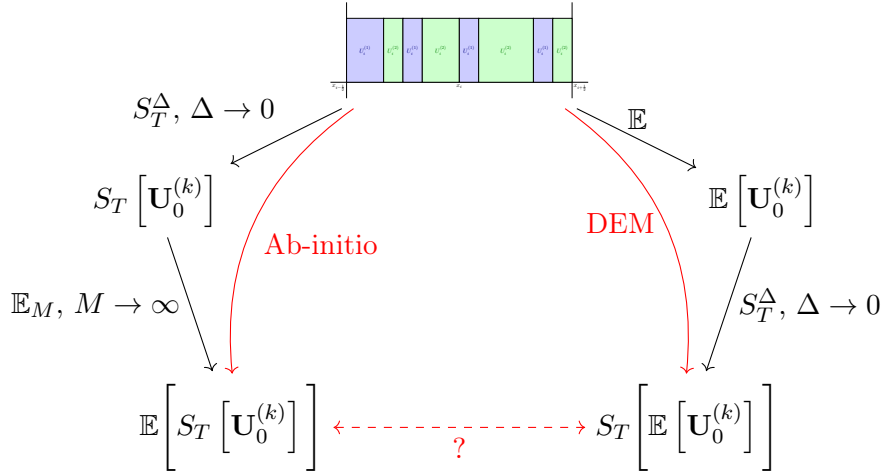
Finally, we would like to make comments from a higher point of view. 
The hereby collected evidences point into the direction of refining the approach for the relaxation step of macroscopic schemes: due to a finite number of interfaces involved in the ab-initio simulations, physical quantities of interest relax towards equilibrium at a significantly slower speed as compared to the one provided by macroscopic algorithms. 
This to say that letting the $r$-parameter vary is not enough to correctly capture quantities of interest.\\
A conspicuous amount of theories dealing with ensemble modeling of two-phase flow rely on the assumption that the evolution operator and the mean one commute, under reasonable assumptions, in analogy to procedures employed for turbulent flow simulations.
A prototypical example of such view point is the DEM: one carefully designs an evolution scheme for the mean operator and aim at capturing averages in the limit of vanishing resolution. 
At the practical level, such strategy \emph{virtually} construct the mean across all possible regimes, inherently assuming commutation of the operators involved.\\
Conversely, the ab-initio approach approximates the moments of evolved data, by \emph{concretely} constructing (approximations of) the microscale regime.\\
In this work we showed that for some test cases, the two road-maps (see Fig. \ref{fig:schemes_discrepancy}) lead to the same predictions, thus underlying two points
\begin{itemize}
\item one has constructed numerical evidences that, under specific physical conditions, the relaxation procedure of the DEM acts correctly. 
This is non-trivial to establish, as the relative number of interfaces is typically hard to compute, while being responsible for the speed at which the two-phases relax towards equilibrium;
\item from the mathematical point of view, one found test cases where commutation between such operators is indeed happening.
\item For such test cases where commutation is happening, one understands DEM predictions as an approximation of the first moment of solutions generated by the ab-initio.
\end{itemize}
However, in this work we showed that such commutation is not achieved in general, and solutions converge to different models, posing the question of what are we computing with the DEM.\\
The gap between the two methodologies seems to be possibly eliminated by tracking the correct number of interface per volume of resolution, which defines a time-parametrized distribution over the domain of interest.
Such an insight, suggests that the DEM is in fact lacking information about the time variation of the underlying law associated to the regime.
In turn, one should target a measure over a functions rather than just statistical moments of interest. In this light, it is authors belief that the ab-initio simulation seems to provide a much more fundamental information as compared to $r$-constant-valued mesoscopic models, since makes almost no assumption about the driving forces that shape solutions.\\
This change in perspective present a many similarities to the recent outbreak of new solutions paradigma for systems of hyperbolic conservation laws \cite{Fjordholm2017,fjordholm17}.
Interestingly, such solutions concepts have been observed to be stable under variations of the underlying distribution \cite{LyePhD}, as it was observed hereby.

\section{Conclusions}
\label{sec:concl}

The inherent uncertainties in the dynamics of multiphase flow necessitate a statistical description, which, when approached in the continuous-theory framework, has been proven to lead to unclosed models.
Hereby we develop the \emph{ab-initio} method for the direct simulation of averaged equations as a blend of the Front-Tracking (FT) method and the Monte-Carlo (MC) sampling.
To the best of our knowledge, this constitutes the first direct simulation of two-phase flow in compressible character for two fluids.

The methodology is composed of three essential ingredients: a regime generating procedure, a numerical evolution operator and a statistical operator.
The ab-initio algorithm is found to be robust at simulating two-phase flows in different underlying flow regimes and provides both the statistical mean of the quantities of interest as well as their variance (and higher moments), allowing one to infer possible uncertainties implicit in the flow description. The algorithm is shown to converge as the number of samples and the number of bubbles within each volume are increased, but convergence fails when the macroscopic resolution is finer than the microscopic one, due to an inconsistent inversion of the inherent hierarchy of scales taken under consideration.

The article compares the results of the ab-initio algorithm to those generated by a macroscopic DEM scheme and finds that as long as the latter leads to a unique solution, then it coincides with the mean of results generated using the ab-initio method. 
In turn, one concludes that the ab-initio algorithm constitutes a generalization of the macroscopic approach. 
However, in more generic situations where there is no uniqueness in the macroscopic results, significant discrepancies are found between the ab-initio and macroscopic results. 
Carefully designed test cases show that much of the discrepancy resides in the speed at which the two phases are driven towards equilibrium, representing a concerning bottleneck for the accurate simulation of two-phase flows.

In fact, the relaxation procedure employed by the DEM assumes a fixed (infinite) rate for the relaxation parameters, without any apparent justification.
The convergence of the volume fraction obtained on all test cases considered in fact implies a convergence of the average number of interfaces per volume, thus resulting in a finite rate relaxation.

In addition, the ab-initio method allows for the synthetic approximation of such hyperparameters, which could be used to enhance macroscopic models predictions.
Such a task seems to be viable in the context of data driven algorithms, with the primary example of Machine-Learning framework \cite{BishopPR}.

Interestingly, we find that the ab-initio method produces stable and converging results, which are, in addition, stable with respect to variations of the underlying distribution.

Conceptually, this work shows how it is imperative to track over space and time the distribution of the relative number of interfaces in order to properly simulate two-phase flow phenomena, where microscale effects play an essential role in shaping solutions and granting convergence.

These as well many other fundamental questions can be further investigated by means of the newly developed ab-initio framework, which, in its current formulation, already defines a flexible generalization of any known two-phase flow continuous model.

\appendix

\section{About the scheme of \cite{Petrella2022}}
\label{app:r-model}

In \cite{Petrella2022} it was provided a one-parameter family of schemes for multiphase flows, as a generalization of the classical Godunov scheme.\\
This scheme take the form

\begin{equation}
\label{semi-discrete_form}
\frac{d \left(\alpha^{(k)}_i\tilde{\textbf{U}}^{(k)}_i\right)}{dt} +
\frac{\mathbb{E}_{i+\frac{1}{2}}\left[X^{(k)}\tilde{\textbf{F}}\right]
-\mathbb{E}_{i-\frac{1}{2}}\left[X^{(k)}\tilde{\textbf{F}}\right]}{\Delta x} 
= 
\frac{\mathbb{E}_{boundary}\left[\tilde{\textbf{F}}^{lag}\right]_i}{\Delta x} 
+ 
\mathbb{E}_{relax}\left[\tilde{\textbf{F}}^{lag}\right]_i
\end{equation}
where
\begin{itemize}
\item the time-dependent vector of (extended) conserved variables $\tilde{\textbf{U}}^{(k)}_i$ and associated flux $\tilde{\textbf{F}}^{(k)}$ for phase $k\in\lbrace 1,2\rbrace$ at space location $x=x_{i}$ read
\[
\tilde{\textbf{U}}^{(k)}_i = 
\begin{bmatrix}
1\\
\rho^{(k)}\\
\rho^{(k)}u^{(k)}\\
\rho^{(k)}E^{(k)}
\end{bmatrix}_i
\qquad
\tilde{\textbf{F}}^{(k)}_i =
\begin{bmatrix}
0\\
\rho^{(k)}u^{(k)}\\
\rho^{(k)}{u^{(k)}}^2 + p^{(k)}\\
u^{(k)}\left(\rho^{(k)}E^{(k)} + p^{(k)}\right)\\
\end{bmatrix}_i
\quad
E^{(k)} = \frac{1}{2}{u^{(k)}}^2 + e^{(k)}(\rho^{(k)},p^{(k)})
\] 
\item the $i+\frac{1}{2}$-th flux-indicator $\beta_{i+\frac{1}{2}}^{(p,q)}$ reads
\begin{equation*}
\beta_{i+\frac{1}{2}}^{(p,q)} := \mathrm{sign}\left(\sigma_{i+\frac{1}{2}}\left(\tilde{\textbf{U}}^{(p)}_{i},\tilde{\textbf{U}}^{(q)}_{i+1}\right)\right)
= 
\begin{cases}
1 & \sigma_{i+\frac{1}{2}}\left(\tilde{\textbf{U}}^{(p)}_{i},\tilde{\textbf{U}}^{(q)}_{i+1}\right) > 0\\
-1 & \sigma_{i+\frac{1}{2}}\left(\tilde{\textbf{U}}^{(p)}_{i},\tilde{\textbf{U}}^{(q)}_{i+1}\right) < 0
\end{cases}
\end{equation*}
where $\sigma_{i+\frac{1}{2}}\left(\tilde{\textbf{U}}^{(p)}_{i},\tilde{\textbf{U}}^{(q)}_{i+1}\right)$ denotes the interface velocity between $\tilde{\textbf{U}}^{(p)}_{i}$ and $\tilde{\textbf{U}}^{(q)}_{i+1}$.
\item the $\pm$-notation denotes
\[
a^+ := \max(a,0),
\qquad
a^- := \min(a,0)
\]
\item the $i+\frac{1}{2}$-th $r$-dependent probability coefficients $\mathbb{P}_{i+\frac{1}{2}}\left[\Sigma_k,\Sigma_l\right]$ of finding phase $k$ on the left and phase $l$ on the right of the cell-interface at $x=x_{i+\frac{1}{2}}$ reads
\begin{subequations}
\label{eq:Convex}
\begin{align}
\mathbb{P}_{i+\frac{1}{2}}\left[\Sigma_p, \Sigma_p\right] &= r \max\left(\alpha_{i}^p-\alpha_{i+1}^q,0\right) + (1-r)\min\left(\alpha_{i}^p,\alpha_{i+1}^p\right)\label{Convex-pp}\\
\mathbb{P}_{i+\frac{1}{2}}\left[\Sigma_p, \Sigma_q\right] &= r\min\left(\alpha_{i}^p,\alpha_{i+1}^q\right) + (1-r)\max\left(\alpha_{i}^p-\alpha_{i+1}^p,0\right)\label{Convex-pq}
\end{align}
\end{subequations}
for some hyperparameter $r = r_{i+\frac{1}{2}}(t)\in [0,1]$.
\item the $i+\frac{1}{2}$ contribution of the average flux reads
\begin{align*}
\mathbb{E}_{i+\frac{1}{2}}\left[X^{(k)}\tilde{\textbf{F}}\right]& 
:= \mathbb{P}_{i+\frac{1}{2}}\left[\Sigma_k,\Sigma_k\right]
\tilde{\textbf{F}}\left(\tilde{\textbf{U}}^{(k)}_{i}, \tilde{\textbf{U}}^{(k)}_{i+1}\right) 
+ 
\left(\beta_{i+\frac{1}{2}}^{(k,l)}\right)^+ \mathbb{P}_{i+\frac{1}{2}}\left[\Sigma_k,\Sigma_l\right] \tilde{\textbf{F}}\left(\tilde{\textbf{U}}^{(k)}_{i},\tilde{\textbf{U}}^{(l)}_{i+1}\right)\\
&\qquad  + \left(-\beta_{i+\frac{1}{2}}^{(l,k)}\right)^+ \mathbb{P}_{i+\frac{1}{2}}\left[\Sigma_l,\Sigma_k\right]\tilde{\textbf{F}}\left(\tilde{\textbf{U}}^{(l)}_{i},\tilde{\textbf{U}}^{(k)}_{i+1}\right)
\end{align*}
\item the $i$-th contribution of the Lagrangian fluxes coming from the boundary of each volume reads
\begin{align*}
\mathbb{E}_{boundary}\left[\tilde{\textbf{F}}^{lag}\right]_i& 
:= 
\left(\beta_{i-\frac{1}{2}}^{(l,k)}\right)^+ \mathbb{P}_{i-\frac{1}{2}}\left[\Sigma_l, \Sigma_k\right] \tilde{\textbf{F}}^{lag}\left(\tilde{\textbf{U}}^{(l)}_{i-1},\tilde{\textbf{U}}^{(k)}_{i}\right) \\
-\left(\beta_{i-\frac{1}{2}}^{(k,l)}\right)^+ & \mathbb{P}_{i-\frac{1}{2}}\left[\Sigma_k, \Sigma_l\right] \tilde{\textbf{F}}^{lag}\left(\tilde{\textbf{U}}^{(k)}_{i-1},\tilde{\textbf{U}}^{(l)}_{i}\right)
+ \left(-\beta_{i+\frac{1}{2}}^{(l,k)}\right)^+ \mathbb{P}_{i+\frac{1}{2}}\left[\Sigma_l, \Sigma_k\right] \tilde{\textbf{F}}^{lag}\left(\tilde{\textbf{U}}^{(l)}_{i},\tilde{\textbf{U}}^{(k)}_{i+1}\right)\\
& -\left(-\beta_{i+\frac{1}{2}}^{(k,l)}\right)^+ \mathbb{P}_{i+\frac{1}{2}}\left[\Sigma_k, \Sigma_l\right] \tilde{\textbf{F}}^{lag}\left(\tilde{\textbf{U}}^{(k)}_{i},\tilde{\textbf{U}}^{(l)}_{i+1}\right)
\end{align*}
\item the $i$-th contribution of the internal Lagrangian fluxes reads
\begin{align*}
\mathbb{E}_{relax}\left[\tilde{\textbf{F}}^{lag}\right]_i& := \mathbb{E}\left[\frac{N_{int}(\omega)}{\Delta x}\right]\left( \tilde{\textbf{F}}^{lag}(\tilde{\textbf{U}}^{(l)}_i,\tilde{\textbf{U}}^{(k)}_i)-\tilde{\textbf{F}}^{lag}(\tilde{\textbf{U}}^{(k)}_i,\tilde{\textbf{U}}^{(l)}_i)\right)\\
\end{align*}
\end{itemize}

The probability coefficients (\ref{eq:Convex}) were found by assuming the following relations to hold true
\begin{subequations}
\label{ConsistencyConds_ansatz}
\begin{align}
\mathbb{P}_{i+\frac{1}{2}}\left[\Sigma_p, \Sigma_p\right] + \mathbb{P}_{i+\frac{1}{2}}\left[\Sigma_p, \Sigma_q\right] &= \alpha_i^{(p)}(t)\\
\mathbb{P}_{i+\frac{1}{2}}\left[\Sigma_p, \Sigma_p\right] + \mathbb{P}_{i+\frac{1}{2}}\left[\Sigma_q, \Sigma_p\right] &= \alpha_{i+1}^{(p)}(t)
\end{align}
\end{subequations}
for any cell index $i$, and any phase indexes $p,q\in\lbrace 1,2\rbrace$.\\
Such assumption essentially tells that the probability of finding phase $p$ at the left of the interface $x=x_{i+\frac{1}{2}}$ coincides with the volume fraction $\alpha^{(p)}_i$ on the same side. Similarly, for the right side.\\

Such a relation turns out also to automatically verify the so called Abgrall criterion: phases under uniform mechanical conditions (i.e. moving with a unique velocity and a unique pressure), will evolve preserving the same conditions.\\
In order to show such result, we first make a fundamental assumption on the Riemann Solver under use.
As a numerical procedure, scheme (\ref{semi-discrete_form}) makes use of a Riemann Solver \cite{ToroRS}, whose employment has become routine in fluid-dynamics simulations.
Hereby, we assume that the solver under use to concretely build scheme (\ref{semi-discrete_form}) admits the following decomposition
\begin{subequations}
\label{eq:RS-decomposition}
\begin{align}
& \textbf{F}\left(\textbf{V}, \textbf{W}\right) = u\left(\textbf{V}, \textbf{W}\right) \textbf{U}\left(\textbf{V}, \textbf{W}\right) + p\left(\textbf{V}, \textbf{W}\right)\textbf{D}\left(\textbf{V}, \textbf{W}\right), \qquad \textbf{D}\left(\textbf{V}, \textbf{W}\right) = \begin{bmatrix}
0\\
1\\
u\left(\textbf{V}, \textbf{W}\right)
\end{bmatrix}\\
&
\textit{ and }
\qquad
u\left(
\begin{bmatrix}
\rho_L\\
u\\
p_L
\end{bmatrix}
,
\begin{bmatrix}
\rho_L\\
u\\
p_R
\end{bmatrix}
\right)
=u,
\qquad
p\left(
\begin{bmatrix}
\rho_L\\
u_L\\
p
\end{bmatrix}
,
\begin{bmatrix}
\rho_L\\
u_R\\
p
\end{bmatrix}
\right)
=p
\end{align}
\end{subequations}
where $u\left(\textbf{V}, \textbf{W}\right)$, $\textbf{U}\left(\textbf{V}, \textbf{W}\right)$ and $p\left(\textbf{V}, \textbf{W}\right)$ denote the velocity, the vector of unknowns and the pressure, respectively, resulting from the resolution of the RP between the initial states $\textbf{V}$ and $\textbf{W}$ evaluated at the same sampling point where the flux is. 
The decomposition (\ref{eq:RS-decomposition}) is a well-known property of the physical flux, as well as of several RS of common use, see \cite{ToroRS}.\\

Under the assumption (\ref{eq:RS-decomposition}), 
and the uniform conditions 
\begin{equation}
\label{eq:uniform_conditions}
u^{(k)}_i = u>0 
\qquad
p^{(k)}_i = p
\end{equation}
one gets that
\[
\tilde{\textbf{F}}^{lag}
\left(\textbf{V},\textbf{W}\right) 
= 
\textbf{F}\left(\textbf{V},\textbf{W}\right) 
- 
\sigma\left(\textbf{V},\textbf{W}\right)
\textbf{U}\left(\textbf{V},\textbf{W}\right)
 =
 \begin{bmatrix}
 -u\\
  p\textbf{D}
\end{bmatrix}  
,
 \qquad
 \textbf{D} = 
 \begin{bmatrix}
 0\\
 1\\
 u
 \end{bmatrix}
\]
so that
\begin{align*}
\tilde{\textbf{F}}^{lag} \left(\tilde{\textbf{U}}^{(l)}_i,\tilde{\textbf{U}}^{(k)}_{i}\right)
-
\tilde{\textbf{F}}^{lag}\left(\tilde{\textbf{U}}^{(k)}_i,\tilde{\textbf{U}}^{(l)}_i\right) = 
\begin{bmatrix}
 -u\\
  p\textbf{D}
\end{bmatrix}
-
\begin{bmatrix}
 -u\\
  p\textbf{D}
\end{bmatrix}
=
\textbf{0}
\end{align*}
and the relaxation terms vanishes. 
Moreover the Lagrangian fluxes coming from the boundary then reads
\begin{align*}
\mathbb{E}_{boundary}\left[\tilde{\textbf{F}}^{lag}\right]_i 
&= 
\mathbb{P}_{i-\frac{1}{2}}\left[\Sigma_l,\Sigma_k\right]
\tilde{\textbf{F}}^{lag}
\left(
\tilde{\textbf{U}}^{(l)}_{i-1},\tilde{\textbf{U}}^{(k)}_{i}
\right) 
- 
\mathbb{P}_{i-\frac{1}{2}}\left[\Sigma_k,\Sigma_l\right]
\tilde{\textbf{F}}^{lag}
\left(
\tilde{\textbf{U}}^{(k)}_{i-1},\tilde{\textbf{U}}^{(l)}_{i}
\right)\\
&=
\left(
\mathbb{P}_{i-\frac{1}{2}}\left[\Sigma_l,\Sigma_k\right] - \mathbb{P}_{i-\frac{1}{2}}\left[\Sigma_k,\Sigma_l\right]
\right)
\begin{bmatrix}
-u\\
p\textbf{D}
\end{bmatrix}
\end{align*}
The $i+\frac{1}{2}$-th contribution of the flux instead reads
\begin{align*}
\mathbb{E}_{i+\frac{1}{2}}\left[X^{(k)}\tilde{\textbf{F}}\right]
& 
:= 
\mathbb{P}_{i+\frac{1}{2}}\left[\Sigma_k,\Sigma_k\right]
\tilde{\textbf{F}}\left(\tilde{\textbf{U}}^{(k)}_{i}, \tilde{\textbf{U}}^{(k)}_{i+1}\right) 
+ 
\mathbb{P}_{i+\frac{1}{2}}\left[\Sigma_k,\Sigma_l\right] \tilde{\textbf{F}}\left(\tilde{\textbf{U}}^{(k)}_{i},\tilde{\textbf{U}}^{(l)}_{i+1}\right)\\
&=
u
\begin{bmatrix}
0\\
\mathbb{P}_{i+\frac{1}{2}}\left[\Sigma_k,\Sigma_k\right]
\textbf{U}\left(\textbf{U}^{(k)}_{i}, 
\textbf{U}^{(k)}_{i+1}\right) 
+ 
\mathbb{P}_{i+\frac{1}{2}}\left[\Sigma_k,\Sigma_l\right] 
\textbf{U}\left(\textbf{U}^{(k)}_{i},
\textbf{U}^{(l)}_{i+1}\right)
\end{bmatrix}
\\
&
\qquad
+ 
\left(
\mathbb{P}_{i+\frac{1}{2}}\left[\Sigma_k,\Sigma_k\right] 
+
\mathbb{P}_{i+\frac{1}{2}}\left[\Sigma_k,\Sigma_l\right]
\right)
\begin{bmatrix}
0\\
p\textbf{D}
\end{bmatrix}
\end{align*}
Applying the definition of various terms in the latter system,
one gets the following update formulas 
at time $t = t^{n+1}=t^n+\Delta t$
\begin{align*}
\alpha^{(k),n+1}_i &= \alpha^{(k),n}_i - u \frac{\Delta t}{\Delta x}
\left(
\mathbb{P}_{i-\frac{1}{2}}\left[\Sigma_l,\Sigma_k\right] - \mathbb{P}_{i-\frac{1}{2}}\left[\Sigma_k,\Sigma_l\right]
\right)\\
\left(
\alpha^{(k)} 
\rho^{(k)}
\right)^{n+1}_i
&= \left(
\alpha^{(k)} 
\rho^{(k)}
\right)^{n}_i 
- u \frac{\Delta t}{\Delta x}\Delta \rho_i\\
\left(
\alpha^{(k)} 
\rho^{(k)}u^{(k)}
\right)^{n+1}_i
&= \left(
\alpha^{(k)} 
\rho^{(k)}
\right)^{n}_i 
u 
- u^2 \frac{\Delta t}{\Delta x}\Delta \rho_i\\
&
\qquad
+
p\frac{\Delta t}{\Delta x}
\left(
\mathbb{P}_{i+\frac{1}{2}}\left[\Sigma_k,\Sigma_k\right] 
+ \mathbb{P}_{i+\frac{1}{2}}\left[\Sigma_k,\Sigma_l\right]
- \mathbb{P}_{i-\frac{1}{2}}\left[\Sigma_k,\Sigma_k\right] 
- \mathbb{P}_{i-\frac{1}{2}}\left[\Sigma_l,\Sigma_k\right]
\right)
\end{align*}
where 
\begin{align*}
\Delta\rho_i 
&:=
\mathbb{P}_{i+\frac{1}{2}}\left[\Sigma_k,\Sigma_k\right]
\rho\left(\textbf{U}^{(k)}_{i}, 
\textbf{U}^{(k)}_{i+1}\right) 
+ 
\mathbb{P}_{i+\frac{1}{2}}\left[\Sigma_k,\Sigma_l\right] 
\rho\left(\textbf{U}^{(k)}_{i},
\textbf{U}^{(l)}_{i+1}\right)\\
&-
\mathbb{P}_{i-\frac{1}{2}}\left[\Sigma_k,\Sigma_k\right]
\rho\left(\textbf{U}^{(k)}_{i-1}, 
\textbf{U}^{(k)}_{i}\right) 
- 
\mathbb{P}_{i-\frac{1}{2}}\left[\Sigma_k,\Sigma_l\right] 
\rho\left(\textbf{U}^{(k)}_{i-1},
\textbf{U}^{(l)}_{i}\right)
\end{align*}

Such an update scheme yields the following proposition.
\begin{proposition}
[Necessary conditions for Abgrall's criterion fulfillment]
\label{prop:uniform}
The scheme (\ref{semi-discrete_form}) with Riemann Solver verifying (\ref{eq:RS-decomposition}) and Forward Euler time-stepping under uniform conditions (\ref{eq:uniform_conditions}) preserves the same conditions only if
the probability coefficients verify
\begin{equation}
\label{eq:necessary_proba}
\mathbb{P}_{i+\frac{1}{2}}\left[\Sigma_k,\Sigma_k\right] 
+ \mathbb{P}_{i+\frac{1}{2}}\left[\Sigma_k,\Sigma_l\right]
= 
\mathbb{P}_{i-\frac{1}{2}}\left[\Sigma_k,\Sigma_k\right] 
+ \mathbb{P}_{i-\frac{1}{2}}\left[\Sigma_l,\Sigma_k\right]
\end{equation}
\end{proposition}

\begin{proof}
The thesis follows by previous computations and the fact that
the value of the velocity associated to phase $k$ can be computed according to
\begin{align*}
\left(
u^{(k)}
\right)^{n+1}_i
&
=
\frac{
\left(
\alpha^{(k)} 
\rho^{(k)}u^{(k)}
\right)^{n+1}_i
}{
\left(
\alpha^{(k)} 
\rho^{(k)}
\right)^{n+1}_i
}
= u + 
p\frac{\Delta t}{\Delta x}
\frac{
\mathbb{P}_{i+\frac{1}{2}}\left[\Sigma_k,\Sigma_k\right] 
+ \mathbb{P}_{i+\frac{1}{2}}\left[\Sigma_k,\Sigma_l\right]
- \mathbb{P}_{i-\frac{1}{2}}\left[\Sigma_k,\Sigma_k\right] 
- \mathbb{P}_{i-\frac{1}{2}}\left[\Sigma_l,\Sigma_k\right]
}{
\left(
\alpha^{(k)} 
\rho^{(k)}
\right)^{n+1}_i
}
\end{align*} 
\end{proof}

\begin{remark}
This means that, under uniform conditions, the probability of finding phase $k$ on the right of the volume $i$ should be the same to the one of finding it on its left. 
One could also derive necessary conditions for the probability coefficients by computing the update formula for the pressure, which would require the knowledge of the specific form of the internal energy, though. 
For the point we would like to make, equation (\ref{eq:necessary_proba}) suffices: relations (\ref{ConsistencyConds_ansatz}) are \emph{one} of the infinite relations linking the probability coefficients, which render (\ref{eq:necessary_proba}) trivially true.
Alternatively, choosing the probability coefficients to verify (\ref{ConsistencyConds_ansatz}) ensure that the Abgrall criterion holds true, for many RS of common use (i.e. those verifying \ref{eq:RS-decomposition}).
\end{remark}

\bibliographystyle{plain}
\bibliography{./biblio}

\end{document}